\documentclass[11pt]{amsart}
\usepackage{amssymb,mathrsfs,graphicx,enumerate,color}
\usepackage{colortbl}
\definecolor{black}{rgb}{0.0, 0.0, 0.0}
\definecolor{red}{rgb}{1.0, 0.5, 0.5}

\title[From Navier-Stokes to BV solutions of barotropic Euler]{From Navier-Stokes to BV solutions of the barotropic Euler equations}

\author[Chen]{Geng Chen}
\address[Geng Chen]{ Department of mathematics, \newline
University of Kansas, Lawrence, KS, 66045, USA.}
\email{gengchen.edu}

\author[Kang]{Moon-Jin Kang}
\address[Moon-Jin Kang]
{ Department of Mathematical Sciences, \newline
Korea Advanced Institute of
Science and Technology \\
 Daejeon 34141, Korea}
 \email{moonjinkang@kaist.ac.kr}

\author[Vasseur]{Alexis F. Vasseur}
\address[Alexis F. Vasseur]{\newline Department of Mathematics, \newline The University of Texas at Austin, Austin, TX 78712, USA}
\email{vasseur@math.utexas.edu}

\newtheorem{theorem}{Theorem}[section]
\newtheorem{lemma}{Lemma}[section]

\newtheorem{proposition}{Proposition}[section]
\newtheorem{remark}{Remark}[section]

\newcommand{\bbr}{\mathbb R}

\numberwithin{figure}{section}

\newcommand{\beq}{\begin{equation}}
\newcommand{\eeq}{\end{equation}}
\newcommand{\bsp}{\begin{split}}
\newcommand{\esp}{\end{split}}

\newcommand{\Dt}{{\partial_t}}

\newcommand{\lam}{\lambda}

\newcommand{\RR}{{\mathbb R}}

\def\eps{\varepsilon }

\newcommand\adots{\mathinner{\mkern2mu\raise1pt\hbox{.}
\mkern3mu\raise4pt\hbox{.}\mkern1mu\raise7pt\hbox{.}}}

\newtheorem{theo}{Theorem}[section]

\newtheorem{lem}[theo]{Lemma}

\def\charf {\mbox{{\text 1}\kern-.30em {\text l}}}

\def\lam{\lambda}  

\newcommand{\Div}{{\mathrm{div \!\! \ }}}

\newcommand{\deltao} {\delta_1}

\newcommand{\ds} {\delta_*}

\newcommand{\s}{\sigma}

\newcommand{\rn}{{\rho_\nu}}

\newcommand{\uv}{\underline{v}}
\newcommand{\uu}{\underline{U}}

\newcommand{\weakto}{\rightharpoonup}

\newcommand{\ba}{\overset{\raisebox{0pt}[0pt][0pt]{\text{\raisebox{-.5ex}{\scriptsize$\leftharpoonup$}}}}}
\newcommand{\fa}{\overset{\raisebox{0pt}[0pt][0pt]{\text{\raisebox{-.5ex}{\scriptsize$\rightharpoonup$}}}}}

\begin{document}
\bibliographystyle{plain}

\date{\today}

 \keywords{}

\thanks{\textbf{Acknowledgment.}  G. Chen is partially supported by the NSF grants 
DMS-2306258 and DMS-2008504. M.-J. Kang was partially supported by the National Research Foundation of Korea  (NRF-2019R1C1C1009355).
A. F. Vasseur is partially supported by the NSF grants 2306852 and 2219434.
}

\begin{abstract}
In the realm of mathematical fluid dynamics, a formidable challenge lies in establishing  inviscid limits from the Navier-Stokes equations to the Euler equations, wherein physically admissible solutions can be discerned. The pursuit of solving this intricate problem, particularly concerning singular solutions, persists in both compressible and incompressible scenarios.

This article focuses on small $BV$ solutions to the barotropic Euler equation in one spatial dimension. Our investigation demonstrates that these solutions are inviscid limits for solutions to the associated compressible Navier-Stokes equation. Moreover, we extend our findings by establishing the well-posedness of such solutions within the broader class of inviscid limits of Navier-Stokes equations with locally bounded energy initial values.

\end{abstract}
\maketitle \centerline{\date}

\tableofcontents

\section{Introduction}
\setcounter{equation}{0}

The Euler equation is the first partial differential equation ever written down in fluid dynamics more than 250 years ago \cite{Euler}. It can be written as:
\begin{align}
\begin{aligned}\label{eq:Euler}
\left\{ \begin{array}{ll}
\partial_t \rho+\Div (\rho u)=0, \qquad t>0, \ x\in \RR^N, \\[0.3cm]
\partial_t (\rho u)+\Div(\rho u\otimes u)+\nabla p=0, \qquad t>0, \ x\in \RR^N,\\[0.3cm]
(\rho(0,x),\rho u(0,x))=(\rho^0,\rho^0 u^0) \qquad  x\in \RR^N,
 \end{array} \right.
\end{aligned}
\end{align}
where the unknown quantities are the density of the gas $\rho:\RR^+\times\RR^N\to \RR$ and  its velocity $u:\RR^+\times\RR^N\to \RR^N$. The barotropic case corresponds to pressure functionals of the form:
$$
p(\rho)=\rho^\gamma,
$$
for adiabatic coefficients $\gamma>1$. This system of equations is well-posed on short times  in the functional space of Lipschitz functions (see, for instance, \cite{Dafermos2}). However, it is well-known that  such solutions can blow up in finite time \cite{Sideris}. When considering possible discontinuous solutions, the Euler system has to be endowed with the entropy inequality:
\begin{equation}\label{eq:Entropy}
\Dt \eta(\rho,\rho u)+\Div(q(\rho,\rho u))\leq0, \qquad t>0, \ x\in \RR^N,
\end{equation}
where the entropy of the system (in this case, the physical energy) and its flux are given as:
\begin{align}
\begin{aligned}\label{defentpair}
&\eta(\rho,\rho u)=\rho \frac{|u|^2}{2}+\frac{\rho^\gamma}{\gamma-1},\\
&q(\rho,\rho u)=(\eta(\rho,\rho u) +\rho^\gamma)u.
\end{aligned}
\end{align}
Solutions of \eqref{eq:Euler} verifying \eqref{eq:Entropy} are called entropic solutions. It is easy to show that Lipschitz solutions to \eqref{eq:Euler} verifiy \eqref{eq:Entropy} as an equality.
However, for $N>1$, this condition \eqref{eq:Entropy} is not enough to ensure the uniqueness of non-Lipschitz physical solutions. Using convex integration techniques, Chiodaroli, De Lellis, and Kreml showed in \cite{CDK} that planar shock discontinuities are not unique in the class of bounded entropic solutions verifying (\eqref{eq:Euler} \eqref{eq:Entropy}). De Lellis and Kwon showed in the multi-D periodic case that the non-uniqueness anomaly holds even in the class of continuous weak entropic solutions \cite{DeLellisKwon}. 
\vskip0.3cm
Another classical attempt to select admissible solutions is through the vanishing viscosity criterion (see \cite{Dafermos2}). For $\nu>0$ fixed, consider the Navier-Stokes equation:
\begin{align}
\begin{aligned}\label{eq:Navier-Stokes}
\left\{ \begin{array}{ll}
\partial_t \rho^\nu+\Div (\rho^\nu u^\nu)=0, \qquad t>0, \ x\in \RR^N, \\[0.3cm]
\partial_t \rho^\nu u^\nu+\Div(\rho^\nu u^\nu\otimes u^\nu)+\nabla p(\rho^\nu) =\nu \Div(\bar\mu(\rho^\nu) Du^\nu), \qquad t>0, \ x\in \RR^N,\\[0.3cm]
(\rho^\nu(0,x),\rho^\nu u^\nu(0,x))=(\rho_0^\nu,\rho_0^\nu u_0^\nu) \qquad  x\in \RR^N,
 \end{array} \right.
\end{aligned}
\end{align}
where $Du=(\nabla u+\nabla^t u)/2$ and $\bar\mu$ is the viscosity functional which may depend on $\rho$. We say that $U=(\rho,\rho u)$ is a vanishing viscosity solution to the Euler equation \eqref{eq:Euler} if there exists a sequence of initial value $U^\nu_0=(\rho_0^\nu, \rho_0^\nu u_0^\nu)$ converging to $U^0=(\rho^0,\rho^0u^0)$ such that, for the double limit (on the initial value and the equation) $\nu\to 0$, the solution to \eqref{eq:Navier-Stokes} converges, up to a subsequence, to $U$.
In most case, solutions to  \eqref{eq:Navier-Stokes} can be constructed verifying the energy inequality (\cite{Lions2,FNP,BJ}):
\begin{equation}\label{eq:EntropyNavier-Stokes}
\Dt \eta(\rho^\nu,\rho^\nu u^\nu)+\Div(q(\rho^\nu,\rho^\nu u^\nu))+\nu\bar\mu(\rho^\nu) |Du^\nu|^2\leq0, \qquad t>0, \ x\in \RR^N.
\end{equation}
Especially this justifies the notion of entropic solutions \eqref{eq:Entropy} at the limit $\nu=0$. 
\vskip0.3cm
In 1979, Dafermos and DiPerna \cite{Dafermos4, DiPerna} introduced the Weak/Strong principle. They showed that Lipschitz solutions (strong solutions) are unique and $L^2$-stable (for $L^2$ perturbations on the initial values) in the larger class of bounded entropic solutions (weak solutions) of \eqref{eq:Euler} \eqref{eq:Entropy}. This Weak/Strong principle, based on the relative entropy method,  can be easily extended to the class of vanishing viscosity solutions as weak solutions. It shows that Lipschitz solutions to \eqref{eq:Euler} are stable with respect to initial values perturbation and viscous perturbation on the system itself. Especially it shows that Lipschitz solutions are vanishing viscosity solutions to Euler \eqref{eq:Euler}. 
\vskip0.3cm 
The situation is completely different when considering discontinuous patterns instead of Lipschitz solutions. For $N>1$, it is not even known if the notion of vanishing viscosity limit differs from the notion of entropic solutions. Note that in the simpler incompressible case, it was shown in \cite{BVNavier} that the non-uniqueness pathology exists at the level of the Navier-Stokes equation and persists via viscosity limit. However, the solutions constructed at the Navier-Stokes equation are not Leray-Hopf and do not verify the Energy equality related to \eqref{eq:EntropyNavier-Stokes}. 
\vskip0.3cm
Surprisingly, the vanishing viscosity  limit problem is already extremely difficult for the mono-dimensional case, even in the well-studied framework of small BV solutions for  hyperbolic conservation laws. In this framework, Glimm showed in  \cite{Glimm} the existence of global solutions to $n\times n$ conservation laws in one dimension (including \eqref{eq:Euler}-\eqref{eq:Entropy} in dimension one) for any initial value small enough in $BV(\RR)$. Later, the small BV existence was proved using the front tracking scheme \cite{Bressan,Dafermos2,Holden_Book}. The $L^1$ stability of the front tracking scheme was proved in  Bressan, Liu and Yang \cite{Bressan2} and Bressan, Crasta and Piccoli \cite{Bressan1} see also Liu and Yang \cite{Liu-Yang} and Bressan and Colombo \cite{Bressan3} in the context of $2\times 2$ systems as \eqref{eq:Euler}.  The uniqueness of these solutions was established by Bressan and Goatin under the Tame Oscillation Condition \cite{UniquenessGoatin}. It improved the previous theorem by Bressan and LeFloch \cite{UniquenessLefloch}. Uniqueness was also known to prevail when the Tame Oscillation Condition is replaced by the Bounded Variation Condition along space-like curves, see Bressan and Lewicka  \cite{UniquenessLewicka}.  These technical extra conditions were first removed in the case of $2\times 2$ systems in \cite{CKV20} where the uniqueness result was extended to a weak/BV principle for a family of weak entropic solutions with strong traces. Note that the technical conditions have been recently removed in the general case, but without the weak/strong principle in \cite{bressan2023unique, bressan2023remark}. In this general context, Bianchini and Bressan were able to show that these solutions are vanishing viscosity limits for artificial viscosities in the celebrated paper \cite{BB} (see also \cite{tongest, Feiminest} for convergence estimates).  However, up to now, the result could not be extended to physical systems as the Navier-Stokes equation. The main difficulty is to obtain $BV$ estimates uniform with respect to the viscosity (see \cite{LiuYu} for estimates at viscosity fixed).
\vskip0.3cm
In 2010, using the compensated compactness method \cite{Tartar},  Chen and Perepelitsa showed in \cite{CP} the existence of vanishing viscosity solutions to Euler \eqref{eq:Euler} in one-space dimension with $\bar\mu=1$. This convergence was obtained for merely energy bounded initial values. However, in the case where  initial values are small in $BV$, it was not known until now whether the vanishing viscosity solutions were small $BV$ and whether they were  the unique solutions in the class of inviscid limits from Navier-Stokes. 
\vskip0.3cm
In this paper, we are showing the following two fundamental statements about vanishing viscosity limits and weak/BV stability in 1D:
\vskip0.3cm
{\it Any small BV solutions to \eqref{eq:Euler} is a vanishing viscosity solution, that is, can be constructed as double limit of solutions to the Navier-Stokes equation \eqref{eq:Navier-Stokes}.}

\vskip0.3cm
{\it Small BV solutions are stable via initial  perturbations among vanishing viscosity solutions.}
\vskip0.3cm
Note that the second statement does not need extra conditions as the strong traces property or a priori $L^\infty$ bound  as in \cite{CKV20}.
These two results are the later development of the $a$-contraction theory  up to shift first coined in \cite{KVARMA}.   It is based on the relative entropy method, and is a natural extension of the the work of Dafermos and DiPerna \cite{Dafermos4, DiPerna}.

\section{Precise statement of the results}
\setcounter{equation}{0}
From now on, we fix the dimension $N=1$, and so consider the system in the form
\begin{align}
\begin{aligned}\label{1NS}
\left\{ \begin{array}{ll}
       \rho^\nu_t + (\rho^\nu u^\nu)_x = 0\\
       (\rho^\nu u^\nu)_t +(\rho^\nu (u^\nu)^2)_x + p(\rho^\nu)_x = \nu \Big( \bar\mu(\rho^\nu) u^\nu_x\Big)_x . \end{array} \right.
\end{aligned}
\end{align}
 We will consider the viscosity functionals satisfying the following property:
\begin{align}
\left\{
\begin{aligned} \label{hyp:mu}
&\hspace{0.5cm} \mbox{Assume  $\gamma>1$, and that there exists $0<C_1<C_2$ and  $\alpha>0$ such that }\\
&\hspace{0.5cm}\mbox{$\gamma-1\leq \alpha\leq \gamma$ and }\displaystyle{C_1(1+\rho^\alpha)\leq \frac{\bar\mu(\rho)}{\gamma}\leq C_2(1+\rho^\alpha)} \quad\forall \rho>0.\\
&\hspace{0.5cm} \mbox{If $\gamma\leq 5/3$, we fix the value $\alpha=\gamma-1$.} 
\end{aligned}
\right.
\end{align} 

 This hypothesis allows very general viscosity functionals. It requires, however, a polynomial growth for large densities. Note that this growth is consistent with the case of the shallow water equation, which corresponds to  $\gamma=2$, with a growth of the viscosity functional as $\alpha=1$.
\vskip0.3cm
As discussed before, the theory is based on the relative entropy. 
For $U=(\rho,\rho u), \bar{U}=(\bar\rho, \bar\rho\bar u)$ with $\bar\rho>0$, we define the relative entropy of $U$ with respect to $\bar U$ as
$$
\eta(U|\bar U)=\rho\frac{|u-\bar u|^2}{2}+\frac{p(\rho|\bar\rho)}{\gamma-1}, \qquad  p(\rho|\bar\rho)=\rho^\gamma-\bar\rho^\gamma-\gamma\bar\rho^{\gamma-1}(\rho-\bar\rho)\geq0.
$$

For any fixed state $U_*=(\rho_*, \rho_*u_*)\in \RR^+\times\RR$ defining the limit state on the left and on the right, we consider the following set of Euler initial values:
\begin{equation}\label{initialEuler}
E^0:=\left \{  U^0=(\rho^0,\rho^0u^0)\in (L^1_{\mathrm{loc}} \cap L^\infty)(\RR) ~ \bigg| ~   \inf(\rho^0)>0, \ \  \int_\RR \left( \rho^0 |u^0 -u_*|+|\rho^0 - \rho_*| \right) \,dx< \infty           \right\}.
\end{equation}
Notice  that for any $U^0\in E^0$,
\beq\label{einiq}
 \int_\RR \eta(U^0(x)|U_*)\,dx< \infty.
\eeq

As in \cite{CP}, we consider slightly regularized initial values at the Navier-Stokes level. For any  Euler initial value \eqref{initialEuler}, we say that $U_0^\nu:=(\rho_0^\nu, \rho_0^\nu u_0^\nu)$ is an adapted family of initial values for the Navier-Stokes equation (associated with $U^0$) if  $U_0^\nu$ is smooth and satisfies: 
\begin{equation}\label{initialNS}
\begin{array}{l}
 \displaystyle{\lim_{\nu\to 0} \left(\int_\RR|U_0^\nu-U^0|\,dx+\int_\RR\eta(U_0^\nu|U^0) dx +\nu^2 \int_\RR|\partial_x\phi(\rho_0^\nu)|^2\,dx\right)=0, }
\end{array}
\end{equation}
where we define $\phi$ as $\phi'(\rho)=\frac{\bar\mu(\rho)}{\rho^{3/2}}$. We will prove the following lemma in Section \ref{sec:proof}.
\begin{lem}\label{lem_initialval}
For any $U^0=(\rho^0,\rho^0u^0)\in E^0$ defined in \eqref{initialEuler}, there exists an adapted family of initial values for the Navier-Stokes equation $U^\nu_0$, that is, verifying \eqref{initialNS}. 
\end{lem}
This lemma shows that for any Euler initial values in $E^0$, there exists an associated adapted family of initial values for the Navier-Stokes equation. 
\vskip0.3cm
 Our first result is as follows. 
 
 \begin{theorem}\label{th1}
 Assume \eqref{hyp:mu}.
 For any $U_*=(\rho_*, \rho_*u_*)\in \RR^+\times\RR$, there exists $\varepsilon>0$ such that the following is true. Let $U^0\in E^0$ with $\|U^0\|_{BV(\RR)}\leq \varepsilon$, 
 and $U_0^\nu$ be any adapted family of initial values satisfying \eqref{initialNS}. Consider the solution $U^\nu$ to the Navier-Stokes equations \eqref{1NS} with initial value $U_0^\nu$. Then, for some $q>1$, $U^\nu$ converges in $L^q_{\mathrm{loc}}(\RR^+\times\RR)$  to the small $BV$ solution to \eqref{eq:Euler} associated to the initial value $U^0$, as $\nu\rightarrow0$. 
  \end{theorem}

Note that thanks to the hypotheses \eqref{initialNS}  and  \eqref{initialEuler} on the initial value, and the condition \eqref{hyp:mu} on the viscosity,  the existence (and uniqueness) of the solution  to the Navier-Stokes system \eqref{eq:Navier-Stokes} follows from  \cite{MV-sima}. More precisely,    
 \cite[Theorem 2.1]{MV-sima} ensures that this solution verifies:
for any given $T>0$, there exist $\kappa^\nu_2(T)>\kappa^\nu_1(T)>0$ such that
\begin{equation}\label{hyp2}
\begin{array}{l}
\rho^\nu -\rho_* \in L^\infty(0,T;H^1(\bbr)), \quad \kappa^\nu_1(T) \le \rho^\nu(t,x) \le \kappa^\nu_2(T),\quad \forall x\in\bbr, t\le T, \\
u^\nu -u_* \in L^\infty(0,T;H^1(\bbr)) \cap  L^2(0,T;H^2(\bbr)).
\end{array}
\end{equation}
\vskip0.3cm
For any $U^0$ in \eqref{initialEuler}, and any $U^\nu$ solution to the Navier-Stokes equation \eqref{1NS}  with initial values verifying \eqref{initialNS}, the entropy inequality  \eqref{eq:EntropyNavier-Stokes} provides that 
\beq\label{invdef}
\limsup_{t>0, \nu>0}\int_{\RR}\eta(U^\nu(t,x)|U_*)\,dx\leq \int_\RR \eta(U^0(x)|U_*)\,dx.  
\eeq
Notice that by the definition of the entropy, $U^\nu$ is uniformly bounded in $L_{\mathrm{loc}}^q(\RR^+\times \RR)$ for $1<q<\min\{\gamma,2\}$. So, there exists $U\in L_{\mathrm{loc}}^q(\RR^+\times \RR)$ and a subsequence $\nu\to0$ such that $U^\nu$ converges at least weakly  in $L_{\mathrm{loc}}^q(\RR^+\times \RR)$ to $U$. We call any such $U$ an {\it inviscid limit} of Navier-Stokes associated to the initial value $U^0$. 
\vskip0.3cm
Our second theorem is the following.
\begin{theorem}\label{th2}
 Assume \eqref{hyp:mu}.
 For any $U_*=(\rho_*, \rho_*u_*)\in \RR^+\times\RR$, there exists $\varepsilon>0$ such that the following is true. Let $U^0\in E^0$ with $\|U^0\|_{BV(\RR)}\leq \varepsilon$, and $U^0_n$ a sequence of initial values in 
 $E^0$ such that 
 $$
\lim_{n\to\infty} \left(\int_\RR\eta(U^0_n(x)|U^0(x))\,dx+\int_\RR|U^0_n(x)-U^0(x)|\,dx\right)=0.
 $$
 Then, for any $U_n$ inviscid limit of Navier-Stokes associated to the initial value $U^0_n$, $U_n$ converges  strongly in $[L_{\mathrm{loc}}^q(\RR^+\times\RR)]^2$, as $n\to\infty$, to the small $BV$ solution $U$ associated to the initial value $U^0$.
\end{theorem}
Note that only $U$ is in $BV$. The perturbations $U_n$ does not need to be $BV$. Therefore, this theorem implies a Weak/BV principle (in the spirit of the weak/strong principle of Dafermos and DiPerna): If $U^0$ is small in $BV$, then any inviscid limit of Navier-Stokes associated to the initial value $U^0$ is indeed the $BV$ solution to Euler.  
\vskip0.3cm

Theorem \ref{th1} shows that, in dimension 1,  any small $BV$ solution to \eqref{eq:Euler} is a vanishing viscosity solution. Theorem \ref{th2} shows that these solutions are stable via initial perturbations among vanishing viscosity solutions.

\section{The $a$-contraction theory with shifts for Euler}
\setcounter{equation}{0}

We will focus in this section on the results of the $a$-contraction with shifts that can be applied to the  barotropic Euler and Navier-Stokes equation. Note, however, that the theory has already been developed in several directions for extremal shocks of general systems \cite{KVARMA}, stability of general solutions for $2\times2$ systems \cite{CKV20}, the general Riemann problem for the Full Euler system and its associated Navier-Stokes-Fourier system \cite{SV-16dcds,KVW3}. A multi-D stability result for contact discontinuities in this context was proved in \cite{KVW2}. The inviscid limits results Theorem \ref{th1} and Theorem \ref{th2} are based on previous work on weak/BV stability at the level of the Euler equation itself. However, notice that previous results are usually weaker than Theorem \ref{th2}. A priori uniform boundedness and strong traces properties are needed on the weak solutions in \cite{CKV20}, whereas these conditional assumptions can  now be dropped using the notion of vanishing viscosity solution. For a general review of recent results using the $a$-contraction theory, see \cite{VReview}.  We first describe the ideas working directly on the hyperbolic equation. We will then explain the extra difficulties encountered while working on the inviscid limits. 

\subsection{The relative entropy method}

For a fixed $U_*\in \RR^+\times \RR$, $U\to \eta(U|U_*)$ is itself an entropy associated to the flux of entropy $q(U;U_*)$, and it holds from \eqref{eq:Euler} and \eqref{eq:Entropy} that any entropic solution verifies in the sense of distribution:
\begin{equation}\label{relative}
\partial_t \eta(U|U_*)+\partial_x q(U;U_*)\leq 0.
\end{equation}
The strict convexity of $\eta$ showed that, as long as $U, U_*$ have values in a compact set $\mathcal{C}$ away from the vacuum, 
\begin{equation}\label{eq:L2}
\frac{1}{C} |U-U_*|^2\leq \eta(U|U_*)\leq C |U-U_*|^2,
\end{equation}
where the constant $C$ depends only on $\mathcal{C}$. Integrating the entropy inequality with respect to $x$ provides the $L^2$ stability (and contraction with respect to the relative entropy), of  bounded entropic weak solutions with respect to the constant state:
$$
\int_{\RR} |U(t,x)-U_*|^2\,dx\leq C^2\int_{\RR}|U^0(x)-U_*|^2\,dx.
$$
The relative entropy method introduced by DiPerna and Dafermos \cite{Dafermos4, DiPerna} consists in modulating the constant $U_*$ with respect to a Lipschitz solution $V$. Using  the finite speed of propagation, the method provides the uniqueness and  $L^2$ stability of the Lipschitz solutions in a wider class of bounded entropic weak solutions. The method is also very efficient for asymptotic analysis when the approximated models are entropic such as the Navier-Stokes equation,  and as long as the limit solution is Lipschitz. The idea is that Lipschitz solutions are stable with respect to $L^2$ perturbations on the initial values and with respect to entropic perturbations on the equation itself. The situation is more complicated when considering discontinuous pattern $V$. DiPerna already pushed the method in \cite{DiPerna} to obtain the uniqueness of shocks. The method was extended in \cite{CHY} to Riemann solutions. The aim to the $a$-contraction with shifts is to extend this method to a $L^2$ based theory for the stability and asymptotic analysis of general discontinuous solutions.  

\subsection{A Kruzkov-like theory for systems}  For scalar conservation laws with flux $U\to f(U)\in \RR$, Kruzkov developed a well-posedness theory for bounded solutions \cite{K1} based on the family of entropies:
$$
\eta_K(U|U_*)=|U-U_*|.
$$
Any weak entropic solutions $U$ verifies the entropy inequalities \eqref{relative} for the entropies $U\to \eta_K(\cdot,U_*)$, and the associated entropy flux functional $U\to q_K(U;U_*)=\mathrm{sgn}(U-U_*)(f(U)-f(U_*))$. By modulating the constant via the doubling variables method, he showed that  two such weak entropic solutions $U,V$ still verify:
$$
\partial_t \eta_K(U|V)+\partial_x q_K(U;V)\leq 0,
$$
in the sense of distribution. This provides a $L^1$ contraction for the solutions
\begin{equation}\label{contraction_scalar}
\frac{d}{dt}\int_{\RR}\eta_K(U(t,x)|V(t,x))\,dx\leq 0,
\end{equation}
from which the well-posedness follows easily.
This theory does not generally extend to systems because the set of convex entropies is more restrictive, and does not include the Kruzkov entropies. However, from one convex entropy, we can still derive the associated Kruzkov-like family of relative entropies verifying \eqref{relative}. In general, rarefactions preserve  the contraction property measured  by the relative entropy. However, this is not true for general solutions anymore. This is because the relative entropy method is basically based on the norm $L^2$, while the Kruzkov entropies are based on $L^1$. For shocks, because of the Rankine-Hugoniot condition, $L^2$ perturbations of order $\varepsilon$ can induce a shift that produces an error at time $t$ of order $\sqrt{t\varepsilon}$. Because of that, shocks cannot have a  $L^2$ contraction  property. However, it was shown in \cite{KVARMA}, that a similar property holds after weighting (via a weight $a>0$), and factoring out this shift. This is the basis of the $a$-contraction theory with shift. More precisely, it has been shown that for any shock $(U_L,U_R,\sigma_{LR})$, (that is, discontinuous wave $S(t,x)=U_L {\bf 1}_{\{x<\sigma_{LR} t\}}+U_R {\bf 1}_{\{x>\sigma_{LR}t\}}$ as an entropic solution to \eqref{eq:Euler}) there exists a weight $a>0$ such that the following is true. For any bounded weak entropic solution verifying the strong trace property, there exists a Lipschitz  function  $t\to h(t)$ (called shift)  such that for all $t>0$:
\begin{equation}\label{acontraction}
\frac{d}{dt}\left\{a\int_{-\infty}^{\sigma_{LR}t+h(t)}\eta(U(t,x)|U_L)\,dx+\int_{\sigma_{LR}t+h(t)}^{\infty}\eta(U(t,x)|U_R)\,dx\right\} \leq 0.
\end{equation}
For the scalar case (see  Leger \cite{Leger}), we can take $a=1$. We then have a contraction similar to the $L^1$ theory \eqref{contraction_scalar}, but for the $L^2$ based relative entropy, and up to a shift:
$$
\frac{d}{dt}\int_{\RR} \eta(U(t,x)|S(t,x+h(t)))\,dx\leq 0.
$$
However, for the Euler equation, the inequality is not true without the weight (see \cite{Serre-Vasseur}). 
Note that, at this stage, these inequalities are valid only for elementary waves: shocks and rarefactions.

\subsection{Weak/BV stability principle for Euler} Conservation laws have finite speed of propagation. Thanks to this principle, if the contraction property is true for elementary waves one can hope  to extend it to   Riemann problems, and from here to general small $BV$ solutions. Because of the shifts, it is however not clear how to extend the contraction property to general  solutions. Such a program was first done for the scalar case in \cite{Krupa-V}. In \cite{CKV20}, the method was applied to weak/BV principle for $2\times 2$ systems.  Instead of looking for a pseudo-distance between the weak solutions $U_n$ and the specific small $BV$ solution $V$, we use the relative entropy to control the $L^2$ distance from $U_n$ to the set of small $BV$ functions. More precisely, given $U_*\in \RR^+\times\RR$ let $\eps>0$ such that the small $BV$ solution $V$ is in 
$\mathcal{S}=\{U\in L^\infty(\RR) \ | \ \|U-U_*\|_{L^\infty(\RR)}+\|U\|_{BV(\RR)}\leq \eps \}$ (more precisely, $\eps$ is the maximum of such norm that can be obtained via the front tracking method). The method of $a$-contraction with shift was used to show that for all time $t>0$:
\begin{equation}\label{strong/weak-contraction}
\inf_{\{W\in \mathcal{S}\}}\int_{\RR} |U_n(t,x)-W(x)|^2\,dx\leq C\int_\RR |U_n^0(x)-V^0(x)|^2\,dx.
\end{equation}
We can match this result with the small $BV$  uniqueness theory of Bressan \cite{Bressan} to obtain the weak/BV stability result. Indeed, consider a sequence $U^0_n$ of initial values converging in $L^2$ to the initial value $V^0$ of the small BV solution $V$. The inequality \eqref{strong/weak-contraction} implies that the limit $U$ of $U_n$ is in $\mathcal{S}$ for all time. Assuming that the $U_n$ are uniformly bounded, we can show that the convergence is strong and so $U$ is a small $BV$ solution to Euler with initial value $V^0$. By the uniqueness theorem in this class of solutions, we have $U=V$. Note that, for this program, it is not needed for the sequence $U_n$ to be with values in $\mathcal{S}$.
\vskip0.3cm

To obtain the control \eqref{strong/weak-contraction}, the idea is to construct, via a modified front tracking method, a function $\bar{U}_\delta$ keeping values in $\mathcal{S}$ for all time $t>0$, and a weight function $a$ such that for all $t,x$, $1/2\leq a(t,x)\leq 2$, and such that 
\begin{equation}\label{deltacontraction}
\int_{\RR} a(t,x) \eta(U_n(t,x)|\bar{U}_\delta(t,x))\,dx
\end{equation}
is non-increasing in time (up to error terms of order $\delta$, the small parameter associated to the front tracking approximation). The front tracking method is a numerical approximation method which involves solely the evolution of fronts. It proves itself to be a powerful  tool to study the well-posedeness of 1-D conservation laws (see for instance \cite{Bressan}). The approximated solutions are made piecewise constant by an ad-hoc treatment of the rarefactions. To fit the relative entropy framework, the front tracking method is modified by introducing the shifts on each front in order to keep the decrease of the quantity via the $a$-contraction theory. The piecewise constant weight function $a$ has to be recomputed after the interaction of two fronts. Because the front does not verify the Rankine-Hugoniot condition (due to the artificial shifts), the function $\bar{U}_\delta$ may not be  a meaningful approximation of the Euler equation anymore. However, the general theory of the front tracking method based on the evolution of the interaction potential is still valid, and insures that $\bar{U}_\delta$ stays in $\mathcal{S}$ for all time $t$.  Note also that the function $\bar{U}_\delta$ actually depends on the wild weak solutions $U_n$ via the shifts. The function $\bar U_\delta$ can be seen as a $\delta$-approximation of a sort of $L^2$ projection of $U_n$ onto the set $\mathcal{S}$ of small $BV$ functions.

\subsection{Towards the inviscid limit} The Navier-Stokes equation is compatible with the entropy, and so with the relative entropy method. It is then natural to run the same methodology in order to obtain, up to error terms of order $\nu$ and $\delta$, a similar control of  \eqref{deltacontraction} for the solution $U^\nu$ to the Navier-Stokes equation. However, the task is extremely difficult. This is due to three main factors: 1) Navier-Stokes has a destabilization effect on shocks, with the formation of viscous layers (viscous shocks). 2) The Navier-Stokes equation is not local anymore: there are interactions between the simple waves. 3) There is no uniform $L^\infty$ bound on the solutions with respect to the viscosity. This last condition is due to the fact that only one entropy  (the physical energy) is compatible with the Navier-Stokes equation. Therefore, the invariant regions known for the Euler equation which imply $L^\infty$ bounds on the Euler solutions, are not preserved at the level of the Navier-Stokes equation. This is a major difference with the artificial viscosity situation \cite{BB}. As in the work \cite{CP}, we have to deal with  possible points of high velocity or high density.

\section{Main idea of the proof}
\setcounter{equation}{0}

Following \cite{Kang-V-NS17,KV-Inven}, the relative entropy method is performed on the Navier-Stokes equation in its Lagrangian form. The mass Lagrangian variables are the specific volume $v=1/\rho$ and the velocity $u$ in the mass variables:
$y(t,x)=\int_0^x\rho(t,z)\,dz$. The Navier-Stokes system in the $(t,y)$ variables has the form (where we keep the variable $x$ for $y$ without confusion):
\begin{align}
\begin{aligned}\label{NS-1}
\left\{ \begin{array}{ll}
       v^\nu_t - u^\nu_x = 0\\
   u^\nu_t+p_L(v^\nu)_x =\nu \Big(  \frac{\bar\mu_L(v^\nu)}{v^\nu} u^\nu_x\Big)_{x} , \end{array} \right.
\end{aligned}
\end{align}
where $p_L(v)=v^{-\gamma}, \gamma>1$, $\bar\mu_L(v)=\bar\mu(1/v)$ (the subscript $L$ is for Lagrange). The  initial data $(v_0^\nu,u_0^\nu)$ verifies
\[
\lim_{|x|\to\infty}(v_0^\nu,u_0^\nu)=(v_*,u_*) ,
\]
and so has infinite mass. Moreover it has no vacuum.  This is preserved by the Navier-Stokes equation, and therefore the system \eqref{NS-1} is equivalent to the system \eqref{eq:Navier-Stokes} in dimension one. The Lagrangian setting is better suited to study the stability of the shock layer which involves a subtle balance between the hyperbolic terms (from Euler) and the parabolic ones (from the viscosity). Note that we need a powerful stability result on these shock layers.
Shock layers are objects at the scale $\nu$. We need them, however, to be stable with respect to perturbations of order $\delta$, at the level of the front tracking approximation. This is a stability property {\it uniform} with respect to the viscosity (see \cite{Kang-V-NS17}). To obtain this result,   
 we consider  the  effective velocity  $h := u - \nu \frac{\bar \mu_L(v)}{v} v_x $ corresponding to the Bresch and Desjardins entropy \cite{BD-cmp03}. The equation on $(v^\nu,h^\nu)$ is then 
\begin{align}
\begin{aligned}\label{NS}
\left\{ \begin{array}{ll}
       v^\nu_t - h^\nu_x = -\nu \Big( \bar\mu_L(v^\nu)(v^\nu)^\gamma [p_L(v^\nu)]_x\Big)_{x}\\
       h^\nu_t+p_L(v^\nu)_x =0. \end{array} \right.
\end{aligned}
\end{align}

We actually run the relative entropy method on this system. In the Lagrange variables, the relative entropy is 
\begin{equation}\label{etaL}
 \eta^L(U|\bar U)=\eta^L((v,h)|(\bar v,\bar h)) :=\frac{|h- \bar h |^2}{2} + Q(v|\bar v),
\end{equation}
where $Q(v)=v^{-\gamma+1}/(\gamma-1)$.
Note that the nonlinear terms of the Euler equation is the pressure and is a function of $v$ only. The advantage of  System \eqref{NS} compared to \eqref{NS-1} is that 
the diffusion is in $v$ rather than $u$, giving control on the main nonlinear term. After a maximization in the variable $h$, the stability relies on a nonlinear Poincar\'e estimate in the $v$ variable (see Section \ref{sec:shock}. This derivation was first  derived in \cite{Kang-V-NS17}. We rely in this paper on the version proved in  \cite{KV-2shock}).
\vskip0.3cm
Because we do not have a uniform $L^\infty$ bound on $U^\nu=(v^\nu,h^\nu)$ with respect to the viscosity, the control of arbitrarily small or big values of $v^\nu$ is not trivial. This is the reason why the condition on the viscosity \eqref{hyp:mu} for big values of density is needed. Under this assumption, a viscous version of the result \eqref{acontraction} via the $a$-contraction with shift  can be obtained for the viscous layer (see \cite{KV-Inven}): Let $U_L,U_R, \sigma_{LR}$  be a shock for the Euler equation, and consider $\tilde{U}^\nu$ the associated viscous layer, that is the solution to system \eqref{NS} with end states $U_L, U_R$. Then there exists a weight function $x\to a(x/\nu)$ (with end values associated to the number $a>0$) such that the following is true. For any solution $U^\nu$ of the system \eqref{NS}, there exists a shift $t\to X(t)$ such that 
$$
\frac{d}{dt}\left\{\int_{-\infty}^{+\infty}a((x-X(t))/\nu)\eta^L(U^\nu(t,x)|\tilde{U}^\nu)\,dx\right\} \leq 0.
$$
Notice that the use of the effective velocity imposes the condition on the initial values \eqref{initialNS}.
\vskip0.3cm
The aim is now  to extend this stability result of \cite{KV-Inven}, valid for a single shock, to a stability result for a general solution, building up from the general weak/BV principle for Euler described above.

\subsection{Construction of approximate solutions} The function $U^\nu$ is an exact solution to Navier-Stokes equation \eqref{NS}. What we call here ``approximate solution"  can be seen as a sort of projection of $U^\nu$ on the space of functions with small $BV$ norms. The basic idea is to construct these ``approximate solutions" $\bar U_{\nu,\delta}$ as a combination of shifted simple waves $\tilde U^\nu$. As for the classical front tracking method, by construction, these simple waves can interact only  a finite amount of times. At each of these times,   the waves are recomputed at the Euler level  following the front tracking method, and then replaced by their viscous counterparts. This construction is described in Section \ref{sec:appsoln}. These simple waves are then shifted between times of interaction, following the a-contraction principle to keep the combined ``approximate solutions" $\bar U_{\nu,\delta}$ close (in the sense of relative entropy) to the exact solutions $U^\nu$ to Navier-Stokes.  The whole theory is based on the necessity to shift the shock waves in $\bar{U}_{\nu,\delta}$ because of the incompatibility of the Rankine-Hugoniot condition and the $L^2$ norm.  These shifts are defined in Subsection \ref{subsec-shock}. An additional difficulty is that we also need to shift the rarefaction fronts. (These shifts are defined in  Subsection \ref{subsec-rar}). The purpose of these shifts is  to control the  worst part of the errors due to the rarefaction discretization at the level of the  front tracking method. The difficulty is due to the lack of (uniform in $\nu$) $L^\infty$ bound on $U^\nu$. The rarefaction discretization may produce $L^1$ error that we mitigate with the construction of the shifts (see difficulty 1. below). 

Let us fix the main coefficients involved as follows. The global  $BV$ strength is $\eps>0$. The whole result is under the assumption that this number is small, but it will be fixed for the whole paper. The strength of viscosity $\nu>0$ is the smallest coefficients converging to 0. There are two coefficients associated to the front tracking method.  The  jump strength of the approximation rarefaction, denoted by $\delta$, converges to 0 the slowest. We keep the standard notation and denote the cut-off on the interaction strength used to choose the Riemann solver as $\rn$ (depending on $\nu$). It should not be confused with the density $\rho^\nu$. (The construction of the approximate solution using the coefficient $\rho_\nu$ is done in the Lagrange frame on $\bar{U}_{\nu,\delta}=(\bar{v}_{\nu,\delta}, \bar{h}_{\nu,\delta})$, reducing the chance of confusion with the density).
This coefficient $\rho_\nu$ is also used to define the smallest possible size of physical waves (together with the biggest possible size of the pseudo-shocks). This coefficient has to converges to zero together with $\nu$, but at a lower rate. We fix the dependence as $\rn=\nu^{1/3}$. The construction faces several severe difficulties:

\begin{itemize}
\item[1.] We do not have uniform $L^\infty$ bound with respect to the viscosity on the wild Navier-Stokes solution $U^\nu$. As a consequence, we lose the uniform finite speed of propagation on the approximation   $\bar{U}_{\nu,\delta}$ (because we do not have uniform bounds on the shifts). Another difficulty is that for all $t$ the relative entropy $\eta(U^\nu(t,x)|\tilde{U}^\nu)$ is bounded only in $L^1$. Many errors in the $a$-contraction estimates can only be  controlled by carefully keeping track of signs (see above the issue for rarefactions for instance). Especially, to control errors due to pseudo-shocks, it is crucial that the velocity of pseudo-shocks are increasing in the pseudo-shock layer.
\item[2.] Long range interactions at the level of Navier-Stokes. The superposition of viscous waves is not an exact solution to Navier-Stokes due to their non-linear interactions. It is therefore important to control the spread of these interactions. The transition zone of a shock wave of strength $|\sigma|$ has width $\nu/ |\sigma|$. It is actually large for small shocks. It is then imperative to have a sharp control on the size of the shock by below, not only on the interaction strength. This requires  new Riemann solvers  at the level of the Front tracking method. 
\end{itemize}

The standard front tracking algorithm (see \cite{Bressan}) involves two Riemann solvers. The first accurate solver only discretizes the rarefaction to ensure that solutions stay piece-wise constant. It uses a simplified solver when the product of the strength of incoming waves $|\sigma_1\sigma_2|$ is smaller than the threshold $\rn$. This is to avoid the production of infinitely many waves in finite time. The idea is to make the physical waves either cross each other (if they are of different families) or combine (if they are of the same family). The error produces the so-called pseudo-shocks, also named non-physical waves, that can be controlled. An important feature is to make all the pseudo-shocks travel at a same particular speed, to avoid any collision between them, and such that they interact with any physical wave at most once (counting the order of generation). The traditional way is to use a fixed speed faster than any physical wave. However, because of the loss of uniform control on the speed of propagation, we fix the speed 0 for the pseudo-shocks, using that, in the Lagrange frame, no physical wave travel at that speed  (even after shifts).  

In our context, we need three more crucial properties on the front tracking algorithm. We need to ensure that the strength of shocks is bounded below proportionally to $\rn$. We  need that  the velocity $h$ is increasing in $x$ in the pseudo-shocks. Finally we need a precise control on the total number of interactions as a power of $log (1/\nu)$. To do this, we introduce the three families of  Riemann solvers: we called them the accurate, simplified and adjusted solvers. Note that even the accurate and simplified solvers are modified from their traditional definitions. A major reason is that  our pseudo-shocks need to have some monotonicity property  ($h$ has to increase in $x$). 
We also need to consider   adjusted solvers for so-called  overtaking interactions between rarefaction and shock waves (interaction between the  same family of waves). The construction insures a $O(\rn)$ lower bound on outgoing shock(s) by slightly adjusting the outgoing waves.
Moreover,  we take advantage of these modifications and of  the barotropic regime to control the total number of interactions on a lifespan $T$ by a power of $log (1/\nu)$.

Finally, we need to resolve the  interactions between incoming waves before they get too close to each other, when a certain distance between them is reached. We slightly decrease  this distance after each interaction to avoid freezing everything together. It is denoted by $d_j$ and is of order $ \rn^{1/4}$ (see \eqref{defdj}). 

Ultimately, we can consider the  coefficient $\delta$ to converge (very slowly) to 0 as a function of $\nu$. For this reason, the approximation function $\bar U_{\nu,\delta}=\bar U_\nu$ can be denoted as a function of $\nu$ only.


\subsection{Construction of the weight function and $a$-contraction estimates}
In order to obtain an approximate $a$-contraction, we need to construct a weight function $a(t,x)$ associated to the simple shock waves of $\bar U_\nu$. Its construction is described in Subsection \ref{subsec-weight}. Similarly to the construction of $\bar U_\nu$ itself,  the construction of $a(x,t)$ is based on the definition as in \cite{CKV20} for the Euler system, and first defined as a discontinuous function. After each wave interaction, the weight coefficients are recomputed based on the simple shock waves present in $\bar U_\nu$.  Following \cite{Kang-V-NS17}, the weight functions are then smoothed out (spread) at the scale $\nu$ as for the simple viscous shock waves. Finally,  they are shifted along the shocks shifts between the times of wave interactions. 
\vskip0.3cm
With both the approximate solution and the weight function constructed, it remains to check that the construction provides a $a$-contraction up to the small quantities $\nu$ and $\delta$. This computation is described from Section \ref{sec:wrel} to Section \ref{sec:pfthm}. The basic relative entropy computation is performed in Section \ref{sec:wrel}. Section \ref{sec:rare} deals with the errors due to the approximations of the rarefaction waves of $\bar U_\nu$. Section \ref{sec:para} localizes the viscous dissipation functional, so it can be used for the $a$-contraction estimates associated to each simple wave of  $\bar U_\nu$. Section \ref{sec:shock} deals with the $a$-contraction computation associated to the simple shock waves of $\bar U_\nu$. Finally, Section \ref{sec:pfthm}  insures that there is no growth of relative entropy at the interaction times because of the rebuilding of the weights.
Especially, we show, in  Subsection \ref{subsec-time}, the decay of the weight function $a(x,t)$ at each time when some waves interaction happens, i.e. $$\Delta a(x,t)=a(x,t+)-a(x,t-)\leq 0$$ for any $(x,t)$. Because of the spread of the waves, this can be achieved only by 
repositioning the outgoing waves strategically. 
This is because we need to show $\Delta a(x,t)\leq 0$ at any $x$, including the region between two incoming and outgoing waves.
The decay of $a(x,t)$ on these regions of size $\nu$ is much more complicated. To make this happens, we need to adjust the position of outgoing waves, as in Figures \ref{Fig1}, \ref{Fig2} and \ref{pic_a_0}: each outgoing wave is slightly shifted  along its main direction. 

Among all cases of pairwise interactions, the worst case  scenario in showing the decay of  $a(x,t)$ is the overtaking interaction between a rarefaction wave and a shock wave. We take advantage that in this case, the total variation function $L(t)$ decays of more that twice  the strength of the smaller incoming wave. This ensures a net drop of the  non-increasing functional  $L(t)+\kappa Q(t)$ (where $Q(t)$ is the Glimm potential). Adding  $L(t)+\kappa Q(t)$ in the weight function $a(x,t)$ ensures the decrease of $a(x,t)$ in this scenario for the accurate solvers. Some subtle estimates, however,  are still  needed for the adjusted and simplified solvers.

\subsection{Uniform estimate from the $a$-contraction with shifts} We can now state our main technical theorem based on the $a$-contraction with shifts in the Lagrange frame. First, notice from \eqref{hyp:mu} that
\beq\label{assmu}
C_1(1+v^{-\alpha}) \le \mu_L(v) \le C_2 (1+v^{-\alpha}).
\eeq
Therefore, 
 we can decompose the viscous coefficient $\mu_L(v)$ as follows:
\beq\label{mudecom}
\mu_L(v) = \mu_1(v) + \mu_2(v),
\eeq
where $\mu_1$ and $\mu_2$ satisfy
\beq\label{mu1}
C^{-1} v^{-\alpha} \le \mu_1(v) \le C v^{-\alpha},\quad \forall v>0,
\eeq
and
\beq\label{mu2}
0\le \mu_2(v) = \mu_2(v) \mathbf{1}_{v>2v_*} (v) \le C.
\eeq
Indeed, to decompose as above, set
\[
\mu_2(v) := \big( \mu(v) - C_1 v^{-\alpha} \big) \phi(v),\quad \mu_1(v):= \mu(v) - \mu_2(v)
\]
where a cutoff function $\phi$ is Lipschitz such that $\phi(v)=1$ for $v\ge 3v_*$; $\phi(v)=0$ for $v\le 2v_*$. Then, the assumption \eqref{assmu} implies the desired estimate.\\
Our main technical theorem is then the following. Its proof will stretch  from Section \ref{sec:wrel} to Section \ref{sec:pfthm}.

\begin{theorem}\label{thm:uniform}
Assume that the first two lines of \eqref{hyp:mu} hold true.  Then, for any  $U_*:=(v_*,h_*)$ with $v_*>0$, there exists $ C,\eps_0>0,$ such that  the following holds true.
\vskip0.3cm
\noindent Let  $U^0=(v^0,h^0)\in L^\infty(\RR)$ such that   $\int_\RR \eta^L(U^0(x)|U_*)\,dx<\infty$, and $\|U^0\|_{BV(\RR)}\leq \eps$ for some $0<\eps\leq \eps_0$. 
 Consider  $U_0^\nu:=(v_0^\nu,h_0^\nu)$ with 
\beq\label{inihuv}
\lim_{\nu\to 0} \int_\RR \eta^L(U_0^\nu(x)|U^0(x))\,dx =0.
\eeq 
Denote  $U^\nu:=(v^\nu,h^\nu)$ the solution to \eqref{NS} with initial data $U_0^\nu$. 
\vskip0.3cm
\noindent Then, for any  $\nu$ small enough,
 there exists  a  function $\overline U_{\nu}=(\bar v_{\nu}, \bar h_{\nu})\in L^\infty(\RR^+\times\RR)$  such that 
\begin{align}\label{BVbarU}
\begin{aligned}
\sup_{t>0} \left(\|\overline U_\nu(\cdot,t)\|_{BV(\RR)}+\|\overline U_\nu(\cdot,t)-U_*\|_{L^\infty(\RR)}\right)\leq C\eps,
\end{aligned}
\end{align}
and  for all $T>0$,
\[
\lim_{\nu\to0}\sup_{t\in[0,T]}\int_{\bbr} \eta^L(U^\nu(x,t) | \overline U_{\nu} (x,t) ) dx=0.
\]
Moreover,  $(\overline U_{\nu})_x$ can be written as a finite sum $V^{PS}_{\nu}+\sum_{i=1}^{N_{\nu}} V_{i,\nu}$ with:
$$
\lim_{\nu\to 0}\|V^{PS}_{\nu}\|_{L^\infty(\RR^+\times\RR)}=0,
$$
and such that 
\begin{align*}
\begin{aligned}
G_\nu(t)&:= \nu  \int_{\bbr} \mu_1(v^\nu) (v^\nu)^{\gamma} |(p_L(v^\nu)-p_L(\bar v_{\nu}))_x|^2 dy + \nu \int_{\bbr} \mu_2(v^\nu) (v^\nu)^{\gamma} |\partial_x p_L(v^\nu) |^2 dy \\
&\quad +  \sum_{i}  \int_{\bbr} |V_{i,\nu}| p_L(v^\nu |\bar v_{\nu}) dx
\end{aligned}
\end{align*}
verifies for all $T>0$:
\begin{equation}\label{eq_G}
\lim_{\nu\to0}\int_0^T G_\nu(t) dt=0.
\end{equation}
\end{theorem}
In the decomposition of $(\overline U_{\nu})_x$, the $V_{i,\nu}$ correspond to the derivatives in $x$ of the physical waves of $\overline U_{\nu}$ (shocks and rarefactions), while $V^{PS}_{\nu}$ is the derivative in $x$ of the combination of all the pseudo-shocks (see Sections \ref{sub:conapp} and \ref{sec:pfthm}). The extra estimate \eqref{eq_G} will be useful for the proof of Proposition \ref{prop:cc}.
\vskip0.3cm
Theorem \ref{thm:uniform} shows that if the initial value $U^\nu_0$ converges to a small BV function, then, up to a subsequence,  the associated solution $U^\nu$ to  \eqref{NS} will also converge (possibly only weakly) to a small BV function. Indeed, the theorem shows that $U^\nu-\overline{U}_\nu$ converge (strongly) to 0. But because of the uniform bound \eqref{BVbarU}, up to a subsequence, $\overline{U}_\nu$ converges (possibly only weakly) to a function verifying the same bound.

 
\subsection{Stability for the system \eqref{1NS} in the Eulerian frame} We will recast Theorem \ref{thm:uniform} in the Eulerian frame to obtain the following proposition.
\begin{proposition}\label{main_prop}
Assume that the first two lines of \eqref{hyp:mu} hold true.  Then, for any  $U_*:=(\rho_*,\rho_*u_*)$ with $\rho_*>0$, there exists $C, \eps_0>0$ such that  the following holds true.
\vskip0.3cm\noindent
Let  $U^0=(\rho^0,\rho^0u^0)\in E^0$ with $\|U^0\|_{BV(\RR)}\leq \eps<\eps_0$, and let $U^\nu_0$ be an adapted family of initial values satisfying \eqref{initialNS}. Consider the solution $U^\nu$ to the Navier -Stokes equation \eqref{1NS} with initial value $U^\nu_0$. 
\vskip0.3cm\noindent 
Then, there exists a family of functions $\bar U_\nu$ such that 
\begin{equation}\label{BVest}
\sup_{t>0} \left(\|\bar U_\nu(\cdot,t)\|_{BV(\RR)}+\|\bar U_\nu(\cdot,t)-U_*\|_{L^\infty(\RR)}\right)\leq C\eps,
\end{equation}
and 
$(U^\nu-\bar U_\nu)$ converges weakly to 0 in the sense of distribution.

\end{proposition}

This proposition shows how the weighted relative entropy, coupled with the front tracking with shifts,  allows the construction of an approximate solution $\bar{U}_{\nu}$ which has the two properties: 
\begin{itemize}
\item Its BV norm is uniformly bounded in time and $\nu$.
\item $U^\nu-\bar{U}_{\nu}$ converges weakly to 0. 
\end{itemize}
Note that we lose the strong convergence because of the combination of the  BD entropy, and the lack of control on extremely small or large values of density.
Nevertheless, this ensures that, up to subsequences, any (possibly weak) limit $U_\infty$ of $U^\nu$ is uniformly of small $BV$. It remains to show that these limits are, actually,  entropic solutions to the Euler equation. This can be obtained via strong convergence. 
However, because we do not control well the shifts, we do not control the  possible time oscillations of the functions $\bar{U}_{\nu}$. This prevents to obtain strong convergence directly on the functions 
$\bar{U}_{\nu}$.
Instead, following Chen and Perepelitsa \cite{CP}, we use the compensated compactness on $U^\nu$ to show that the limit $U_\infty$ is an entropic solution to the Euler equation.

\subsection{Inviscid limit via compensated compactness} We will prove in Section \ref{sec:13} the following proposition. 

\begin{proposition}\label{prop:cc}
Assume that \eqref{hyp:mu} holds true. 
Let  $(\rho_*,\rho_* u_*)\in \RR^+\times \RR$ with $\rho_*>0$. Consider $U^0\in E^0$. Assume that $U_0^\nu=(\rho_0^\nu, \rho^\nu_0u_0^\nu)$ is an adapted family of initial values for Navier-Stokes, that is, verifies \eqref{initialNS}.
Then, up to a subsequence,
\[
\rho^\nu \to \rho,\quad \rho^\nu u^\nu \to \rho u \quad \mbox{in } L^q_{loc} (\bbr_+\times\bbr), \quad\mbox{for some } q>1,
\]
where $(\rho,u)$ is solution to  the Euler system \eqref{eq:Euler} and verifies the entropy inequality  \eqref{eq:Entropy} \eqref{defentpair}. 
\end{proposition}

The core of the proof is similar to  \cite{CP}. 
Because our viscosity functional is not constant, we are giving the details in Section \ref{sec:13}.  To pass  to the limit, it is important to obtain new a priori estimates. In Lemma \ref{lem:high1}, depending on the values of $\gamma, \alpha$, we obtain extra integrability on $\rho$. Our hypotheses on the viscosity functional requires to consider only the case $\alpha=\gamma-1$ for values $\gamma\leq 5/3$.
As in the case of  Chen and Perepelitsa in \cite{CP}, the most challenging part of the proof is to obtain the convergence on the flux of entropy and the decay of the entropy dissipation in Section \ref{sub:lim}.
For this part, our proof is different from the one of Chen and Perepelitsa in \cite{CP}. 
Instead, we  use crucially our  uniform bound \eqref{eq_G} obtained in Theorem \ref{thm:uniform}. This allows us to use an entropy pair $(\eta^M, q^M)$ corresponding to a truncated density function $\psi_M$, and to pass into the limit when $M$ goes to infinity (see Section \ref{sub:lim}). 

\subsection{Conclusion of the sketched proof}\label{sub:concl} We have shown that the  inviscid limits have small BV norm (thanks to Proposition \ref{main_prop}) and that they are entropic solutions to Euler (thanks to Proposition \ref{prop:cc}).
We can use the uniqueness theorem of Chen, Krupa and Vasseur \cite{CKV20} (see also Bressan and De Lellis \cite{bressan2023remark})  to ensure that this limit is the unique small BV solution to Euler with the corresponding initial value. From the uniqueness, the whole sequence converges. 
Note that it is important to use  uniqueness theorems for small BV solutions  without extra conditions as small BV along space-like curves as these conditions may not be verified  by $U_\infty$ due, again, to the wild shifts.
\vskip0.3cm
The rest of the paper is structured as follows. In Section \ref{sec:proof}, we first prove Lemma \ref{lem_initialval}.  Then we prove that the main theorems  \ref{th1} and \ref{th2} are consequences of Propositions \ref{main_prop} and  \ref{prop:cc}. Sections  \ref{sec:wrel} to Section \ref{sec:pfthm} are dedicated to the proof of the main $a$-contraction Theorem \ref{thm:uniform}.  We are working only in the Lagrangian frame in this part of the paper. To simplify the notations, the subscripts $L$ (which are used for Lagrangian, as $p_L, \eta^L, \mu_L$) will be dropped in these sections.  Proposition \ref{main_prop} is proved in Section \ref{sec:main}, and 
 Proposition \ref{prop:cc} is proved in   Section \ref{sec:13}. Both proofs use Theorem \ref{thm:uniform}. Finally, we provide a short introduction to wave interactions for the p-system in the appendix.



\section{Proof of the main Theorems} \label{sec:proof}

We prove in this section Theorems \ref{th1} and \ref{th2}, using Proposition \ref{main_prop}  and Proposition \ref{prop:cc}. We will follow the sketch described in Subsection \ref{sub:concl}. We begin with the construction of the adapted family of initial data.

\subsection{Construction for the adapted family of initial data} We first prove Lemma \ref{lem_initialval}.

For a given datum $U^0$ satisfying \eqref{initialEuler}, 
consider its mollification as follows. \\
Using $\psi_{\nu}(x):=\frac{1}{\sqrt\nu}\psi_1\big(\frac{x}{\sqrt\nu}\big)$ where $\psi_1$ is a smooth mollifier with supp $\psi_1=[-1,1]$, define
\[
\rho_0^{\nu}=\rho^0*\psi_{\nu},\quad u_0^{\nu}=u^0*\psi_{\nu}.
\]
Note that $\rho^0$ is bounded by above and by below, so $u^0$ is bounded above, $\rho^\nu_0$ is uniformly bounded (both by above and below), and 
\begin{eqnarray*}
&&|\rho^\nu_0-\rho^0|\leq |(\rho^0-\rho_*)-(\rho^0-\rho_*)*\psi_\nu|,\\
&&|\rho^\nu_0 u^\nu_0-\rho^0u^0|\leq C(|u^0-u^0*\psi_\nu|+ |\rho^0-\rho^0*\psi_\nu|)\\
&&\qquad\qquad = C(|(u^0-u_*)-(u^0-u_*)*\psi_\nu|+ |(\rho^0-\rho_*)-(\rho^0-\rho_*)*\psi_\nu|),\\
&&\eta(U^\nu_0|U^0)\leq C\left(|u^0-u^0*\psi_\nu|^2+|\rho^0-\rho^0*\psi_\nu|^2\right)\\
&&\qquad\qquad = C(|(u^0-u_*)-(u^0-u_*)*\psi_\nu|^2+ |(\rho^0-\rho_*)-(\rho^0-\rho_*)*\psi_\nu|^2).
\end{eqnarray*}
However, since $(\rho^0,\rho^0 u^0)$ is in $E^0$, both $\rho^0-\rho_*$ and $u^0-u_*$ are bounded in $L^1$ and $L^2$. 
Thus:
\[
\int_\RR|U_0^\nu-U^0|\,dx+\int_\bbr  \eta\big(U_0^{\nu}| U^0 \big) dx \to 0.
\]

Now, since 
\[
|\partial_x\phi(\rho_0^\nu)| \le |\phi'(\rho_0^\nu)| \left| \rho_0^{\nu}* \partial_x\psi_\nu \right| =  |\phi'(\rho_0^\nu)| \left| (\rho_0^{\nu}-\rho_*)* \partial_x\psi_\nu \right| \]
we have
$$
\|\partial_x\phi(\rho_0^\nu)\|_{L^2}\leq C\|\rho_0^\nu-\rho_*\|_{L^2}\|\partial_x\psi_\nu\|_{L^1}\leq   \frac{C}{\nu}  \int_\bbr \big|\psi_1' \big(\frac{y}{\sqrt\nu}\big)\big|dy \le  \frac{C}{\sqrt\nu}.
$$
Therefore,
\[
\nu^2 \int_\bbr  |\partial_x\phi(\rho_0^\nu)|^2 dx \to 0,
\]
and so, we have \eqref{initialNS}.\\

\subsection{Proof of Theorem \ref{th1}}
We now prove Theorem \ref{th1} by using Proposition \ref{main_prop}  and Proposition \ref{prop:cc}.\\

We first use Proposition \ref{prop:cc}  to have
\beq\label{ruc}
\rho^\nu \to \rho,\quad \rho^\nu u^\nu \to \rho u \quad \mbox{in  $L^q_{\mathrm{loc}} ((0,T)\times\bbr)$ as $ \nu\to 0$}.
\eeq
where $U=(\rho,\rho u)$ is an entropy solution to the Euler system. From Proposition \ref{main_prop}, there exists $\bar U_\nu$ uniformly small BV, such that  $U^\nu-\bar U_\nu$ converges weakly to 0.
By weak compactness, there exists a function $\bar U=(\bar\rho, \bar \rho \bar u)$ such that, up to a subsequence, $\bar U_\nu$ converges weakly to $\bar U$. Thanks to the uniform bound \eqref{BVest},
the limit $\bar U$ still verifies the small BV condition \eqref{BVest}. And by the uniqueness of the limit, $(\rho, \rho u)=\bar U$. 
We  use the uniqueness theorem of Chen, Krupa and Vasseur \cite{CKV20} (see also Bressan and De Lellis \cite{bressan2023remark})  to ensure that this limit is the unique small BV solution to Euler with the corresponding initial value. From the uniqueness, the whole sequence converges.





\subsection{Proof of Theorem \ref{th2}}
For any $U_n$ inviscid limit of Navier-Stokes associated to the initial value $U^0_n$, there exists a sequence of solutions $U^\nu_n$ to the Navier-Stokes equations such that
\[
U^\nu_n \weakto U_n \quad \mbox{in $L^q_{\mathrm{loc}}((0,T)\times\bbr)$ as $ \nu\to 0$},\quad \mbox{for some } q>1.
\]
which implies, for $n$ fixed, that
\[
\|U^{\nu}_n- U_n\|_{W^{-1,q}_{\mathrm{loc}}((0,T)\times\bbr)} \to 0 \quad \mbox{as $ \nu\to 0$}.
\]
We want to find a sequence $\nu_n$ converging to 0 for which we can apply Theorem  \ref{th1}. \\
To ensure that $U^{\nu_n}_{n,0}$ (for such a sequence $\nu_n$) is an adapted family of initial values associated with $U^0$, we observe that 
letting $U_{n,0}^\nu$ be an adapted family of initial values associated with $U^0_n$, for $n$ fixed, it holds from  \eqref{initialNS} that for $\nu$ small enough,
$$
\int_\RR |U_{n,0}^\nu -U^0|\,dx\leq  \int_\RR |U_{n,0}^\nu -U^0_n|\,dx+\int_\RR |U^0_n-U^0|\,dx\leq \frac{1}{n} +\int_\RR |U^0_n-U^0|\,dx.
$$
Similarly, by taking $\nu$ small enough,
\begin{eqnarray*}
&&\int_\RR \eta(U_{n,0}^\nu|U^0)\,dx=\int_\RR \eta(U_{n,0}^\nu|U^0_n)\,dx+\int_\RR \eta(U^0_n|U^0)\,dx+\int_\RR \big(\eta'(U_n^0)-\eta'(U^0)\big) (U_{n,0}^\nu-U^0_n)\,dx\\
&&\qquad \leq \int_\RR \eta(U_{n,0}^\nu|U^0_n)\,dx+\int_\RR \eta(U^0_n|U^0)\,dx+C_n\int|U_{n,0}^\nu-U^0_n|\,dx\\
&&\qquad \leq \int_\RR  \eta(U^0_n|U^0)\,dx+\frac{1}{n}.
\end{eqnarray*}
Therefore, there exists a sequence $\nu_n$ converging fast enough to 0 as $n\to \infty$ such that $U^{\nu_n}_{n,0}$ is an adapted family of initial values for $U^0$, and  
\[
\|U^{\nu_n}_n- U_n\|_{W^{-1,q}_{\mathrm{loc}}((0,T)\times\bbr)} \to 0 \quad \mbox{as $ n\to \infty$}.
\]
Then, applying Theorem \ref{th1} to $U^{\nu_n}_n$,   we have
\[
U^{\nu_n}_n  \to U \quad \mbox{in $L^q_{\mathrm{loc}}((0,T)\times\bbr)$ as $n \to \infty$},
\]
where $U$ is the small $BV$ solution associated to the initial value $U^0$.\\
Therefore,
\[
U_n = (U_n-U^{\nu_n}_n ) +U^{\nu_n}_n \weakto U\quad \mbox{in $L^q_{\mathrm{loc}}((0,T)\times\bbr)$ as $n \to \infty$}.
\]
We claim that this convergence should be strong. Suppose not. Then, there exists a compact set $K\subset (0,T)\times\bbr$ such that 
\[
\liminf_{n\to\infty} \int_K |U_n|^q >  \int_K |U|^q.
\]
But, since $U^{\nu_n}_n$ converges to $U_n$ weakly in $L^q_{\mathrm{loc}}$, 
\[
\liminf_{n\to\infty} \int_K |U^{\nu_n}_n|^q > \int_K |U|^q.
\]
This implies that $U^{\nu_n}_n$ does not converges to $U$ strongly in $L^q_{\mathrm{loc}}$, which is a contradiction.
Hence we have the desired result.

\section{Construction of the approximate solutions} \label{sec:appsoln}

\subsection{Modified front tracking algorithm\label{subsection_4.2}}
We will define an approximate solution through a modification of the  front tracking algorithm for the Euler equation. In a second step, we will replace the discontinuous fronts with viscous fronts, and move the fronts with ad-hoc shifts. Therefore the front tracking algorithm is designed to re-engineer the simple wave patterns of the approximate solution right after  wave interactions.   Some modifications to the front tracking algorithm would be related to the estimates for  the numbers of waves and interactions, and for the lower bound of the strength of physical waves, and for the upper bound of the strength of non-physical waves, which depend on the viscous strength $\nu$. 

First, for some technical reason, we need to construct an approximation solution such that the strength of all physical waves has a positive lower bound as in Lemma \ref{lemma_key2}.  \\
That is why we consider, as in Appendix \ref{app:algorithm}, a pair of Riemann invariants $s$ and $r$ as
\[
s=h+z,\quad r=h-z
\]
where 
\[
z=\frac{2\sqrt \gamma}{\gamma-1}v^{-\frac{\gamma-1}{2}}.
\]
We define the (signed) size of shock/rarefaction, in the first family as
\[
\sigma=r_R-r_L,
\]
and in the second family as
\[
\sigma=s_R-s_L,
\]
where $L$ and $R$ denote the left and right states of the wave, respectively. So, $\sigma<0$ for shocks and $\sigma>0$ for rarefactions by Figure \ref{riem_coords}.

For pseudo-shock, we define the size by $\sigma=r_R-r_L$ if $|r_R-r_L|\geq|s_R-s_L|$, otherwise
$\sigma=s_R-s_L$. So the strength of wave is
\[
|\sigma|=\max\{|r_R-r_L|,\ |s_R-s_L|\}.
\]
We will sometimes call pseudo-shock ``non-physical shock''. We will simply call physical (viscous) shock ``shock''.

Our front tracking algorithm is adapted from the one used in Bressan's book \cite{Bressan}. However, in order to make this algorithm work well in our analysis at the level of Navier-Stokes, we need to make a lot of necessary adaptions on the algorithm. We can roughly summarize our adaptions as follows. 
\begin{itemize}
\item For each approximate solution, we need a low bound on the strength of any shock and rarefaction jump to control the high frequency oscillation. In order to get that,
we introduce some new Riemann solvers other than the accurate and simplified solvers, and change the simplified solvers.
\item We do not want to give the non-physical shock a speed faster than all physical fronts. Instead, we give all non-physical shocks a zero speed in the approximate solution, between the first and second families.
 \item We need to control the number of waves and wave interactions. So some adjustments of the classical algorithm are added. Here we take advantage of the system with two variables, where less new waves are produced than system with more than two variables.
\end{itemize}

For the (typical) Riemann problem with two constant states $U_L=(r_L,s_L)$ and $U_R=(r_R,s_R)$ sufficiently close, a Riemann solution with at most three constant states, connected by shocks or rarefaction fans, can always be found, see the appendix or \cite{Young, sm}. More precisely, there exist $C^2$ curves  $s\mapsto T_i(s)(U_-)$ of $i$-family where $i=1,2$, parametrized by arc-length $s$, such that
\begin{align}\label{RP_curves}
U_R=T_2(s_2)\circ T_1(s_1)(U_L),
\end{align}
for some arc-lengths $s_1$ and $s_2$. We define $U_0:= U_L$ and  $U_2:= U_R$.  Then \eqref{RP_curves} can be also written as
\begin{align}
U_1:= T_1(s_1)(U_0),\\
U_2:= T_2(s_2)\circ T_1(s_1)(U_0).
\end{align}
Here, as in  \cite{Bressan}, we use the convention that $s_i$ is negative (resp. positive) as the arc-length for the i-shock (resp. i-rarefaction) curve connecting the states $U_{i-1}$ to $U_i$. 
Notice that $|\sigma_i|$ the strength of the i-wave is equivalent to the strength of the associated arc-length. For example, if a 2-shock generated by $T_2(s_2)$ has size $\s_1$, then $|\s_1|$ is equivalent to $|s_2|$.
Notice that the strength and the arc-length of a rarefaction are the same by the definition of size and the fact that each rarefaction curve is parallel to either $r$ axis or $s$ axis. 

For two positive parameters $\rho_\nu$ and $\delta$ small enough, we will introduce a modified front tracking algorithm composed of three types of Riemann solvers:
{\emph{Accurate solver, Simplified solver and Adjusted solver.}}

As in Section \ref{sub:conapp}, we will consider that at each time we have at most one collision, which will involve only two wavefronts. Suppose that at some time $t>0$ there is a collision between two waves from the $i_\alpha$th and $i_\beta$th families. Denote the sizes of the two waves by $\sigma_\alpha$ and $\sigma_\beta$, respectively. The Riemann problem generated by this interaction is solved as follows. 
\begin{itemize}
\item For shock-rarefaction overtaking interactions, if $|\sigma_\alpha\sigma_\beta|> \rho_\nu$ and any of the outgoing wave has strength less than $\rho_\nu$, use adjusted solver (C).
\item For shock-rarefaction overtaking interactions with $|\sigma_\alpha\sigma_\beta|\leq\rho_\nu$, or shock-shock overtaking interactions with the strength of reflected rarefaction  in the exact interaction less than $\rho_\nu$ or for any interaction including one non-physical shock, we use the simplified solver (B).
\item For all other cases, we use the accurate solver (A).
\end{itemize}
The solution of any Riemann solver contains finitely many constant states, connected by rays starting from the interaction point on the $(x,t)$-plane, called wave fronts. Later, we will prove that for each the approximate solution $\psi_{\mathcal F}$, there will be only finitely many wave fronts, so there will be finitely many wave interactions. 

For convenience, we use the notation of $\fa R$, $\fa S$ and $\ba R$, $\ba S$, first used in \cite{ChenJenssen}, to denote the $2$-rarefaction, $2$-shock and $1$-rarefaction, $1$-shock.

Now to define all types of Riemann solvers, we always assume that at a positive time $\bar{t}$, there is only one interaction at the point $\bar{x}$ between two waves with strengths $|\sigma_1|,\ |\sigma_2|$. Let $U_L$, $U_R$ be the constant states that generate the interaction, then we will solve this Riemann problem using three types of solvers. 

In the mean time of defining Riemann solvers, we also prove some important estimates on total variation and Glimm potential. We denote the total variation of $\psi_{\mathcal F}$ as 
\[
L(t)=\sum |\sigma_i|=\hbox{TV}(\psi_{\mathcal F})(t),
\]
namely the sum of the strengths of all jump discontinuities that cross the $t$-time line, including all physical and non-physical fronts. Here all non-physical fronts have zero speed. Clearly, $L(t)$ stays constant along time intervals between consecutive collisions of fronts and changes only across points of wave interaction.

A $j$-wave $\alpha_j$ and an $i$-wave $\alpha_i$, with the former crossing the $t$-time line to the left of the latter, are called {\it approaching} except when $j<i$ or $i=j=1.5$, where for convenience we set $i$ or $j$ equals $1.5$ if any of them is a non-physical shock.
We then define the Glimm potential for wave interactions as
\[
Q(t)=\sum_{\alpha_j,\alpha_i:\hbox{approaching waves}}|\sigma_{\alpha_i}| |\sigma_{\alpha_j}|,
\]
where the summation runs over all pairs of approaching waves. Here we note we consider $\ba R$-$\ba R$ (and $\fa R$-$\fa R$) as approaching wave to gain some technical advantage for some adjusted solvers, although they won\rq{}t really interact.

We always assume $L(0)\leq \eps$ with $\eps$ small enough, so $Q(0)\leq O(\eps^2)$ is much smaller.\\
The following two lemmas are crucially used throughout the paper.
\vskip0.1cm

\begin{lemma}\label{lemma_key1}
We use three types of Riemann solvers to solve the Riemann problem generated by the interaction between two waves with strength $|\sigma_i|$ and $|\sigma_j|$ at time $t$. Denote 
\[\Delta L=L(t+)-L(t-),\qquad 
\Delta Q=Q(t+)-Q(t-).
\]
There exist $C, \eps_0>0$ such that for all $\eps<\eps_0$, the following holds: for all $|\sigma_i|$ and $|\sigma_j|$ with $|\sigma_i|, |\sigma_j|<\eps$,
\[
\Delta Q\leq -\frac{3}{4}|\sigma_i| |\sigma_j| ;
\]
and
$$
\Delta L\leq C|\sigma_i| |\sigma_j|.
$$
Therefore, there exists $\kappa_0>0$ such that for all $\kappa>\kappa_0$,
$$
\Delta L+\kappa \Delta Q(t)\leq -\frac{\kappa}{2}|\sigma_i| |\sigma_j|.
$$
\end{lemma}
Using the decay of $L+\kappa Q$ and $Q$, it is easy to conclude that when $\eps$ is small enough, $L(t)\leq 2\eps$ and $Q(t)\leq O(\eps^2)$.\\

\begin{lemma}\label{lemma_key2}
For each Riemann solver, assume the strength of each incoming shock and rarefaction is larger than $\rho_\nu$. Then, the following holds: 
\begin{itemize}
\item[1.] There exists $C>0$ such that the strength of each rarefaction is less than $C\delta$.
\item[2.] The strength of each outgoing shock and rarefaction is larger than $\rho_\nu$.
\item[3.] For any non-physical shock, $h$ is not decreasing from the left state to the right state.
\item[4.] There exists $C>0$ such that the strength of each non-physical shock is less than $C\rho_\nu$. 
\end{itemize}
\end{lemma}
In what follows, we will prove the above two lemmas for each of the three Riemann solvers: 
{\emph{Accurate solver, Adjusted solver and Simplified solver}} in order. \\
The proof of the first statement of Lemma \ref{lemma_key2} on rarefaction wave is very similar to Part 5 of Chapter 7.3 in \cite{Bressan}. This statement is correct, basically because rarefaction waves in the same family never merge. 
So we omit the proof of statement 1, and refer the reader to \cite{Bressan}. 

Similarly, the strength of non-physical shock might increase after the interaction, but  we can show as in \cite{Bressan} that its strength is still less than  $C_1\rho_\nu$ for some constant $C_1>C$, where non-physical waves never merge too, by assuming that the strength of each incoming non-physical shock  $C\rho_\nu$. \\


Next, we will construct three types of Riemann solvers and prove statements 2, 3, 4.

\vskip0.2cm

\paragraph{\bf (i) Accurate solver}
\vskip0.1cm

The accurate solver (sometimes called exact solver) is only for interaction between physical waves.

We let 
\begin{align}\label{control_rarefaction_932020}
q_i:= \left \lceil{|s_i|/\delta}\right \rceil,
\end{align}
where $\left \lceil{s}\right \rceil$ denotes the smallest integer number greater than $s$. \\
If $s_1>0$,
\begin{align}\label{raredef1}
U_{1,l}:= T_1(ls_1/q_1)(U_0),\qquad x_{1,l}(t):= \bar{x}+(t-\bar{t})\lambda_1(U_{1,l}),\qquad l=1,\ldots,q_1.
\end{align}
If $s_2>0$, then
\begin{align}\label{raredef2}
U_{2,l}:= T_2((l-1)s_2/q_2)(U_{1}),\qquad x_{2,l}(t):= \bar{x}+(t-\bar{t})\lambda_2(U_{2,l})\quad l=1,\ldots,q_2.
\end{align}
On the other hand, if $s_i<0$, we define $q_i:= 1$ and 
\begin{align}
U_{i,l}:= U_i ,\hspace{.7in} x_{i,l}(t):= \bar{x}+(t-\bar{t})\tilde\lam_i ,
\end{align}
where $\tilde\lam_i$ is the Rankine-Hugoniot speed.

Then, for the left state $U_L=(r_L,s_L)$ and the right state $U_R=(r_R,s_R)$ sufficiently close, we define an accurate solver as follows:
\begin{align}\label{accurate_RP_sol}
v_a(t,x):=
\begin{cases}
U_-, &\mbox{ if } x < x_{1,1}(t),\\
U_+ ,&\mbox{ if } x>x_{2,q_2}(t),\\
U_i,&\mbox{ if } x_{i,q_i}(t)<x<x_{i+1,1}(t),\\
U_{i,l},&\mbox{ if } x_{i,l}(t)<x<x_{i,l+1}(t)\hspace{.3in}(l=1,\ldots,q_i-1).
\end{cases}
\end{align}


\vskip0.1cm

For future use, we consider the following {\bf adjustment}: (especially, it is used in the proof of Lemma \ref{lem:num})

\vskip0.1cm
\emph{In each interaction, if there are more than one outgoing 1-rarefactions (or 2-rarefactions), and the wave strength $\sigma$ (arc-length) of the last rarefaction is less than $\delta$, then we average wave strengths for the last two rarefactions, such that each of them is larger than $\frac{\delta}{2}$. }

\vskip0.1cm

Similar as in \cite{Bressan}, we adopt the following {\bf provision:} (especially, it is used in the proof for the statement 1 of Lemma \ref{lemma_key2} as in \cite{Bressan}  )

\vskip0.1cm
\emph{In the accurate solver, rarefaction fronts of the same family of one of the incoming fronts are never partitioned (even if their wave strength $\sigma$ is larger than $\delta$).}

\vskip0.1cm

For any pairwise interaction considered in the accurate solver in the front tracking scheme, one has the following estimates (see Lemma 7.2 in \cite{Bressan} or \cite{BCZ2}).

\begin{lemma}\label{prop1} We consider the Euler equations \eqref{eq:Euler}.
Let $\sigma_1, \sigma_2$ be  sizes for two interacting physical wave-fronts, and let 
$\sigma', \sigma''$ be the  sizes of the outgoing $1$-wave and $2$-wave, respectively, for accurate solver  connecting the left and right states.

\begin{itemize}
\item If  wave with size $\sigma_1$ is a 1-wave and wave with size $\sigma_2$ is a 2-wave, then
\beq\label{299}
|\sigma'-\sigma_1|+|\sigma''-\sigma_2|~\leq~C\,|\sigma_1\sigma_2|(|\sigma_1|+|\sigma_2|).\eeq
\item If both waves with strengths $\sigma_1$ and $\sigma_2$ belong to the $1$-st family, then 
\beq\label{21}
|\sigma'-(\sigma_1+\sigma_2)|+|\sigma''|~\leq~C\,|\sigma_1\sigma_2|(|\sigma_1|+|\sigma_2|).\eeq
\item If both waves with sizes $\sigma_1$ and $\sigma_2$ belong to the $2$-nd family, then 
\beq\label{22}
|\sigma\rq{}\rq{}-(\sigma_1+\sigma_2)|+|\sigma'|~\leq~C\,|\sigma_1\sigma_2|(|\sigma_1|+|\sigma_2|).\eeq
\end{itemize}
This result does not depend on the choice of wave strength.
\end{lemma}

\vskip0.1cm

Clearly, waves in head-on interactions do not change type. For overtaking interactions:
$\ba S\ba S$ produces $\ba S$ and $\fa R$; $\ba S\ba R$ or $\ba R\ba S$ produces $\ba S$ or $\ba R$ and $\fa S$. Interaction of forward waves are symmetric. The proof of these properties is classical. See the appendix for more details on pairwise interactions and also \cite{ChenJenssen,sm,Young}.

For accurate solvers, Lemma \ref{lemma_key1} can be directly proved by Lemma \ref{prop1} using some standard method as in \cite{Bressan}. One exception we need to address is the shock-shock overtaking interaction, where multiple rarefaction jumps might be reflected. Since we add $\ba R$-$\ba R$ and $\fa R$-$\fa R$ pairs in the potential $Q$, we need to calculate the impact of this change to $Q$. Let\rq{}s only consider 
$\ba S\ba S$ interaction with incoming strengths $\sigma\rq{}$ and $\sigma\rq{}\rq{}$. Then the total strength of outgoing rarefactions is less than $C_0\varepsilon |\sigma\rq{}\sigma\rq{}\rq{}|$ by Lemma \ref{prop1}. So the interaction potential between these multiple reflected rarefactions is at most in the order of $\varepsilon^2 |\sigma\rq{}\sigma\rq{}\rq{}|^2$ which is very small, and can be well controlled by $O(1)|\sigma_i||\sigma_j|\,L(\bar t-)$. It is clear that the result in Lemma \ref{lemma_key1} still holds.

To prove Lemma \ref{lemma_key2} for accurate solvers, we need to prove the following lemma by considering the accurate interactions case by case. 

\begin{lemma}\label{lemma_key3}
The accurate solver satisfies the following property.

For head-on interactions, waves do not change type after the interaction.
\begin{itemize}
\item[1.] For $\fa R\ba R$ interaction, each wave keeps its strength after interaction.
\item[2.] For $\fa S\ba R$, the strength of shock does not change, the strength of rarefaction increases.
\item[3.] For $\fa S\ba S$, the strengths of shocks in both families increase.
\end{itemize}

For overtaking interactions,
\begin{itemize}
\item[4.] For $\ba S\ba S$ and $\fa S\fa S$ interactions, the reflected wave is rarefaction. And the outgoing shock strength is the sum of two incoming ones.
\end{itemize}
\end{lemma}

The proof for the claim on $\fa R\ba R$ interaction is trivial, since rarefaction curves in the same family are alway parallel. The claim on $\ba S\ba S$ and $\fa S\fa S$ interactions can be easily proved. The proof of cases 2 and 3 can be found in \cite{BCZ1,BCZ2}, where we give a schetch on the idea how to prove cases 2 and 3 in the Appendix. One can also find other proofs in \cite{Young,sm,ChenJenssen}.
Since all waves in the accurate solvers are physical, we prove Lemma \ref{lemma_key2} for all accurate solvers. \\

\vskip0.1cm

\paragraph{\bf (ii) Adjusted solver}

\vskip0.1cm

In this part, we consider any shock-rarefaction overtaking interaction generating a wave whose strength is less than $\rho_\nu$. Here we only consider the case when $|\sigma_1\sigma_2|> \rho_\nu$, otherwise, we will use the simplified solver. The goal of this solver is to make sure the physical wave strength is always larger than $\rho_\nu$.

For the adjusted solver and simplified solver, we introduce the \emph{non-physical shock} (also known as \emph{pseudoshock} or non-physical wave), which is a jump discontinuity connecting two constant states (denoted by $U_L$ and $U_R$ from left to right), and traveling with zero velocity. For the future use, we need $h_R\geq h_L$ for any non-physical wave. This will be verified for all outgoing non-physical shocks in all Riemann solvers. Later after taking limit $\rho_\nu\rightarrow 0$, non-physical waves will disappear.

We only consider backward interactions. The forward interaction will be treated symmetrically.

\vskip0.1cm

For $\ba S\ba R$ interaction, we consider two cases: Case 1 when the outgoing 1-wave is a shock; Case 2, when the outgoing 1-wave is a rarefaction. By  Figure \ref{Simp6} and \ref{Simp7}, it is easy to see that the reflected wave with strength less than $\varepsilon|\sigma_1\sigma_2|$ is always a shock in the accurate solution, see \cite{ChenJenssen}.

If the strength of one of the outgoing waves is less than $\rho_\nu$, then by adding a non-physical wave whose strength is at most $\rho_\nu$, we can make all outgoing waves have strength larger than $\rho_\nu$. See in Figure \ref{Simp6} and \ref{Simp7}, where $a'$ is the middle state in the accurate solver.

  More precisely, we choose the slope of $bc$ as $1$ in Figure \ref{Simp6} and $-1$ in Figure  \ref{Simp7}. So it is clear that $h_c\geq h_b$ always hold. Since we adjust one  (or two) outgoing wave whose strength is less than $\rho_\nu$ to $\rho_\nu$, the additional strength of each outgoing wave comparing to the exact solution in the accurate solver is at most $k_0\rho_\nu<k_0|\sigma_1\sigma_2|$ for some constant $k_0$, i.e. 
\beq\label{G_add}
\max(|r_{a'}-r_b|, |s_{a'}-s_c|)\leq k_0\rho_\nu<k_0|\sigma_1\sigma_2|.
\eeq 
 Clearly, by \eqref{G_add}, the strength of outgoing non-physical wave,  i.e. $|s_b-s_c|=|r_b-r_c|$, is less than $C\rho_\nu$ for some constant $C$.\\
Observe that \eqref{21} and $\s_1<0<\s_2$ with $\s'<0$ for Case 1;  $\s'>0$ for Case 2 imply that \\
 for Case 1,
\[
\Delta L= |\s'|+ |\s''|- (|\s_1|+ |\s_2|) \le -2\s_2 + C|\s_1||\s_2| \le -\s_2 <0 ;
\]
and for Case 2,
\[
\Delta L \le 2\s_1 + C|\s_1||\s_2| \le \s_1 <0.
\]

In Case 2, if $r_b-r_L>r_R-r_a$, then move point $b$ to the left such that $r_b-r_L=r_R-r_a$, so the total strength of the backward rarefaction is less than $\delta$, where $\delta\gg \rho_\nu$. In summary, $L+\kappa Q$ always decays for sufficiently large $\kappa$, and Lemma  \ref{lemma_key1} and \ref{lemma_key2} both hold.

\vskip0.1cm
For $\ba R\ba S$ interaction with $|\sigma_1\sigma_2|> \rho_\nu$ which generates a wave whose strength is less than $\rho_\nu$, the adjusted interactions can be defined similarly as in  Figure \ref{Simp8} and \ref{Simp9}. Similarly, Lemma \ref{lemma_key1} and \ref{lemma_key2} hold. In Case 2, we choose point $b$ to be between points $L$ and $a$, so there is only one outgoing backward rarefaction whose strength is not larger than $\delta$.

\begin{figure}
	\centering
	\includegraphics[scale=.5]
	{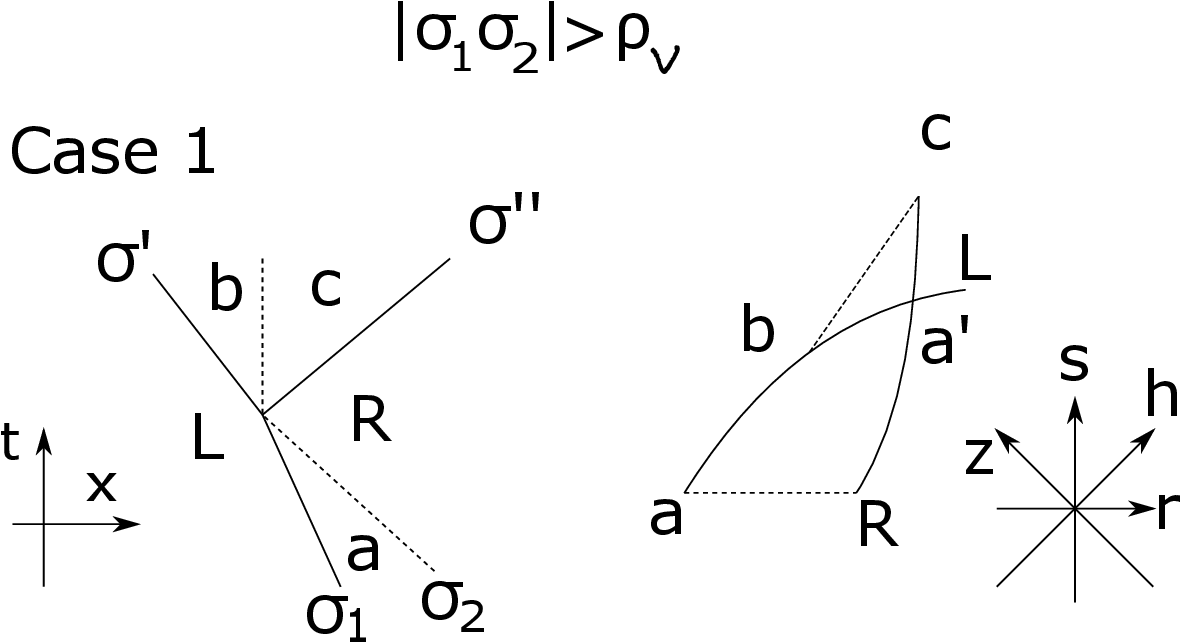}
	\caption{Adjusted solver: Rarefaction overtakes shock: Case 1. Solid lines for shocks, dot lines for other waves.  }\label{Simp6}
\end{figure} 

\begin{figure}
	\centering
	\includegraphics[scale=.5]
	{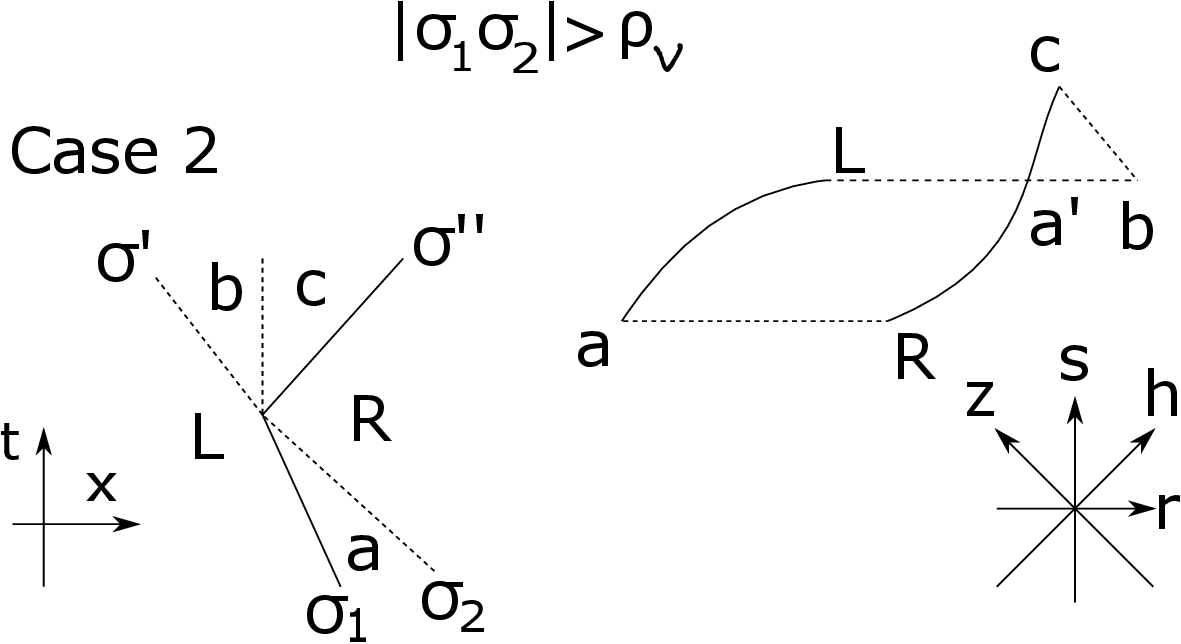}
	\caption{Adjusted solver: Rarefaction overtakes shock: Case 2. 
	First choose $bc$ in the positive $z$ direction, then move $b$ to the left if needed.} \label{Simp7}
\end{figure}

\begin{figure}
	\centering
	\includegraphics[scale=.5]
	{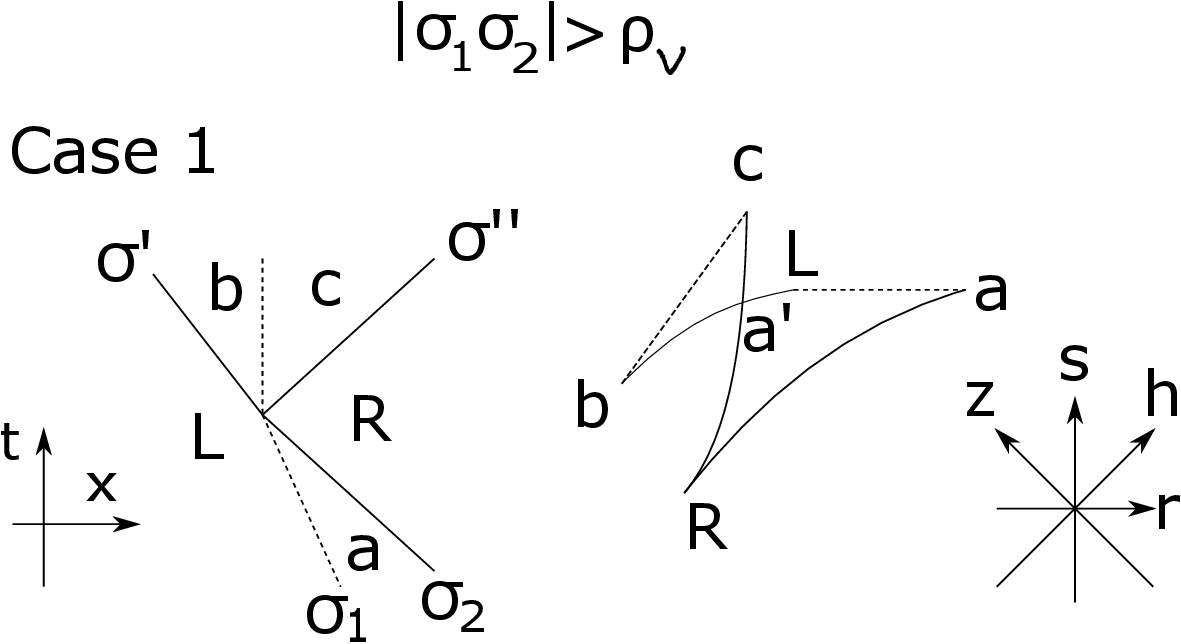}
	\caption{Adjusted solver: Shock overtakes Rarefaction: Case 1.} \label{Simp8}
\end{figure} 

\begin{figure}
	\centering
	\includegraphics[scale=.5]
	{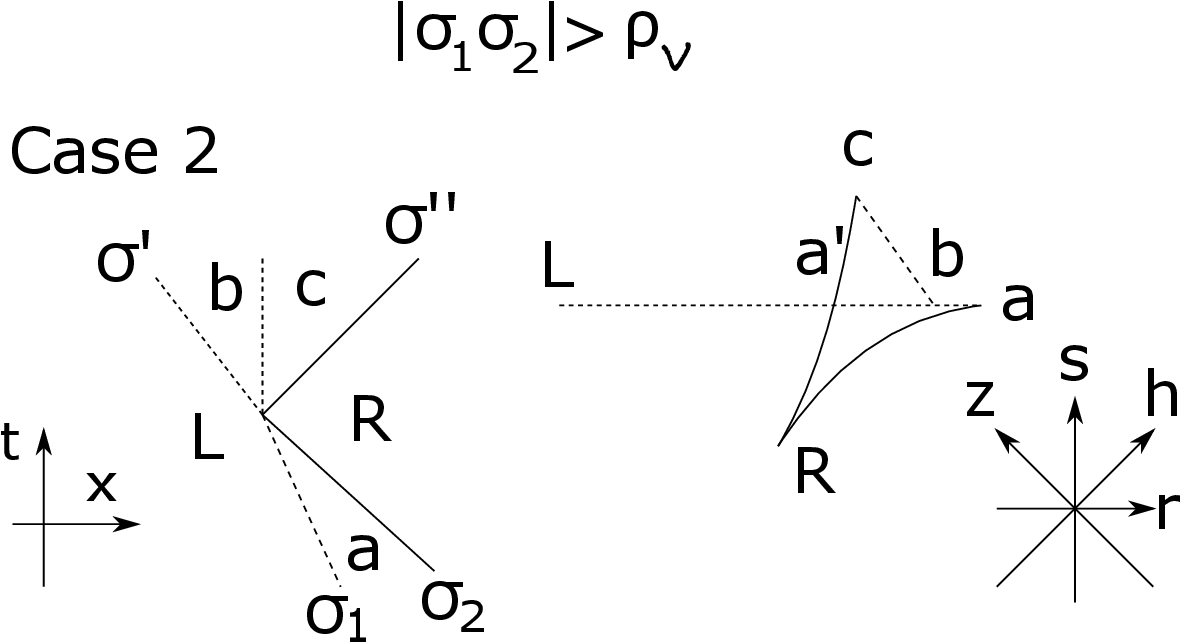}
	\caption{Adjusted solver: Shock overtakes Rarefaction: Case 2. 
	First choose $bc$ in the positive $z$ direction, then move $b$ to the left if needed.}\label{Simp9}
\end{figure}

\vskip0.3cm

\paragraph{\bf (iii) Simplified solver}

\vskip0.1cm

Finally, to control the total number of wave interactions, we introduce simplified solvers in the following cases:
\begin{itemize}
\item[(iii-1)] shock-shock overtaking interactions when the strength of reflected rarefaction in the exact interaction is less than $\rho_\nu$,
\item[(iii-2)]  shock-rarefaction overtaking interactions with $|\sigma_\alpha\sigma_\beta|\leq\rho_\nu$, 
\item[(iii-3)] any interaction including one non-physical wave.
\end{itemize} 
Here $\sigma_1$ and $\sigma_2$ denote the strengths of two incoming waves. For the shock-shock overtaking interaction, if $|\sigma_\alpha\sigma_\beta|\leq\rho_\nu$, then the strength of reflected rarefaction might be less than $\rho_\nu$ with $\eps$ small by Lemma \ref{prop1}, then the simplified solver will be used.
We do not use simplified solver for head-on interactions, since these interactions do not add wave numbers or generate reflected waves. We note that our simplified solver is different from the simplified solver in the standard front tracking scheme as in \cite{Bressan}.

Now, we define the simplified solver case by case. 

\begin{figure}
	\centering
	\includegraphics[scale=.5]
	{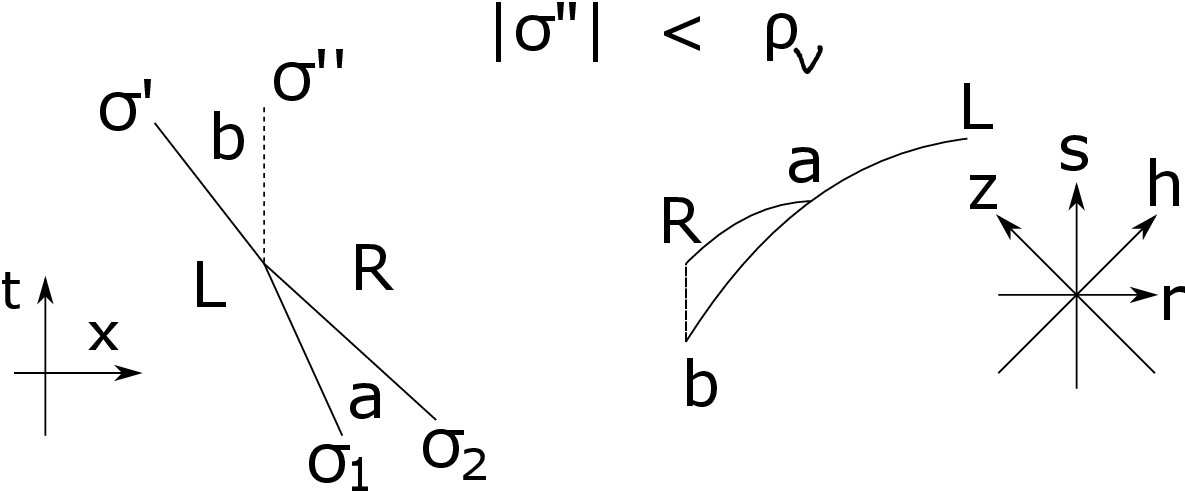}
	\caption{Simplified solver: Shock overtakes Shock}\label{Simp1}
\end{figure}

\vskip0.2cm
\paragraph{\bf (iii-1)} For the shock-shock overtaking interaction, when the strength of reflected rarefaction is less than $\rho_\nu$, change the reflected rarefaction into a zero speed non-physical shock with the same strength, and keep the same outgoing shock, as in Figure \ref{Simp1}. Lemma \ref{lemma_key1} and \ref{lemma_key2} clearly hold as the accurate solver case. It is easy to check that 
\[\sigma'=\sigma_1+\sigma_2,\]
i.e. $r_L-r_b=r_L-r_a+r_a-r_R$ since $r_b=r_R$.

\vskip0.2cm
\paragraph{\bf  (iii-2)} For $\ba S\ba R$ interaction with $|\sigma_\alpha\sigma_\beta|\leq\rho_\nu$, there are two cases. Use Figure \ref{Simp2} and \ref{Simp3} as the reference for Case 1 and 2, respectively. The $\fa R\fa S$ simplified solver can be defined symmetrically.

More precisely,  on the $(r,s)$ plane, draw a line in the positive $z$ direction, starting from the point $R$. If this line intersects with the shock curve from $a$ to $L$ except $L$ as in Figure \ref{Simp2}, where the intersection is named $a'$, then this case is called Case 1. Otherwise, we call it Case 2, as in Figure \ref{Simp3}.

\begin{figure}
	\centering
		\includegraphics
	[scale=.5]
	{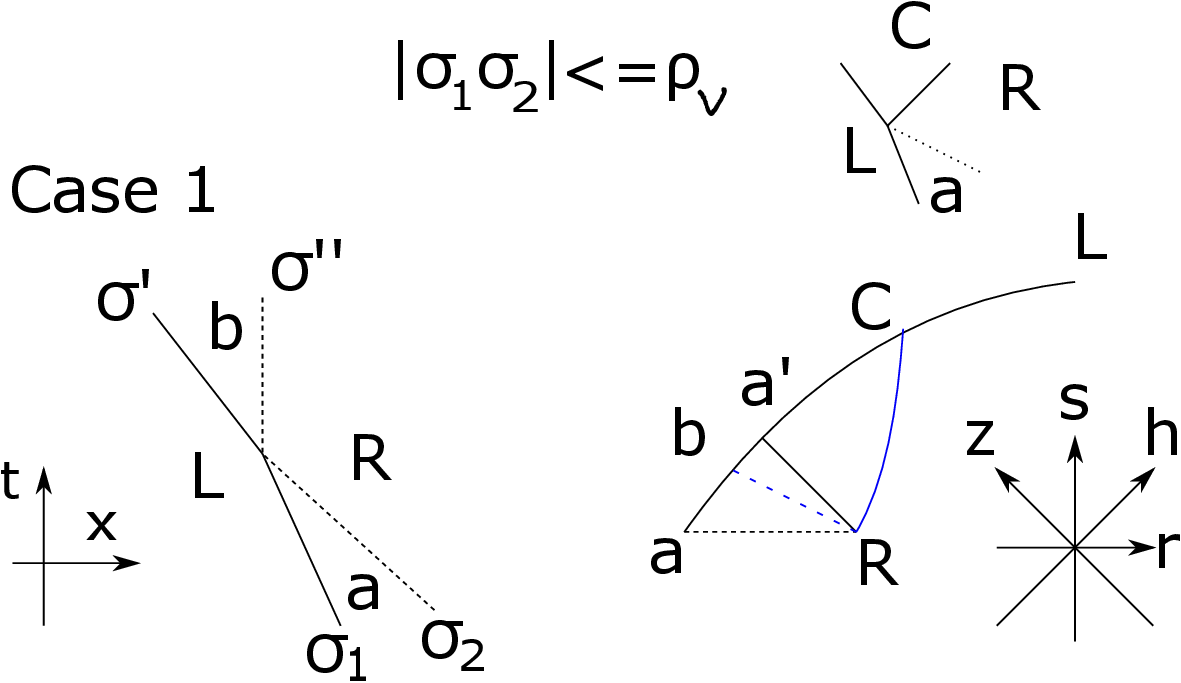}
	\caption{Simplified solver: Rarefaction overtakes Shock. Case 1 when shock dominates. The exact solution of this $\ba S\ba R$ interaction is in the right-up corner, with outgoing middle state $C$.}\label{Simp2}
\end{figure} 	

\begin{figure}
	\centering
		\includegraphics
	[scale=.5]
	{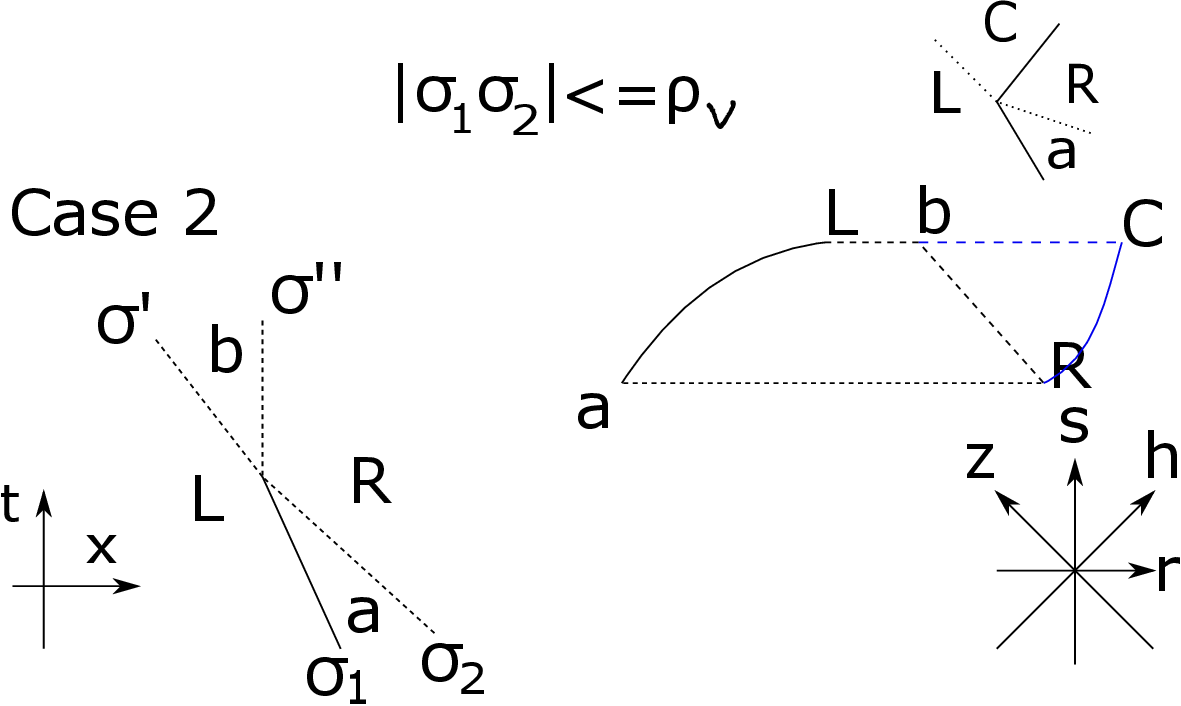}
	\caption{Simplified solver: Rarefaction overtakes Shock. Case 2 when rarefaction dominates.}\label{Simp3}
\end{figure}

\vskip0.1cm
In Case 1, we use Figure \ref{Simp2} as a reference. Recall we assume that the incoming strength $|\sigma_1|$ and $|\sigma_2|$ are always greater than or equal to $\rho_\nu$.

If the strength of shock from $L$ to $a\rq{}$ (in short, $L$-$a\rq{}$) is not less than $\rho_\nu$. Then we will choose $b=a\rq{}$. 
Then clearly $|\sigma\rq{}|<|\sigma_1|$, $|\sigma\rq{}\rq{}|<|\sigma_2|$.

Otherwise, choose state $b$ on the shock curve $L$-$a$ between $a$ and $a'$, to be the closest point to $a\rq{}$ on the $(r,s)$-plane, such that 
\beq\label{shock_strength1}
\rho_\nu\leq -\sigma'=|r_b-r_L|.
\eeq
Here we can always find such a point  between $a$ and $a'$, since the strength of incoming shock from $L$ to $a$ is not less than $\rho_\nu$. 
Then it is clear that \eqref{shock_strength1} is satisfied, and in fact, 
\[\rho_\nu= -\sigma'=|r_b-r_L|\leq \sigma_1.\] 
And clearly 
\beq\label{sigma_est_simp1}
|\sigma\rq{}\rq{}|=r_R-r_b=r_R-r_{a\rq{}}+r_{a\rq{}}-r_b<\sigma_2
\eeq 
by the definition of wave strengths of rarefaction and non-physical shock. 
Here  $r_{a\rq{}}-r_b$ is at most in the order of $\rho_\nu$ by our construction. We will show that $r_R-r_{a\rq{}}$ is also at most in the order of $\rho_\nu$, so does $|\sigma\rq{}\rq{}|$ by \eqref{sigma_est_simp1}.

Now we estimate $r_R-r_{a\rq{}}$. We denote the strength for the shock wave from $L$ to $a'$ as $-\hat \sigma$. Now we only consider the case when $0<-\hat \sigma<\rho_\nu$. Then we will show that 
\beq\label{case_special}
|\hat\sigma-(\sigma_1+\sigma_2)|=|r_{a’}-r_L-(r_R-r_L)|=|r_{a’}-r_R|=|s_{a’}-s_R|\leq 
O(\varepsilon|\sigma_1\sigma_2|)\leq \frac{1}{2}\rho_\nu,
\eeq
using  Lemma \ref{prop1} and the fact that $|s_{a’}-s_R|$ is in the same order of the strength of reflected forward outgoing wave in the exact  $\ba S\ba R$ interaction.
Let's only check when the outgoing backward wave in the exact $\ba S\ba R$  interaction is a shock, as in the up-right picture of Figure \ref{Simp2}, where
\[
|\hat\sigma-(\sigma_1+\sigma_2)|=|r_{a’}-r_R|=|s_{a’}-s_R|\leq |s_C-s_R|
=O(\varepsilon|\sigma_1\sigma_2|)\leq \frac{1}{2}\rho_\nu,
\]
where  $|s_C-s_R|$ is the strength of reflected forward outgoing wave in the exact  $\ba S\ba R$ interaction, so it is in the order of $O(\varepsilon|\sigma_1\sigma_2|)$ by  Lemma \ref{prop1}. The proof of \eqref{case_special} when  the outgoing backward wave in the exact $\ba S\ba R$  interaction is a rarefaction, is similar. We omit it here.
Therefore, Lemmas \ref{lemma_key1} and \ref{lemma_key2} hold for Case 1.

\vskip0.1cm

For Case 2, as in Figure \ref{Simp3}, we choose a state $b$ to be at the intersection point between a horizontal line starting from the point $L$ and a line starting from point $R$ in the positive $z$ direction. In this case, the point $b$ is always to the right of $L$, so the outgoing 1-wave is a rarefaction. \\
If the strength of rarefaction $L$-$b$ is less than $\rho_\nu$, i.e., $r_b-r_L<\rn$, then we use a non-physical shock to connect  directly from $L$ to $R$. In this case, the strength of non-physical shock is less than $2\rho_\nu$. Indeed, since $s_C-s_R$ as the strength of reflected forward shock in the exact solver is smaller than $O(\varepsilon|\sigma_1\sigma_2|)$ by Lemma \ref{prop1}, 
\beq\label{2.15}
r_b-r_L<\rho_\nu, \quad r_R-r_b=s_b-s_R=s_C-s_R\leq O(\eps|\sigma_1\sigma_2|)\leq \rho_\nu,
\eeq
which yields that the strength of non-physical shock$=r_R-r_L<2\rn$.\\
 In other cases, as in the second inequality of \eqref{2.15},
\[|\sigma''|=s_b-s_R<\rho_\nu.\]
And the outgoing backward rarefaction, if exists, always has a strength larger than $\rho_\nu$. \\
In summary, it holds from Lemma \ref{prop1} with $\sigma_1<\sigma_2, \s'$ and above estimates that
\[
-\sigma'+(\sigma_1+\sigma_2)=|\sigma''|=r_R-r_b\leq 2\rho_\nu,\qquad h_b \le h_R,\qquad |\sigma''|\leq 2\rho_\nu,
\]
and
\[
\Delta L= \s' + |\s''|+\s_1-\s_2 = 2\s_1 <0.
\]
Thus, Lemmas \ref{lemma_key1} and \ref{lemma_key2} hold.

\vskip0.2cm

The simplified $\ba R\ba S$ interaction can be defined similarly, as shown in Figure \ref{Simp4} and \ref{Simp5}. On the $(r,s)$ plane, draw a line in the positive $z$ direction, starting from the point $R$. If this line intersects with the backward shock curve starting from $L$, as in Figure \ref{Simp4}, where the intersection is named $a'$, then this case is called Case 1. Otherwise, we call it Case 2, as in Figure \ref{Simp5}.

In Case 1, use  Figure \ref{Simp4} as a reference. Recall we assume the incoming strength $|\sigma_1|$ and $|\sigma_2|$ are always greater than or equal to $\rho_\nu$.
If the strength of shock $L$-$a\rq{}$  is not less than $\rho_\nu$, then we choose $b=a\rq{}$. In this case,  we can show in the appendix that
\beq\label{key_est}
|\sigma\rq{}|=r_L-r_{a\rq{}}<r_a-r_R=|\sigma_2|\quad \hbox{and}\quad |\sigma\rq{}\rq{}|= r_R-r_{a\rq{}}<r_a-r_L=|\sigma_1|.
\eeq
Otherwise, choose the state $b$ on the shock curve $L$-$a\rq{}$, to be the closest point to $a\rq{}$ on the $(r,s)$-plane, such that 
\beq\label{shock_strength1_2}
\rho_\nu\leq -\sigma'=r_L-r_b.
\eeq
Thus, $\rn = |\s'| = r_L-r_b.$ Then, since $ |\sigma_2|\geq \rho_\nu$,
\[
|\sigma'|\leq |\sigma_2|,\quad \hbox{i.e.}\quad r_L-r_b\leq r_a-r_R,
\]
which yields $\quad r_a-r_L\geq r_R-r_b$, and so, $|\s''|\le |\s_1|$.
%
%

\vskip0.2cm
In Case 2, similar as before, if the rarefaction $L$-$b$ on Figure \ref{Simp5} is less than $\rho_\nu$, then directly use a non-physical shock to connect states $L$ and $R$. So the outgoing rarefaction always has a strength larger than $\rho_\nu$. The strength of non-physical shock is less than $2\rho_\nu$.

So for both cases, we can prove Lemma \ref{lemma_key1} and \ref{lemma_key2}, where we use the fact that all waves approaching the outgoing non-physical wave also approach the incoming $\ba R$.

The $\fa S\fa R$ simplified solver can be defined symmetrically.

\begin{figure}
	\centering
		\includegraphics
	[scale=.5]
	{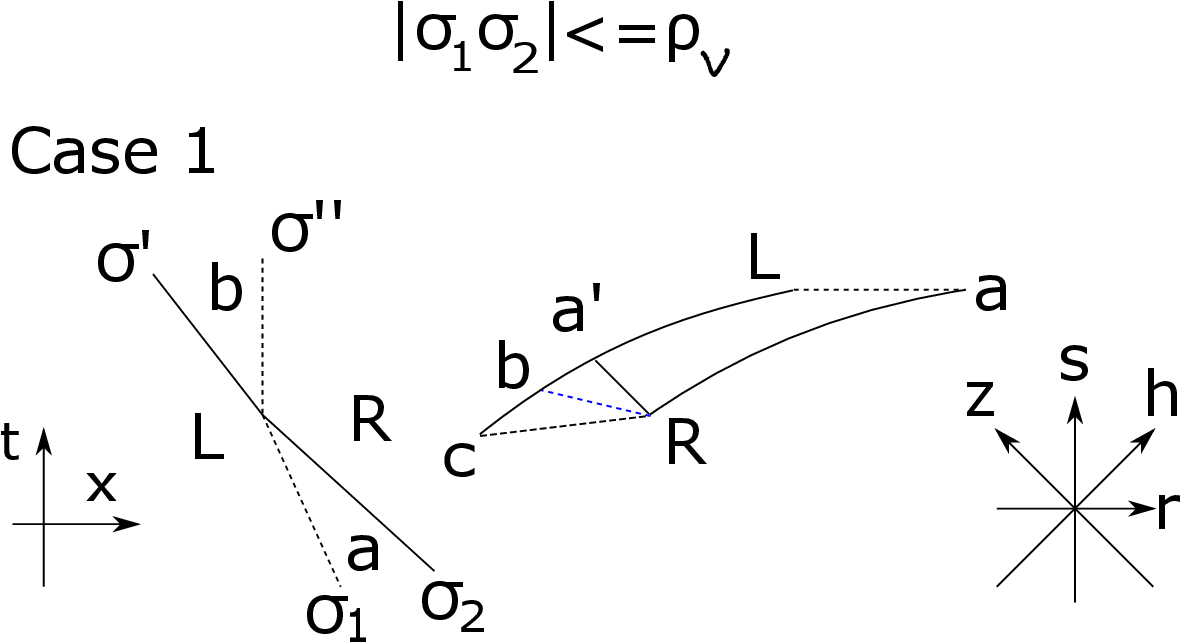}
	\caption{Simplified solver: Shock overtakes Rarefaction. Case 1 when shock dominates.}\label{Simp4}
\end{figure} 
\begin{figure}
	\centering
		\includegraphics
	[scale=.5]
	{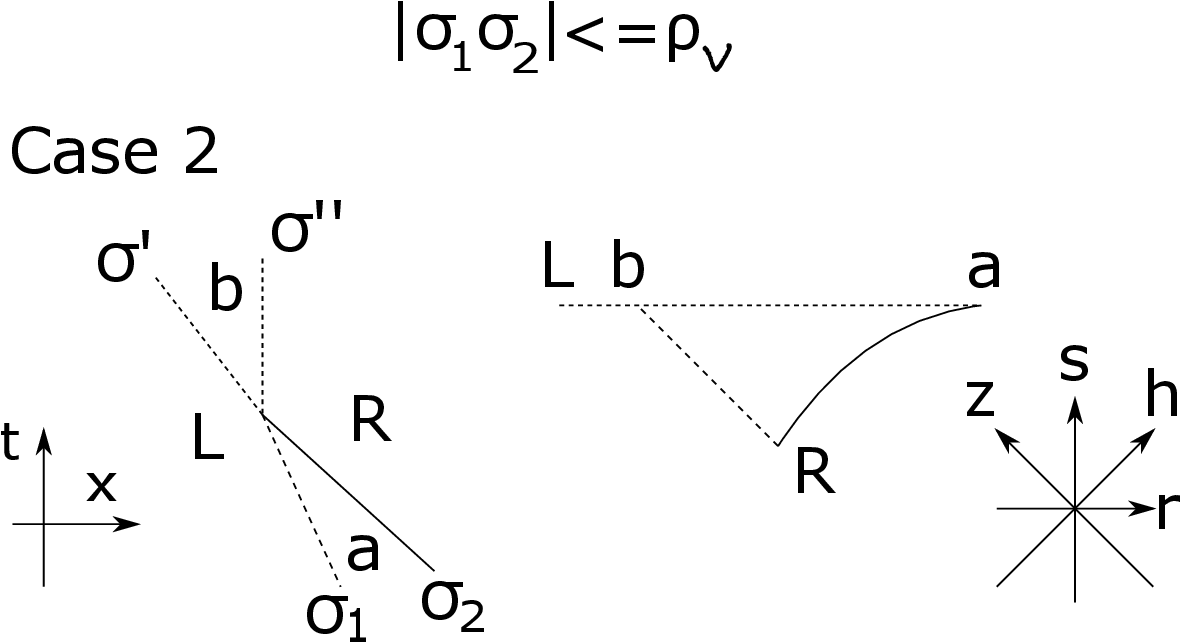}
	\caption{Simplified solver: Shock overtakes Rarefaction. Case 2 when rarefaction dominates.}\label{Simp5}
\end{figure}

\vskip0.3cm
\paragraph{\bf  (iii-3)}
\vskip0.1cm
For interaction between a non-physical shock and a rarefaction wave, we keep the strength of outgoing rarefaction, the slope of non-physical wave on the $(r,s)$-plane and its strength, see Figure \ref{Simp_non_2}. Since the strength of both waves do not change, Lemmas \ref{lemma_key1} and \ref{lemma_key2} clearly hold.

When a shock interacts with a non-physical wave, we first choose the outgoing shock strength to be the same as the incoming one. 
Then, by the similar proof as in \cite{Bressan}, the change of strength of non-physical wave is in the order of $|\sigma_1\sigma_2|$, and the slop of non-physical wave in the $(r,s)$-plane changes in the order of $\sigma_1$.   
Thus, if the slop of the incoming non-physical wave on the $(r,s)$-plane lies on a $\eps$-neighborhood of $-1$, then $h$ might decrease from the left to the right state for the outgoing non-physical wave.  In this case, we can always make the outgoing and incoming non-physical lines on the $(r,s)$-plane parallel by extending the outgoing shock curve, as in Figure \ref{Simp_non}, so that $h$ does not decrease from the $L$ to $b$. Since the strength of outgoing shock is not less than that of incoming shock, and thus it is bigger than $\rn$. Moreover, since  the increase of shock strength is in the order of $|\sigma_1\sigma_2|$, 
\[
\Delta L \le C|\sigma_1\sigma_2|.
\]
Therefore, Lemmas \ref{lemma_key1} and \ref{lemma_key2} hold.

\begin{figure}
	\centering
		\includegraphics
	[scale=.5]
	{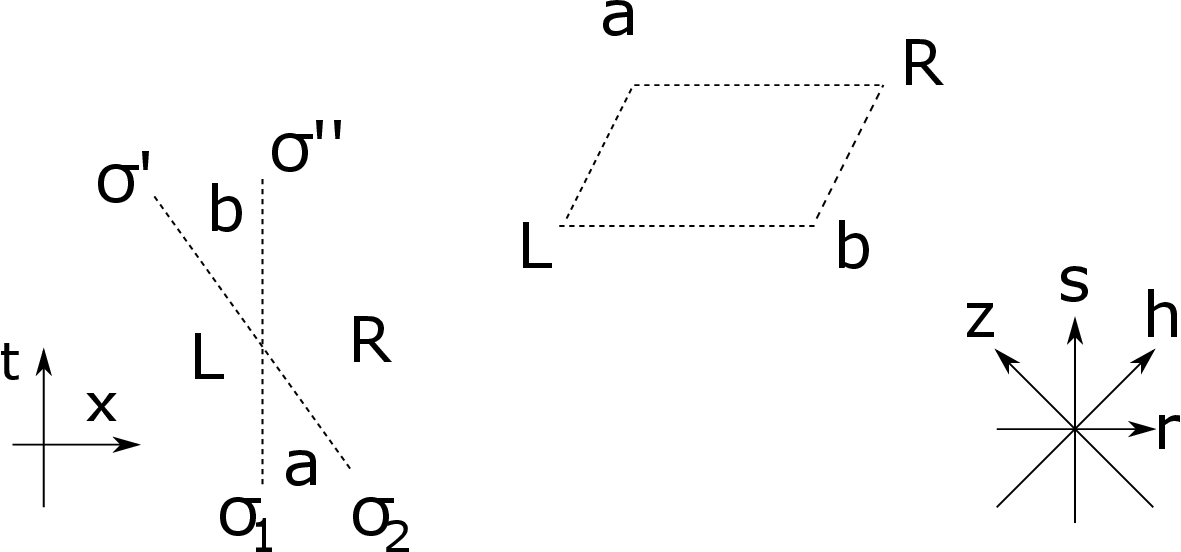}
	\caption{Simplified solver: rarefaction and non-physical wave. }\label{Simp_non_2}
\end{figure} 

\begin{figure}
	\centering
		\includegraphics
	[scale=.5]
	{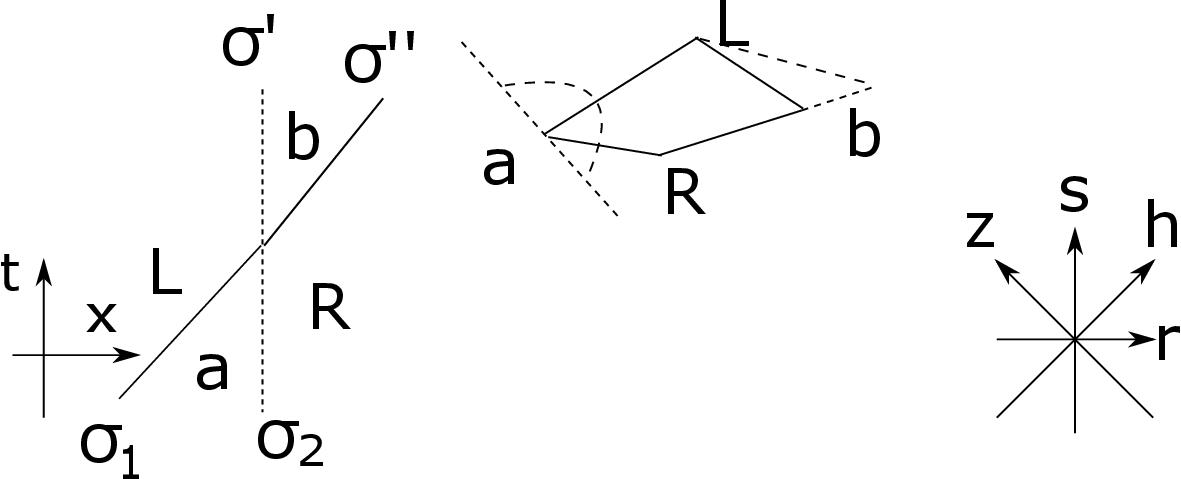}
	\caption{Simplified solver: shock and non-physical wave. The slope of curve $a$-$R$ might vary. We may need to adjust the slop of $L$-$b$ when the slope of curve $a$-$R$ is very close to $-1$, by slightly increasing the outgoing shock strengh.} \label{Simp_non}
\end{figure} 

\vskip0.1cm


\subsection{Construction of three kinds of waves}\label{sub:waves}
Given $\delta, \nu>0$, we will use the modified front tracking algorithm to construct an approximate solution $\overline U_{\nu,\delta}=(\bar v_{\nu,\delta}, \bar h_{\nu,\delta})$ composed of viscous shocks, rarefaction waves and smooth pseudo-shocks.
For that, we here introduce the three kinds of smooth waves.
We still use a parameter $\s$ to denote the sizes of waves, as introduced in Section \ref{subsection_4.2}.    \\
{\bf 1) Viscous shocks:} 
Consider 
\[
\tilde U(x-\tilde\lambda t;U_-,U_+,\s,k):=\big(\tilde v(x-\tilde\lambda t;v_-,v_+,\s,k),\tilde h(x-\tilde\lambda t;u_-,u_+,\s,k) \big)
\]
which is a $k$-viscous shock of Rankine-Hugoniot speed $\tilde\lam$ and size $\s<0$, as a traveling wave solution to \eqref{NS} with $\nu=1$:
\begin{align}
\begin{aligned}\label{sen}
\left\{ \begin{array}{ll}
       -\tilde\lambda \tilde v_x -\tilde h_x = - \Big( \mu(\tilde v) \tilde v^\gamma  p(\tilde v)_x\Big)_{x}\\
       -\tilde\lambda \tilde h_x + p(\tilde v)_x =0, \end{array} \right.
\quad \lim_{x\to-\infty}\tilde U(x)=U_- ,
\quad  \lim_{x\to \infty}\tilde U(x)=  U_+,
\end{aligned}
\end{align}
Then, we will consider $\tilde U^\nu(x-\tilde\lambda t;U_-,U_+,\s,k):=\tilde U(\frac{x-\tilde\lambda t}{\nu};U_-,U_+,\s,k)$ that is a $k$-viscous shock solution to \eqref{NS}. \\
{\bf 2) Viscous rarefactions:}
Define  $U^{R,\nu} (x;U_-,U_+,\s,k) = (v^{R,\nu}, h^{R,\nu})(x;U_-,U_+,\s,k)$ a smooth rarefaction as a smooth profile by 
\beq\label{defsmrare}
U^{R,\nu}(x;U_-,U_+,\s,k) = {\mathbf{R}}_{U_-}^k\Big(\s\phi \left(\s \frac{x}{\nu}\right) \Big) , \quad \s>0,
\eeq
where $\phi(x):=\frac12\int_{-\infty}^x e^{-|s|} ds $ (note: $\phi(+\infty)=1$), and $ {\mathbf{R}}_{U_-}^k$ is the rarefaction curve of $k$-family passing through $U_-={\mathbf{R}}_{U_-}^k(0)$, 
defined as in Appendix \ref{app:algorithm}. 
This wave connects $U_-$ to $U_+$. More precisely, $\lim_{x\to-\infty}U^{R,\nu}(x)=U_-$ and $ \lim_{x\to \infty} U^{R,\nu}(x)= {\mathbf{R}}_{U_-}^k (\s)=:U_+$.
Notice that $\s$ is the size from $U_-$ to $U_+$, by the definitions of $\sigma$ and arc-length for rarefaction.
Notice also that $\frac{d}{d\s}{\mathbf{R}}^k_{U_-}(\s) = {\mathbf{r}}_{k}({\mathbf{R}}^k_{U_-}(\s))$ for the right eigenvector ${\mathbf{r}}_{k}$ of $k$-family. \\
{\bf 3) Smooth pseudo-shocks:}
Define a  pseudo-shock $U^{P,\nu}(x;U_-,\s) =(v^{P,\nu} , h^{P,\nu})(x;U_-,\s)$ as a smooth stationary profile as follows.
 Let $\phi_{P}$ be a smooth monotone function such that $\phi_{P}(x)=0$ for $x\le0$, and $\phi_{P}(x)=1$ for $x\ge1$. Then, for the left end state $U_- = (v_-,u_-)$, we define $v^{P,\nu}(x)=v_- + \s \phi_P\Big(\frac{x}{\sqrt{\rho_\nu}}\Big)$ and $\partial_x h^{P,\nu}(x) \ge 0$, where the constant $\rho_{\nu}$ is chosen so small such that 
 \begin{equation}\label{pssize}
 \rho_\nu := \nu^{1/3}
 \end{equation}as  a threshold. We will make all sizes $\s$ of pseudo-shock to be smaller than $O(\rho_\nu)$ as in Lemma \ref{lemma_key2}.

\subsection{Construction of approximate solutions} \label{sub:conapp}
For positive parameters $\nu, \rho_\nu(=\nu^{1/3})$ and $\delta$ small enough, we here construct the desired approximate solutions. 
Eventually, we will take $\delta$ to be a function of $\nu$ in the end of Section \ref{sec:pfthm}.

\subsubsection{Initial step}

At initial time $t_0:=0$, consider a sequence of step functions $U_m$ that converges to the initial function $U^0$ in $L^2$ as $m\to\infty$. 
Notice that $\|U_m\|_{BV(\bbr)} \le \|U^0\|_{BV(\bbr)}\le \eps$.
Especially, for the two parameters $\nu,\delta$, we consider a step function $U_{n_{\nu,\delta}}$ (as a subsequence $n_{\nu,\delta}\to\infty$ as $\nu,\delta\to 0$) that takes a value $U^0(\bar x_{i0})$ on $(\bar x_{i0},\bar x_{i+1,0})$ for a finite number of points 
$\bar x_{10}<\bar x_{20}<\bar x_{30}<\cdots$ such that 
\beq\label{ptzero}
|\bar x_{i0}-\bar x_{i-1,0}| > \frac{\rn^{1/5}}{\delta}, \quad \forall i.
\eeq
Notice that the lower bound $\frac{\rn^{1/5}}{\delta}$ vanishes as $\nu\to 0$ and then $\delta\to 0$.\\
We apply the accurate Riemann solver at each point $x_{i0}$ connecting $U^0(\bar x_{i-1,0})$ and $U^0(\bar x_{i0})$. Then, we translate those  accurate solvers along $t$ axis as $-\sqrt\rn/\delta$ scale such that the intersection points $x_{10}<x_{20}<\cdots$ of the line $t=0$ and the wave fronts of the accurate solvers satisfy:
\beq\label{zerodis}
|x_{i+1,0}-x_{i0}|\ge d_0:= \rn^{1/4},\quad \forall i
\eeq
and  
\[
 \sum_{i\ge1} |U_{i-1/2}^0-U_{i+1/2}^0| =L(0) <C\eps,
\] 
where $U_{i-1/2}^0$ denotes the left constant state at $x_{i0}$ where $U_{1/2}^0:=U_*$, that is, the two constants $U_{i-1/2}^0$ and $U_{i+1/2}^0$ are connected by a wave-front at $x_{i0}$.
Indeed, we can get \eqref{zerodis} by \eqref{ptzero} and the facts that the number of fronts is at most $O(\delta^{-1})$, and
the difference of speeds of consecutive rarefaction fronts is at most $O(\delta)$ due to the definition of rarefaction fronts \eqref{raredef1}-\eqref{raredef2}.

We now use the above wave-fronts to construct the approximate solution as a viscous counterpart composed of physical viscous waves with possible shifts as follows.\\
First, if $U_{i-1/2}^0$ and $U_{i+1/2}^0$ are connected by a shock front at $x_{i0}$, then we consider a viscous shock with shift $X^\nu_{i0}(t)$ as
\beq\label{vsn1}
\tilde U_{i0}^\nu (x,t):=\tilde U_{i0}^\nu \Big(x-x_{i0}-\tilde\lambda_{i0} t -X^\nu_{i0}(t) ;U_{i-1/2}^0,U_{i+1/2}^0, \s_{i0}, k_{i0}\Big),
\eeq
which is a $k_{i0}$-viscous shock of size $\s_{i0}<0$, started at $(x_{i0},t_0)$ and shifted by an absolutely continuous function $X^\nu_{i0}(t)$ satisfying $X^\nu_{i0}(0)=0$ to be defined later, and $\tilde\lam_{i0}$ is the Rankine-Hugoniot speed. 
Notice that $\tilde U^\nu_{i0}:=(\tilde v^\nu_{i0},\tilde h^\nu_{i0})$ satisfies (with $j:=0$)
\begin{align}
\begin{aligned}\label{seo1}
\left\{ \begin{array}{ll}
      -\tilde\lambda_{ij} (\tilde v^\nu_{ij})_x -(\tilde h^\nu_{ij})_x = -\nu \Big( \mu(\tilde v^\nu_{ij}) (\tilde v^\nu_{ij})^\gamma  p(\tilde v^\nu_{ij})_x\Big)_{x}\\
      -\tilde\lambda_{ij} (\tilde h^\nu_{ij})_x + p(\tilde v^\nu_{ij})_x =0 \end{array} \right.
\end{aligned}
\end{align}
and thus,
\begin{align}
\begin{aligned}\label{SE1}
\left\{ \begin{array}{ll}
      (\tilde v^\nu_{ij})_t +\dot X^\nu_{ij}(\tilde v^\nu_{ij})_x  -(\tilde h^\nu_{ij})_x = -\nu \Big( \mu(\tilde v^\nu_{ij}) (\tilde v^\nu_{ij})^\gamma  p(\tilde v^\nu_{ij})_x\Big)_{x}\\
       (\tilde h^\nu_{ij})_t +\dot X^\nu_{ij}(\tilde h^\nu_{ij})_x  + p(\tilde v^\nu_{ij})_x =0. \end{array} \right.
\end{aligned}
\end{align}
Similarly, if $U_{i-1/2}^0$ and $U_{i+1/2}^0$ are connected by a rarefaction front at $x_{i0}$, then we consider a viscous rarefaction with shift $\Lambda^\nu_{i0}(t)$ as
\beq\label{urdef1}
U^{R,\nu}_{i0} (x,t):=U^{R,\nu}\Big(x-x_{i0}-\Lambda^\nu_{i0}(t);U_{i-1/2}^0,U_{i+1/2}^0, \s_{i0}, k_{i0}\Big), 
\eeq
which is of $k_{i0}$-family and size $\s_{i0}>0$, started at $(x_{i0},t_0)$ and shifted by a Lipschitz function $\Lambda^\nu_{i0} (t)$  satisfying $\Lambda^\nu_{i0} (0) =0$ to be defined in \eqref{rode}. 
Notice that this viscous wave moves at velocity $(\Lambda^\nu_{i0})'(t)$ instead of $\lambda_{k_{i0}} (U_{i+1/2}^0)$ defined by \eqref{raredef1}-\eqref{raredef2}. 
But, by \eqref{lamest}, we have 
\beq\label{rarewsp}
 \lambda_{k_{i0}} (U_{i-1/2}^0) \le (\Lambda^\nu_{i0})'(t) \le  \lambda_{k_{i0}} (U_{i+1/2}^0).
\eeq
For a short time $t>0$, we define an approximate solution $\overline U_{\nu,\delta}=(\bar v_{\nu,\delta}, \bar h_{\nu,\delta})$ by
\begin{align}
\begin{aligned} \label{apprr0}
\overline U_{\nu,\delta} (x,t) = \int_{-\infty}^x \left(\sum_{i\in \mathcal{S}_0} \partial_x \tilde U_{i0}^\nu (y,t)  +  \sum_{i\in \mathcal{R}_0} \partial_x U^{R,\nu}_{i0} (y,t)  \right) dy +U_*.
\end{aligned}
\end{align}
where  $\mathcal{S}_0, \mathcal{R}_0$ denote the sets of locations of shocks, rarefactions, respectively, at $t=t_0+$.\\
This will persist until the next interaction time, as follows.

\subsubsection{The first Interaction}
Let $y_{i0}^\nu (t)$ denote the trajectories for the above physical waves as
\beq\label{ycurve0}
y_{i0}^\nu(t):=
\left\{ \begin{array}{ll}
       x_{i0}+\tilde\lambda_{i0} t+X_{i0}^\nu(t),\quad\mbox{if } x_{i0}\in \mathcal{S}_0,  \\
       x_{i0}+\Lambda_{i0}^{\nu} (t) ,\quad\mbox{if } x_{i0}\in \mathcal{R}_0 . \end{array} \right.
\eeq
Define the first interaction time by
\beq\label{deftj}
t_1 := \inf\{ t ~|~ t > 0,\, \min_{i} |y_{i-1,0}^\nu(t)-y_{i0}^\nu(t)| = d_1:= \rn^{1/4} - \rn^{1/3} \}.
\eeq
Notice from \eqref{zerodis} that $t_1>0$ by $d_1 <  d_0$ and the continuity of $y_{i0}^\nu(t)$.

In order to consider only one interaction at $t_1$, we define
\[
i_0:=\min\{ i~|~ |y_{i-1,0}^\nu(t_{1})-y_{i0}^\nu(t_{1})| = d_{1} \}.
\]
So, for the two adjacent indices $i_0$ and $i_0+1$, we consider the interaction at $t_{1}$ between two waves moving along trajectories $y_{i_0,0}^\nu$ and $y_{i_0+1,0}^\nu$, that is, no interaction of waves along $y_{i_0}^\nu$ for $i\notin\{i_0, i_0+1\}$. 

\subsubsection{Initial points after the first interaction}
For the two incoming waves of sizes $\s_1:=\s_{i_0}$ and $\s_2:=\s_{i_{i_0+1}}$ at the first interaction, we apply the modified front tracking algorithm composed of the three types of solvers as in Section \ref{subsection_4.2}. 
For the two outgoing waves of sizes $\s'$ and $\s''$, newly generated by the modified front tracking algorithm, we define the location for
the initial points $x_{i_0,1}, x_{i_0+1,1}, x_{i_0+2,1},...$ as follows, where we may have $O(1/\delta)$ initial points if one of the outgoing waves, as a reflected wave, is a bunch of rarefaction fronts defined as in \eqref{raredef1}-\eqref{raredef2}.  
First, if $\s_1$ and $\s'$ have negative as the same family, i.e., they are both shocks, then we slightly perturb the arrival point $y_{i_0,0}^\nu(t_1)$ such that
\beq\label{ssx1}
|x_{i_0,1}-y_{i_0,0}^\nu(t_1)| < 3\sqrt\nu. 
\eeq
Likewise, if $\s_2$ and $\s''$ have negative as the same family, then we consider $x_{i_0+1,1}$ such that
\beq\label{ssx2}
|x_{i_0+1,1}-y_{i_0+1,0}^\nu(t_1)| < 3\sqrt\nu. 
\eeq
This is necessary in the proof of Proposition \ref{prop:deca} (see Remark \ref{rem:conf}).  

If an outgoing wave of size $\s'$ is a bunch of rarefaction fronts, then we translate those fronts in the sense of Appendix \ref{app:mod} so that the distance of any two rarefaction fronts is bigger than $d_1$. See Remark \ref{rem:mod} for the necessity of this adjustment.

Now, let $x_{i,1}$ be the initial points for trajectories $y_{i,2}^\nu(t)$ where all points generated by the first interaction are defined as above, and the other points are the same as the terminal points $y_{i,1}^\nu(t_1)$. \\
Then, it holds from \eqref{deftj} and \eqref{ssx1}-\eqref{ssx2} that for any two points $x_{i,1}$ and $x_{k,1}$, 
\beq\label{assinicon0}
|x_{i,1}-x_{k,1}|\ge \rn^{1/4}-2\rn^{1/3}.
\eeq
Indeed, that is true by $\nu\ll\delta$ such that $\frac{\sqrt\rn}{\delta}\ll \rn^{1/3}$.\\



\subsubsection{Inductive steps}
At $t=t_j$ for $j\ge 1$, let $\{x_{ij}\}$ be the set of initial points satisfy the following two conditions:\\
(i) For all rarefaction fronts generated by the interaction at $t_j$, we use the same translation as in Appendix \ref{app:mod} to find the initial points generated by rarefaction fronts such that the distance between any two points of them is bigger than $d_j:= \rn^{1/4} - (2j-1) \rn^{1/3}$.\\  
(ii) For any two points $x_{i,j}, x_{k,j}$ except for initial points generated by rarefaction fronts as above,
\beq\label{assinicon}
|x_{i,j}- x_{k,j}|\ge  \rn^{1/4}-2j\rn^{1/3}.
\eeq
For each $i$, let
\[
U_{i-1/2}^j, \s_{ij}, k_{ij}, x_{ij},
\] 
respectively denote the left end state, and size of wave of $k_{ij}$-characteristic field, starting at $x_{ij}$, which is obtained by  the modified front tracking algorithm.   \\
As in the initial step, we construct an approximate solution composed of smooth waves as follows.\\
If $U_{i-1/2}^j$ and $U_{i+1/2}^j$ are connected by a shock front at $x_{ij}$, then we consider
\beq\label{vsn}
\tilde U_{ij}^\nu (x,t):=\tilde U^\nu \Big(x-x_{ij}-\tilde\lambda_{ij} (t-t_j)-X^\nu_{ij}(t-t_j) ;U_{i-1/2}^j,U_{i+1/2}^j, \s_{ij}, k_{ij}\Big),
\eeq
which is a $k_{ij}$-viscous shock of speed $\tilde\lam_{ij}$ and size $\s_{ij}<0$, starting at $(x_{ij},t_j)$, and shifted by an absolutely continuous function $t\mapsto X^\nu_{ij}$ satisfying $X^\nu_{ij}(0)=0$ that is defined later.\\ 
Notice that $\tilde U^\nu_{ij}:=(\tilde v^\nu_{ij},\tilde h^\nu_{ij})$ satisfies
\begin{align}
\begin{aligned}\label{seo}
\left\{ \begin{array}{ll}
      -\tilde\lambda_{ij} (\tilde v^\nu_{ij})_x -(\tilde h^\nu_{ij})_x = -\nu \Big( \mu(\tilde v^\nu_{ij}) (\tilde v^\nu_{ij})^\gamma  p(\tilde v^\nu_{ij})_x\Big)_{x}\\
      -\tilde\lambda_{ij} (\tilde h^\nu_{ij})_x + p(\tilde v^\nu_{ij})_x =0 \end{array} \right.
\end{aligned}
\end{align}
and thus,
\begin{align}
\begin{aligned}\label{SE}
\left\{ \begin{array}{ll}
      (\tilde v^\nu_{ij})_t +\dot X^\nu_{ij}(\tilde v^\nu_{ij})_x  -(\tilde h^\nu_{ij})_x = -\nu \Big( \mu(\tilde v^\nu_{ij}) (\tilde v^\nu_{ij})^\gamma  p(\tilde v^\nu_{ij})_x\Big)_{x}\\
       (\tilde h^\nu_{ij})_t +\dot X^\nu_{ij}(\tilde h^\nu_{ij})_x  + p(\tilde v^\nu_{ij})_x =0. \end{array} \right.
\end{aligned}
\end{align}
If $U_{i-1/2}^j$ and $U_{i+1/2}^j$ are connected by a rarefaction front at $x_{ij}$, then we consider
\beq\label{urdef}
U^{R,\nu}_{ij} (x,t):=U^{R,\nu}\Big(x-x_{ij}-\Lambda^\nu_{ij} (t-t_j);U_{i-1/2},U_{i+1/2}^j, \s_{ij}, k_{ij}\Big), 
\eeq
that is a rarefaction of $k_{ij}$-family, as a traveling wave starting at $(x_{ij},t_j)$, shifted by a Lipschitz function $t\mapsto \Lambda_{ij}^\nu$ (to be defined in \eqref{rode}).

If $U_{i-1/2}^j$ and $U_{i+1/2}^j$ are connected by a pseudo-shock at $x_{ij}$, then we consider
\[
U^{P,\nu}_{ij} (x):=U^{P,\nu}\Big(x-x_{ij};U_{i-1/2},U_{i+1/2}^j, \s_{ij}\Big),
\]
which is a smooth pseudo-shock of size $\s_{ij}$, located at $x_{ij}$.\\

Let $\mathcal{S}_j, \mathcal{R}_j, \mathcal{NP}_j$ be the sets of locations of shocks, rarefactions and pseudo-shocks, respectively, at $t=t_j+$. \\
For a short time $t>t_j$, we define an approximate solution $\overline U_{\nu,\delta}=(\bar v_{\nu,\delta}, \bar h_{\nu,\delta})$ by
\begin{align}
\begin{aligned} \label{apprr}
\overline U_{\nu,\delta} (x,t) = \int_{-\infty}^x \left(\sum_{i\in \mathcal{S}_j} \partial_x \tilde U_{ij}^\nu (y,t)  +  \sum_{i\in \mathcal{R}_j} \partial_x U^{R,\nu}_{ij} (y,t) +  \sum_{i\in \mathcal{NP}_j} \partial_x U^{P,\nu}_{ij} (y) \right) dy +U_*,
\end{aligned}
\end{align}
This persists until the next interaction time, which is determined as follows.\\
For simplicity, let $y_{ij}^\nu (t)$ denote the continuous trajectories for the above waves as
\beq\label{ycurve}
y_{ij}^\nu(t):=
\left\{ \begin{array}{ll}
       x_{ij}+\tilde\lambda_{ij} (t-t_j)+X_{ij}^\nu(t-t_j),\quad\mbox{if } x_{ij}\in \mathcal{S}_j,  \\
       x_{ij}+\Lambda_{ij}^{\nu} (t-t_j) ,\quad\mbox{if } x_{ij}\in \mathcal{R}_j, \\
       x_{ij} ,\quad\mbox{if } x_{ij}\in \mathcal{NP}_j . \end{array} \right.
\eeq

Since it holds from \eqref{lamest} that
\[
 \lambda_{k_{ij}} (U_{i-1/2}^j) \le (\Lambda^\nu_{ij})'(t-t_j) \le  \lambda_{k_{ij}} (U_{i+1/2}^j),
\]
we see that for all $x_{ij}, x_{kj}\in \mathcal{R}_j$ satisfying the above condition (i),
\beq\label{rareyij}
|y_{ij}^\nu(t)-y_{kj}^\nu(t)|\ge 2\sqrt\rn.
\eeq
This estimate will be crucial to control the wave interaction between any two rarefaction waves as in 5) of Lemma \ref{lem:waves}.

We now define the next interaction time by
\beq\label{deftjg}
t_{j+1}:= \inf\{ t ~|~ t > t_{j},\, \min_{i,k: \mbox{approaching waves}} |y_{ij}^\nu(t)-y_{kj}^\nu(t)| = d_{j+1}\},
\eeq
where \beq\label{defdj}
d_{j}:= \rn^{1/4} - (2j-1) \rn^{1/3}.
\eeq
Notice that $t_{j+1}>t_{j}$. Indeed, that is true because of $d_{j+1}= \rn^{1/4} - (2j+1) \rn^{1/3}<  \rn^{1/4}-2j\rn^{1/3}$ and the fact that 
rarefaction fronts of the same family (satisfying \eqref{rareyij}) do not approach, and the waves starting only from the initial points satisfying \eqref{assinicon} can interact.\\

As before, let
\[
i_0:=\min\{ i~|~ |y_{i-1,j}^\nu(t_{j+1})-y_{ij}^\nu(t_{j+1})| = d_{j+1} \}.
\]

So, we consider only one interaction at $t_{j+1}$ between two waves moving along trajectories $y_{i_0,j}^\nu$ and $y_{i_0+1,j}^\nu$, that is, no interaction of waves along $y_{ij}^\nu$ for $i\notin\{i_0, i_0+1\}$. 

\begin{remark} \label{rem:conf}
In the proof of Proposition \ref{prop:deca}, we need a special configuration for the initial points of the outgoing waves, as follows. First, we fix the center of each viscous wave moving along the associated trajectory, by making the average point of the end states of each wave starting at $x_{i}$ lie at that point $x_{i}$. 
As in Figures \ref{Fig1} and \ref{pic_a_0}, let $x_{L\pm}, x_{m\pm}, x_{M\pm}, x_{R\pm}$ be the boundary points of the transition zones of width $\sqrt\nu$ for  physical waves. In particular, we need to consider the special configuration as in Figure \ref{pic_a_0}  when the condition of the adjustment in the proof of Proposition \ref{prop:deca} is satisfied.  For other cases,  we simply identify the arrival points of incoming waves and the initial points of outgoing waves at the interaction time, that is, $x_{L-}=x_{L+}, x_{m-}=x_{m+}, x_{M-}=x_{M+}, x_{R-}=x_{R+}$. 
This configuration for initial points of outgoing waves does not affect the number of waves and interactions.
\end{remark}

\begin{remark} \label{rem:mod}
When a bunch of reflected rarefaction fronts interacts with another approaching wave after the $j$-th interaction time $t_j$, in which case, the interaction between the approaching wave and the closest rarefaction front occurs, the adjustment as in Appendix \ref{app:mod} is essential to prevent successive interactions between the approaching wave and the remaining rarefaction fronts within a distance far smaller than $d_{j+1}$.
\end{remark}

\begin{remark} \label{rem:djtj}
By \eqref{deftjg} we first find that for all $i,k$ approaching waves,
\[
|y_{ij}^\nu(t)-y_{kj}^\nu(t)|\ge d_{j+1}= \rn^{1/4} - (2j+1) \rn^{1/3} \quad \forall t\in (t_j,t_{j+1}),
\]
which together with Lemma \ref{lem:num} below, and taking $\nu\ll \delta$ implies
\[
d_{j+1}\ge  \rn^{1/4}  - C\left(\frac{1}{\delta} \log \frac{1}{\rho_\nu}\right)^3 \rn^{1/3} \gg 2\sqrt\rn.
\]
This and \eqref{rareyij} imply that for any $j$,
\beq\label{distij}
|y_{ij}^\nu(t)-y_{i'j}^\nu(t)|\ge 2 \sqrt\rn , \quad\mbox{for any $i,i'$ with $i\neq i'$ and any } t\in (t_j,t_{j+1}).
\eeq
\end{remark}



\subsection{The number of waves and interactions}

We here estimate the total number of waves and wave interactions as follows.

\begin{lemma}\label{lem:num}
There exists $C>0$ such that for all $t>0$, the following holds:
\begin{itemize}
\item[i.]
The total number of physical waves is bounded by 
$$C\Big(\frac{1}{\delta}\ln \rho_\nu^{-1}\Big).$$
\item[ii.]
The total number of all waves is bounded by 
$$C\Big(\frac{1}{\delta}\ln \rho_\nu^{-1}\Big)^2.$$
\item[iii.] The total number of wave interactions is bounded by 
$$C\Big(\frac{1}{\delta}\ln \rho_\nu^{-1}\Big)^3.$$
\end{itemize}
\end{lemma}

\subsubsection{Proof of Lemma \ref{lem:num}}
We first need to define the meaning of new wave, and the order of generation. 
We adopt a similar concept on the order of generation as in the book of Bressan \cite{Bressan}. First, we say:
\begin{itemize}
\item All fronts generated by the accurate Riemann solvers at the initial time $t=0$ have generation order $k=1$.
\end{itemize}
Then we define the orders of outgoing waves after each Riemann solver as follows.
\begin{itemize}
\item[(1).] Order of wave in each family, for all head-on interactions of a $2$-wave and a $1$-wave (always using accurate solver) and all simplified solvers including a non-physical shock,  does not change.
\end{itemize}
For overtaking interactions of two $i$-waves with generation orders $k, k'$, for a fixed $i=1,2$, there might be outgoing waves in three families. For convenience, we regard a non-physical shock as a wave in the $1.5$-th family. Then we define
\begin{itemize}
\item[(2).] For simplified solvers of rarefaction-shock interactions (including $\ba S\ba R$, $\ba R\ba S$, $\fa S\fa R$ and $\fa R\fa S$): In Case 1 when the incoming shock dominates, the outgoing physical shock takes the same order of the incoming shock, the non-physical shock takes the same order of the incoming rarefaction; In Case 2, when the incoming rarefaction dominates, two outgoing waves both take the same order of the incoming rarefaction.
\item[(3).] For other overtaking interactions, the outgoing $i$-wave has order $\min(k,k')$, and the outgoing wave in other families has order  $\max(k,k')+1$.
\end{itemize}

Then we introduce an important lemma to estimate the strength for waves with generation order $k$.
\begin{lemma}\label{lemma_key5} There exists a constant $C$ such that for any positive integer $k$, the strength of any wave with generation order $k$ is less than $(C\varepsilon)^k$.
\end{lemma}
This lemma can be proved using the same method as in \cite[(7.75)]{Bressan}. In fact, we only have to show that the estimates in \cite[$(7.68)_{1-5}$]{Bressan} hold for our front tracking solutions, then the rest of the proof is entirely same. When we define the order of wave, we treat the interactions in group (2) separately, because the outgoing non-physical shock in this case may be in the first order, i.e. in the order of $\sigma_1$ or $\sigma_2$. So we consider the non-physical shock as an extension of the whole or a part of the incoming rarefaction, i.e. keep its order. It is not hard to check that the estimates in \cite[$(7.68)_{1-5}$]{Bressan} still hold. For interactions in other groups, we can use the fact that the strength of each reflected outgoing wave is at most in the second order, i.e. $O(|\sigma_1\sigma_2|)$, to show the estimates in \cite[$(7.68)_{1-5}$]{Bressan} as Bressan did in his book. We omit the details.

\vskip0.3cm

Next, we define the new wave and the continue wave. After each interaction, if an outgoing wave is in the same family and same order of an incoming wave, then these two waves are regarded as the same wave, i.e. the outgoing one is the continuation of the existing incoming wave. Other outgoing wave(s) are considered to be new waves. It is important to notice that any pair of different waves will interact at most {\emph{once}} in their lifetimes, because no wave will change its family (unless it changes to a new wave), and when two waves in the same family meet, there will be only one left after their interaction.\\

We are ready to prove Lemma \ref{lem:num} on the number of waves and interactions.

First, since the strength of each physical wave is larger than $C\rho_\nu$ by Lemma \ref{lemma_key2} and the strength of any wave of any order $k$ (generation) is less than $(C\varepsilon)^k$ in Lemma \ref{lemma_key5}, we find that $(C\eps)^k \ge \rn$, which implies that 
 \beq\label{maxog}
\mbox{ the maximum order of physical wave is less than } C\frac{ \ln \rho_\nu^{-1}}{\ln \eps^{-1}}.
\eeq
 We note that the order of non-physical wave will not impact the order of physical wave. The order of non-physical wave is at most one larger than the maximum order of physical waves.

Next, we count the number of physical waves in each generation.

Let $N^0_{\mathcal F}$ denote the number of wave fronts in the first generation, i.e., at $t=0+$.

After each interaction, a physical wave of new order will be created only by the reflected wave in the overtaking accurate or adjusted solvers.

We always consider the reflected outgoing rarefaction after a shock-shock overtaking interaction (as in Lemma \ref{lemma_key3}), as {\em switched} from one incoming shock with order 1 less, 
where another incoming shock is {\em continued} by the (same) wave in its own family with the same order. The other reflected outgoing rarefactions are called {\em non-switched}. For other overtaking interaction, there are at most two outgoing physical waves, including one extended wave and one switched reflected wave.

In summary, a physical wave in the new order is created either by a switched wave or a non-switched wave. Each physical wave can only generate one switched wave in the next order, in their whole life. 
So the total number of switched physical waves in each generation is not larger than $N^0_{\mathcal F}$.

Here, the wave strength $\sigma$ (arc-length) for each non-switched outgoing rarefaction is larger than $\delta/2$, when it is generated, by the adjustment for rarefaction as in Section \ref{subsection_4.2}. In addition, Lemma \ref{prop1} implies that the decay of $Q$ is larger than the total (all time) strength of reflected rarefactions  (for shock-shock interaction), by which the total strength of non-switched rarefactions when generated is less than $Q=O(\eps^2)$. Therefore, the total number of non-switched rarefactions is at most in the order of $C(\eps^2/\delta)$.

We have shown that 
\[
\hbox{the number of all physical waves in each generation is less than } C(N^0_{\mathcal F}+\eps^2/\delta).
\]
This with \eqref{maxog} implies that  the total (all time) number of all physical waves is less than 
$$C\left(\frac{ \ln \rho_\nu^{-1}}{\ln \eps^{-1}} \Big(N^0_{\mathcal F}+\eps^2/\delta\Big)\right),$$
which is less than $\frac{C}{\delta}\ln \rho_\nu^{-1}$ since $N^0_{\mathcal F} = O(1/\delta)$.

Since a new non-physical wave can be generated only by an interaction between two physical waves,  the total number of all waves is bounded by 
$$C\Big(\frac{1}{\delta}\ln \rho_\nu^{-1}\Big)^2.$$

Hence, the total number of wave interactions is bounded by 
$$C\Big(\frac{1}{\delta}\ln \rho_\nu^{-1}\Big)^3,$$
since interactions only happen between two physical waves or between one physical wave and one non-physical wave. \\

\begin{remark}\label{rem6.3}
To prove the convergence of the algorithm, we need the total strength of non-physical shocks converges to zero as $\rho_\nu\rightarrow 0$, i.e.,
\[
\rn\Big(\frac{1}{\delta}\ln \rho_\nu^{-1}\Big)^2 \to 0,\quad\hbox{as}\quad \rho_\nu\rightarrow 0.
\]
This holds by $\rho_\nu\ll \delta^3\ll 1$.
\end{remark}

Since the BV norm of the inviscid front tracking approximate solution is less than $C\varepsilon_0$, using the standard argument of the decay of Glimm potential, 
we obtain  \eqref{BVbarU}.

%

\section{Weighted relative entropy with shifts}\label{sec:wrel}
\setcounter{equation}{0}

Let $0=:t_0<t_1<t_2<\cdots$ denote a sequence of the interaction times.

We here obtain the desired estimates for a weighted relative entropy with shifts on a fixed time interval $(t_j,t_{j+1})$, $j\ge0$. \\

For notational simplicity, without confusion, we omit the dependence of the above solution and quantities on the parameters $\nu, \delta$ and the $j$-th interval, that is,
\beq\label{abbu}
U:=(v,h):= U^\nu .
\eeq
and
\begin{align}
\begin{aligned} \label{abb}
&\overline U:=\overline U_{\nu,\delta},\quad (\bar v, \bar h):=(\bar v_{\nu,\delta}, \bar h_{\nu,\delta}),\quad \tilde U_{i}:=\tilde U_{ij}^\nu,\quad U^R_{i}:=U^{R,\nu}_{ij},\quad U^P_{i}:=U^{P,\nu}_{ij} ,\\
&U_{i-1/2}:=U_{i-1/2}^j,\quad \s_i:=\s_{ij},\quad \tilde\lambda_i:=\tilde\lambda_{ij}, \quad X_i:=X^\nu_{ij},\quad \Lambda_{i}:= \Lambda^\nu_{ij},\quad k_i:=k_{ij}, \quad y_i := y_{ij}^\nu .
\end{aligned}
\end{align}

\subsection{Construction of weights}\label{subsec-weight}
First of all, for each shifted viscous shock $ \tilde U_{i} = ( \tilde v_{i}, \tilde h_{i})$ defined by \eqref{vsn}, we define the associated weight $a_i$ by
\[
a^\nu_i(x,t)= 1-\frac{p(\tilde v_i)-p(v^j_{i\pm 1/2})}{\delta_*}
\]
so that
\beq\label{dera}
(a^\nu_i)_x = -\frac{\partial_x p(\tilde v_i)}{\delta_*},
\eeq 
where $\delta_*$ is to be chosen sufficiently small (such that $\delta_*< \delta_1^{10}$ for the truncation size $\delta_1$ later), but depends only on the reference point $U_*$.\\
Notice that
\beq\label{avr}
|(a^\nu_i)_x| \sim \frac{|(\tilde v_i)_x|}{\delta_*} \sim  \frac{|(\tilde h_i)_x|}{\delta_*} .
\eeq
Since
\beq\label{sigmai}
|\s_i| \le L(0) \le \ds^2 \quad \mbox{for any shock sizes,} 
\eeq
we have
\[
\|(a^\nu_i)_x\|_{L^1} \le C\frac{|\s_i|}{\delta_*} \le  C\frac{L(0)}{\delta_*} < C\delta_*.
\]
Based on the above weights, we will construct a global weight function that is non-increasing in time.
For that, it would be easier to consider the following two functions on the total variation and the Glimm potential for the jump sizes of pressure:
\[
 \bar L(t) :=\sum_i |\bar\s_i|,
\]   
and
\[
\bar Q(t)=\sum_{j,i:\hbox{approaching waves}}|\bar\s_i| |\bar\s_j|,
\]
where each $\bar\s_i$ represents the jump size of pressure, more precisely,
\begin{align}
\begin{aligned}\label{barsi}
\bar\s_i := 
\left\{ \begin{array}{ll}
       (-1)^{k_i} \int_{-\infty}^\infty \partial_x  p(\tilde v_i) dx,\quad\mbox{if } i\in \mathcal{S}_j,  \\
      (-1)^{k_i} \int_{-\infty}^\infty \partial_x  p(v^R_i)  dx,\quad\mbox{if } i\in \mathcal{R}_j, \\
       (-1)^{k_i} \int_{-\infty}^\infty \partial_x p(v^P_i) dx,\quad\mbox{if } i\in \mathcal{NP}_j . \end{array} \right.
\end{aligned}
\end{align}
Note that $\bar\s_i<0$ for each shock and  $\bar\s_i>0$ for each rarefaction, and
$\bar\s_i$ is equivalent to $\s_i$ :
\[
C^{-1} |\s_i| \le |\bar\s_i| \le C|\s_i| .
\]

We now define the (global) weight $a^\nu$ by
\beq\label{defweight}
a^\nu(x,t) = 1+ \frac{1}{\ds} \Big( \bar L(t)+\kappa \bar Q(t) \Big) + \int_{-\infty}^x \sum_{i\in \mathcal{S}_j} \partial_x a_i^\nu(y,t) dy,
\eeq
where $\kappa>0$ is the same constant as before.\\
Equivalently, we have from \eqref{dera} that
\beq\label{defweight1}
a^\nu(x,t) = 1+ \frac{1}{\ds} \bigg[ \Big(\bar L(t)+\kappa \bar Q(t) \Big) + \int_{-\infty}^x \sum_{i\in \mathcal{S}_j}  \partial_x \Big(- p(\tilde v_i)\Big) dy \bigg],
\eeq
In Section \ref{sec:pfthm}, we will show that the weight $a^\nu$ is non-increasing in time, and satisfies  
\[
\|a^\nu(x,t)-1\|_{L^\infty(\bbr)} \ll1,\quad\forall t.
\] 
Again, without confusion, we omit the dependence of the above quantities on the parameter $\nu$, that is,
\beq\label{abba}
a := a^\nu,\quad a_i := a_i^\nu .
\eeq

\subsection{Relative entropy method}
First, we rewrite \eqref{NS} into the  viscous hyperbolic system of conservation laws:
\beq\label{system}
\partial_t U +\partial_x A(U)= { -\nu\partial_{x}\big(\mu(v)v^\gamma \partial_x p(v)\big) \choose 0},
\eeq
where 
\[
U:={v \choose h},\quad A(U):={-h \choose p(v)}.
\]
Consider the entropy $\eta(U):=\frac{h^2}{2}+Q(v)$ in the Lagrangian coordinate, where $Q(v)=\frac{v^{-\gamma+1}}{\gamma-1}$, i.e., $Q'(v)=-p(v)$.  From now on, we will drop the subscript $L$ in the notation since we are working only in the Lagrangian coordinates.  
We use the same notation for the entropy as \eqref{defentpair} in the Eulerian coordinate without confusion. \\
Consider the relative entropy functional defined by
\beq\label{defent}
\eta(U|V)=\eta(U)-\eta(V) -\nabla\eta(V) (U-V),
\eeq
and the relative flux defined by
\beq\label{defa}
A(U|V)=A(U)-A(V) -\nabla A(V) (U-V).
\eeq
Let $G(\cdot;\cdot)$ be the flux of the relative entropy defined by
\beq\label{defg}
G(U;V) = G(U)-G(V) -\nabla \eta(V) (A(U)-A(V)),
\eeq
where $G$ is the entropy flux of $\eta$, i.e., $\partial_{i}  G (U) = \sum_{k=1}^{2}\partial_{k} \eta(U) \partial_{i}  A_{k} (U),\quad 1\le i\le 2$.

The relative entropy method implies that for any functions $U, V$,
\[
- \big(\nabla\eta(U)-\nabla\eta(V) \big)\partial_x A(U) + \nabla^2\eta(V) (U-V) \partial_x A(V) = -\partial_x G(U;V) - \partial_x \nabla\eta(V) A(U|V).
\]
Thus, for a solution $U$ of \eqref{system} and any function $V(x,t)$, it holds that
\begin{align}
\begin{aligned} \label{etaform}
\partial_t \eta(U | V ) &= \big(\nabla\eta(U)-\nabla\eta(V) \big)\partial_t U - \nabla^2\eta(V) (U-V) \partial_t V\\
&= - \big(\nabla\eta(U)-\nabla\eta(V) \big)\partial_x A(U) + \nabla^2\eta(V) (U-V) \partial_x A(V) \\
&\quad - \nabla^2\eta(V) (U-V) \Big(\partial_t V +\partial_x A(V)\Big) +  \big(\nabla\eta(U)-\nabla\eta(V) \big) { -\nu\partial_{x}\big(\mu(v)v^\gamma \partial_x p(v)\big) \choose 0}\\
&=-\partial_x G(U;V) - \partial_x \nabla\eta(V) A(U|V)- \nabla^2\eta(V) (U-V) \Big(\partial_t V +\partial_x A(V)\Big)\\
&\quad+  \big(\nabla\eta(U)-\nabla\eta(V) \big) { -\nu\partial_{x}\big(\mu(v)v^\gamma \partial_x p(v)\big) \choose 0}.
\end{aligned}
\end{align}
Now, consider the approximate solution $V=\overline U$ as the second component.\\
First, we observe
\begin{align*}
\begin{aligned}
\partial_t \overline U + \partial_x A(\overline U) 
&= \sum_{i\in \mathcal{S}} \Big( \partial_t \tilde U_i + \nabla A(\overline U) \partial_x \tilde U_i \Big) \\
&\quad +  \sum_{i\in \mathcal{R}} \Big( \partial_t U^R_i + \nabla A(\overline U) \partial_x U^R_i \Big) +  \sum_{i\in \mathcal{NP}} \nabla A(\overline U)  \partial_x U^P_i .
\end{aligned}
\end{align*}
Since it holds from  \eqref{SE} that $ \tilde U_i$ satisfies 
\[
\partial_t  \tilde U_i +\dot X_i \partial_x \tilde U_i +\partial_x A(\tilde U_i)= { -\nu\partial_{x}\big(\mu(\tilde v_i )\tilde v_i ^\gamma \partial_x p(\tilde v_i )\big) \choose 0},
\]
and it holds from \eqref{defsmrare} and \eqref{urdef} that 
\[
\partial_t U^R_i = -\dot\Lambda_i \partial_x U^R_i,
\]
we have
\begin{align*}
\begin{aligned}
\partial_t \overline U + \partial_x A(\overline U) 
&=  \sum_{i\in \mathcal{S}} \left[ -\dot X_i \partial_x \tilde U_i + \big( \nabla A(\overline U) - \nabla A(\tilde U_i) \big) \partial_x \tilde U_i +{ -\nu\partial_{x}\big(\mu(\tilde v_i )\tilde v_i ^\gamma \partial_x p(\tilde v_i )\big) \choose 0} \right] \\
&\quad +  \sum_{i\in \mathcal{R}} \left[\Big( -\dot \Lambda_i  +  \nabla A(U^R_i)\Big)  \partial_x U^R_i  + \Big(\nabla A(\overline U) -  \nabla A(U^R_i)\Big)  \partial_x U^R_i \Big) \right]  \\
&\quad +  \sum_{i\in \mathcal{NP}}  \nabla A(\overline U) \partial_x U^P_i.
\end{aligned}
\end{align*}
Thus, substituting the above equality into \eqref{etaform}, and using 
\begin{align}
\begin{aligned} \label{formbase}
&\nabla\eta(U)={-p(v)\choose h},\quad \nabla^2\eta(U)={-p'(v)\quad 0\choose 0\quad\quad 1}, \quad \nabla A(U)={0\quad -1 \choose p'(v)\quad 0},\\
& A(U|\overline U)={0 \choose p(v|\bar v)}, \quad G(U;\overline U)=(p(v)-p(\bar v)) (h-\bar h) ,
\end{aligned}
\end{align}
we have
\begin{align*}
\begin{aligned}
\partial_t \eta(U | \overline U ) 
&= -\partial_x \Big((p(v)-p(\bar v)) (h-\bar h)\Big)   -  \sum_{i\in \mathcal{S}} \partial_x \tilde h_{i} p(v|\bar v) \underbrace{- \sum_{i\in \mathcal{NP}} \partial_x h^P_{i} p(v|\bar v) }_{=: I}
-\sum_{i\in \mathcal{R}}  \nabla^2\eta(\overline U)  \partial_x U^R_i  A(U|\overline U) \\
&\quad +\nabla^2\eta( \overline U) (U- \overline U)   \sum_{i\in \mathcal{S}} \dot X_i \partial_x \tilde U_i  - (h-\bar h)  \sum_{i\in \mathcal{S}} (p'(\bar v)-p'(\tilde v_i)) (\tilde v_i)_x \\
&\quad \underbrace{- p'(\bar v) (v-\bar v)  \sum_{i\in \mathcal{S}} \nu\partial_{x}\big(\mu(\tilde v_i )\tilde v_i ^\gamma \partial_x p(\tilde v_i )\big)}_{=:J} - (h-\bar h)  \sum_{i\in \mathcal{R}} (p'(\bar v)-p'(v^R_i)) (v^R_i)_x\\
&\quad  -\nabla^2\eta( \overline U) (U-\overline U) \sum_{i\in \mathcal{R}} \Big( -\dot \Lambda_i  +  \nabla A(U^R_i)\Big)  \partial_x U^R_i   \\
&\quad   - (h-\bar h)  \sum_{i\in \mathcal{NP}} p'(\bar v) \partial_x v^P_i 
+\nu (p(v)-p(\bar v) ) \partial_{x}\big(\mu(v ) v^\gamma \partial_x p(v )\big) .
\end{aligned}
\end{align*}
Notice that by the third property of Lemma \ref{lemma_key2},
\[
\partial_x h_i^P \ge 0.
\]
So, we have $I\le 0$.\\
For $J$, using the fact that (by \eqref{seo})
\beq\label{seq2}
\tilde\lambda_i \partial_x \tilde v_i +\partial_x \tilde h_i  = \nu\partial_{x}\big(\mu(\tilde v_i )\tilde v_i ^\gamma \partial_x p(\tilde v_i )\big) ,
\eeq
we have
\begin{align*}
\begin{aligned}
J &=- p'(\bar v) (v-\bar v)  \sum_{i\in \mathcal{S}} \big( \tilde\lambda_i \partial_x \tilde v_i +\partial_x \tilde h_i \big) \\
& = p(v|\bar v)  \sum_{i\in \mathcal{S}} \big( \tilde\lambda_i \partial_x \tilde v_i +\partial_x \tilde h_i \big) - (p(v)-p(\bar v))  \sum_{i\in \mathcal{S}} \big( \tilde\lambda_i \partial_x \tilde v_i +\partial_x \tilde h_i \big).
\end{aligned}
\end{align*}
Especially, for the last term above, we use $\tilde\lambda_i \partial_x \tilde v_i +\partial_x \tilde h_i =\nu\partial_x \big(\mu(\tilde v_i) \tilde v_i^{\gamma} \partial_x p(\tilde v_i) \big)$ by \eqref{seo}.
Thus, we have
\begin{align*}
\begin{aligned}
\partial_t \eta(U | \overline U ) 
&\le -\partial_x \Big((p(v)-p(\bar v)) (h-\bar h)\Big)  +  \sum_{i\in \mathcal{S}}  \tilde\lambda_i \partial_x \tilde v_i  p(v|\bar v) -\sum_{i\in \mathcal{R}}  \nabla^2\eta(\overline U)  \partial_x U^R_i  A(U|\overline U)  \\
&\quad +\nabla^2\eta(\overline U) (U- \bar U)   \sum_{i\in \mathcal{S}} \dot X_i \partial_x \tilde U_i  - (h-\bar h)  \sum_{i\in \mathcal{S}} (p'(\bar v)-p'(\tilde v_i)) (\tilde v_i)_x \\
&\quad - (h-\bar h)  \sum_{i\in \mathcal{R}} (p'(\bar v)-p'(v^R_i)) (v^R_i)_x\\
&\quad  -\nabla^2\eta(\overline U) (U- \overline U) \sum_{i\in \mathcal{R}} \Big( -\dot \Lambda_i  +  \nabla A(U^R_i)\Big)  \partial_x U^R_i   \\
&\quad   - (h-\bar h)  \sum_{i\in \mathcal{NP}} p'(\bar v) \partial_x v^P_i +\nu (p(v)-p(\bar v) ) \partial_{x}\Big(\mu(v ) v^\gamma \partial_x p(v ) - \sum_{i\in \mathcal{S}} \mu(\tilde v_i) \tilde v_i^{\gamma} \partial_x p(\tilde v_i)  \Big) .
\end{aligned}
\end{align*}

Therefore, we have
\begin{align}
\begin{aligned}\label{ineq-0}
&\frac{d}{dt}\int_{\bbr} a(x,t) \eta(U(x,t)|\overline U(x,t) ) dx \\
&\quad\le \sum_{i\in \mathcal{S}}\dot X_i(t) Y_i^\nu(U)  +\mathcal{H}_1^\nu(U)+\mathcal{H}_2^\nu(U)+\mathcal{H}_3^\nu(U) + \mathcal{P}^\nu(U) - \mathcal{J}^\nu_{good}(U),
\end{aligned}
\end{align}
where
\begin{align}
\begin{aligned}\label{ybg-first}
&Y^\nu_i(U):= -  \int_{\bbr} (a_i)_x  \eta(U|\overline U ) dx + \int_{\bbr} a \nabla^2\eta(\overline U) (U- \bar U)   \partial_x \tilde U_i dx ,\\
&\mathcal{H}_1^\nu(U):= \sum_{i\in \mathcal{S}} \bigg[  \int_\bbr (a_i)_x \big(p(v)-p(\bar v)\big) \big(h-\bar h \big) dx +  \tilde\lam_i \int_\bbr a (\tilde v_i)_x  p(v| \bar v) dx \bigg]\\
&\quad\quad\quad   - \sum_{i\in \mathcal{S}}\int_{\bbr}  a   (h-\bar h)  (p'(\bar v)-p'(\tilde v_i)) (\tilde v_i)_x  dx ,\\
&\mathcal{H}_2^\nu(U):=  -\sum_{i\in \mathcal{R}}  \int_{\bbr} a \nabla^2\eta(\overline U)  \partial_x U^R_i  A(U|\overline U) dx  - \sum_{i\in \mathcal{R}}   \int_{\bbr} a (h-\bar h) (p'(\bar v)-p'(v^R_i)) (v^R_i)_x dx  \\
&\quad\quad\quad  - \sum_{i\in \mathcal{R}}\int_{\bbr} a \nabla^2\eta(\overline U) (U- \overline U) \Big( -\dot \Lambda_i  +  \nabla A(U^R_i)\Big)  \partial_x U^R_i dx,\\
&\mathcal{H}_3^\nu(U):= -  \sum_{i\in \mathcal{NP}}\int_{\bbr} a (h-\bar h)  p'(\bar v) \partial_x v^P_i dx,\\
&\mathcal{P}^\nu(U):= \nu \int_{\bbr} a (p(v)-p(\bar v) )\partial_{x}\Big(\mu(v ) v^\gamma \partial_x p(v ) - \sum_{i\in \mathcal{S}} \mu(\tilde v_i) \tilde v_i^{\gamma} \partial_x p(\tilde v_i)  \Big) dx,\\
& \mathcal{J}^\nu_{good}(U):=  \sum_{i\in \mathcal{S}} \int_{\bbr} \tilde\lambda_i (a_i)_x \eta(U|\overline U ) dx  .
\end{aligned}
\end{align}

\begin{remark}\label{rem:0}
The hyperbolic part $\mathcal{H}_1(U)$ collects bad terms localized by derivatives of weights and shocks, whereas the hyperbolic part $\mathcal{H}_2(U)$ collects bad terms localized by derivatives of rarefactions. In Section \ref{sec:rare}, we will first estimate $\mathcal{H}_2(U)$ by defining a suitable shifts $\Lambda_i(t)$.
In Section \ref{sec:shock}, we will handle the shift part $\dot X_i(t) Y_i(U)$, the hyperbolic term $\mathcal{H}_1(U)$ and good term $ \mathcal{J}_{good}(U)$, which are localized by derivatives of weights and shocks. Those together with the diffusion term of perturbations will be estimated. The diffusion term will be extracted from the parabolic term $\mathcal{P}(U)$ by decomposing it in two different ways depending on the size of values of $v$ in Section \ref{sec:para}.
\end{remark}

\begin{proposition}\label{prop:main}
For all $ t\in (t_j,t_{j+1}) $,
\begin{align}
\begin{aligned}\label{ineq-mid0}
&\frac{d}{dt} \int_{\bbr} a^\nu \eta(U^\nu |\overline U_{\nu,\delta} ) dx + G(U^\nu)(t) \\
&\le  C \int_\bbr a^\nu \eta(U^\nu|\bar U_{\nu,\delta}) dx +C\delta^{1/4} \overline{\mathcal{D}}(t) + C \delta^{3/4}  +   C   \frac{\delta^2}{\delta_*}  +\mathcal{C} (\delta, \nu),
\end{aligned}
\end{align}
where
\begin{align*}
\begin{aligned}
G(U)(t)&:=\sqrt\ds\nu  \int_{\bbr} \mu_1(v) v^{\gamma} |(p(v)-p(\bar v))_x|^2 dx + \nu \int_\bbr  \mu_2(v) v^\gamma |\partial_x p(v) |^2 dx \\
 &\quad + \ds \sum_{i\in \mathcal{S}}  \int_{\bbr} |(\tilde U_i)_x| p(v|\bar v) dx +  \sum_{i\in \mathcal{R}}  \int_{\bbr} |(U^R_i)_x| p(v|\bar v) dx,
\end{aligned}
\end{align*}
\beq\label{odiff}
\overline{\mathcal{D}}(t) := \nu \int_\bbr \frac{\bar\mu(v)}{v} |u_x|^2 dx,
\eeq
and $\mathcal{C} (\delta, \nu)$ is the constant that vanishes when $\nu\to0$ for any fixed $\delta>0$. 
\end{proposition}

The proof of the main Proposition \ref{prop:main} is presented in the following sections and eventually completed in Section \ref{sec:shock} .

\subsection{Useful estimates on relative functions}
We here present useful estimates on relative functions, as in \cite[Lemmas A.1, A.2, A.3]{KV-2shock} (see also \cite[Lemmas 2.2, 2.5]{Kang-V-NS17}).

\begin{lemma}\label{lem-pro}  
For given constants $\gamma>1$, and $v_*>0$, there exists constants $C, \eps_*>0$ such that  the following inequalities hold.\\
1)  For any $w\in (0,2v_*)$,
\begin{align*}
\begin{aligned}
& Q(v|w)\ge C^{-1} |v-w|^2,\quad \mbox{for all } 0<v\le 3v_*,\\
 & Q(v|w)\ge  C^{-1} |v-w|,\quad  \mbox{for all } v\ge 3v_*.
\end{aligned}
\end{align*}
2) If $0<w\leq u\leq v$ or $0<v\leq u\leq w$ then 
\[
Q(v|w)\geq Q(u|w),
\]
and for any $\eps>0$ there exists a constant $C>0$ such that if, in addition, 
$|w-v_*|\le \eps/2$ and $|w-u|>\eps$, we have
\[
Q(v|w)-Q(u|w)\geq C|u-v|. 
\]
3) For any $w>v_*/2$,
\[
|p(v)-p(w)| \le C |v-w|,\qquad \forall v\ge v_*/2,
\]
\[
p(v|w) \le C |v-w|^2,\qquad \forall v\ge v_*/2 ,
\]
\[
p(v|w)\leq C(|v-w|+|p(v)-p(w)|),\qquad \forall v>0.
\]
4) For any $0<\eps<\eps_*$, the following is true.\\
For any $(v, w)\in \bbr_+^2$  
satisfying $|p(v)-p(w)|<\eps$ and  $|p(w)-p(v_*)|<\eps$,
\begin{align*}
\begin{aligned}
p(v|w)&\le \bigg(\frac{\gamma+1}{2\gamma} \frac{1}{p(w)} + C\eps \bigg) |p(v)-p(w)|^2,
\end{aligned}
\end{align*}
\[
Q(v|w)\ge \frac{p(w)^{-\frac{1}{\gamma}-1}}{2\gamma}|p(v)-p(w)|^2 -\frac{1+\gamma}{3\gamma^2} p(w)^{-\frac{1}{\gamma}-2}(p(v)-
p(w))^3,
\]
\[
Q(v|w)\le \bigg( \frac{p(w)^{-\frac{1}{\gamma}-1}}{2\gamma} +C\eps  \bigg)|p(v)-p(w)|^2.
\]
For any $(v, w)\in \bbr_+^2$ such that  $|p(w)-p(v_*)|\leq \eps$,  and satisfying either $Q(v|w)<\eps$ or $|p(v)-p(w)|<\eps$,
\[
|p(v)-p(w)|^2 \le C Q(v|w).
\]
\end{lemma}
\begin{lemma}\cite[Lemma 4.1]{KVARMA} \label{lem:tri}
For any function $F$, its relative function $F(\cdot|\cdot)$ satisfies
\[
F(u|w)+F(w|v) = F(u|v) + (F'(w)-F'(v)) (w-u),\quad\forall u, v, w.
\]
\end{lemma}

\subsection{Uniform bounds}
We present uniform bounds for entropy and dissipation. The following lemmas written in the Eulerian frame will be useful in the proofs of the two main Propositions \ref{main_prop} and \ref{prop:cc}. In particular, as in Remark \ref{rem:origin}, the uniform bound for diffusion written in Lagrangian frame will be used  in the proof of Theorem \ref{thm:uniform}.

Notice first that from \eqref{initialNS} and \eqref{initialEuler}, there exists a constant $C>0$ such that for any $\nu>0$:
\[
 \int_\RR \eta(U^\nu_0|U_*)\,dx\leq C.
\]
Indeed, this follows from Lemma \ref{lem:tri}:
$$
\eta(U^\nu_0|U^*)=\eta(U^\nu_0|U^0)+\eta(U^0|U_*)+(\eta'(U^0)-\eta'(U_*))\cdot(U^\nu_0-U^0),
$$
and \eqref{initialNS}, \eqref{initialEuler} with the fact that $\eta'(U^0)-\eta'(U_*)$ is bounded.

\begin{lemma}\label{lem:origin}
Given $(\rho_*,u_*)$, $\gamma>1$ and $\alpha$, assume that 
\beq\label{inif1}
\int_\bbr \eta\big((\rho_0^\nu,u_0^\nu)|(\rho_*,u_*)\big) dx \le C,
\eeq
for some constant $C$ independent of $\nu$.
Then, for all $ t>0$, 
\[
\int_\bbr \eta\big((\rho^\nu ,u^\nu)|(\rho_*,u_*)\big) dx +\nu \int_0^t \int_\bbr \bar\mu(\rho^\nu) |u^\nu_x|^2 dx ds \le C.
\]
\end{lemma}
\begin{proof}
For the simplicity of proof, we use the Lagrangian frame. That is, let $(t,y)$ denote the mass Lagrangian coordinate  where $y(t,x)=\int_0^x\rho(t,z)dz$, to distinguish it from the Eulerian coordinate $(t,x)$. So, we have
\beq\label{ccoord}
\frac{dy}{dx}=\rho.
\eeq
We recall the definition of the relative entropy in the Lagrangian coordinate (see \eqref{etaL}): 
\[
 \eta^L((v,u)|(v_*,u_*)) :=\frac{|u- u_* |^2}{2} + Q(v|v_*),
\]
where $\rho_*=1/v_*$.
Notice from \eqref{ccoord} that 
\beq\label{ccrel}
\int_\bbr \eta((\rho,u)(x)|(\rho_*,u_*)) dx= \int_\bbr \eta^L((v,u)(y)|(v_*,u_*)) dy.
\eeq
Then, we rewrite the Navier-Stokes system \eqref{NS-1} in the form:
\[
\partial_t U^\nu +\partial_y A(U^\nu)= { 0 \choose  \nu \Big(  \frac{\bar\mu(v^\nu)}{v^\nu} u^\nu_y \Big)_y },
\]
where 
\[
U^\nu:={v^\nu \choose u^\nu},\quad A(U^\nu):={-u^\nu \choose p(v^\nu)}.
\]
Then, it holds from the relative entropy method (especially form the equality \eqref{etaform} with the constant $V:=(v_*,u_*)$) that
\[
\partial_t \eta^L(U^\nu | V ) = -\partial_y  G(U^\nu;V) +  \big(\nabla\eta(U^\nu)-\nabla\eta(V) \big) { 0 \choose  \nu \Big(  \frac{\bar\mu(v^\nu)}{v^\nu} u^\nu_y \Big)_y },
\]
which implies
\[
\frac{d}{dt}\int_\bbr \eta^L\big((v^\nu,u^\nu)|(v_*,u_*)\big) dy = \nu \int_\bbr (u^\nu- u_*) \Big(  \frac{\bar\mu(v^\nu)}{v^\nu} u^\nu_y \Big)_y  = - \nu \int_\bbr  \frac{\bar\mu(v^\nu)}{v^\nu} \big|u^\nu_y \big|^2 dy. 
\]
Since it holds from \eqref{ccoord} that
\[
 \int_\bbr  \frac{\bar\mu(v)}{v} \big|u_y \big|^2 dy = \int_\bbr \bar\mu(\rho) |u_x|^2 dx,
\]
we have the desired result. 

\end{proof}

\begin{remark}\label{rem:origin}
The proof of Lemma \ref{lem:origin} implies that for all $ t>0$, 
\[
\int_0^t \overline{\mathcal{D}}(s)ds \le C
\]
where
\[
\overline{\mathcal{D}}(t) := \nu \int_\bbr \frac{\bar\mu(v)}{v} |u_y |^2 dy.
\]
This bound will be used in the proof of Theorem \ref{thm:uniform}.
\end{remark}

\begin{lemma}\label{lem:derho}
Given $(v_*,u_*)$, $\gamma>1$ and $\alpha$, assume that 
\beq\label{inif2}
\int_\bbr \frac{\rho_0^\nu |u_0^\nu- u_* |^2}{2} + \frac{p(\rho_0^\nu|\rho_*)}{\gamma-1} dx + \nu^2 \int_\bbr |\partial_x \phi(\rho_0^\nu) |^2  dx \le C,
\eeq
for some constant $C$ independent of $\nu$.
Then, for all $ t>0$, 
\[
\nu \int_0^t \int_\bbr \mu(\rho^\nu) (\rho^\nu)^{\gamma-3} |\rho^\nu_x|^2 dx ds \le C,
\]
and
\[
\nu^2 \int_\bbr |\partial_x \phi(\rho^\nu) |^2  dx \le C.
\]
\end{lemma}
\begin{proof}
We apply the same computation as in Lemma \ref{lem:origin} to the system \eqref{NS} written in the form:
\[
\partial_t U^\nu +\partial_y  A(U^\nu)= {  -\nu \Big( \mu(v^\nu)(v^\nu)^\gamma  p(v^\nu)_y \Big)_y  \choose 0},
\]
where 
\[
U^\nu:={v^\nu \choose h^\nu},\quad A(U^\nu):={-h^\nu \choose p(v^\nu)}.
\]
Then, using the equality \eqref{etaform} with the constant $V:=(v_*,u_*)$, we have
\begin{eqnarray*}
\frac{d}{dt}\int_\bbr \eta^L\big((v^\nu,h^\nu)|(v_*,u_*)\big) dy &=& \nu \int_\bbr (p(v^\nu)- p(v_*)) \Big( \mu(v^\nu)(v^\nu)^\gamma p(v^\nu)_y  \Big)_y \\
 &=& - \nu \int_\bbr  \mu(v^\nu)(v^\nu)^\gamma \big| p(v^\nu)_y \big|^2 dy. 
\end{eqnarray*}
Note that
\[
\int_\bbr \eta^L\big((v_0^\nu,h_0^\nu)|(v_*,u_*)\big) dy = \int_\bbr \frac{\rho_0^\nu |u_0^\nu- u_* |^2}{2} + \frac{p(\rho_0^\nu|\rho_*)}{\gamma-1} dx + \nu^2 \int_\bbr |\partial_x \phi(\rho_0^\nu) |^2  dx \le C.
\]
Thus, we have the first result by the change of coordinates: 
\[
 \int_\bbr  \mu(v^\nu)(v^\nu)^\gamma \big| p(v^\nu)_y \big|^2 dy =  \int_\bbr \mu(\rho^\nu) (\rho^\nu)^{\gamma-3} |\rho^\nu_x|^2 dx.
\]
For the second estimate, observe that by the definition of effective velocity,
\[
h^\nu(y,t)=u^\nu(x,t)-\nu\frac{\bar \mu_L(v^\nu)}{v^\nu}v^\nu_y(y,t)=u^\nu(x,t)+\nu\frac{1}{(\rho^\nu(x,t))^{1/2}}\partial_x \phi(\rho^\nu(x,t)).
\]
This and \eqref{ccrel} yield
\[
\int_\bbr \big| \sqrt{\rho^\nu} (u^\nu - u_*) +\nu \partial_x\phi(\rho^\nu) \big|^2 dx =   \int_\bbr \eta^L\big((v_0^\nu,h_0^\nu)|(v_*,u_*)\big) dy \le C.
\]
Thus, this and Lemma \ref{lem:origin} imply
\[
\nu^2 \int_\bbr |\partial_x \phi(\rho^\nu) |^2  dx \le C \int_\bbr \big| \sqrt{\rho^\nu} (u^\nu - u_*) +\nu \partial_x\phi(\rho^\nu) \big|^2 dx + \int_\bbr \rho^\nu \big| u^\nu - u_*\big|^2 dx \le C.
\]

\end{proof}

\subsection{Useful estimates for waves, and those interactions}
We here present useful estimates on waves and  interactions.
 
\begin{lemma} \label{lem:waves}
Given $U_{i-1/2}, \s_{i}, k_{i}, x_{i}$, let $y_i$ denote the trajectories \eqref{ycurve}. Assume \eqref{distij} with $ \rho_\nu=\nu^{1/3}$. Then,
there exists $C>0$ (depending only on $U_*$) such that for all $x$, and $t\in (t_j,t_{j+1})$, the following holds. \\
1) For each $i \in \mathcal{R}$,
\begin{align*}
\begin{aligned}
&\partial_x U^R_i(x,t) = \frac{\s_i^2}{2\nu}\exp \left( - \Big| \s_i \frac{x-x_i-\Lambda_i(t-t_j)}{\nu} \Big| \right)  {\mathbf{r}}_{k} \Big( U^R_i(x,t) \Big),\\
&|U^R_{i}(x,t) - U_{i -1/2}| \le C \s_{i} \exp\Big(-\frac{\s_{i} |x-y_i| }{\nu} \Big),\quad x< y_i ,\\
&|U^R_{i}(x,t) - U_{i +1/2}| \le C \s_{i} \exp\Big(-\frac{\s_{i} |x-y_i| }{\nu} \Big),\quad x> y_i ,\\
&C^{-1} \frac{\s_i^2}{\nu}\exp \left( -\frac{\s_{i} |x-y_i| }{\nu}\right) \le |\partial_x U_{i}^R(x,t)| \le C\frac{\s_i^2}{\nu}\exp \left( -\frac{\s_{i} |x-y_i| }{\nu} \right);\\
&\big|\partial_x^2 U_i^R(x,t) \big| \le C \frac{\s_i}{\nu} \big|\partial_x U_i^R(x,t) \big|;
\end{aligned}
\end{align*}
2) For each $i \in \mathcal{S}$,
\begin{align*}
\begin{aligned}
&|\tilde U_{i}(x,t) - U_{i -1/2}| \le C |\s_{i}| \exp\Big(-\frac{C^{-1} |\s_{i}| |x-y_i| }{\nu} \Big),\quad x< y_i ,\\
&|\tilde U_{i}(x,t) - U_{i +1/2}| \le C |\s_{i}| \exp\Big(-\frac{C^{-1} |\s_{i}| |x-y_i| }{\nu} \Big),\quad x> y_i ,\\
&C^{-1} \frac{\s_i^2}{\nu}\exp \left( -  \frac{C^{-1} |\s_{i}||x-y_{i}|}{\nu} \right) \le |\partial_x \tilde U_{i}(x,t)| \le C\frac{\s_i^2}{\nu}\exp \left( - \frac{C^{-1} |\s_{i}||x-y_{i}|}{\nu}  \right);\\
&\big|\partial_x^2 \tilde U_i(x,t) \big| \le C \frac{|\s_i|}{\nu} \big|\partial_x \tilde U_i(x,t) \big|;
\end{aligned}
\end{align*}
3)  For each $i \in \mathcal{NP}$,
\begin{align}
\begin{aligned} \label{vpest}
&|(v^P_i)_x|\le C\frac{|\s_i|}{\sqrt\rn} \mathbf{1}_{|x-y_i|\le \sqrt\rn},\\
&|(v^P_i)_{xx}|\le C\frac{|\s_i|}{\rn} \mathbf{1}_{|x-y_i|\le \sqrt\rn}.
\end{aligned}
\end{align}
4) For each $i \in \mathcal{R}$, for all $ x\in I_i:=\{ x~|~ |x-y_i| \le \sqrt{\rho_\nu} \} $,
\begin{align*}
\begin{aligned}
&i>i'\in \mathcal{S}\quad\Rightarrow\quad |\tilde v_{i'} - v_{i' +1/2}| \le C \frac{\nu}{\sqrt\rn} ; \qquad i<i'\in \mathcal{S}\quad\Rightarrow\quad |\tilde v_{i'} - v_{i' -1/2}| \le C \frac{\nu}{\sqrt\rn}, \\
&i>i'\in \mathcal{R}\quad\Rightarrow\quad |v^R_{i'} - v_{i' +1/2}| \le C \frac{\nu}{\sqrt\rn} ; \qquad i<i'\in \mathcal{R}\quad\Rightarrow\quad |v^R_{i'} - v_{i' -1/2}| \le C \frac{\nu}{\sqrt\rn}, \\
&i>i'\in \mathcal{NP}\quad\Rightarrow\quad |v^P_{i'} - v_{i' +1/2}| =0 ; \qquad i<i'\in \mathcal{NP}\quad\Rightarrow\quad |v^P_{i'} - v_{i' -1/2}| =0, 
\end{aligned}
\end{align*}
5) We have the following interaction estimates of waves: let $I_i:=\{ x~|~ |x-y_i| \le \sqrt{\rho_\nu} \} $.
\begin{align*}
\begin{aligned}
& |(\tilde v_i)_x| |(\tilde v_{i'})_x| \le \left\{ \begin{array}{ll}
       C  \frac{|\s_{i}|^2}{\rn} ,\quad  i \in \mathcal{S},~i'\in (\mathcal{S}-\{i\}),~  x\in I_i   \\
       C  \frac{|\s_{i'}|^2}{\rn}  ,\quad  i \in \mathcal{S},~i'\in (\mathcal{S}-\{i\}),~  x\in I_i^c ; \end{array} \right.
\\
&  |( v^R_{i})_x| |(v^R_{i'})_x| \le \left\{ \begin{array}{ll}
       C  \frac{|\s_{i}|^2}{\rn} ,\quad  i \in \mathcal{R},~i'\in (\mathcal{R}-\{i\}),~  x\in I_i   \\
       C  \frac{|\s_{i'}|^2}{\rn}  ,\quad  i \in \mathcal{R},~i'\in (\mathcal{R}-\{i\}),~  x\in I_i^c ; \end{array} \right.
\\
&
 |(\tilde v_i)_x| |(v^R_{i'})_x| \le \left\{ \begin{array}{ll}
       C  \frac{|\s_{i}|^2}{\rn} ,\quad  i \in \mathcal{S},~i'\in \mathcal{R},~  x\in I_i   \\
       C  \frac{|\s_{i'}|^2}{\rn}  ,\quad  i \in \mathcal{S},~i'\in \mathcal{R},~  x\in I_i^c ; \end{array} \right.
\\
&
|(\tilde v_i)_x| |(v^P_{i'})_x|    \le 
\left\{ \begin{array}{ll}
       0,\quad  i \in \mathcal{S},~i'\in \mathcal{NP},~  x\in I_i   \\
       C  \frac{|\s_{i'}| \nu}{\rn} ,\quad  i \in \mathcal{S},~i'\in \mathcal{NP},~  x\in I_i^c ; \end{array} \right.
       \\
       &
|(v^R_i)_x| |(v^P_{i'})_x|    \le 
\left\{ \begin{array}{ll}
       0,\quad  i \in \mathcal{R},~i'\in \mathcal{NP},~  x\in I_i   \\
       C  \frac{|\s_{i'}| \nu}{\rn} ,\quad  i \in \mathcal{R},~i'\in \mathcal{NP},~  x\in I_i^c ; \end{array} \right.
       \\
&|(v^P_i)_x| |(v^P_{i'})_x|  =0, \quad i \neq i'\in \mathcal{NP}    ;
\end{aligned}
\end{align*}

\end{lemma}
\begin{proof}
\noindent $\bullet$ {\bf proof of 1), 2) and 3) :} 
The proof of 2) can be found in \cite[Lemma 2.1 and Section 2.3]{Kang-V-NS17}.
The first relation of 1) follows from the definition in Section \ref{sub:waves}. The other estimates can be shown by the same argument as in \cite[Lemma 2.1 and Section 2.3]{Kang-V-NS17}. The proof of 3) follows from the definition in Section \ref{sub:waves}.   \\

\noindent $\bullet$ {\bf proof of 4) :} 
Observe from \eqref{distij} that for each $i \in \mathcal{R}$,
\[
|x-y_{i'}| \ge \sqrt{\rho_\nu},\quad \forall i'\in (\mathcal{R}-\{i\})\cup \mathcal{S}\cup \mathcal{NP},\quad \forall x\in I_i.
\]
So, it follows from 1) and 2) that  for each $i \in \mathcal{R}$, and $ \forall x\in I_i$,
\begin{align*}
\begin{aligned}
& C |\s_{i'}| \exp\Big(-\frac{C^{-1}|\s_{i'}|  \sqrt{\rho_\nu} }{\nu} \Big)  \ge \left\{ \begin{array}{ll}
        |\tilde v_{i'} - v_{i' +1/2}|  ,\quad i>i'\in \mathcal{S},  \\
        |\tilde v_{i'} - v_{i' -1/2}|  ,\quad i<i'\in \mathcal{S},  \end{array} \right. \\
& C |\s_{i'}| \exp\Big(-\frac{|\s_{i'}|  \sqrt{\rho_\nu} }{\nu} \Big)  \ge \left\{ \begin{array}{ll}
        |v^R_{i'} - v_{i' +1/2}|  ,\quad i>i'\in \mathcal{R},  \\
        | v^R_{i'} - v_{i' -1/2}|  ,\quad i<i'\in \mathcal{R}.  \end{array} \right. \\
\end{aligned}
\end{align*}
Then, the desired estimates are obtained by using
\[
 |\s_{i'}| \exp\Big(-\frac{C^{-1}|\s_{i'}|  \sqrt{\rho_\nu} }{\nu} \Big) =  \frac{\nu}{\sqrt\rn} \Big(\frac{|\s_{i'}| \sqrt\rn }{\nu} \Big) \exp\Big(-\frac{C^{-1}|\s_{i'}|  \sqrt{\rho_\nu} }{\nu} \Big) \le C \frac{\nu}{\sqrt\rn}.
\]
On the other hand, since $v^P_{i'}(x)$ connects from $v_{i' -1/2}$ to $v_{i' +1/2}$ with transition width $\sqrt\rn$, we have the desired one. \\

\noindent $\bullet$ {\bf proof of 5) :} 
As done above,  we will use the fact from \eqref{distij} that
\begin{align*}
\begin{aligned}
& i \in \mathcal{S} \quad\Rightarrow\quad |x-y_{i'}| \ge \sqrt{\rho_\nu},\quad \forall i'\in (\mathcal{S}-\{i\})\cup \mathcal{R}\cup \mathcal{NP},\quad \forall x\in I_i:=\{ x~|~ |x-y_i| \le \sqrt{\rho_\nu} \}, \\
& i \in \mathcal{R} \quad\Rightarrow\quad |x-y_{i'}| \ge \sqrt{\rho_\nu},\quad \forall i'\in \mathcal{R}-\{i\},\quad \forall x\in I_i, \\
& i \in \mathcal{NP} \quad\Rightarrow\quad |x-y_{i'}| \ge \sqrt{\rho_\nu},\quad \forall i'\in \mathcal{NP}-\{i\},\quad \forall x\in I_i.
\end{aligned}
\end{align*}
This together with 1) and 2) implies that if $(i,i')$ is an element of one of the following sets: \\
$ \mathcal{S}\times (\mathcal{S}-\{i\}), \mathcal{R}\times (\mathcal{R}-\{i\}), \mathcal{S}\times \mathcal{R}$, then we have the following estimates respectively:
\begin{align*}
\begin{aligned}
& |(\tilde v_i)_x| |(\tilde v_{i'})_x|, \, |( v^R_i)_x| |(v^R_{i'})_x|,  \, |(\tilde v_i)_x| |(v^R_{i'})_x| \\
&\quad   \le 
\left\{ \begin{array}{ll}
       C  \frac{|\s_{i}|^2}{\nu} \frac{|\s_{i'}|^2}{\nu} \exp\Big(-\frac{C^{-1}|\s_{i'}|  \sqrt{\rho_\nu} }{\nu} \Big),\quad  \forall x\in I_i   \\
       C  \frac{|\s_{i}|^2}{\nu} \frac{|\s_{i'}|^2}{\nu} \exp\Big(-\frac{C^{-1}|\s_{i} | \sqrt{\rho_\nu} }{\nu} \Big),\quad  \forall  x\in I_i^c , \end{array} \right.
\end{aligned}
\end{align*}
This implies the desired estimates.\\
In addition, using $|(v^P_i)_x|\le C\frac{|\s_i|}{\sqrt\rn} \mathbf{1}_{|x-y_i|\le \sqrt\rn}$, we obtain that for each $i \in \mathcal{S},~i'\in \mathcal{NP}$,
\begin{align*}
\begin{aligned}
&\forall x\in I_i, \quad |(\tilde v_i)_x| |(v^P_{i'})_x|  = 0,\\
&\forall x\in I_i^c, \quad |(\tilde v_i)_x| |(v^P_{i'})_x| \le C  \frac{|\s_{i'}|}{\sqrt{\rho_\nu}}  \frac{|\s_{i}|^2}{\nu} \exp\Big(-\frac{C^{-1}|\s_{i}|  \sqrt{\rho_\nu} }{\nu} \Big)\le   C  \frac{|\s_{i'}| \nu}{\rn},
\end{aligned}
\end{align*}
and for each $i,i'\in \mathcal{NP}$
\[
|(v^P_i)_x| |(v^P_{i'})_x|  =0, \quad \forall  x.
\]
Likewise, for each $i \in \mathcal{R},~i'\in \mathcal{NP}$,
\begin{align*}
\begin{aligned}
&\forall x\in I_i, \quad |(v^R_i)_x| |(v^P_{i'})_x|  = 0,\\
&\forall x\in I_i^c, \quad |(v^R_i)_x| |(v^P_{i'})_x| \le C  \frac{|\s_{i'}|}{\sqrt{\rho_\nu}}  \frac{|\s_{i}|^2}{\nu} \exp\Big(-\frac{\s_{i}  \sqrt{\rho_\nu} }{\nu} \Big)\le   C  \frac{|\s_{i'}| \nu}{\rn}.
\end{aligned}
\end{align*}

\end{proof}

\begin{lemma} \label{lem:finsout}
Under the same assumptions above, we have
\begin{align*}
\begin{aligned}
&\int_{\bbr} \Big( \sum_{i\in \mathcal{R}}  |\bar v-v^R_i | |(v^R_i)_x|  \Big)^2 dx \le C \frac{\nu^{1/3}}{\delta},\\
& \int_\bbr  \Big( \sum_{i\in \mathcal{S}}   |\bar v - \tilde v_i| | (\tilde v_i)_x| \Big)^2 dx  \le C \frac{\nu^{1/3}}{\delta},\\
& \int_\bbr \sum_{i\in \mathcal{S}}   |\bar v - \tilde v_i| | (\tilde v_i)_x| dx\le  \frac{C\nu^{1/2}}{\delta}, \\
& \int_\bbr \sum_{i\in \mathcal{R}}   |\bar v - v^R_i| | (v^R_i)_x| dx\le  \frac{C\nu^{1/2}}{\delta}.
\end{aligned}
\end{align*}
In addition, if $|\s_i|\ge \rn$ for all $i\in \mathcal{S}$, then
\begin{align*}
\begin{aligned}
&\int_{v>2v_*}  \Big( \sum_{i\in \mathcal{S}}   |\bar v - \tilde v_i| | (\tilde v_i)_x| \Big)^2 v^\beta dx \le \frac{C\nu^{1/3}}{\delta^3} \sum_{i\in \mathcal{S}} \int_\bbr |(\tilde v_i)_x| p(v|\bar v)  dx +  C  \frac{C\nu^{1/3}}{\delta^3}  \int_\bbr Q(v|\bar v)  dx.
\end{aligned}
\end{align*}
\end{lemma}
\begin{proof}
First, we apply Lemma \ref{lem:num} and 3) of Lemma \ref{lem:waves} to the estimate: \\for all $ x\in I_i :=\{ x~|~ |x-y_i| \le \sqrt{\rho_\nu} \} $, 
\begin{align}
\begin{aligned} \label{nodbar}
|\bar v-v^R_i | &\le \left| \sum_{i'\in \mathcal{R},~i'>i} (v^R_{i'} - v_{i' -1/2}) \right| + \left| \sum_{i'\in \mathcal{R},~i'<i} (v^R_{i'} - v_{i' +1/2}) \right| \\
& +  \left| \sum_{i'\in \mathcal{S},~i'>i} (\tilde v_{i'} - v_{i' -1/2}) \right| + \left| \sum_{i'\in \mathcal{S},~i'<i} (\tilde v_{i'} - v_{i' +1/2}) \right|  \\
&\le  C \frac{\ln(\rho_\nu^{-1})}{\delta} \frac{\nu}{\sqrt\rn}.\\ 
\end{aligned}
\end{align}
Thus,
\begin{align}
\begin{aligned} \label{estins}
 &\sum_{i\in \mathcal{R}}   \int_{I_i} |\bar v-v^R_i |^2 |(v^R_i)_x|^2 dx \\
 &\le C\left(\frac{\ln(\rho_\nu^{-1})}{\delta} \frac{\nu}{\sqrt\rn}\right)^2 \sum_{i\in \mathcal{R}} \int_{I_i} \frac{\s_i^2}{\nu} |(v^R_i)_x| dx \\ 
&\le C\nu^{2/3} \sum_{i\in \mathcal{R}} \|(v^R_i)_x\|_{L^1} \le CL(0) \nu^{2/3}.
\end{aligned}
\end{align}
On the other hand, we use $ |\bar v-v^R_i | \le C$ and Lemma \ref{lem:waves} to have
\begin{align}
\begin{aligned} \label{estout}
 &\sum_{i\in \mathcal{R}}  \int_{I_i^c} |\bar v-v^R_i |^2 |(v^R_i)_x|^2 dx \\
 &\le C \sum_{i\in \mathcal{R}}  \int_{|x-y_i| > \sqrt{\rho_\nu} } \frac{\s_i^4}{\nu^2} \exp\Big(-\frac{2\s_{i}|x-y_i|}{\nu} \Big)  dx \\
 &= C \sum_{i\in \mathcal{R}}  \frac{\s_i^4\sqrt{\rho_\nu}}{\nu^2} \int_{|z| > 1 } \exp\Big(-\frac{2\s_{i} \sqrt{\rho_\nu} |z|}{\nu} \Big)  dz \\
 &\le C \sum_{i\in \mathcal{R}}  \frac{\s_i^3}{\nu} \exp\Big(-\frac{2\s_{i} \sqrt{\rho_\nu}}{\nu} \Big) \le C  \frac{\nu}{ \rho_\nu} \sum_{i\in \mathcal{R}} \s_i \le CL(0) \nu^{2/3}.
\end{aligned}
\end{align}
Those estimates with Lemma \ref{lem:num} imply
\[
\int_{\bbr} \Big( \sum_{i\in \mathcal{R}}  |\bar v-v^R_i | |(v^R_i)_x|  \Big)^2 dx \le C\frac{\ln(\rho_\nu^{-1})}{\delta}  \int_{\bbr} \sum_{i\in \mathcal{R}}  |\bar v-v^R_i |^2 |(v^R_i)_x|^2  dx  \le C \frac{\nu^{1/3}}{\delta}.
\]

Likewise, using the same estimates as above, we have
\[
|\bar v-\tilde v_i | \le  C \frac{\ln(\rho_\nu^{-1})}{\delta} \frac{\nu}{\sqrt\rn},\quad \mbox{on } I_i,
\]
and so,
\begin{align*}
\begin{aligned} 
&  \int_\bbr  \Big( \sum_{i\in \mathcal{S}}   |\bar v - \tilde v_i| | (\tilde v_i)_x| \Big)^2 dx  \\
&\le   C\frac{\ln(\rho_\nu^{-1})}{\delta} \sum_{i\in \mathcal{S}} \int_\bbr    |\bar v - \tilde v_i|^2 | (\tilde v_i)_x|^2  dx\\
&\le  C\frac{\ln(\rho_\nu^{-1})}{\delta}  \left(\frac{\ln(\rho_\nu^{-1})}{\delta} \frac{\nu}{\sqrt\rn}\right)^2 \sum_{i\in \mathcal{S}}  \int_{I_i} \frac{|\s_i|^2}{\nu} |(\tilde v_i)_x| dx\\
&\quad +  C\frac{\ln(\rho_\nu^{-1})}{\delta} \sum_{i\in \mathcal{S}}  \int_{I_i^c} \frac{\s_i^4}{\nu^2} \exp\Big(-\frac{C^{-1}|\s_{i}||x-y_i|}{\nu} \Big)  dx \\
&\le  \frac{C\nu^{1/3}}{\delta^3} .
\end{aligned}
\end{align*}
Also, we have
\begin{align*}
\begin{aligned} 
\int_\bbr \sum_{i\in \mathcal{S}}   |\bar v - \tilde v_i| | (\tilde v_i)_x| dx &\le  C\left(\frac{\ln(\rho_\nu^{-1})}{\delta} \frac{\nu}{\sqrt\rn}\right) \sum_{i\in \mathcal{S}}  \int_{I_i} |(\tilde v_i)_x| dx +  \sum_{i\in \mathcal{S}}  \int_{I_i^c} |(\tilde v_i)_x| dx \\
&\le  C\left(\frac{\ln(\rho_\nu^{-1})}{\delta} \frac{\nu}{\sqrt\rn}\right) \sum_{i\in \mathcal{S}} |\s_i| + \sum_{i\in \mathcal{S}}  \int_{I_i^c} \frac{\s_i^2}{\nu} \exp\Big(-\frac{C^{-1}|\s_{i}||x-y_i|}{\nu} \Big)  dx \\
&\le  \frac{C\nu^{1/2}}{\delta}.
\end{aligned}
\end{align*}
Likewise, we have
\[
\int_\bbr \sum_{i\in \mathcal{R}}   |\bar v - v^R_i| | (v^R_i)_x| dx\le  \frac{C\nu^{1/2}}{\delta}.
\]

For the last estimate, using $v^\beta \mathbf{1}_{v>2v_*} \le C Q(v|\bar v)$ or $v^\beta \mathbf{1}_{v>2v_*} \le C p(v|\bar v)$, we have
\begin{align*}
\begin{aligned} 
&\int_{v>2v_*}  \Big( \sum_{i\in \mathcal{S}}   |\bar v - \tilde v_i| | (\tilde v_i)_x| \Big)^2 v^\beta dx \le C\frac{\ln(\rho_\nu^{-1})}{\delta} \int_{v>2v_*} \sum_{i\in \mathcal{S}}   |\bar v - \tilde v_i|^2 | (\tilde v_i)_x|^2 v^\beta dx\\
&\le  C\frac{\ln(\rho_\nu^{-1})}{\delta}  \left(\frac{\ln(\rho_\nu^{-1})}{\delta} \frac{\nu}{\sqrt\rn}\right)^2 \sum_{i\in \mathcal{S}}  \int_{I_i} \frac{|\s_i|^2}{\nu} |(\tilde v_i)_x| p(v|\bar v)  dx\\
&\quad +  C\frac{\ln(\rho_\nu^{-1})}{\delta}\sum_{i\in \mathcal{S}}  \frac{\s_i^4}{\nu^2} \exp\Big(-\frac{C^{-1}|\s_{i}|\sqrt\rn}{\nu} \Big) \int_{I_i^c}  Q(v|\bar v)  dx \\
\end{aligned}
\end{align*}
Then, using $|\s_i|\ge \rn$ for all $i\in \mathcal{S}$, we have
\begin{align*}
\begin{aligned} 
&\int_{v>2v_*}  \Big( \sum_{i\in \mathcal{S}}   |\bar v - \tilde v_i| | (\tilde v_i)_x| \Big)^2 v^\beta dx \\
&\le   \frac{C\nu^{1/3}}{\delta^3} \sum_{i\in \mathcal{S}} \int_{I_i} |(\tilde v_i)_x| p(v|\bar v)  dx\\
&\quad +C\frac{\ln(\rho_\nu^{-1})}{\delta} \frac{1}{\nu^2} \exp\Big(-\frac{1}{C\nu^{1/3}} \Big) \int_\bbr  Q(v|\bar v)  dx  \\
&\le  \frac{C\nu^{1/3}}{\delta^3} \sum_{i\in \mathcal{S}} \int_{I_i} |(\tilde v_i)_x| p(v|\bar v)  dx +  C  \frac{C\nu^{1/3}}{\delta^3} \int_\bbr Q(v|\bar v)  dx.
\end{aligned}
\end{align*}

\end{proof}

\section{Estimates for main terms localized by derivatives of rarefactions}\label{sec:rare}
\setcounter{equation}{0}

We here handle the following functional of \eqref{ybg-first}:
\begin{align*}
\begin{aligned}
\mathcal{H}_2^\nu(U) &:=  -\sum_{i\in \mathcal{R}}  \int_{\bbr} a \nabla^2\eta(\overline U)  \partial_x U^R_i  A(U|\overline U) dx  - \sum_{i\in \mathcal{R}}   \int_{\bbr}a (h-\bar h) (p'(\bar v)-p'(v^R_i)) (v^R_i)_x dx  \\
&\quad  - \sum_{i\in \mathcal{R}}\int_{\bbr} a\nabla^2\eta(\overline U)   (U- \overline U) \cdot \Big( -\dot \Lambda_i I +  \nabla A(U^R_i)\Big)  \partial_x U^R_i dx,\\
&=: \mathcal{H}_{21} + \mathcal{H}_{22} +\mathcal{H}_{23}.
\end{aligned}
\end{align*}

\subsection{Estimates on $\mathcal{H}_{21}$ and $\mathcal{H}_{22}$}

We first show that $\mathcal{H}_{21}$ is non-positive as follows. \\
Using Lemma \ref{lem:waves} and \eqref{formbase}, we have
\[
\mathcal{H}_{21}= -\sum_{i\in \mathcal{R}}  \int_{\bbr} a  \frac{\s_i^2}{2\nu}\exp \left( - \Big| \s_i \frac{x-x_i-\Lambda_i(t-t_j)}{\nu} \Big| \right)  \Big[ {\mathbf{r}}_{k_i} \Big( U^R_i \Big)  \Big]_2 p(v|\bar v) dx ,
\]
where $\Big[ {\mathbf{r}}_{k_i} \Big( U^R_i \Big)  \Big]_2$ denotes the second component of $ {\mathbf{r}}_{k_i} \Big( U^R_i \Big) $, which is positive, more specifically, there exists a constant $C>0$ (depending only on $U_*$) such that
\[
C<\Big[ {\mathbf{r}}_{k_i} \Big( U^R_i \Big)  \Big]_2 < 2C.
\]
This with Lemma \ref{lem:waves} implies
\[
\mathcal{H}_{21}  \le   -C \sum_{i\in \mathcal{R}}  \int_{\bbr} |(U^R_i)_x| p(v|\bar v) dx.
\]
Set 
\beq\label{gir}
G_i^R:=  \int_{\bbr} |(U^R_i)_x| p(v|\bar v) dx.
\eeq

To estimate $\mathcal{H}_{22}$, we first have
\begin{align*}
\begin{aligned}
\mathcal{H}_{22} & \le   \int_{\bbr} |h-\bar h|  \sum_{i\in \mathcal{R}} |\bar v-v^R_i | |(v^R_i)_x| dx  \\
 & \le 2 \int_{\bbr} |h-\bar h|^2 dx +  2  \int_{\bbr} \Big( \sum_{i\in \mathcal{R}}  |\bar v-v^R_i | |(v^R_i)_x|  \Big)^2 dx. 
\end{aligned}
\end{align*}
To estimate the last term, we use Lemma \ref{lem:finsout}. 
Therefore,
\[
\mathcal{H}_{22} \le 2 \int_{\bbr} |h-\bar h|^2 dx + C \frac{\nu^{1/3}}{\delta}.
\]

\subsection{Estimates on $\mathcal{H}_{23}$ with construction of shifts for rarefactions}\label{subsec-rar}

For $\mathcal{H}_{23}$, by Lemma \ref{lem:waves} and the definition of rarefaction, we first observe
\begin{align*}
\begin{aligned}
\mathcal{H}_{23} &=  - \sum_{i\in \mathcal{R}} \int_{\bbr}  \frac{\s_i^2}{2\nu}\exp \left( - \Big| \s_i \frac{x-y_i}{\nu} \Big| \right) \nabla^2\eta(\overline U) (U- \overline U) \cdot \Big[ -\dot \Lambda_i  I  +  \nabla A\left(  U^R_i  \right)\Big] {\mathbf{r}}_{k_i} \Big(  U^R_i  \Big) dx \\
&=  - \sum_{i\in \mathcal{R}} \int_{\bbr}  \frac{\s_i^2}{2\nu}\exp \left( - \Big| \s_i \frac{x-y_i}{\nu} \Big| \right) \nabla^2\eta(\overline U) (U- \overline U) \cdot \Big[ -\dot \Lambda_i   +  \lam_{k_i} \left(  U^R_i  \right)\Big] {\mathbf{r}}_{k_i} \Big(  U^R_i  \Big) dx\\
&=  \sum_{i\in \mathcal{R}} \int_{\bbr} p'(\bar v) (v- \bar v)  \Big[ -\dot \Lambda_i   +  \lambda_{k_i} \left({\mathbf{r}}_{k_i} \Big( U^R_i \Big) \right)\Big] \partial_x v^R_i dx \\
&\quad - \sum_{i\in \mathcal{R}} \int_{\bbr}  (h- \bar h)  \Big[ -\dot \Lambda_i   +  \lambda_{k_i} \left({\mathbf{r}}_{k_i} \Big( U^R_i  \Big) \right)\Big] \partial_x h^R_i dx\\
&=: \mathcal{H}_{231}+\mathcal{H}_{232},
\end{aligned}
\end{align*}
where $\lam_{k_i}$ is the eigenvalue associated with ${\mathbf{r}}_{k_i}$.
First, $\mathcal{H}_{231}$ is easily estimated as
\begin{align*}
\begin{aligned}
\mathcal{H}_{231} &\le C  \sum_{i\in \mathcal{R}} |\s_i| \int_{v\le 2v_*}  |v- \bar v|  | \partial_x v^R_i| dx +C  \sum_{i\in \mathcal{R}} |\s_i| \int_{v> 2v_*}  |v- \bar v|  | \partial_x v^R_i| dx \\
&\le C \delta \sum_{i\in \mathcal{R}} \int_{v\le 2v_*}  |v- \bar v|^2  | \partial_x v^R_i| dx + C \sum_{i\in \mathcal{R}} |\s_i| \int_{v\le 2v_*}  | \partial_x v^R_i| dx \\
&\quad  +C\delta  \sum_{i\in \mathcal{R}} \int_{v> 2v_*}  p(v| \bar v)  | \partial_x v^R_i| dx \\
&\le C \delta \sum_{i\in \mathcal{R}} G_i^R + C \delta L(0).
\end{aligned}
\end{align*}
On the other hand, handling $\mathcal{H}_{232}$ is subtle. For that, we use $h = u- \nu\mu(v) v^\gamma p(v)_x$ and the average of $u$ as
\[
(u)_R : =  \frac{1}{\int_\bbr \partial_x h^R_i  dx}\int_\bbr u \, \partial_x h^R_i  dx.
\] 
That is,
\begin{align*}
\begin{aligned}
\mathcal{H}_{232} & = - \sum_{i\in \mathcal{R}} \int_{\bbr}  (u- (u)_R)  \Big[ -\dot \Lambda_i   +  \lambda_{k_i}  \left({\mathbf{r}}_{k_i}  \Big( U^R_i  \Big) \right)\Big] \partial_x h^R_i dx  \\
& - \sum_{i\in \mathcal{R}} \int_{\bbr}  (u)_R  \Big[ -\dot \Lambda_i   +  \lambda_{k_i}  \left({\mathbf{r}}_{k_i}  \Big( U^R_i  \Big) \right)\Big] \partial_x h^R_i dx \\
& +\nu \sum_{i\in \mathcal{R}} \int_{\bbr} \mu(v) v^\gamma p(v)_x  \Big[ -\dot \Lambda_i   +  \lambda_{k_i}  \left({\mathbf{r}}_{k_i}  \Big( U^R_i  \Big) \right)\Big] \partial_x h^R_i  dx \\
& + \sum_{i\in \mathcal{R}} \int_{\bbr}  \bar h  \Big[ -\dot \Lambda_i   +  \lambda_{k_i} \left({\mathbf{r}}_{k_i} \Big( U^R_i  \Big) \right)\Big] \partial_x h^R_i dx \\
&=:J_1+J_2+J_3 +J_4.
\end{aligned}
\end{align*}
Define the shift $\Lambda_i$ as a unique Lipschitz solution of the following ODE:
\beq\label{rode}
\dot \Lambda_i (t) = F_i \big((u)_R \big),\qquad \Lambda_i(0)=0,
\eeq
where $F_i$ is a bounded Lipschitz function defined as \\
\[
F_i(y) =\left\{ \begin{array}{ll}
      \lambda_{k_i}  (U_{i+1/2}) ,\quad\mbox{if $y\le -1$} \\
       \lambda_{k_i} (U_{i-1/2}) ,\quad\mbox{if $y\ge 1$} \\
       linear,\quad\mbox{otherwise} \end{array} \right.
\]
Notice that since $(u)_R$ depends on $\Lambda_i(t)$ as $\partial_x h^R_i= \partial_x h^R_i(y_i(t))$ (see \eqref{urdef}) where the trajectory $y_i(t)$ is a function of $\Lambda_i(t)$ by \eqref{ycurve}, the above ODE is nonlinear. In addition, \eqref{rode} has a unique Lipschitz solution $\Lambda_i$, by Lemma \ref{lem:waves} and the definition of the rarefaction $h^R_i$.\\
Since
\beq\label{lamest}
 \lambda_{k_i} (U_{i-1/2}) \le \dot \Lambda_i \le  \lambda_{k_i}  (U_{i+1/2}) ,
\eeq
and $| \lambda_{k_i}  (U_{i+1/2}) - \lambda_{k_i}  (U_{i-1/2}) |\le C\s_i$ and
\[
\dot\Lambda =
\left\{ \begin{array}{ll}
      \lambda_{k_i}  (U_{i+1/2}) ,\quad\mbox{if $(u)_R \le -1$} \\
       \lambda_{k_i} (U_{i-1/2}) ,\quad\mbox{if $(u)_R \ge 1$}, \end{array} \right.
\]
we have
\[
J_2 \le C \sum_{i\in \mathcal{R}} \s_i \mathbf{1}_{|(u)_R|\le 1} \int_{\bbr} | \partial_x h^R_i  | dx  \le  C\sum_{i\in \mathcal{R}} \s_i^2 \le C\delta L(t) \le C\delta L(0).
\]

To estimate $J_1$, observe that using
\begin{align*}
\begin{aligned}
u(x) - (u)_R &=  \frac{1}{\int_\bbr  ( h^R_i )_x dx} \int_\bbr \big(u(x)-u(y)\big) \, ( h^R_i )_x (y) dy \\
& =  \frac{1}{\int_\bbr  ( h^R_i )_x dx} \int_\bbr  \left( \int^x_y \partial_z u(z) dz  \right) ( h^R_i )_x (y) dy,
\end{aligned}
\end{align*}
the estimates \eqref{lamest} and $\int_\bbr |\partial_x h^R_i | dx \sim C\s_i$ yields
\begin{align*}
\begin{aligned}
J_1
& \le C  \sum_{i\in \mathcal{R}} \int_{\bbr}  | ( h^R_i )_x (x)|   \int_\bbr  \left| \int^x_{y_i} \partial_z u(z) dz  \right| |( h^R_i )_x (y)| dy
dx\\
  & + C  \sum_{i\in \mathcal{R}} \int_{\bbr}  | ( h^R_i )_x (x)|   \int_\bbr  \left| \int^{y_i}_y \partial_z u(z) dz  \right| |( h^R_i )_x (y)| dy
dx \\
&=: J_{11} + J_{12}  .
\end{aligned}
\end{align*}
Using $|( h^R_i )_x (z)|\le |( h^R_i )_x (x)|$ for all ${y_i}\le |z|\le |x|$, we have
\begin{align*}
\begin{aligned}
J_{11} & \le C \sum_{i\in \mathcal{R}} \int_{\bbr}  \sqrt{|( h^R_i )_x (x)|}   \int_\bbr \int_\bbr  \sqrt{|( h^R_i )_x (z)|} |\partial_z u(z)| dz  |( h^R_i )_x (y)| dy
dx\\
& \le C \sum_{i\in \mathcal{R}} \int_{\bbr}  \sqrt{|( h^R_i )_x (x)|}   \int_\bbr \| v^{-1/2} u_x \|_{L^2} \| v ( h^R_i )_x \|_{L^1}^{1/2} |( h^R_i )_x (y)| dy
dx\\
& \le C \sum_{i\in \mathcal{R}} \s_i  \| v^{-1/2} u_x \|_{L^2} \| v ( h^R_i )_x \|_{L^1}^{1/2} \int_{\bbr}  \sqrt{|( h^R_i )_x (x)|} dx.
\end{aligned}
\end{align*}
Then, using $ \int_{\bbr}  \sqrt{|( h^R_i )_x (x)|} dx \le C\sqrt\nu$,
\begin{align*}
\begin{aligned}
J_{11} 
& \le C\sum_{i\in \mathcal{R}} \s_i \sqrt{\nu}  \| v^{-1/2} u_x \|_{L^2}  \| v ( h^R_i )_x \|_{L^1}^{1/2} .
\end{aligned}
\end{align*}
Using $\s_i\le \delta$, $\sum_{i\in \mathcal{R}} |\s_i| \le L(0) \ll1$, and 
\[
\| v ( h^R_i )_x \|_{L^1} \le C \int_\bbr \Big( 1 + p(v|\bar v) \Big) | ( h^R_i )_x| dx \le C (\s_i + G^R_i),
\] 
we have
\begin{align*}
\begin{aligned}
J_{11} 
& \le C\sum_{i\in \mathcal{R}} \s_i \sqrt{\nu}  \| v^{-1/2} u_x \|_{L^2}  \sqrt{\s_i + G^R_i} \\
&\le C\sum_{i\in \mathcal{R}} \s_i^{5/4} \nu  \| v^{-1/2} u_x \|_{L^2}^2 + C\sum_{i\in \mathcal{R}} \s_i^{3/4} (\s_i + G^R_i) \\
&\le C\delta^{1/4} L(0) \nu  \| v^{-1/2} u_x \|_{L^2}^2 + C\delta^{3/4} L(0) + C\delta^{3/4} \sum_{i\in \mathcal{R}} G^R_i .
\end{aligned}
\end{align*}
To control the first term on r.h.s. above, we may use the diffusion term $\overline{\mathcal{D}}$ in \eqref{odiff} that is bounded after integrating in time by Remark \ref{rem:origin}. Especially, since the condition \eqref{assmu} and $\mu(v):=\frac{\bar\mu(v)}{\gamma}$ implies $\bar\mu(v)\ge C_1\gamma$,  and so
\[
\overline{\mathcal{D}} \ge C_1\nu\gamma  \int_\bbr \frac{1}{v} |u_x|^2 dx,
\]
we have
\[
J_{11}  \le C\delta^{1/4} \overline{\mathcal{D}}  + C\delta^{3/4}  + C\delta^{1/4} \sum_{i\in \mathcal{R}} G^R_i .
\]
Likewise, $J_{12}$ is estimated as above. Therefore,
\[
J_1\le C\delta^{1/4} \overline{\mathcal{D}}  + C\delta^{3/4}  + C\delta^{1/4} \sum_{i\in \mathcal{R}} G^R_i .
\]

To estimate $J_3$, we use the decompositions \eqref{mudecom} on $\mu=\mu_1+\mu_2$ to separate $J_3$ into two parts: $J_3 = J_b + J_s$, where
\begin{align*}
\begin{aligned}
J_b &:= \nu \sum_{i\in \mathcal{R}} \int_\bbr \mu_2(v) v^\gamma p(v)_x  \Big[ -\dot \Lambda_i   +  \lambda_{k_i}  \left({\mathbf{r}}_{k_i}  \Big( U^R_i  \Big) \right)\Big] \partial_x h^R_i  dx ,\\
J_s &:= \nu \sum_{i\in \mathcal{R}} \int_\bbr \mu_1(v) v^\gamma p(v)_x  \Big[ -\dot \Lambda_i   +  \lambda_{k_i}  \left({\mathbf{r}}_{k_i}  \Big( U^R_i  \Big) \right)\Big] \partial_x h^R_i  dx .
\end{aligned}
\end{align*}
To control $J_b$, we use a function 
\beq\label{phif}
\Phi(v):= \int^v_{2v_*} \frac{\gamma\mu_2(s)}{s} ds.
\eeq
Since $\Phi'(v) = - \mu_2(v) v^\gamma p'(v)$,
\begin{align*}
\begin{aligned}
J_b &= - \nu \sum_{i\in \mathcal{R}} \int_\bbr \partial_x \Phi(v)  \Big[ -\dot \Lambda_i   +  \lambda_{k_i}  \left({\mathbf{r}}_{k_i}  \Big( U^R_i  \Big) \right)\Big] \partial_x h^R_i  dx \\
&=  \nu \sum_{i\in \mathcal{R}} \int_\bbr  \Phi(v)\partial_x \left( \Big[ -\dot \Lambda_i   +  \lambda_{k_i}  \left({\mathbf{r}}_{k_i}  \Big( U^R_i  \Big) \right)\Big] \partial_x h^R_i \right) dx.
\end{aligned}
\end{align*}
Observe that $\mu_2(v) \le C \mathbf{1}_{v>2v_*} (v)$, and so
\beq\label{estphi}
 \Phi(v) =  \Phi(v)\mathbf{1}_{v>2v_*} (v) \le C \int^v_{2 v_*} \frac{1}{s} ds \le C |\ln v -\ln(2 v_*) | \le C p(v|\bar v). 
\eeq
In addition, since $|\Phi(\bar v)|\le C$ and (by Lemma \ref{lem:waves})
\[
\left| \partial_x \left( \Big[ -\dot \Lambda_i   +  \lambda_{k_i}  \left({\mathbf{r}}_{k_i}  \Big( U^R_i  \Big) \right)\Big] \partial_x h^R_i \right) \right| \le C \big( |(U^R_i)_{xx}| +|(U^R_i)_x|^2 \big)  \le C\frac{|\s_i|}{\nu} |(U^R_i)_x|,
\]
we have
\[
|J_b| \le C \sum_{i\in \mathcal{R}} |\s_i| \int_\bbr p(v|\bar v)  |(U^R_i)_x| dx \le C \delta  \sum_{i\in \mathcal{R}} G^R_i.
\]
To control $J_s$, we use the following diffusion to be derived from the parabolic part $\mathcal{P}^\nu$ in Section \ref{sec:shock}:
\beq\label{diffs}
D(U) : = \int_\bbr a \mu_1 (v) v^\gamma |(p(v)-p(\bar v))_x|^2 dx.
\eeq
Using \eqref{lamest}, we first have
\[
|J_s| \le C \nu \sum_{i\in \mathcal{R}} |\s_i| \int_{\bbr} \mu_1(v) v^\gamma |(p(v)_x|  |\partial_x h^R_i|  dx.
\]
 we have
\begin{align*}
\begin{aligned}
|J_s| &\le C \nu \sum_{i\in \mathcal{R}} |\s_i| \int_{\bbr} \mu_1(v) v^\gamma \Big( |(p(v)-p(\bar v))_x| + |p(\bar v)_x|\Big)  |\partial_x h^R_i|  dx \\
&\le  \nu \sum_{i\in \mathcal{R}} |\s_i| \int_{\bbr} a \mu_1 (v) v^\gamma |(p(v)-p(\bar v))_x|^2 dx + C \nu  \sum_{i\in \mathcal{R}} |\s_i|  \int_{\bbr} v^{\beta} |\partial_x h^R_i|^2  dx \\
&\quad + C \nu \sum_{i\in \mathcal{R}} |\s_i| \int_{\bbr} v^{\beta} |\bar v_x|  |\partial_x h^R_i|  dx \\
&\le L(0) \nu D + C \nu  \sum_{i\in \mathcal{R}} |\s_i|  \int_{\bbr} (1+ p(v|\bar v)) |\partial_x h^R_i|^2  dx \\
&\quad + C \nu \sum_{i\in \mathcal{R}} |\s_i| \int_{\bbr} (1+ p(v|\bar v)) |\bar v_x|  |\partial_x h^R_i|  dx.
\end{aligned}
\end{align*}
Using Lemma \ref{lem:waves} and
\begin{align*}
\begin{aligned}
& |\bar v_x| \le C \sum_{j\in \mathcal{S}} |(\tilde v_j)_x|+ \sum_{j\in \mathcal{R}} |(v^R_j)_x| + \sum_{j\in \mathcal{NP}} |(v^P_j)_x| \le C\frac{L(0)}{\nu}, \\
& \|\bar v_x\|_{L^1(\bbr)} \le  \sum_{j\in \mathcal{S}} \|(\tilde v_j)_x\|_{L^1(\bbr)}+ \sum_{j\in \mathcal{R}} \|(v^R_j)_x\|_{L^1(\bbr)} + \sum_{j\in \mathcal{NP}} \|(v^P_j)_x\|_{L^1(\bbr)} \le L(0) ,
\end{aligned}
\end{align*}
we have
\[
|J_s| \le  L(0) \nu D + C \delta^2 L(0) +C\delta^2  \sum_{i\in \mathcal{R}} G^R_i .
\]

Therefore, 
\[
J_3 \le C \delta  \sum_{i\in \mathcal{R}} G^R_i + L(0) \nu D + C \delta^2 L(0).
\]
Since $J_4 \le C\sum_{i\in \mathcal{R}} |\s_i|  \|\partial_x h^R_i\|_{L^1(\bbr)} \le C\delta L(0)$, we have
\[
\mathcal{H}_{232} \le  CC\delta^{1/4} \overline{\mathcal{D}}  + C\delta^{3/4}  + C \delta^{3/4}  \sum_{i\in \mathcal{R}} G^R_i + L(0) \nu D .
\]

Thus,
\[
\mathcal{H}_{23} \le C \delta^{1/4} \overline{\mathcal{D}}  + C \delta^{3/4}  + C  \delta^{3/4} \sum_{i\in \mathcal{R}} G^R_i + L(0) \nu D .
\]

Hence, combining the above estimates, and taking $L(0)<\delta_*$, we have
\begin{align}
\begin{aligned}\label{fh2}
\mathcal{H}_2^\nu &\le  -C^{-1} \sum_{i\in \mathcal{R}}  G_i^R + \delta_* \nu D + C \int_{\bbr} |h-\bar h|^2 dx +  C \delta^{1/4} \overline{\mathcal{D}}  + C \delta^{3/4}   + C \frac{\nu^{1/3}}{\delta}.
\end{aligned}
\end{align}
where 
\begin{align*}
\begin{aligned}
&G_i^R:=  \int_{\bbr} |(U^R_i)_x| p(v|\bar v) dx,\\
&D : = \int_{\bbr} a \mu_1 (v) v^\gamma |(p(v)-p(\bar v))_x|^2 dx.
\end{aligned}
\end{align*}

\subsection{Estimates of minor terms}
Before performing the normalization, we here estimate the remaining minor terms such as $\mathcal{H}_3^{\nu}$ and the third term of $\mathcal{H}_1^{\nu}$ in  \eqref{ybg-first}:
 \begin{align*}
\begin{aligned}
&\mathcal{H}_3^\nu := -  \sum_{i\in \mathcal{NP}}\int_{\bbr} a (h-\bar h)  p'(\bar v) \partial_x v^P_i dx,\\
& \mathcal{H}_{13}^\nu := - \int_{\bbr}  a (h-\bar h) \sum_{i\in \mathcal{S}} (p'(\bar v)-p'(\tilde v_i)) (\tilde v_i)_x dx,
 \end{aligned}
\end{align*}
First, we estimate $\mathcal{H}_3^\nu$:
\[
\mathcal{H}_3^\nu \le    \sum_{i\in \mathcal{NP}}\int_{\bbr} a |h-\bar h| |(v^P_i)_x| dx \le  \int_\bbr a |h-\bar h|^2 dx + C \int_\bbr \Big(\sum_{i\in \mathcal{NP}} |(v^P_i)_x| \Big)^2 dx
\]
Using Lemmas \ref{lem:waves}, \ref{lemma_key2},
\[
 \int_\bbr \Big(\sum_{i\in \mathcal{NP}} |(v^P_i)_x| \Big)^2 dx =  \int_\bbr \sum_{i\in \mathcal{NP}} |(v^P_i)_x|^2 dx \le C \sum_{i\in \mathcal{NP}} \frac{|\s_i|^2}{\sqrt\rn} \le C\sqrt\rn .
\]
Thus,
\beq\label{fh3}
\mathcal{H}_3^\nu \le  \int_\bbr a |h-\bar h|^2 dx + C\sqrt\rn .
\eeq
We use Lemma \ref{lem:finsout} to estimate 
\begin{align}
\begin{aligned} \label{fh13}
\mathcal{H}_{13}^\nu &\le C  \int_{\bbr}  a |h-\bar h| \sum_{i\in \mathcal{S}} |\bar v - \tilde v_i| |(\tilde v_i)_x| dx \\
&\le \int_\bbr a |h-\bar h|^2 dx + C \int_\bbr  \Big( \sum_{i\in \mathcal{S}}   |\bar v - \tilde v_i| | (\tilde v_i)_x| \Big)^2 dx \\
&\le \int_\bbr a |h-\bar h|^2 dx +  \frac{C\nu^{1/3}}{\delta^3} .
\end{aligned}
\end{align}

\section{Decomposition and estimates for parabolic term} \label{sec:para}
\setcounter{equation}{0}

Since the diffusion term $D$ introduced in \eqref{diffs} should be extracted from the parabolic term $\mathcal{P}^\nu(U)$ of \eqref{ybg-first}, we will decompose $\mathcal{P}^\nu(U)$ as follows. First, since $\mu_2(\tilde v_i)=0$ for all $i$ by \eqref{mu2} and $\tilde v_i <2v_*$, we have
\begin{align}
\begin{aligned} \label{deco-para}
\mathcal{P}^\nu(U) &= \mathcal{P}^b(U) +\mathcal{P}^s(U), \\
\mathcal{P}^b(U) &:= \nu \int_{\bbr} a (p(v)-p(\bar v) ) \partial_{x}\big(\mu_2(v ) v^\gamma \partial_x p(v )\big) dx, \\
\mathcal{P}^s(U) & :=  \nu \int_{\bbr} a (p(v)-p(\bar v) )\partial_{x}\Big(\mu_1(v ) v^\gamma \partial_x p(v ) - \sum_{i\in \mathcal{S}} \mu_1(\tilde v_i) \tilde v_i^{\gamma} \partial_x p(\tilde v_i)  \Big) dx .
\end{aligned}
\end{align}

\subsection{ Estimate of $\mathcal{P}^b$ }
First, integrate by parts as follows:
\begin{align*}
\begin{aligned}
\mathcal{P}^b(U) &= \nu \int_{\bbr} a p(v) \partial_{x}\big(\mu_2(v ) v^\gamma \partial_x p(v )\big) dx
- \nu \int_{\bbr} a p(\bar v)  \partial_{x}\big(\mu_2(v ) v^\gamma \partial_x p(v )\big) dx \\
& =: - \nu \mathcal{D}_b +  \mathcal{P}^b_1+\mathcal{P}^b_2+\mathcal{P}^b_3,
\end{aligned}
\end{align*}
where we have a good term 
\beq\label{defdb}
\mathcal{D}_b := \int_\bbr a  \mu_2(v ) v^\gamma |\partial_x p(v) |^2 dx,
\eeq
and
\begin{align*}
\begin{aligned}
\mathcal{P}^b_1 & := - \nu  \sum_{i\in \mathcal{S}}   \int_\bbr   (a_i)_x p(v) \mu_2(v ) v^\gamma \partial_x p(v ) dx, \\
\mathcal{P}^b_2 & := \nu \int_\bbr a  p'(\bar v) \bar v_x \mu_2(v ) v^\gamma p'(v ) v_x dx ,\\
\mathcal{P}^b_3 & :=  \nu  \sum_{i\in \mathcal{S}}   \int_\bbr  (a_i)_x p(\bar v) \mu_2(v ) v^\gamma p'(v ) v_x  dx.
\end{aligned}
\end{align*}

Since $\mu_2(v) \le C\mathbf{1}_{v>2v_*}(v)$ by \eqref{mu2}, and so $p(v)^2 \mu_2(v) v^\gamma \le C \mathbf{1}_{v>2v_*}(v) \le C p(v|\bar v)$, we have
\begin{align*}
\begin{aligned}
\mathcal{P}^b_1 
  & \le \frac{\nu}{2}\mathcal{D}_b + C\nu \int_\bbr \Big(\sum_{i\in \mathcal{S}} |(a_i)_x| \Big)^2 p(v|\bar v) dx.
\end{aligned}
\end{align*}
Using \eqref{avr} and Lemma \ref{lem:waves}, we have
\begin{align}
\begin{aligned} \label{p21}
&\nu   \int_\bbr \Big(\sum_{i\in \mathcal{S}} |(a_i)_x| \Big)^2  p(v|\bar v)  dx  \\
&\le  \frac{C\nu}{\delta_*^2}  \int_\bbr \Big(\sum_{i\in \mathcal{S}} |(\tilde v_i)_x|\Big)^2 p(v|\bar v)  dx \\
&\le \frac{C\nu}{\delta_*^2}  \int_\bbr   \left( \sum_{i\in \mathcal{S}} \frac{ |\s_i|^2 }{\nu} |(\tilde v_i)_x|  +  \sum_{i\in \mathcal{S}} |(\tilde v_i)_x| \bigg( \sum_{i'\in \mathcal{S}-\{i\}}  \frac{ |\s_{i'}|^2 }{\nu}  \bigg) \right) p(v|\bar v)  dx \\
&\le  \frac{CL(0)}{\delta_*^2}   \sum_{i\in \mathcal{S}}  \int_\bbr |(\tilde v_i)_x| p(v|\bar v) dx .
\end{aligned}
\end{align}
Thus, using $L(0)<\delta_*^3$,
\[
\mathcal{P}^b_1  \le  \frac{\nu}{2}\mathcal{D}_b +  C\delta_*   \sum_{i\in \mathcal{S}}  \int_\bbr |(\tilde v_i)_x| p(v|\bar v) dx.
\]

To control $\mathcal{P}^b_2$, as done before, we use \eqref{phif}
to have
\begin{align*}
\begin{aligned}
\mathcal{P}^b_2  &= -  \nu \int_\bbr a  p'(\bar v) \bar v_x \partial_x \Phi(v) dx \\
 & =  \nu \sum_{i\in \mathcal{S}} \int_\bbr (a_i)_x  p'(\bar v) \bar v_x  \Phi(v) dx +  \nu \int_\bbr a  (p'(\bar v) \bar v_x )_x  \Phi(v) dx.
\end{aligned}
\end{align*}
Using \eqref{avr} and \eqref{estphi}, we have
\begin{align*}
\begin{aligned}
\mathcal{P}^b_2 & \le C \delta_*^{-1} \nu  \int_{v>2v_*} \Big| \sum_{i\in \mathcal{S}} (\tilde v_i)_x \Big| \Big|  \sum_{i\in \mathcal{S}}(\tilde v_i)_x+ \sum_{i\in \mathcal{R}}(v^R_i)_x+ \sum_{i\in \mathcal{NP}}(v^P_i)_x \Big| p(v|\bar v) dx\\
&\quad\quad + C  \nu  \int_{v>2v_*}  \Big|  \sum_{i\in \mathcal{S}}(\tilde v_i)_x+ \sum_{i\in \mathcal{R}}(v^R_i)_x+ \sum_{i\in \mathcal{NP}}(v^P_i)_x \Big|^2 p(v|\bar v) dx\\
&\quad\quad + C  \nu  \int_{v>2v_*} \Big|  \sum_{i\in \mathcal{S}}(\tilde v_i)_{xx}+ \sum_{i\in \mathcal{R}}(v^R_i)_{xx}+ \sum_{i\in \mathcal{NP}}(v^P_i)_{xx} \Big| p(v|\bar v) dx,
\end{aligned}
\end{align*}
Observe that using 1) and 2) of  Lemma \ref{lem:waves},
\begin{align*}
\begin{aligned}
& \frac{C\nu}{ \delta_*} \Big| \sum_{i\in \mathcal{S}} (\tilde v_i)_x \Big| \Big|  \sum_{i\in \mathcal{S}}(\tilde v_i)_x+ \sum_{i\in \mathcal{R}}(v^R_i)_x+ \sum_{i\in \mathcal{NP}}(v^P_i)_x \Big| \\
&+  C\nu \bigg( \Big|  \sum_{i\in \mathcal{S}}(\tilde v_i)_x+ \sum_{i\in \mathcal{R}}(v^R_i)_x+ \sum_{i\in \mathcal{NP}}(v^P_i)_x \Big|^2 
+\Big|  \sum_{i\in \mathcal{S}}(\tilde v_i)_{xx}+ \sum_{i\in \mathcal{R}}(v^R_i)_{xx}+ \sum_{i\in \mathcal{NP}}(v^P_i)_{xx} \Big| \bigg)\\
&\le \frac{C\nu}{ \delta_*}  \bigg[ \sum_{i\in \mathcal{S}}  \frac{ |\s_i|^2 }{\nu} |(\tilde v_i)_x|  +  \sum_{i\in \mathcal{S}} |(\tilde v_i)_x| \bigg( \sum_{i'\in \mathcal{S}-\{i\}}|(\tilde v_{i'})_x| + \sum_{i'\in \mathcal{R}}|(v^R_{i'})_x|+ \sum_{i'\in \mathcal{NP}}|(v^P_{i'})_x| \bigg)\bigg] \\
&\quad + C\nu    \bigg[ \sum_{i\in \mathcal{R}}  \frac{ |\s_i|^2}{\nu}|(v^R_i)_x| + \sum_{i\in \mathcal{NP}} |(v^P_i)_x|^2  \bigg] \\
&\quad + C\nu    \bigg[  \sum_{i , i' \in \mathcal{R}, ~ i\neq i'} |(v^R_i)_x|  |(v^R_{i'})_x|  + \sum_{i , i' \in \mathcal{NP}, ~ i\neq i'} |(v^P_i)_x|  |(v^P_{i'})_x|+ \sum_{i \in \mathcal{R}, ~ i'\in \mathcal{NP} } |(v^R_i)_x|  |(v^P_{i'})_x|  \bigg] \\
&\quad +C\nu    \Big(  \sum_{i\in \mathcal{S}} \frac{ |\s_i|}{\nu} |(\tilde v_i)_{x}| +\sum_{i\in \mathcal{R}} \frac{ |\s_i|}{\nu} |(v^R_i)_{x}| + \sum_{i\in \mathcal{NP}} |(v^P_i)_{xx}| \Big) .
\end{aligned}
\end{align*}
Then, using 3) and 5) of Lemma \ref{lem:waves}, 
we have
\begin{align*}
\begin{aligned}
\mathcal{P}^b_2 &\le \frac{C}{\delta_*}  \sum_{i\in \mathcal{S}}  \int_{v>2v_*}  |\s_i| |(\tilde v_i)_x| p(v|\bar v) dx
 + C\delta \sum_{i\in \mathcal{R}}   \int_{v>2v_*} |(v^R_i)_x| p(v|\bar v) dx
+  \frac{C\sqrt\nu L(0)}{ \delta_*} \int_{v>2v_*} p(v|\bar v) dx .
\end{aligned}
\end{align*}
In addition, for the last term, using
\beq\label{pqbig}
p(v|\bar v) \le C Q(v|\bar v) \quad \mbox{for } v>2v_*,
\eeq
we have
\begin{align*}
\begin{aligned}
\mathcal{P}^b_2& \le \frac{C}{\delta_*}  \sum_{i\in \mathcal{S}}  \int_\bbr  |\s_i| |(\tilde v_i)_x| p(v|\bar v) dx
 + C\delta \sum_{i\in \mathcal{R}}   \int_\bbr |(v^R_i)_x| p(v|\bar v) dx
+  \frac{C\sqrt\nu L(0)}{ \delta_*} \int_\bbr Q(v|\bar v) dx .
\end{aligned}
\end{align*}
Especially, using $|\s_i|\le L(0)< \delta_*^3$, we have
\[
\mathcal{P}^b_2 \le C \delta_*  \sum_{i\in \mathcal{S}}  \int_\bbr |(\tilde v_i)_x| p(v|\bar v) dx
 + C\delta \sum_{i\in \mathcal{R}}   G_i^R
+  C \int_\bbr Q(v|\bar v) dx .
\]
Next, as above, we use  \eqref{phif} and  \eqref{estphi} to have
\begin{align*}
\begin{aligned}
\mathcal{P}^b_3 &=  -\nu  \sum_{i\in \mathcal{S}}   \int_\bbr  (a_i)_x p(\bar v) \partial_x \Phi(v)  dx \\
& \le C\nu   \int_{v>2v_*} \left( \sum_{i\in \mathcal{S}} | (a_i)_{xx}| + \sum_{i\in \mathcal{S}} | (a_i)_{x}| \Big|  \sum_{j\in \mathcal{S}}(\tilde v_i)_x+ \sum_{j\in \mathcal{R}}(v^R_i)_x+ \sum_{j\in \mathcal{NP}}(v^P_i)_x \Big| \right) p(v|\bar v) dx \\
&  \le \frac{C\nu}{\delta_*} \int_{v>2v_*} \left( \sum_{i\in \mathcal{S}}  \frac{ |\s_i|}{\nu} |(\tilde v_i)_x|  + \sum_{i\in \mathcal{S}} |(\tilde v_i)_x|  \Big|  \sum_{j\in \mathcal{S}}(\tilde v_i)_x+ \sum_{j\in \mathcal{R}}(v^R_i)_x+ \sum_{j\in \mathcal{NP}}(v^P_i)_x \Big| \right) p(v|\bar v) dx \\
& \le \frac{C}{\delta_*}  \sum_{i\in \mathcal{S}}  \int_{v>2v_*}  |\s_i| |(\tilde v_i)_x| p(v|\bar v) dx
+  \frac{C\sqrt\nu L(0)}{ \delta_*} \int_{v>2v_*} Q(v|\bar v) dx \\
&\le C \delta_* \sum_{i\in \mathcal{S}}  \int_\bbr |(\tilde v_i)_x| p(v|\bar v) dx
+  C \int_\bbr Q(v|\bar v) dx .
\end{aligned}
\end{align*}

Therefore, we have
\begin{align}
\begin{aligned} \label{fpb}
\mathcal{P}^b
  \le -\frac{\nu}{2}\mathcal{D}_b + C \delta_*  \sum_{i\in \mathcal{S}}  \int_\bbr |(\tilde v_i)_x| p(v|\bar v) dx
 + C\delta \sum_{i\in \mathcal{R}}   G_i^R
+  C \int_\bbr Q(v|\bar v) dx .
\end{aligned}
\end{align}

\subsection{ Estimate of $\mathcal{P}^s$ }
First, integrate by parts as follows:
\begin{align*}
\begin{aligned}
\mathcal{P}^s(U) &= \nu \int_{\bbr} a (p(v)-p(\bar v) )\partial_{x}\Big(\mu_1(v ) v^\gamma \partial_x p(v ) - \sum_{i\in \mathcal{S}} \mu_1(\tilde v_i) \tilde v_i^{\gamma} \partial_x p(\tilde v_i)  \Big) dx \\
& =:  \mathcal{P}^s_1+\mathcal{P}^s_2+\mathcal{P}^s_3+\mathcal{P}^s_4,
\end{aligned}
\end{align*}
where
\begin{align*}
\begin{aligned}
& \mathcal{P}^s_1 : = -\nu \int_{\bbr} a \partial_x (p(v)-p(\bar v) ) \mu_1(v ) v^\gamma \Big( \partial_x p(v ) - \sum_{i\in \mathcal{S}} \partial_x p(\tilde v_i)  \Big) dx \\
& \mathcal{P}^s_2 : =  -\nu \int_{\bbr} a \partial_x (p(v)-p(\bar v) )  \sum_{i\in \mathcal{S}}\Big(\mu_1(v ) v^\gamma -\mu_1(\tilde v_i) \tilde v_i^{\gamma}\Big)  \partial_x p(\tilde v_i) dx \\
&\mathcal{P}^s_3 : = -\nu \int_{\bbr} \Big(\sum_{i\in \mathcal{S}} |(a_i)_x| \Big)  (p(v)-p(\bar v) ) \mu_1(v ) v^\gamma \Big( \partial_x p(v ) - \sum_{i\in \mathcal{S}} \partial_x p(\tilde v_i)  \Big) dx\\
&\mathcal{P}^s_4 : = -\nu \int_{\bbr}\Big(\sum_{i\in \mathcal{S}} |(a_i)_x| \Big)  (p(v)-p(\bar v) )  \sum_{i\in \mathcal{S}}\Big(\mu_1(v ) v^\gamma -\mu_1(\tilde v_i) \tilde v_i^{\gamma}\Big)  \partial_x p(\tilde v_i) dx .
\end{aligned}
\end{align*}

\noindent $\bullet$ {\bf Estimate on $ \mathcal{P}^s_1$ : }
First, we obtain the desired good term
\[D(U) = \int_\bbr a \mu_1(v ) v^\gamma |( p(v) -p(\bar v))_x |^2 dx,\]
from 
\begin{align*}
\begin{aligned} 
\mathcal{P}^s_1 &= - \nu D  \underbrace{-\nu \int_{\bbr} a \partial_x (p(v)-p(\bar v) ) \mu_1(v ) v^\gamma \Big( \partial_x p(\bar v ) - \sum_{i\in \mathcal{S}} \partial_x p(\tilde v_i)  \Big) dx  }_{=:\mathcal{P}^s_{11}}.
\end{aligned}
\end{align*}
Rewrite
\begin{align*}
\begin{aligned} 
\mathcal{P}^s_{11} &= - \nu \int_\bbr a  \partial_x (p(v)-p(\bar v)) \mu_1(v ) v^\gamma  p'(\bar v ) \Big( \sum_{i\in \mathcal{R}} (v^R_i)_x + \sum_{i\in \mathcal{NP}} (v^P_i)_x \Big) dx \\
& -\nu \int_{\bbr} a \partial_x (p(v)-p(\bar v) ) \mu_1(v ) v^\gamma \sum_{i\in \mathcal{S}} \Big(  p'(\bar v ) -  p'(\tilde v_i)  \Big) (\tilde v_i)_x dx \\
&=: \mathcal{P}^s_{111} + \mathcal{P}^s_{112}.
\end{aligned}
\end{align*}
Since  $\mu_1(v ) v^\gamma\le C v^\beta $ by the conditions \eqref{mu1},  
we have
\[
 \mathcal{P}^s_{111}  
\le \delta_* \nu D + C \underbrace{\frac{\nu}{\delta_*}  \int_\bbr \Big(\sum_{i\in \mathcal{R}} |(v^R_i)_x| +\sum_{i\in \mathcal{NP}} |(v^P_i)_x| \Big)^2 v^\beta\,dx}_{=:J}.
\]
Then, using $v^{\beta} \le C+ C p(v|\bar v) \mathbf{1}_{v>2v_*}$ by $\beta:=\gamma-\alpha \in (0,1)$, and Lemma \ref{lem:waves} and \eqref{pqbig} to have
\begin{align}
\begin{aligned} \label{jest}
J & \le  C \frac{\nu}{\delta_*}  \left( \sum_{i\in \mathcal{R}} \frac{\s_i^2}{\nu}  \int_\bbr  |(v^R_i)_x| dx +\sum_{i\in \mathcal{R}, ~ i'\in \mathcal{NP}}  \int_{|x-y_{i'}|\le\sqrt\rn}  |(v^R_i)_x| \frac{|\s_{i'}|}{\rn}  dx + \sum_{i\in \mathcal{NP}} \frac{|\s_i|^2}{\rn\sqrt\rn}   \right) \\
&\quad +  C \frac{\nu}{\delta_*} \int_{v>2v_*}  \left( \sum_{i\in \mathcal{R}} \frac{\s_i^2}{\nu}   |(v^R_i)_x|  +\sum_{i\in \mathcal{R}, ~ i'\in \mathcal{NP}}    |(v^R_i)_x| \frac{|\s_{i'}|}{\rn}  + \sum_{i\in \mathcal{NP}} \frac{|\s_i|^2}{\rn\sqrt\rn}   \right)p(v|\bar v)  dx\\
& \le  C  \frac{\nu}{\delta_*}  \left(  \frac{\delta^2}{\nu} \sum_{i\in \mathcal{R}} \s_i +\sum_{i\in \mathcal{R}, ~ i'\in \mathcal{NP}}   \frac{\s_i |\s_{i'}|}{\sqrt\rn} + \sum_{i\in \mathcal{NP}} \frac{|\s_i|^2}{\rn\sqrt\rn} \right) \\
&\quad +  C \frac{\delta^2+ \nu}{\delta_*} \sum_{i\in \mathcal{R}}  \int_\bbr   |(v^R_i)_x|  p(v|\bar v) dx + C\frac{\sqrt\nu}{\delta_*}  \int_\bbr   Q(v|\bar v) dx\\
&\le   C   \frac{\delta^2}{\delta_*}  +  C   \frac{\sqrt\nu}{\delta_*} +  C \frac{\delta^2+ \nu}{\delta_*} \sum_{i\in \mathcal{R}} G_i^R + C  \int_\bbr   Q(v|\bar v) dx.
\end{aligned}
\end{align}
Thus, especially by $\nu<\delta^2$,
\[
 \mathcal{P}^s_{111} \le  \delta_* \nu D +  C   \frac{\delta^2}{\delta_*}  +  C   \frac{\sqrt\nu}{\delta_*} +  C \frac{\delta^2}{\delta_*} \sum_{i\in \mathcal{R}} G_i^R + C  \int_\bbr   Q(v|\bar v) \,dx.
\]
Likewise, using Lemma \ref{lem:finsout}, we have
\begin{align*}
\begin{aligned}
 \mathcal{P}^s_{112}  &\le \nu \int_{\bbr} a |\partial_x (p(v)-p(\bar v) )| \mu_1(v ) v^\gamma \sum_{i\in \mathcal{S}} \Big|  \bar v  -  \tilde v_i  \Big| |(\tilde v_i)_x| dx \\
&\le \delta_* \nu D + C \frac{\nu}{\delta_*}  \int_\bbr \Big( \sum_{i\in \mathcal{S}} \Big|  \bar v  -  \tilde v_i  \Big| |(\tilde v_i)_x|  \Big)^2 \Big(1+ v^\beta \mathbf{1}_{v>2v_*} \Big)\, dx \\
&\le  \delta_* \nu D + C \frac{\nu^{4/3}}{\delta\delta_*} +  \frac{C\nu^{4/3}}{\delta^3\delta_*} \sum_{i\in \mathcal{S}} \int_\bbr |(\tilde v_i)_x| p(v|\bar v)  dx +  C  \frac{C\nu^{4/3}}{\delta^3\delta_*} \int_\bbr Q(v|\bar v)  dx.
\end{aligned}
\end{align*}
Thus, taking $\nu\ll1$ such that $\frac{\nu^{4/3}}{\delta^3\delta_*} \le C\delta_*$, we have
\begin{align*}
\begin{aligned} 
\mathcal{P}^s_1 &= - \nu(1-2\delta_*) D    +  C \delta \sum_{i\in \mathcal{R}} G_i^R + C  \int_\bbr   Q(v|\bar v) \,dx + C\delta_* \sum_{i\in \mathcal{S}} \int_\bbr |(\tilde v_i)_x| p(v|\bar v)  dx \\
&\quad +  C   \frac{\delta^2}{\delta_*}  +  C   \frac{\sqrt\nu}{\delta_*} + C \frac{\nu^{4/3}}{\delta\delta_*} .
\end{aligned}
\end{align*}

\noindent $\bullet$ {\bf Estimate on $ \mathcal{P}^s_2$ : }
\begin{align*}
\begin{aligned}
\mathcal{P}^s_2 & \le  \nu \int_{\bbr} a |\partial_x (p(v)-p(\bar v) )| \Big|\mu_1(v ) v^\gamma -\mu_1(\bar v) \bar v^{\gamma}\Big|   \sum_{i\in \mathcal{S}}
|(\tilde v_i)_x| dx \\
&\quad +C \nu \int_{\bbr} a |\partial_x (p(v)-p(\bar v) )|  \sum_{i\in \mathcal{S}}\big|\bar v -\tilde v_i \big|  
|(\tilde v_i)_x| dx \\
&=: \mathcal{P}^s_{21} + \mathcal{P}^s_{22}.
\end{aligned}
\end{align*}
Observe that since $0\le \beta:=\gamma-\alpha \le 1 < \gamma$ and  $\mu_1(v ) v^\gamma \sim v^\beta$ by \eqref{mu1},
\[
 \frac{| \mu_1(v ) v^\gamma-\, \mu_1(\bar v) \bar v^\gamma |^2}{\mu_1(v ) v^\gamma } \le C v^{-\beta}\mathbf{1}_{v\le v_*/2} + C|v^\beta - \bar v^\beta|^2 \mathbf{1}_{v\ge v_*/2} \le C p(v|\bar v).
\]
This together with using the same estimates as \eqref{p21} implies
\begin{align*}\begin{aligned}
 \mathcal{P}^s_{21}  &\le \nu C \sqrt{ D } \sqrt{ \int_\bbr \Big(\sum_{i\in \mathcal{S}} |(\tilde v_i)_x|\Big)^2  \frac{| \mu_1(v ) v^\gamma- \,\mu_1(\bar v) \bar v^\gamma |^2}{\mu_1(v ) v^\gamma } \, dx } \\
& \le \delta_* \nu D + C \delta_*^{-1} \nu  \int_\bbr \Big(\sum_{i\in \mathcal{S}} |(\tilde v_i)_x|\Big)^2 p(v|\bar v) \, dx \\
& \le \delta_* \nu D  +  \frac{CL(0)}{\delta_*}  \sum_{i\in \mathcal{S}}  \int_\bbr |(\tilde v_i)_x| p(v|\bar v) dx .
\end{aligned}\end{align*}
Likewise, using \eqref{p21} and Lemma \ref{lem:finsout}, we have
\begin{align*}\begin{aligned}
 \mathcal{P}^s_{22}  &\le \nu C \sqrt{ D } \sqrt{ \int_\bbr \Big( \sum_{i\in \mathcal{S}} \Big|  \bar v  -  \tilde v_i  \Big| |(\tilde v_i)_x|  \Big)^2  v^{-\beta} \, dx } \\
& \le \delta_* \nu D + \frac{C\nu}{\delta_*}   \int_{\bbr} \Big( \sum_{i\in \mathcal{S}} \Big|  \bar v  -  \tilde v_i  \Big| |(\tilde v_i)_x|  \Big)^2  \Big( p(v|\bar v) \mathbf{1}_{v\le v_*/2} +  \mathbf{1}_{v< v_*/2} \Big) \, dx \\
& \le \delta_* \nu D + \frac{C\nu}{\delta_*}   \int_{v\le v_*/2} \Big( \sum_{i\in \mathcal{S}} |(\tilde v_i)_x|  \Big)^2  p(v|\bar v) \, dx + \frac{C\nu}{\delta_*}   \int_{\bbr} \Big( \sum_{i\in \mathcal{S}} \Big|  \bar v  -  \tilde v_i  \Big| |(\tilde v_i)_x|  \Big)^2\, dx \\
& \le \delta_* \nu D  +  \frac{CL(0)}{\delta_*}  \sum_{i\in \mathcal{S}}  \int_\bbr |(\tilde v_i)_x| p(v|\bar v) dx +  C \frac{\nu^{4/3}}{\delta\delta_*}.
\end{aligned}\end{align*}
Thus,
\[
\mathcal{P}^s_2 \le 2\delta_* \nu D  +  C\delta_*  \sum_{i\in \mathcal{S}}  \int_\bbr |(\tilde v_i)_x| p(v|\bar v) dx + +  C \frac{\nu^{4/3}}{\delta\delta_*}.
\]

\noindent $\bullet$ {\bf Estimate on $ \mathcal{P}^s_3$ : }

\begin{align*}
\begin{aligned}
\mathcal{P}^s_3 & = -\nu \int_{\bbr} \Big(\sum_{i\in \mathcal{S}} |(a_i)_x| \Big)  (p(v)-p(\bar v) ) \mu_1(v ) v^\gamma  \partial_x \big( p(v ) - p(\bar v)  \big) dx\\
 &- \nu \int_\bbr \Big(\sum_{i\in \mathcal{S}} |(a_i)_x| \Big)  (p(v)-p(\bar v) ) \mu_1(v ) v^\gamma p'(\bar v) \Big( \sum_{i\in \mathcal{R}} (v^R_i)_x + \sum_{i\in \mathcal{NP}} (v^P_i)_x \Big) dx \\
 & -\nu \int_{\bbr} \Big(\sum_{i\in \mathcal{S}} |(a_i)_x| \Big) (p(v)-p(\bar v) ) \mu_1(v ) v^\gamma \sum_{i\in \mathcal{S}} \Big(  p'(\bar v ) -  p'(\tilde v_i)  \Big) (\tilde v_i)_x dx \\
 & =:  \mathcal{P}^s_{31} +  \mathcal{P}^s_{32} +  \mathcal{P}^s_{33} .   
\end{aligned}
\end{align*}

\begin{align*}\begin{aligned}
 \mathcal{P}^s_{31}  &\le \nu C \sqrt{ D } \sqrt{ \int_\bbr \Big(\sum_{i\in \mathcal{S}} |(a_i)_x| \Big)^2 v^\beta |p(v)-p(\bar v) |^2 \, dx } \\
& \le \delta_* \nu D +  \frac{C\nu}{\delta_*} \int_\bbr \Big(\sum_{i\in \mathcal{S}} |(a_i)_x| \Big)^2 v^\beta |p(v)-p(\bar v) |^2 \, dx ,
\end{aligned}\end{align*}
where the last term will be controlled later.\\
Using \eqref{jest},
\begin{align*}\begin{aligned}
 \mathcal{P}^s_{32}  
& \le C\nu \int_\bbr \Big(\sum_{i\in \mathcal{S}} |(a_i)_x| \Big)^2 v^\beta |p(v)-p(\bar v) |^2 \, dx + C\nu  \int_\bbr \Big(\sum_{i\in \mathcal{R}} |(v^R_i)_x| +\sum_{i\in \mathcal{NP}} |(v^P_i)_x| \Big)^2 v^\beta\,dx \\
&\le C\nu \int_\bbr \Big(\sum_{i\in \mathcal{S}} |(a_i)_x| \Big)^2 v^\beta |p(v)-p(\bar v) |^2 \, dx\\
&\quad + C  \delta^2  +  C\sqrt\nu +  C (\delta^2+ \nu) \sum_{i\in \mathcal{R}} G_i^R + C  \int_\bbr   Q(v|\bar v) \,dx.
\end{aligned}\end{align*}
Likewise, we use Lemma \ref{lem:finsout} to have
\begin{align*}\begin{aligned}
 \mathcal{P}^s_{33}  
& \le C\nu \int_\bbr \Big(\sum_{i\in \mathcal{S}} |(a_i)_x| \Big)^2 v^\beta |p(v)-p(\bar v) |^2 \, dx \\
&\quad +C\nu \int_{v\le 2v_*}  \Big( \sum_{i\in \mathcal{S}}   |\bar v - \tilde v_i| | (\tilde v_i)_x| \Big)^2 dx +C\nu \int_{v>2v_*}  \Big( \sum_{i\in \mathcal{S}}   |\bar v - \tilde v_i| | (\tilde v_i)_x| \Big)^2 v^\beta dx \\
& \le C\nu \int_\bbr \Big(\sum_{i\in \mathcal{S}} |(a_i)_x| \Big)^2 v^\beta |p(v)-p(\bar v) |^2 \, dx \\
&\quad +C \frac{\nu^{1/3}}{\delta} + \frac{C\nu^{1/3}}{\delta^3} \sum_{i\in \mathcal{S}} \int_\bbr |(\tilde v_i)_x| p(v|\bar v)  dx +   \frac{C\nu^{1/3}}{\delta^3}  \int_\bbr Q(v|\bar v)  dx .
\end{aligned}\end{align*}
Thus,  taking $\nu\ll1$ as $\frac{\nu^{1/3}}{\delta^3} <\delta_*$,
\begin{align*}\begin{aligned}
 \mathcal{P}^s_{3}  
& \le \delta_* \nu D +  \frac{C\nu}{\delta_*} \int_\bbr \Big(\sum_{i\in \mathcal{S}} |(a_i)_x| \Big)^2 v^\beta |p(v)-p(\bar v) |^2 \, dx +  C \delta^2 \sum_{i\in \mathcal{R}} G_i^R + C  \int_\bbr   Q(v|\bar v) \,dx  \\
&\quad  +C\delta_* \sum_{i\in \mathcal{S}} \int_\bbr |(\tilde v_i)_x| p(v|\bar v)  dx + C  \delta^2  +  C\sqrt\nu+C \frac{\nu^{1/3}}{\delta} .
\end{aligned}\end{align*}

\noindent $\bullet$ {\bf Estimate on $ \mathcal{P}^s_4$ : }
\begin{align*}
\begin{aligned}
\mathcal{P}^s_4 & \le 
 \nu \int_{\bbr} \Big(\sum_{i\in \mathcal{S}} |(a_i)_x| \Big)  |p(v)-p(\bar v)| \Big|\mu_1(v ) v^\gamma -\mu_1(\bar v) \bar v^{\gamma}\Big|   \sum_{i\in \mathcal{S}}
|(\tilde v_i)_x| dx \\
&\quad +C \nu \int_{\bbr} \Big(\sum_{i\in \mathcal{S}} |(a_i)_x| \Big)  |p(v)-p(\bar v)|   \sum_{i\in \mathcal{S}}\big|\bar v -\tilde v_i \big|  
|(\tilde v_i)_x| dx \\
&=: \mathcal{P}^s_{41} + \mathcal{P}^s_{42}.
\end{aligned}
\end{align*}

Since 
\[
 |p(v)-p(\bar v)| \Big|\mu_1(v ) v^\gamma -\mu_1(\bar v) \bar v^{\gamma}\Big| \le C v^{-\alpha} \mathbf{1}_{v\le v_*/2} + C p(v|\bar v)\mathbf{1}_{v>v_*/2} \le Cp(v|\bar v),
\]
we use \eqref{p21} to have
\[
 \mathcal{P}^s_{41} \le C\nu\delta_* \int_\bbr \Big(\sum_{i\in \mathcal{S}} |(a_i)_x| \Big)^2  p(v|\bar v)  dx \le  \frac{CL(0)}{\delta_*}   \sum_{i\in \mathcal{S}}  \int_\bbr |(\tilde v_i)_x| p(v|\bar v) dx.
\]
Since 
\[
 |p(v)-p(\bar v)|  \le Cp(v|\bar v) + C,
\]
we use Lemma \ref{lem:finsout} to have
\begin{align*}
\begin{aligned}
 \mathcal{P}^s_{42} 
 &\le C\frac{L(0)}{\delta_*} \int_{\bbr}  (p(v|\bar v) + 1)\sum_{i\in \mathcal{S}}\big|\bar v -\tilde v_i \big|  
|(\tilde v_i)_x| dx \\
 & \le 
 C\frac{L(0)}{\delta_*} \int_{\bbr} \sum_{i\in \mathcal{S}}
|(\tilde v_i)_x| p(v|\bar v) dx +  C\frac{L(0)}{\delta_*} \int_{\bbr} \sum_{i\in \mathcal{S}}\big|\bar v -\tilde v_i \big|  
|(\tilde v_i)_x| dx \\
& \le 
 C\frac{L(0)}{\delta_*} \int_{\bbr} \sum_{i\in \mathcal{S}}
|(\tilde v_i)_x| p(v|\bar v) dx + \frac{C\nu^{1/2}}{\delta\delta_*}.
\end{aligned}
\end{align*}

Hence we have
\begin{align}\begin{aligned} \label{fpara}
 \mathcal{P}
& \le -\frac{\nu}{2}\mathcal{D}_b -\nu(1-C\delta_*)  D +  \frac{C\nu}{\delta_*} \int_\bbr \Big(\sum_{i\in \mathcal{S}} |(a_i)_x| \Big)^2 v^\beta |p(v)-p(\bar v) |^2 \, dx    \\
&\quad+  C \delta \sum_{i\in \mathcal{R}} G_i^R  + C  \int_\bbr   Q(v|\bar v) \,dx +C\delta_* \sum_{i\in \mathcal{S}} \int_\bbr |(\tilde v_i)_x| p(v|\bar v)  dx +C   \frac{\delta^2}{\delta_*}+\mathcal{C} (\delta, \nu),
\end{aligned}\end{align}
where $\mathcal{C} (\delta, \nu)$ is the constant that vanishes when $\nu\to0$ for any fixed $\delta>0$.

\section{Estimates for main terms localized by derivatives of weights and shocks}\label{sec:shock}
\setcounter{equation}{0}

\subsection{Normalization} 
First, combining \eqref{ineq-0}, \eqref{ybg-first}, \eqref{fh2}, \eqref{fpara}, \eqref{fh3} and \eqref{fh13}, and recalling the original notations \eqref{abbu}, \eqref{abb}, \eqref{abba} depending on $\nu$, we obtain the following estimates:  for all $ t\in (t_j,t_{j+1}) $,
\begin{align}
\begin{aligned}\label{ineq-mid}
&\frac{d}{dt} \int_{\bbr} a^\nu \eta(U^\nu |\overline U_{\nu,\delta} ) dx \\
&\le \mathcal{R}^\nu(U^\nu)   + C \int_\bbr a^\nu \eta(U^\nu|\bar U_{\nu,\delta}) dx +C \delta^{1/4}  \overline{\mathcal{D}}  -\frac{\nu}{2}\mathcal{D}_b -C^{-1} \sum_{i\in \mathcal{R}}  G_i^R \\
&\quad  + C \delta^{3/4}  +   C   \frac{\delta^2}{\delta_*}  +\mathcal{C} (\delta, \nu),
\end{aligned}
\end{align}
where
 \begin{align*}
\begin{aligned}
 \mathcal{R}^\nu(U^\nu) &:= \sum_{i\in \mathcal{S}}\dot X_{ij}^\nu(t) Y_i^\nu  +\overline{\mathcal{H}_{1}^\nu}  - \mathcal{J}^\nu_{good}   -  \nu (1-C\delta_*)  D^\nu +  C\delta_*  \sum_{i\in \mathcal{S}}  \int_\bbr |(\tilde v^\nu_{ij})_x| p(v^\nu|\bar v_{\nu,\delta}) dx \\
 &\quad +  \frac{C\nu}{\delta_*} \int_\bbr \Big(\sum_{i\in \mathcal{S}} |(a_i^\nu)_x| \Big)^2 (v^\nu)^\beta |p(v^\nu)-p(\bar v_{\nu,\delta}) |^2 \, dx.
 \end{aligned}
\end{align*}
Here, $G_i^R$, $\mathcal{D}_b$ and $\overline{\mathcal{D}}$ are as in \eqref{gir}, \eqref{defdb} and \eqref{odiff} respectively, and
\begin{align}
\begin{aligned}\label{ybg}
&Y^\nu_i:= -  \int_{\bbr} (a_i^\nu)_x  \eta(U^\nu|\overline U_{\nu,\delta} ) dx + \int_{\bbr} a \nabla^2\eta(\overline U_{\nu,\delta}) (U^\nu- \overline U_{\nu,\delta})   \partial_x \tilde U_{ij}^\nu dx ,\\
&\overline{\mathcal{H}_{1}^\nu}:= \sum_{i\in \mathcal{S}}  \int_\bbr (a_i^\nu)_x \big(p(v^\nu)-p(\bar v_{\nu,\delta})\big) \big(h-\bar h_{\nu,\delta} \big) dx +\sum_{i\in \mathcal{S}}  \tilde\lam_i \int_\bbr a (\tilde v_{ij}^\nu)_x  p(v^\nu| \bar v_{\nu,\delta}) dx \\
& \mathcal{J}^\nu_{good}:=  \sum_{i\in \mathcal{S}} \int_{\bbr} \tilde\lambda_i (a_i^\nu)_x \eta(U^\nu|\overline U_{\nu,\delta} ) dx  ,\\
&D^\nu := \int_\bbr a^\nu\, \mu_1(v^\nu ) (v^\nu)^\gamma |( p(v^\nu) -p(\bar v_{\nu,\delta}))_x |^2 dx .
\end{aligned}
\end{align}
We will estimate $ \mathcal{R}^\nu(U^\nu)$.
 For that, as done in \cite{Kang-V-NS17,KV-Inven,KV-2shock}, it would be easier to normalize $\mathcal{R}^\nu$  as $\nu=1$ by using the following scaling 
: 
\begin{align}
\begin{aligned} \label{newquan}
&U(x,t):=U^{\nu}(\nu x, \nu t),\quad \overline U(x,t):= \overline U_{\nu,\delta} (\nu x, \nu t) , \\
& X_i(t):= \frac{1}{\nu} X^\nu_{ij} (\nu t) ,\quad \Lambda_i(t):= \frac{1}{\nu} \Lambda^\nu_{ij} (\nu t) ,\quad a(x,t):= a^\nu (\nu x, \nu t) ,\quad a_i(x,t):= a_i^\nu (\nu x, \nu t),
\end{aligned}
\end{align}
and so,
\[
\tilde U_i(x,t):= \tilde U^\nu_{ij} (\nu x, \nu t) ,\quad U^R_{i}(x,t):=U^{R,\nu}_{ij}(\nu x, \nu t) ,\quad U^P_{i}(x,t):=U^{P,\nu}_{ij}(\nu x, \nu t) .
\]
Note that $\overline U(x,t)$ still depends on $\nu$ after the above scaling, since the pseudo-shock $U^{P,\nu}_{ij}$ has a transition zone of width $\sqrt{\rho_\nu}$ different from $\nu$. Nonetheless, for simplicity, we omit the dependence of $\overline U$ on the $\nu$, and also the dependence on the $j$-th time interval $(t_j,t_{j+1})$.

Then, using the above scaling, we can rewrite  
\begin{align}
\begin{aligned} \label{normalrr}
 \mathcal{R}^\nu(U^\nu) (t) &= \bigg[\sum_{i\in \mathcal{S}}\dot X_{ij}^\nu  Y_i^\nu (U^\nu)  +\overline{\mathcal{H}_{1}^\nu} (U^\nu)  - \mathcal{J}^\nu_{good} (U^\nu)  -  \nu (1-C\delta_*)  D^\nu  (U^\nu) \\
 &\quad +  C\delta_*  \sum_{i\in \mathcal{S}}  \int_\bbr |(\tilde v_{ij}^\nu)_x| p(v^\nu|\bar v_{\nu,\delta}) dx +  \frac{C\nu}{\delta_*} \int_\bbr \Big(\sum_{i\in \mathcal{S}} |(a_i^\nu)_x| \Big)^2 (v^\nu)^\beta |p(v^\nu)-p(\bar v_{\nu,\delta}) |^2 \, dx \bigg] (t) \\
&= \mathcal{R} (U) (t/\nu),
\end{aligned}
\end{align}
where
\begin{align}
\begin{aligned}\label{mathru}
 \mathcal{R} (U) (t/\nu)&:= \bigg[\sum_{i\in \mathcal{S}}\dot X_{i}  Y_i (U)  +\overline{\mathcal{H}_{1}}(U)  - \mathcal{J}_{good} (U)-  (1-C\delta_*)  D  (U)  \\
 &\quad + C\delta_* \sum_{i\in \mathcal{S}}  \int_\bbr |(\tilde v_{i})_x| p(v|\bar v) dx+  \frac{C}{\delta_*} \int_\bbr \Big(\sum_{i\in \mathcal{S}} |(a_i)_x| \Big)^2 v^\beta |p(v)-p(\bar v) |^2 \, dx \bigg] \Big(\frac{t}{\nu}\Big).
\end{aligned}
\end{align}
Here, the functionals $Y_i, \overline{\mathcal{H}_{1}}, \mathcal{J}_{good}, D$ are the same as in \eqref{ybg}, but replaced by the above quantities \eqref{newquan} independent of $\nu$.

\subsection{Maximization in terms of $h-\bar h$}\label{sec:mini}
First, notice that 
\[
\eta(U|\overline U)=\frac{|h- \bar h |^2}{2} + Q(v|\bar v) .
\]
As in \cite[Lemma 5.3]{KV-2shock}, we have the following representation for the maximization w.r.t. $h-\bar h$.
For any constant $\bar\delta>0$, 
\begin{align}
\begin{aligned}\label{ineq-1}
 \mathcal{J}_{bad}(U) - \mathcal{J}_{good}(U)  =B_{\bar\delta} (U)- G_{\bar\delta}(U),
\end{aligned}
\end{align}
where
\beq\label{jbad}
 \mathcal{J}_{bad}(U) := \overline{\mathcal{H}_{1}} (U) +  C\delta_*  \sum_{i\in \mathcal{S}}  \int_\bbr |(\tilde v_{i})_x| p(v|\bar v) dx+  \frac{C}{\delta_*} \int_\bbr \Big(\sum_{i\in \mathcal{S}} |(a_i)_x| \Big)^2 v^\beta |p(v)-p(\bar v) |^2 \, dx ,
\eeq
\begin{align}
\begin{aligned}\label{bad}
{B}_{\bar\delta}(U)&:=  \sum_{i\in \mathcal{S}} \bigg[ \int_\bbr \Big( |\tilde\lam_i|  a + C\delta_* \Big) |(\tilde v_i)_x|  p(v| \bar v) dx  + \frac{1}{2\tilde\lambda_i} \int_\bbr (a_i)_x |p(v)-p(\bar v) |^2 {\mathbf 1}_{\{p(v)- p(\bar v) \leq{\bar\delta}\}}  dx \\
&\quad\quad\quad\quad + \int_\bbr (a_i)_x \big(p(v)-p(\bar v)\big) \big(h-\bar h \big)  {\mathbf 1}_{\{p(v)- p(\bar v)  > {\bar\delta}\}}  dx \bigg] \\
&\quad +  \frac{C}{\delta_*} \int_\bbr \Big(\sum_{i\in \mathcal{S}} |(a_i)_x| \Big)^2 v^\beta |p(v)-p(\bar v) |^2 \, dx,
\end{aligned}
\end{align}
and
\begin{align}
\begin{aligned}\label{good}
{G}_{\bar\delta}(U) &:=  \sum_{i\in \mathcal{S}} \bigg( \frac{ \tilde\lambda_i}{2}\int_\bbr (a_i)_x  \left|  h-\bar h -\frac{p(v)-p(\bar v)}{ \tilde\lambda_i}\right|^2  {\mathbf 1}_{\{p(v)-p(\bar v)  \leq{\bar\delta}\}}  dx \\
&+\frac{ \tilde\lambda_i}{2}\int_\bbr (a_i)_x \left| h-\bar h \right|^2  {\mathbf 1}_{\{p(v)-p(\bar v)  >{\bar\delta}\}}  dx  + \tilde\lambda_i  \int_\bbr  (a_i)_x  Q(v|\bar v)  dx\bigg) .
\end{aligned}
\end{align}
Here, we maximized the first term of $\overline{\mathcal{H}_{1}}$ in terms of $h-\bar h$ only on $\{p(v)- p(\bar v) \leq{\bar\delta}\}$, as in \cite[Lemma 5.3]{KV-2shock}. Notice that since $\delta_* \ll |\tilde\lam_i|  a$, the second term of \eqref{jbad} would be negligible by the last term $ \sum_{i\in \mathcal{S}}|\tilde\lam_i| \int_\bbr   a  |(\tilde v_i)_x|  p(v| \bar v) dx$ of $\overline{\mathcal{H}_{1}} $.

From \eqref{mathru} and the above maximization, we will estimate
\begin{align}
\begin{aligned} \label{real-R}
 \mathcal{R} (U) 
&= \sum_{i\in \mathcal{S}}\dot X_{i}  Y_i (U) + \mathcal{J}_{bad}(U)  - \mathcal{J}_{good} (U)  -  (1-C\delta_*) D (U)  \\
 &= \sum_{i\in \mathcal{S}}\dot X_{i}  Y_i (U) + {B}_{\bar\delta}(U) -{G}_{\bar\delta}(U)  -  (1-C\delta_*) D (U) . \\
\end{aligned}
\end{align}

\subsection{Construction of shifts for shocks}\label{subsec-shock}
As in  \cite[Section 5.5]{KV-2shock}, consider a Lipschitz function $\Phi_\eps$ defined by (given $\eps>0$)
\[
\Phi_\eps (y)=
\left\{ \begin{array}{ll}
      0,\quad &\mbox{if}~ y\le 0, \\
      -\frac{1}{\eps^4}y,\quad &\mbox{if} ~ 0\le y\le \eps^2, \\
       -\frac{1}{\eps^2},\quad &\mbox{if}  ~y\ge \eps^2, \end{array} \right.
\]
For a fix $t_j$, construct a family of shift functions $\{X_i \}_{i\in \mathcal{S}}$ as a solution to the system of nonlinear ODEs:
\beq\label{X-def}
\left\{ \begin{array}{ll}
       \dot X_i(t/\nu) = \Phi_{\s_i} (Y_i(U)) \Big(2| \mathcal{J}_{bad}(U)|+1 \Big) +\s_i^2\frac{\tilde\lambda_i}{2}  \Phi_{\s_i} (-Y_i(U)) , \quad  \mbox{if $i$ is 1-wave, i.e., } k_i=1  \\
        \dot X_i(t/\nu) = -\Phi_{\s_i} (-Y_i(U)) \Big(2| \mathcal{J}_{bad}(U)| +1 \Big) +\s_i^2 \frac{\tilde\lambda_i}{2}  \Phi_{\s_i} (Y_i(U)), \quad  \mbox{if  $i$ is 2-wave, i.e., } k_i=2 ,\\
       X_i(t_j/\nu)=0, \quad  \mbox{for all } i\in \mathcal{S},  \end{array} \right.
\eeq
where $\mathcal{J}_{bad}$ as in \eqref{jbad}.
Then, as in  \cite[Appendix C]{KV-2shock}, the system \eqref{X-def} has a unique absolutely continuous solution $\{X_i \}_{i\in \mathcal{S}}$ on $[t_j/\nu,t_{j+1}/\nu)$. \\

Since it follows from \eqref{X-def} that for each $i\in \mathcal{S}$,
\beq\label{explicit-X}
\dot X_i (t/\nu) =
\left\{ \begin{array}{ll}
      (-1)^{k_i} \s_i^{-2} \big(2| \mathcal{J}_{bad}(U)|+1 \big),\quad &\mbox{if}~ (-1)^{{k_i}-1} Y_i(U) > \s_i^2, \\
  -\s_i^{-4} Y_i(U) \big(2| \mathcal{J}_{bad}(U)|+1 \big),\quad &\mbox{if} ~ 0\le  (-1)^{{k_i}-1} Y_i(U)\le \s_i^2, \\
       (-1)^{{k_i}-1} \frac{1}{2} \tilde\lambda_i \s_i^{-2}  Y_i(U),\quad &\mbox{if}  ~ -\s_i^2  \le  (-1)^{{k_i}-1} Y_i(U)\le 0 ,\\
       -\frac{1}{2}\tilde\lambda_i , \quad &\mbox{if}  ~ (-1)^{{k_i}-1} Y_i(U) \le -\s_i^2 ,
       \end{array} \right.
\eeq
the shifts satisfy the bounds: $\forall t\in (t_j,t_{j+1})$,
\begin{align*}
\begin{aligned}
&k_i=1 \quad \Rightarrow\quad \dot X_i(t/\nu) \le -\frac{\tilde\lambda_i}{2},\\
&k_i=2 \quad \Rightarrow\quad \dot X_i(t/\nu) \ge -\frac{\tilde\lambda_i}{2}.
\end{aligned}
\end{align*}
Thus, we have
\begin{align}
\begin{aligned}\label{sx12}
&k_i=1 \quad \Rightarrow\quad X_i(t/\nu) + \tilde\lambda_i (t-t_j)  \le \frac{\tilde\lambda_i}{2} (t-t_j) <0 ,\quad\forall t\in (t_j,t_{j+1}) ,\\
&k_i=2 \quad \Rightarrow\quad X_i(t/\nu) + \tilde\lambda_i (t-t_j)  \ge \frac{\tilde\lambda_i}{2} (t-t_j) >0 ,\quad\forall t\in (t_j,t_{j+1}) .
\end{aligned}
\end{align}

\subsection{Proof of Proposition \ref{prop:main}}
First, it follows from \eqref{explicit-X} that for each  $i\in \mathcal{S}$,
\beq\label{xyest}
 \dot X_i(t/\nu) Y_i(U) \le
\left\{ \begin{array}{ll}
    -2| \mathcal{J}_{bad}(U)|,\quad &\mbox{if}~ (-1)^{{k_i}-1} Y_i(U) > \s_i^2, \\
  -\s_i^{-4} |Y_i(U)|^2 ,\quad &\mbox{if} ~ 0\le  (-1)^{{k_i}-1} Y_i(U)\le \s_i^2, \\
       (-1)^{{k_i}-1} \frac{1}{2} \tilde\lambda_i \s_i^{-2}  |Y_i(U)|^2, \quad &\mbox{if}  ~ -\s_i^2  \le  (-1)^{{k_i}-1} Y_i(U)\le 0 ,\\
       -\frac{1}{2}\tilde\lambda_i  Y_i(U), \quad &\mbox{if}  ~ (-1)^{{k_i}-1} Y_i(U) \le -\s_i^2 .
       \end{array} \right.
\eeq
Then, we first find from \eqref{real-R} that whenever $U\in \cup_{i\in \mathcal{S}} \{U~|~ (-1)^{{k_i}-1} Y_i(U) > \s_i^2 \} $, 
\beq\label{asimest}
 \mathcal{R} (U)   \le - |\mathcal{J}_{bad}(U)| - \mathcal{J}_{good} (U)  -  (1-C\delta_*)  D  (U)  \le 0.
\eeq
On the other hand, for  $U\in \cap_{i\in \mathcal{S}} \{U~|~ (-1)^{{k_i}-1} Y_i(U) \le \s_i^2 \}$, we will apply the following main proposition.
\begin{proposition}\label{prop:smain}
There exists a positive constants $ \delta_1 \in(0,1/2)$ such that
  for all $U$ satisfying  $(-1)^{{k_i}-1} Y_i(U) \le \s_i^2$ for all  $i\in \mathcal{S}$, the following holds: 
\begin{align*}
\begin{aligned}
\mathcal{R}(U)
&\le \sum_{i\in \mathcal{S}}\bigg( -\frac{1}{\s_i^4} |Y_i(U)|^2 {\mathbf 1}_{\{0\le  (-1)^{{k_i}-1} Y_i(U)\le \s_i^2\}} 
 -  \frac{|\tilde\lambda_i|}{2 \s_i^2}   |Y_i(U)|^2 {\mathbf 1}_{\{ -\s_i^2  \le  (-1)^{{k_i}-1} Y_i(U)\le 0 \}} \\
&\quad -\frac{|\tilde\lambda_i| }{2} |Y_i(U)| {\mathbf 1}_{\{(-1)^{{k_i}-1} Y_i(U) \le -\s_i^2 \}} \bigg)+  {B}_{\delta_1}(U)+\delta_*  \sum_{i\in \mathcal{S}}  \int_\bbr |(\tilde v_{i})_x| p(v|\bar v) dx \\
&\quad
- \sum_{i\in \mathcal{S}} {G}_{1i}^{-}(U)- \sum_{i\in \mathcal{S}}{G}_{1i}^{+}(U)  - \sum_{i\in \mathcal{S}} {G}_{2i}(U)  - \Big(1-\sqrt\ds\Big) D (U) \\
&\le  C\ds^{\frac{1}{8}}  \sum_{i\in\mathcal{R}} \int_\bbr |(v^R_i)_x| p(v|\bar v) dx  + C\ds^{\frac18}  \sum_{i\in\mathcal{NP}} \int_\bbr |(v^P_i)_x| Q(v|\bar v) dx \\
&\quad+\mathcal{C} (\delta, \nu) +  \frac{C\nu^{4/3}}{\delta\ds}  \int_\bbr a \eta(U|\overline U) dx  ,
\end{aligned}
\end{align*}
where ${G}_{1i}^{-},{G}_{1i}^{+}, {G}_{2i}, {D}$ denote the good terms of ${G}_{\delta_1}$ in  \eqref{good}, and the diffusion in \eqref{ybg} as follows: for each $i\in \mathcal{S}$,
\begin{align}
\begin{aligned}\label{ggd}
&{G}_{1i}^{-}(U):=\frac{ \tilde\lambda_i}{2}\int_\bbr (a_i)_x \left| h-\bar h \right|^2  {\mathbf 1}_{\{p(v)-p(\bar v)  >{\delta_1}\}}  dx ,\\
&{G}_{1i}^{+}(U):=  \frac{ \tilde\lambda_i}{2}\int_\bbr (a_i)_x  \left|  h-\bar h -\frac{p(v)-p(\bar v)}{ \tilde\lambda_i}\right|^2  {\mathbf 1}_{\{p(v)-p(\bar v)  \leq{\delta_1}\}}  dx,\\
&{G}_{2i}(U):=\tilde\lambda_i  \int_\bbr  (a_i)_x  Q(v|\bar v)  dx,\\
&D (U):= \int_\bbr a\, \mu_1(v ) v^\gamma |( p(v) -p(\bar v))_x |^2 dx,
\end{aligned}
\end{align} 
and $\mathcal{C} (\delta, \nu)$ is the constant that vanishes when $\nu\to0$ for any fixed $\delta>0$. 
\end{proposition}

\subsubsection{Proof of Proposition \ref{prop:main} from Proposition \ref{prop:smain}}

First, from \eqref{real-R} with $\bar\delta=\deltao$, and \eqref{xyest}, we have
\begin{align*}
\begin{aligned}
\mathcal{R}(U)
&\le \sum_{i\in \mathcal{S}}\bigg( -\frac{1}{\s_i^4} |Y_i(U)|^2 {\mathbf 1}_{\{0\le  (-1)^{{k_i}-1} Y_i(U)\le \s_i^2\}} 
 -  \frac{|\tilde\lambda_i|}{2 \s_i^2}   |Y_i(U)|^2 {\mathbf 1}_{\{ -\s_i^2  \le  (-1)^{{k_i}-1} Y_i(U)\le 0 \}} \\
&\quad -\frac{|\tilde\lambda_i| }{2} |Y_i(U)| {\mathbf 1}_{\{(-1)^{{k_i}-1} Y_i(U) \le -\s_i^2 \}} \bigg)+  {B}_{\delta_1}(U)-{G}_{\deltao}(U)  -  (1-C\delta_*) D (U).
\end{aligned}
\end{align*}
Then, Proposition \ref{prop:smain} and \eqref{good} with $\bar\delta=\deltao$ implies that\\for all $U\in  \cap_{i\in \mathcal{S}} \{U~|~ (-1)^{{k_i}-1} Y_i(U) \le \s_i^2 \}$,
\begin{align*}
\begin{aligned}
&\mathcal{R}(U) +\delta_*  \sum_{i\in \mathcal{S}}  \int_\bbr |(\tilde v_{i})_x| p(v|\bar v) dx+ (\sqrt\ds-C\ds) D(U) \\
&\le C\ds^{\frac{1}{8}}  \sum_{i\in\mathcal{R}} \int_\bbr |(v^R_i)_x| p(v|\bar v) dx  + C\ds^{\frac18}  \sum_{i\in\mathcal{NP}} \int_\bbr |(v^P_i)_x| Q(v|\bar v) dx + \mathcal{C} (\delta, \nu) +  \frac{C\nu^{4/3}}{\delta\ds}  \int_\bbr a \eta(U|\overline U) dx.
\end{aligned}
\end{align*}
This with \eqref{asimest} and the definition of $\mathcal{J}_{bad}$ implies that for all $U$,
\begin{align*}
\begin{aligned}
&\mathcal{R}(U) +\delta_*  \sum_{i\in \mathcal{S}}  \int_\bbr |(\tilde v_{i})_x| p(v|\bar v) dx+ \sqrt\ds D(U) \\
&\le C\ds^{\frac{1}{8}}  \sum_{i\in\mathcal{R}} \int_\bbr |(v^R_i)_x| p(v|\bar v) dx  + C\ds^{\frac18}  \sum_{i\in\mathcal{NP}} \int_\bbr |(v^P_i)_x| Q(v|\bar v) dx + \mathcal{C} (\delta, \nu)+  \frac{C\nu^{4/3}}{\delta\ds}  \int_\bbr a \eta(U|\overline U) dx.
\end{aligned}
\end{align*}
To complete the proof of Proposition \ref{prop:main}, we first rescale the above estimate as in \eqref{newquan} and \eqref{normalrr} with \eqref{mathru}, as follows:
\begin{align*}
\begin{aligned}
&\mathcal{R}^\nu(U^\nu) +\delta_*  \sum_{i\in \mathcal{S}}  \int_\bbr |(\tilde v^\nu_{ij})_x| p(v^\nu|\bar v_{\nu,\delta}) dx+ \sqrt\ds \nu D^\nu(U^\nu) \\
&\le C\ds^{\frac{1}{8}}  \sum_{i\in\mathcal{R}} \int_\bbr |(v^{R,\nu}_i)_x| p(v^\nu|\bar v_{\nu,\delta}) dx   + C\ds^{\frac18}  \sum_{i\in\mathcal{NP}} \int_\bbr |(v^{P,\nu}_i)_x| Q(v^\nu|\bar v_{\nu,\delta}) dx \\
&\quad + \mathcal{C} (\delta, \nu)+  \frac{C\nu^{1/3}}{\delta\ds}  \int_\bbr a^\nu \eta(U^\nu|\overline U_{\nu,\delta}) dx.
\end{aligned}
\end{align*}
Using the functional $G_i^R(U^\nu)$ and \eqref{vpest} with Lemma \ref{lem:num} and taking $\nu$ small enough as $\nu^{1/6}\ll \delta\ds$, we have
\begin{align*}
\begin{aligned}
&\mathcal{R}^\nu(U^\nu) +\delta_*  \sum_{i\in \mathcal{S}}  \int_\bbr |(\tilde v^\nu_{ij})_x| p(v^\nu|\bar v_{\nu,\delta})  dx+ \sqrt\ds \nu D^\nu(U^\nu) \\
&\quad\le C\ds^{\frac{1}{8}}  \sum_{i\in\mathcal{R}} G_i^R(U^\nu)   + \mathcal{C} (\delta, \nu) + C  \int_\bbr a^\nu \eta(U^\nu|\overline U_{\nu,\delta}) dx.
\end{aligned}
\end{align*}
Therefore, it holds from  \eqref{ineq-mid}, \eqref{normalrr} and the above estimate that for all $ t\in (t_j,t_{j+1}) $,
\begin{align*}
\begin{aligned}
\frac{d}{dt} \int_{\bbr} a^\nu \eta(U^\nu |\overline U_{\nu,\delta} ) dx + G(U^\nu)(t) &\le  C \int_\bbr a^\nu \eta(U^\nu|\bar U_{\nu,\delta}) dx +C \delta^{1/4}  \overline{\mathcal{D}}  \\
&\quad  + C \delta^{3/4}  +   C   \frac{\delta^2}{\delta_*}  +\mathcal{C} (\delta, \nu),
\end{aligned}
\end{align*}
which completes the proof of Proposition \ref{prop:main}.\\

The rest of this section is dedicated to the proof of the main Proposition \ref{prop:smain}.

\subsection{Decomposition of the domain} 
Since we have only one diffusion, but many bad terms as the number of shock fronts, we may localize the diffusion by decomposing the domain $\bbr$.
For that, we recall  \eqref{distij} as follows on the separation of trajectories $y_{ij}^\nu(t)$ satisfying \eqref{ycurve}:
\[
|y_{ij}^\nu(t)-y_{i'j}^\nu(t)|\ge  2\sqrt{\rho_\nu} = 2 \nu^{1/6}, \quad\mbox{for any $i,i'$ with $i\neq i'$ and any } t\in (t_j,t_{j+1}).
\]
This implies the following inequality at the level of $\nu=1$ by the normalization $t\mapsto \nu t$, $x\mapsto \nu x$ under consideration: (setting $y_i(t ;x_{ij},t_j):=\frac{1}{\nu}y_{ij}^\nu(\nu t ;\nu x_{ij},\nu t_j)$ by \eqref{newquan} and \eqref{ycurve})
\beq\label{sepy}
|y_{i}(t)-y_{i'}(t)|\ge  2 \nu^{-5/6}, \quad\mbox{for any $i,i'$ with $i\neq i'$ and any } t\in (t_j,t_{j+1}).
\eeq 
Then, we decompose $\bbr$ into a finite number of subintervals: 
\beq\label{defji}
\mathcal{J}_i(t):= \left\{ \begin{array}{ll}
    \bigg(-\infty, \frac{y_{i+1}(t)+y_i(t)}{2}\bigg] \quad &\mbox{if}~ y_i(t)=\min_{i'} y_{i'}(t), \\
      \bigg(\frac{y_{i-1}(t)+y_i(t)}{2}, +\infty\bigg) \quad &\mbox{if}~ y_i(t)=\max_{i'} y_{i'}(t),\\
     \bigg( \frac{y_{i-1}(t)+y_i(t)}{2},\frac{y_{i+1}(t)+y_i(t)}{2} \bigg]  \quad &\mbox{otherwise,}  \end{array} \right.
\eeq
Notice that
\[
\bbr =  \cup_{i\in \mathcal{S}\cup \mathcal{R}\cup \mathcal{NP}} \mathcal{J}_i(t).
\]
This and \eqref{sepy} imply
\beq\label{estji}
|\mathcal{J}_i(t) | \ge 2\nu^{-5/6},
\eeq
and 
\beq\label{wellom}
|x-y_i(t)|\ge  \nu^{-5/6},\quad \forall x\in \mathcal{J}_i^c = \cup_{i'\neq i} \mathcal{J}_{i'},\quad \forall i.
\eeq

\subsection{Sharp estimates inside truncation}
For a fixed $t$, and for each $i\in\mathcal{S}$, we define a cutoff function $\phi_i$ on $\bbr$ as a non-negative Lipschitz function : 
\beq\label{phii}
\phi_i(x) = \left\{ \begin{array}{ll}
    1\quad &\mbox{if}~ x \in \mathcal{J}_i, \\
   \mbox{linear}\quad &\mbox{if} ~ x\in \Big(y_{i-1},  \frac{y_{i-1}+y_i}{2}\Big)\cup \Big(\frac{y_{i+1}(t)+y_i(t)}{2}, y_{i+1}\Big) , \\
   0  \quad \quad &\mbox{otherwise}. 
       \end{array} \right.
\eeq
Then, consider the following functionals: 
\begin{align}
\begin{aligned}\label{note-in}
&\mathcal{Y}^g_i(v):=-\frac{1}{2 \tilde\lambda_i^2}\int_{\bbr} (a_i)_x   \phi_i^2|p(v)-p(\bar v)|^2 dx -\int_{\bbr} (a_i)_x  \phi_i^2 Q(v|\bar v) dx\\
&\qquad\qquad -\int_{\bbr} a \partial_x p(\tilde v_i)   \phi_i (v-\bar v)dx+\frac{1}{ \tilde\lambda_i}\int_{\bbr} a  (\tilde h_i)_x  \phi_i \big(p(v)-p(\bar v)\big)dx,\\
&\mathcal{I}_{1i}(v):=  \tilde\lambda_i\int_{\bbr} a (\bar v)_x  \phi_i^2  p(v|\bar v) dx,\\
&\mathcal{I}_{2i}(v):= \frac{1}{2\tilde\lambda_i} \int_{\bbr} (a_i)_x   \phi_i^2 |p(v)-p(\bar v)|^2dx,\\
&\mathcal{G}_{2i} (v):= \tilde\lambda_i  \int_{\bbr}  (a_i)_x  \bigg( \frac{1}{2\gamma}  p(\tilde v_i)^{-\frac{1}{\gamma}-1} \phi_i^2 \big(p(v)-p(\bar v)\big)^2\\
&\qquad\qquad - \frac{1+\gamma}{3\gamma^2} p(\tilde v_i)^{-\frac{1}{\gamma}-2} \phi_i^3 \big(p(v)-p(\bar v)\big)^3 \bigg) dx, \\
&\mathcal{D}_i (v):=\int_{\bbr} a\, \mu_1(v ) v^\gamma \big| \big( \phi_i^2 ( p(v) -p(\bar v))\big)_x \big|^2 dx.
\end{aligned}
\end{align}

\begin{lemma}\label{lem:inside}
For any constant $C_2>0$, there exists $\delta_1,\delta_0>0$ such that for any $\theta, \ds\in(0,\delta_1]$ and any $L(0)<\delta_0$, and for each $i\in\mathcal{S}$,
 the following holds.\\
For any function $v:\bbr\to \bbr^+$ such that $\mathcal{D}_i(v)+\mathcal{G}_{2i}(v)$
 is finite, if
\beq\label{assYp}
|\mathcal{Y}^g_i(v)|\leq C_2 |\s_i|\ds,\qquad  \|p(v)-p(\bar v)\|_{L^\infty(\bbr)}\leq \delta_1,
\eeq
then
\begin{align}
\begin{aligned}\label{redelta}
\mathcal{R}^i_{\theta,\ds}(v)&:=-\frac{1}{|\s_i|\theta}|\mathcal{Y}^g_i(v)|^2 +(1+\theta) |\mathcal{I}_{1i}(v)|\\
&\quad\quad+(1+\theta \ds) |\mathcal{I}_{2i}(v)|-\left(1-\theta\ds \right)\mathcal{G}_{2i}(v)-(1-\theta)\mathcal{D}_i(v)\le 0,
\end{aligned}
\end{align}
where note that $\mathcal{I}_{1i}, \mathcal{I}_{2i} \ge 0$.
\end{lemma}
\begin{proof}
We follow the same proof as in \cite[Proposition 6.1 and Appendix B]{KV-2shock} by replacing $\eps/\lambda$ by $\delta_*$, where $\eps/\lambda= \frac{\mbox{jump strength of shock}}{\mbox{jump strength of weight}} $ in \cite[Proposition 6.1 and Appendix B]{KV-2shock}.
So, we omit the details.
\end{proof}

\subsection{Smallness of localized relative entropy}
\begin{lemma}\label{lem:small}
Given a reference point $U_*$, there exist constants $C, C_0, \delta_0$ such that for any $\ds, L(0) <\delta_0$, the following holds. Whenever  $(-1)^{{k_i}-1} Y_i(U) \le \s_i^2$,
\beq\label{enesmall}
\int_\bbr |(a_i)_x| \left| h-\bar h \right|^2 dx +  \int_\bbr  |(a_i)_x|  Q(v|\bar v)  dx \le C|\s_i| \delta_*,
\eeq
and
\beq\label{yismall}
|Y_i(U)| \le C_0 |\s_i| \delta_*.
\eeq
\end{lemma}
\begin{proof}
We follow the same proof as in \cite[Proposition 6.2]{KV-2shock} by replacing $\eps/\lambda$ by $\delta_*$. So, we omit the details.
\end{proof}

\noindent $\bullet$ {\bf Estimates for fixing the size of truncation}
Notice that the choice of constant $\delta_1$ in Lemma \ref{lem:inside} depends on a constant $C_2$. Thus, we should find the bound of $\mathcal{Y}^g_i$ on the unconditional level.
For that, we first consider any constant $\zeta>0$ as a truncation size on $|p(v)-p(\bar v)|$, and then  the special case  $\zeta=\delta_1$ as in Lemma \ref{lem:inside}. But for the moment, consider a general case $\zeta$ to estimate the constant $C_2$ of Lemma \ref{lem:inside}.  For that, let $\psi_\zeta$ be a continuous function on $\bbr$ defined by
\beq\label{psi}
 \psi_\zeta(y)=\inf\left(k,\sup(-\zeta,y)\right),\quad \zeta>0.
\eeq
Then, we define the function $\uv_\zeta$ uniquely (since the function $p$ is one to one) by
\beq\label{trunc-def}
p(\uv_\zeta)-p(\bar v)=\psi_\zeta\big(p(v)-p(\bar v)\big).
\eeq
Notice that $\|p(\uv_\zeta)-p(\bar v)\|_\infty\le \zeta$.\\

As in \cite[Lemma 7.2]{KV-2shock}, we have the following.
\begin{lemma}\label{lem:essy}
Given a reference point $U_*$, there exists positive constants $\delta_0, C_2, \zeta_0$ such that for any $\ds, L(0) <\delta_0$, the following holds.
\[
 |\mathcal{Y}^g_i(\uv_\zeta)|\leq C_2\frac{\eps_i^2}{\lambda}, \qquad \forall \zeta\leq \zeta_0. 
\]
\end{lemma}
\begin{proof}
Since $| \uv_\zeta -\bar v| \ll1$ for any $\zeta\le \zeta_0\ll v_*$, we use \eqref{enesmall} to have
\begin{align*}
\begin{aligned}
 |\mathcal{Y}^g_i(\uv_\zeta)| &\leq C \int_\bbr |(a_i)_x| Q(\uv_\zeta|\bar v) dx + C \int_\bbr |(\tilde v_i)_x| |\uv_\zeta-\bar v| dx \\
 &\le C \int_\bbr |(a_i)_x| Q(v|\bar v) dx + C \sqrt{\int_\bbr |(\tilde v_i)_x| dx} \sqrt{\int_\bbr |(\tilde v_i)_x| Q(\uv_\zeta|\bar v) dx}\\
 &\le C|\s_i|\ds + C\sqrt{|\s_i|} \sqrt{|\s_i| \ds} \le  C|\s_i|\ds.
\end{aligned}
\end{align*}

\end{proof}

\subsection{Estimates outside truncation}

We now fix the constant $\delta_1$ for Lemma \ref{lem:inside} associated to the constant $C_2$ of Lemma \ref{lem:essy}.
For  simplicity, we  use the notations:
\begin{align}
\begin{aligned}\label{strunc}
&\uv:=\uv_{\delta_1}, \qquad  \uu:=(\uv, h),\\
&\Omega:=\{x~|~ (p(v)-p(\bar v))(x) \leq\delta_1\}.
\end{aligned}
\end{align}

In order to present all terms to be controlled on the region $\{|p(v)-p(\bar v)|\geq\delta_1\}$, we split $Y_i$ into four parts $Y^g_i$,  $Y^b_i$, $Y^l_i$ and $Y^s_i$ as follows: for each $i\in \mathcal{S}$,
\begin{align}
\begin{aligned}\label{defyi}
Y_i &= -\int_\bbr  (a_i)_x \Big(\frac{|h-\bar h|^2}{2}  + Q(v|\bar v) \Big) dx +\int_\bbr a \Big(-(\tilde v_i)_x p'(\bar v)(v-\bar v) +(\tilde h_i)_x (h-\bar h) \Big) dx\\
&= Y^g_i +Y^b_i +Y^l_i + Y^s_i ,
\end{aligned}
\end{align}
where 
\begin{align*}
\begin{aligned}
 Y^g_i &:= -\frac{1}{2\tilde\lambda_i^2}\int_\Omega (a_i)_x |p(v)-p(\bar v)|^2 dx -\int_\Omega (a_i)_x Q(v|\bar v) dx -\int_\Omega a (\tilde v_i)_x p'(\bar v)(v-\bar v) dx \\
&\quad+\frac{1}{\tilde\lambda_i}\int_\Omega a  (\tilde h_i)_x\big(p(v)-p(\bar v)\big)dx,\\
Y^b_i &:= -\frac{1}{2}\int_\Omega  (a_i)_x  \Big(h-\bar h-\frac{p(v)-p(\bar v)}{\tilde\lambda_i}\Big)^2 dx \\
&\quad -\frac{1}{\tilde\lambda_i} \int_\Omega  (a_i)_x  \big(p(v)-p(\bar v)\big)\Big(h-\bar h-\frac{p(v)-p(\bar v)}{\tilde\lambda_i}\Big) dx,\\
Y^l_i &:=\int_\Omega a (\tilde h_i)_x \Big(h-\bar h-\frac{p(v)-p(\bar v)}{\tilde\lambda_i}\Big)dx,
\end{aligned}
\end{align*}
and
\[
Y^s_i :=-\int_{\Omega^c} (a_i)_x Q(v|\bar v) dx -\int_{\Omega^c} a (\tilde v_i)_x p'(\bar v)(v-\bar v)dx 
 -\int_{\Omega^c} (a_i)_x \frac{|h-\bar h|^2}{2} dx+\int_{\Omega^c} a (\tilde h_i)_x (h-\bar h) dx.
\]

For the bad terms ${B}_{\delta_1}$ of \eqref{bad} with $\bar\delta=\delta_1$, we will use the following notations :
\beq\label{bad0}
{B}_{\delta_1}= \sum_{i\in \mathcal{S}} \Big({B}_{1i}+{B}_{2i}^- +{B}_{2i}^+  \Big) +  B_3 ,
\eeq
where
\begin{align*}
\begin{aligned}
&{B}_{1i}:=  \int_\bbr \Big( |\tilde\lam_i|  a + \frac{CL(0)}{\bar\delta\delta_*}  \Big) |(\tilde v_i)_x|  p(v| \bar v) dx,\\
&{B}_{2i}^- :=\int_{\Omega^c} (a_i)_x \big(p(v)-p(\bar v)\big) \big(h-\bar h \big) dx ,\\
&  {B}_{2i}^+ := \frac{1}{2\s_i} \int_\Omega (a_i)_x |p(v)-p(\bar v) |^2  dx,\\
& B_3 : =  \frac{C}{\delta_*} \int_\bbr \Big(\sum_{i\in \mathcal{S}} |(a_i)_x| \Big)^2 v^\beta |p(v)-p(\bar v) |^2 \, dx.
\end{aligned}
\end{align*}

\begin{proposition}\label{prop:big}
Let
\[
\widetilde G (t):=  \int_\bbr |(\bar v)_x| Q(v|\bar v)\,dx,\qquad \widetilde G_{i}(t) :=  \int_{\mathcal{J}_i}  |(\bar v)_x| Q(v|\bar v)\,dx,
\]
where $\mathcal{J}_i$ is as in \eqref{defji}.\\
There exist constants $C, C_{\delta_1}, C^*>0$ ($C_{\delta_1}$ depending on $\delta_1$ as $C_{\delta_1}\sim \delta_1^{-r}$ for some $r>0$)  such that the following holds.\\
For each $i\in \mathcal{S}$,
\begin{align}
\begin{aligned} \label{badprop}
& |B_{1i}(U)-B_{1i}(\uu)|  \\
& \le C \ds  (G_{2i}(U)-G_{2i}(\uu))  + C_{\delta_1} \ds \sqrt{\delta_*}  \big( D_i(t) +\widetilde G_i(t) \big) +  C_{\delta_1} \nu|\s_i|  \big( D(U) +\widetilde G(t) \big), \\
& |B_{2i}^+(U)-B_{2i}^+(\uu)|  \le  C_{\delta_1}\sqrt{\delta_*} D_i(t) +  C_{\delta_1} \frac{\nu|\s_i|}{\delta_*} D(U) ,\\
&  |B_{2i}^-(U)|  \le  C_{\delta_1} \sqrt\ds  \big( D_i  +\widetilde G_{i} \big) +  C_{\delta_1}|\s_i| \sqrt\nu (1+ \sqrt\nu/\delta_*)   \big( D(U) +\widetilde G  \big) \\
&\qquad  \qquad +  ( C\delta_1 + C_{\delta_1} \delta_*^{\frac{1}{4}} )  G_{1i}^{-}(U)+ C_{\delta_1}|\s_i| \sqrt\nu,\\
&|B_3(U)| \le   C\ds^2 \sum_{i\in \mathcal{S}} G_{2i}(U) + C_{\delta_1}\sqrt\ds  \big( D(U) +\widetilde G(t) \big) ,
\end{aligned}
\end{align}

and 
\beq\label{roughbad}
 |B_{1i}(\uu)| +  |B_{2i}^+(\uu)| \le C^* \ds |\s_i|.
 \eeq
In addition, for the functionals of \eqref{note-in}, the following holds: 
\begin{align}
\begin{aligned}\label{error-in}
& | Y^g_i(\uu)-\mathcal{Y}^g_i(\uv)| + |{B}_{1i}(\uu)-\mathcal{I}_{1i}(\uv)|+ |{B}_{2i}^+(\uu)-\mathcal{I}_{2i}(\uv)| \le C\frac{|\s_i| \nu}{\delta_*}  ,\\
& -\big(G_{2i}(\uu) - \mathcal{G}_{2i} (\uv) \big) \le \frac{C\sqrt\nu}{\delta\ds}  ,\\
&- D(\uu)\le -\frac{1}{1+\ds}\sum_{i\in\mathcal{S}} \mathcal{D}_{i} (\uv) +  \frac{C\nu^{4/3}}{\delta\ds}  \int_\bbr a \eta(\uu|\bar U) dx.
\end{aligned}
\end{align}
Furthermore, for each $i\in \mathcal{S}$, if
\beq\label{yicond}
\sqrt{\delta_*} D_i(t) +  \frac{\nu|\s_i|}{\delta_*} D(U)  \le 2C^*|\s_i| \ds, \quad\mbox{where } D_i(t) :=\int_{\mathcal{J}_i} a v^\beta |(p(v) - p(\bar v))_x |^2 dx,
\eeq
then the following holds:
\begin{align}
\begin{aligned}\label{finyest}
& |Y^g_i(U)- Y^g_i(\uu)|^2+ |Y^b_i(U)|^2 + |Y^l_i(U)|^2 + |Y^s_i(U)|^2 \\
&\le C_{\delta_1} \ds |\s_i| \bigg[ \sqrt{\delta_*}\Big(D_i(t) +\widetilde G_i(t)\Big) +  \frac{\nu|\s_i|}{\delta_*^2} \Big(D(U) +\widetilde G(t)\Big) + \big(G_{2i}(U)-G_{2i}(\uu) \big) + G_{1i}^-(U) \\
&\quad+ \ds^{-1/4} G_{1i}^+(U) + \ds^{1/4} G_{2i}(\uu)  \bigg].
 \end{aligned}
\end{align}
\end{proposition}

To prove this proposition, 
we will control the bad terms in different ways for each case of small or big values of $v$, which all correspond to the big values of $|p(v)-p(\bar v)|$ (as $|p(v)-p(\bar v)|\ge\delta_1$). \\

\noindent $\bullet$ {\bf Truncations for small and big values of $v$ :} 
Let $\uv_s$ and $\uv_b$ be one-sided truncations of $v$ defined by
\beq\label{schi}
p(\uv_s)- p(\bar v) :=\bar \psi_s \big(p(v)- p(\bar v)\big),\qquad\mbox{where}\quad \bar\psi_s (y)=\inf (\deltao,y),
\eeq
and
\beq\label{bchi}
p(\uv_b)- p(\bar v) :=\bar \psi_b \big(p(v)- p(\bar v)\big),\qquad\mbox{where}\quad \bar\psi_b (y)=\sup (-\deltao,y).
\eeq
Notice that the function $\uv_s$ (resp. $\uv_b$) represents the truncation of small (resp. big) values of $v$ corresponding to $|p(v)-p(\bar v)|\ge\deltao$.\\
By comparing the definitions of \eqref{schi}, \eqref{bchi} and \eqref{trunc-def} with \eqref{strunc}, we see
\beq\label{acompare1}
\uv = \left\{ \begin{array}{ll}
    \uv_b,\quad &\mbox{on}~ {\Omega}, \\
    \uv_s,\quad &\mbox{on}~ {\Omega}^c ,
       \end{array} \right.
\eeq
and
\begin{align*}
\begin{aligned}
& \left( p(\uv_s)-p(\bar v) \right) {\mathbf 1}_{\{p(v)-p(\bar v) \ge -\deltao\}} = \left( p(\uv)-p(\bar v) \right) {\mathbf 1}_{\{p(v)-p(\bar v) \ge -\deltao\}}  ,\\
& \left( p(\uv_b)-p(\bar v) \right) {\mathbf 1}_{\{p(v)-p(\bar v) \le \deltao\}} = \left( p(\uv)-p(\bar v)\right) {\mathbf 1}_{\{p(v)-p(\bar v) \le \deltao\}}  .
\end{aligned}
\end{align*}
We also note that 
\begin{equation}\label{adef-bar}
\begin{array}{rl}
p(v)-p(\uv_s)=& (p(v)-p(\bar v))+(p(\bar v)-p(\uv_s)) \\[0.2cm]
=&\left(I-\bar\psi_s\right)\left(p(v)-p(\bar v)\right) \\[0.2cm]
=& \left(\left(p(v)-p(\bar v) \right)-\deltao\right)_+,\\[0.2cm]
p(\uv_b)-p(v)=& (p(\uv_b)-p(\bar v))+(p(\bar v)-p(v)) \\[0.2cm]
=&\left(\bar\psi_b-I\right)\left(p(v)-p(\bar v)\right) \\[0.2cm]
=& \left(-\left(p(v)-p(\bar v)\right)-\deltao\right)_+,\\
|p(v)-p(\uv)|=& |(p(v)-p(\bar v))+(p(\bar v)-p(\uv))|\\[0.2cm]
=&|(I-\psi)(p(v)-p(\bar v))|\\[0.2cm]
=& (|p(v)-p(\bar v)|-\deltao)_+.
\end{array}
\end{equation}
Therefore, using \eqref{trunc-def}, \eqref{bchi}, \eqref{schi} and \eqref{adef-bar}, we have
\begin{align}
\begin{aligned}\label{aeq_D}
D(U)&=\int_\bbr a\, \mu_1(v ) v^\gamma |\partial_x (p(v)-p(\bar v))|^2 dx\\
&=\int_\bbr a\, \mu_1(v ) v^\gamma |\partial_x (p(v)-p(\bar v))|^2 ( {\mathbf 1}_{\{|p(v)-p(\bar v) |\leq\deltao\}} + {\mathbf 1}_{\{p(v)-p(\bar v) >\deltao\}}+ {\mathbf 1}_{\{p(v)-p(\bar v) <-\deltao\}}  )dx\\
&=D(\uu)+\int_\bbr a\, \mu_1(v ) v^\gamma |\partial_x (p(v)-p(\uv_s))|^2 dx+\int_\bbr  a\, \mu_1(v ) v^\gamma  |\partial_x (p(v)-p(\uv_b))|^2 dx\\
&\ge \int_\bbr a\, \mu_1(v ) v^\gamma |\partial_x (p(v)-p(\uv_s))|^2 dx+\int_\bbr a\, \mu_1(v ) v^\gamma  |\partial_x (p(v)-p(\uv_b))|^2 dx,
\end{aligned}
\end{align}
which also yields 
\beq\label{amonod}
D(U) \ge {D}(\uu).
\eeq
By Lemma \ref{lem:small} and $Q(v|\bar{v})\geq  Q(\uv|\bar v)$, we have
\beq\label{al2}
0\leq  {G}_{2i}(U)-{G}_{2i}(\uu) \leq  {G}_{2i}(U) \leq C \int_\bbr |(a_i)_x| Q(v|\bar v)\,dx \le
C\delta_*|\s_i|.
\eeq

\noindent $\bullet$ {\bf Pointwise estimates :} 
For each $i\in \mathcal{S}$, consider the interval
\[
\underline {\mathcal{J}_i} (t) := \bigg[ y_i(t)-\frac{1}{|\s_i|}, y_i(t)+\frac{1}{|\s_i|}  \bigg].
\]
It holds from Lemma \ref{lemma_key2} and \eqref{estji} that 
\[
\underline  {\mathcal{J}_i} (t) \subset  {\mathcal{J}_i}(t), \quad |\underline  {\mathcal{J}_i} (t) (t)|\le \frac{1}{|\s_i|} \le  C\nu^{-1/3} \ll 2\nu^{-5/6} \le | {\mathcal{J}_i}(t) | .
\]
First, using Lemma \ref{lem:small} and
\[
|(a_i)_x| \ge C^{-1} \frac{\s_i^2}{\delta_*} e^{-C|\s_i||x-y_i|},
\]
we find that for each $i\in \mathcal{S}$
\[
|\s_i| \int_{\underline  {\mathcal{J}_i} (t) } Q(v|\bar v) dx \le C\delta_*^2,
\]
which implies
\[
\mbox{for each $i\in \mathcal{S}$, } \exists x_0^i \in \underline  {\mathcal{J}_i} (t) \quad\mbox{s.t. }\quad |(p(v)-p(\bar v))(x_0^i)| \le C \delta_*.
\]
Then, by $\delta_*\ll \delta_1$, for each  $i\in \mathcal{S}$,
\[
(p(v)-p(\uv))(x_0^i) = 0 ,
\]
and so,
\[
(p(v)-p(\uv_b))(x_0^i)= (p(v)-p(\uv_s))(x_0^i) = 0.
\]
Now, since for each  $i\in \mathcal{S}$, $\forall i',\, \forall x\in  \mathcal{J}_{i'}$,
\begin{align*}
\begin{aligned}
 |(p(v)-p(\uv_b))(x)|  &\le \int_{x_0^{i}}^x \left|\partial_x \big( p(v)-p(\uv_b)\big) \right| {\mathbf 1}_{\{p(v)-p(\bar v) < -\delta_1\}}  \,dx \\
 &\le C \int_{x_0^{i}}^x v^{\beta/2}\left|\partial_x \big( p(v)-p(\uv_b)\big) \right|  \,dx ,
 \end{aligned}
\end{align*}
we use \eqref{aeq_D} to have the point-wise estimates:
\beq\label{pob}
\forall i\in \mathcal{S},\quad\forall i' ,\quad \forall x\in  \mathcal{J}_{i'},\quad |(p(v)-p(\uv_b))(x)| \le 
 \left\{ \begin{array}{ll}
    C\sqrt{|x-x_0^{i}|} \sqrt{D_i}, \quad &\mbox{if}~ i'=i, \\
   C\sqrt{|x-x_0^{i}|} \sqrt{D},\quad &\mbox{if}~ i'\neq i ,
       \end{array} \right.
\eeq
where
\[
D_i (t):=\int_{ {\mathcal{J}_i}} a v^\beta |(p(v) - p(\bar v))_x |^2 dx.
\]
Likewise, following  the proof of \cite[(4.63)]{KV-Inven} and the above argument, we have the point-wise estimates:
\beq\label{pos}
\forall i\in \mathcal{S},\,\forall i' ,\, \forall x\in  \mathcal{J}_{i'}, \quad \big| v^{\beta/2} (p(v)-p(\uv_s))(x) \big| \le 
 \left\{ \begin{array}{ll}
    C\sqrt{|x-x_0^{i}|} \bigg(\sqrt{D_i} +\sqrt{\widetilde G_{i}}  \bigg), \quad &\mbox{if}~ i'=i, \\
   C\sqrt{|x-x_0^{i}|} \bigg(\sqrt{D} +\sqrt{\widetilde G}  \bigg),\quad &\mbox{if}~ i'\neq i ,
       \end{array} \right.
\eeq
where
\[
\widetilde G_{i}(t) :=  \int_{ {\mathcal{J}_i}}  |(\bar v)_x| Q(v|\bar v)\,dx.
\]
Notice that  for each  $i\in \mathcal{S}$,
\[
|x_0^i -y_i|\le  |\underline  {\mathcal{J}_i} (t)|\le \frac{1}{|\s_i|} \quad\mbox{by } x_0^i, y_i \in \underline  {\mathcal{J}_i} (t),
\]
and so,
\beq\label{xyid}
|x-x_0^i |\le |x-y_i|+|y_i-x_0^i |\le |x-y_i|+\frac{1}{|\s_i|}.
\eeq

The proof of Proposition \ref{prop:big} are based on a series of the following lemmas.

\begin{lemma}\label{lem:prop}
There exist $C, C_{\delta_1}>0$ ($C_{\delta_1}$ depending on $\delta_1$ as $C_{\delta_1}\sim \delta_1^{-r}$ for some $r>0$) such that for each $i\in \mathcal{S}$,
\begin{align}
\begin{aligned}\label{p2bad}
& \int_\Omega |(a_i)_x| |p(v)-p(\uv) | dx \le C_{\delta_1}\sqrt{\delta_*} D_i(t) +  C_{\delta_1} \frac{\nu|\s_i|}{\delta_*} D(U)   , \\
&\int_\Omega |(a_i)_x| \Big| |p(v)-p(\bar v) |^2 - |p(\uv) -p(\bar v) |^2 \Big|   dx \le C_{\delta_1}\sqrt{\delta_*} D_i(t) +  C_{\delta_1}\frac{\nu|\s_i|}{\delta_*} D(U) ,
\end{aligned}
\end{align}
\begin{align}
\begin{aligned}\label{l3}
&\int_\bbr  |(a_i)_x| v^\beta |p(v)-p(\uv) |^2 \, dx  \\
&\quad \le  C (G_{2i}(U)-G_{2i}(\uu)) + C_{\delta_1}\sqrt{\delta_*}  \big( D_i(t) +\widetilde G_i(t) \big) +  C_{\delta_1} \frac{\nu|\s_i|}{\delta_*}  \big( D(U) +\widetilde G(t) \big),\\
&\int_\bbr  |(a_i)_x| \Big| v^\beta |p(v)-p(\bar v) |^2 -\uv^\beta |p(\uv)-p(\bar v) |^2 \Big|\, dx \\ 
&\quad \le  C (G_{2i}(U)-G_{2i}(\uu))  + C_{\delta_1}\sqrt{\delta_*}  \big( D_i(t) +\widetilde G_i(t) \big) +  C_{\delta_1} \frac{\nu|\s_i|}{\delta_*}  \big( D(U) +\widetilde G(t) \big),
\end{aligned}
\end{align}
\beq\label{l50}
  \int_\bbr |(a_i)_x| \bigg( \left|Q(v|\bar v)-Q(\uv|\bar v)\right | + |v-\uv| \bigg) \,dx \leq  C(G_{2i}(U)-G_{2i}(\uu) ),
\eeq
\begin{align}
\begin{aligned} \label{l5}
&  \int_\bbr |(a_i)_x| \left|p(v|\bar v)-p(\uv|\bar v)\right | \,dx\\
&\quad \le C (G_{2i}(U)-G_{2i}(\uu))  + C_{\delta_1}\sqrt{\delta_*}  \big( D_i(t) +\widetilde G_i(t) \big) +  C_{\delta_1} \frac{\nu|\s_i|}{\delta_*}  \big( D(U) +\widetilde G(t) \big),
\end{aligned}
\end{align}
\begin{align}
\begin{aligned}
\label{ns1}
& \int_{\Omega^c} |(a_i)_x|  \big| p(v)-p(\bar v) \big|  |h-\bar h|  \le  C_{\delta_1} \sqrt\ds  \big( D_i  +\widetilde G_{i} \big) \\
&\quad  \quad+  C_{\delta_1}|\s_i| \sqrt\nu (1+ \sqrt\nu/\delta_*)   \big( D(U) +\widetilde G  \big)+  ( C\delta_1 + C_{\delta_1} \delta_*^{\frac{1}{4}} )  G_{1i}^{-}(U)+ C_{\delta_1}|\s_i| \sqrt\nu,
\end{aligned}
\end{align}
\begin{align}
\begin{aligned}\label{ns2}
&\int_{\Omega^c}  |(a_i)_x|  dx \Big( Q(\uv|\bar v) +  |\uv-\bar v| \Big) dx \\
&\quad \le C_{\delta_1} |\s_i|^{1/2} \delta_*^{3/4}\sqrt{D_i +\widetilde G_{i} }  + C_{\delta_1}\frac{\nu |\s_i|^{3/2}}{\delta_*}\sqrt{D +\widetilde G } .
\end{aligned}
\end{align}
\end{lemma}
\begin{proof}
The proof follows the proofs of \cite[Proposition 6.3]{KV-2shock} (by replacing the parameter $\eps/\lambda$ of \cite[Proposition 6.3]{KV-2shock} by $\delta_*$), and \cite[Lemma 4.5]{KV-Inven}.\\

\noindent {\bf Estimate on  \eqref{p2bad}  :} 
We estimate
\begin{align*}
\begin{aligned}
 \int_\Omega |(a_i)_x| |p(v)-p(\uv) |^2  dx &= \int_\bbr |(a_i)_x| |p(v)-p(\uv_b) |^2  dx  =   K_i +K_i^c,
\end{aligned}
\end{align*}
where
\begin{align*}
\begin{aligned}
& K_i :=  \int_{ {\mathcal{J}_i}} |(a_i)_x| |p(v)-p(\uv_b) |^2  dx,\\
& K^c_i := \int_{ {\mathcal{J}_i}^c} |(a_i)_x| |p(v)-p(\uv_b) |^2  dx .
\end{aligned}
\end{align*}
First, since
\[
\frac{1}{|\s_i|\sqrt{\delta_*}} \le  \frac{C}{\nu^{1/3}\sqrt{\delta_*}} \ll \frac{2}{\nu^{5/6}} \le | {\mathcal{J}_i} |,
\]
using \eqref{pob} with \eqref{xyid}, 
we have
\begin{align*}
\begin{aligned}
 K_i  &\le \int_{|x-y_i|\le \frac{1}{|\s_i|\sqrt{\delta_*}}} |(a_i)_x| |p(v)-p(\uv_b) |^2 \,dx + \int_{|x-y_i|\ge \frac{1}{|\s_i|\sqrt{\delta_*}}} |(a_i)_x| |p(v)-p(\uv_b) |^2 \,dx\\
&\leq C D_i(t)  \int_{|x-y_i|\le \frac{1}{|\s_i|\sqrt{\delta_*}}} |(a_i)_x| |x-x_0^i | {\mathbf 1}_{\{|p(v)-p(\uv)|>0\}} \,dx \\
&\qquad+C D_i(t)  \int_{|x-y_i|\ge \frac{1}{|\s_i|\sqrt{\delta_*}}} |(a_i)_x| |x-x_0^i | dx \\
&\leq  C_{\delta_1} D_i(t)  \left(  \frac{1}{|\s_i|\sqrt{\delta_*}}
\int_\bbr  |(a_i)_x| Q(v|\bar v)\,d\xi + 2\int_{|x-y_i|\ge \frac{1}{|\s_i|\sqrt{\delta_*}}} |(a_i)_x| |x-y_i| dx \right).
\end{aligned}
\end{align*} 
Thus, using Lemma \ref{lem:small} and 
\[
|(a_i)_x| \le C \frac{\s_i^2}{\delta_*} e^{-C^{-1}|\s_i||x-y_i|},
\]
we have
\[
 K_i  \le C_{\delta_1}\sqrt{\delta_*} D_i(U).
\]
On the other hand, using \eqref{wellom}, \eqref{pob}, \eqref{xyid}, and Lemma \ref{lemma_key2} with $\nu\ll1$, we have
\begin{align}
\begin{aligned} \label{kicest}
K_i^c  &\le CD(U)  \int_{|x-y_i|\ge\nu^{-5/6} }  \frac{\s_i^2}{\delta_*} e^{-C^{-1}|\s_i||x-y_i|} \Big( |x-y_i|+\frac{1}{\nu^{1/3}} \Big)  dx \\
&\le  C D(U) \frac{1}{\delta_*} e^{-C^{-1} |\s_i|\nu^{-5/6}} \Big(|\s_i| \nu^{-5/6} + \frac{|\s_i|}{\nu^{1/3}}  \Big) \\
&\le  C|\s_i| D(U) \frac{\nu}{\delta_*} .
\end{aligned}
\end{align} 
Therefore,
\begin{align*}
\begin{aligned}
 \int_\Omega |(a_i)_x| |p(v)-p(\uv) |^2  dx \le C_{\delta_1}\sqrt{\delta_*} D_i(t) +  C_{\delta_1} \frac{\nu|\s_i|}{\delta_*} D(U).
\end{aligned}
\end{align*} 
Then,  using the non-linearization by $(|p(v)-p(\uv) |-\delta_1)_+ \le \frac{2}{\delta_1} (|p(v)-p(\uv) |-\delta_1/2)_+^2$  as in the proof of \cite[Lemma 4.5]{KV-Inven}, we have
\begin{align*}
\begin{aligned}
 \int_\Omega |(a_i)_x| |p(v)-p(\uv) |  dx &\le \frac{2}{\delta_1} \int_\Omega |(a_i)_x| |p(v)-p(\uv_{\delta_1/2}) |^2  dx \\
&\le C_{\delta_1}\sqrt{\delta_*} D_i(t) +  C_{\delta_1} \frac{\nu|\s_i|}{\delta_*} D(U),
\end{aligned}
\end{align*} 
and thus,
\begin{align*}
\begin{aligned}
& \int_\Omega |(a_i)_x| \Big| |p(v)-p(\bar v) |^2 - |p(\uv) -p(\bar v) |^2 \Big|   dx\\
& \le \int_\Omega |(a_i)_x| \Big( |p(v)-p(\uv) |^2 +2\delta_1 |p(v)-p(\uv) | \Big)   dx \\
&\le C_{\delta_1}\sqrt{\delta_*} D_i(t) +  C_{\delta_1} \frac{\nu|\s_i|}{\delta_*} D(U).
\end{aligned}
\end{align*}

\noindent {\bf Estimate on  \eqref{l3}  :} 
Decompose
\begin{align*}
\begin{aligned}
\int_\bbr  |(a_i)_x| v^\beta |p(v)-p(\uv) |^2 \, dx &= \int_\bbr  |(a_i)_x| v^\beta |p(v)-p(\uv_b) |^2 \, dx \\
&\quad + \int_\bbr  |(a_i)_x| v^\beta |p(v)-p(\uv_s) |^2 \, dx.
\end{aligned}
\end{align*} 
As in the proof of (4.49) of \cite[Lemma 4.6]{KV-Inven}, we use $\beta\le 1$ to have
\[
 \int_\bbr  |(a_i)_x| v^\beta |p(v)-p(\uv_b) |^2 \, dx \le \int_\bbr  |(a_i)_x| \Big( Q(v|\bar v) - Q(\uv|\bar v) \Big) dx \le C (G_{2i}(U)-G_{2i}(\uu)) .
\]
Using \eqref{pos} and the same estimates as in the proof of \eqref{p2bad}, we have
\begin{align*}
\begin{aligned}
 \int_\bbr  |(a_i)_x| v^\beta |p(v)-p(\uv_s) |^2 dx &=  \int_{ {\mathcal{J}_i}} |(a_i)_x| v^\beta |p(v)-p(\uv_s) |^2 dx \\
 &\quad + \int_{ {\mathcal{J}_i}^c} |(a_i)_x| v^\beta |p(v)-p(\uv_s) |^2 dx  \\
 &\le  C_{\delta_1}\sqrt{\delta_*}  \big( D_i(t) +\widetilde G_i(t) \big) +  C_{\delta_1} \frac{\nu|\s_i|}{\delta_*}  \big( D(U) +\widetilde G(t) \big).
\end{aligned}
\end{align*} 
Thus,
\begin{align*}
\begin{aligned}
\int_\bbr  |(a_i)_x| v^\beta |p(v)-p(\uv) |^2 \, dx &\le  C (G_{2i}(U)-G_{2i}(\uu))  \\
&\quad + C_{\delta_1}\sqrt{\delta_*}  \big( D_i(t) +\widetilde G_i(t) \big) +  C_{\delta_1} \frac{\nu|\s_i|}{\delta_*}  \big( D(U) +\widetilde G(t) \big).
\end{aligned}
\end{align*} 
This implies \eqref{l3} as above. \\

\noindent {\bf Estimate on \eqref{l50} and \eqref{l5}  :} 
As in the proof of \cite[Lemma 4.6]{KV-Inven}, we first have
\begin{eqnarray*}
&&\int_\bbr  |(a_i)_x| \left|Q(v|\bar v)-Q(\uv|\bar v)\right | \,d\xi+\int_\bbr |(a_i)_x| |v-\uv| \,d\xi \\
&& \qquad\le C \int_\bbr|(a_i)_x|  \left(Q(v|\bar v )-Q(\uv|\bar v)\right) \,d\xi \le C\left({G}_{2i}(U)-{G}_{2i}(\uu) \right).
\end{eqnarray*}
and
\begin{align*}
\begin{aligned}
\int_\bbr|(a_i)_x|  \left|p(v|\bar v)-p(\uv|\bar v)\right | \,d\xi &\le  \int_\bbr|(a_i)_x|  \left|p(v)-p(\uv) \right | \,d\xi +\int_\bbr|(a_i)_x|   |v-\uv|  \,d\xi
\\
& \le \underbrace{\int_\bbr|(a_i)_x|  \left|p(v)-p(\uv) \right | \,d\xi}_{=:J_i} + C\left({G}_{2i}(U)-{G}_{2i}(\uu) \right)
\end{aligned}
\end{align*} 
We separate $J_i$ into three parts:
\[
J_i =\underbrace{ \int_\bbr|(a_i)_x|  \left|p(v)-p(\uv_b) \right | \,d\xi}_{=:J_{1i}} +\underbrace{\int_{v<v_*/2} |(a_i)_x|  \left|p(v)-p(\uv_s) \right | \,d\xi}_{=:J_{2i}}+\underbrace{\int_{v\ge v_*/2} |(a_i)_x|  \left|p(v)-p(\uv_s) \right | d\xi}_{=:J_{3i}}.
\]
By \eqref{p2bad}, 
\[
 J_{1i} \le  C_{\delta_1}\sqrt{\delta_*} D_i(t) +  C_{\delta_1} \frac{\nu|\s_i|}{\delta_*} D(U).
\]
To estimate $J_{2i}$, we first observe that
\begin{align}
\begin{aligned}\label{pw1}
v^\beta \big| p(v)-p(\uv_s) \big|^2 {\mathbf 1}_{\{p(v)-p(\bar v)>2\delta_1\}} &= p(v)^{-\frac{\gamma-\alpha}{\gamma}} |p(v)-p(\uv_s)|^2  {\mathbf 1}_{\{p(v)-p(\bar v)>2\delta_1\}}\\
&= \Big(\frac{|p(v)-p(\uv_s)|}{p(v)} \Big)^{\frac{\gamma-\alpha}{\gamma}} {\mathbf 1}_{\{p(v)-p(\bar v)>2\delta_1\}}  |p(v)-p(\uv_s)|^{\frac{\gamma+\alpha}{\gamma}}\\
&\ge C\delta_1^{\frac{\gamma-\alpha}{\gamma}}  |p(v)-p(\uv_s)|^{\frac{\gamma+\alpha}{\gamma}}{\mathbf 1}_{\{p(v)-p(\bar v)>2\delta_1\}}.
\end{aligned}
\end{align}
This and \cite[(4.67)]{KV-Inven} implies
\[
\left|p(v)-p(\uv_s) \right | {\mathbf 1}_{\{v<v_-/2\}} \le  C |p(v)-p(\uv_s)|^{\frac{\gamma+\alpha}{\gamma}} \le C v^\beta \big| p(v)-p(\uv_s) \big|^2,
\]
which together with \eqref{pos} implies that $\forall i\in \mathcal{S},\,\forall i' ,\, \forall x\in  \mathcal{J}_{i'}$,
\[
\left|p(v)-p(\uv_s) \right | {\mathbf 1}_{\{v<v_*/2\}}(x)  \le 
 \left\{ \begin{array}{ll}
    C|x-x_0^{i}| \big(D_i +\widetilde G_{i}  \big), \quad &\mbox{if}~ i'=i, \\
   C|x-x_0^{i}| \big(D +\widetilde G  \big),\quad &\mbox{if}~ i'\neq i .
       \end{array} \right.
\]
Then, using the same estimates as in the proofs of \eqref{p2bad} and \eqref{l3}, we have
\[
 J_{2i}  \le   C_{\delta_1}\sqrt{\delta_*}  \big( D_i(t) +\widetilde G_i(t) \big) +  C_{\delta_1} \frac{\nu|\s_i|}{\delta_*}  \big( D(U) +\widetilde G(t) \big).
\]
As in the proof of  \cite[Lemma 4.6]{KV-Inven}, we have
\[
 J_{3i}  \le C \left({G}_{2i}(U)-{G}_{2i}(\uu) \right).
\]

\noindent {\bf Estimate on  \eqref{ns1} :} 
We first separate it into two parts: (by the definition of $\uv_s$)
\begin{align*}
\begin{aligned}
& \int_{\Omega^c} |(a_i)_x|  \big| p(v)-p(\bar v) \big|  |h-\bar h| dx \\
 &\quad\le \underbrace{ \int_{\bbr}  |(a_i)_x| \big| p(v)-p(\uv_s) \big| |h-\bar h| \,d\xi}_{=:J_{1} }  + \underbrace{ \int_{\Omega^c}  |(a_i)_x| \big| p(\uv)-p(\bar v) \big| |h-\bar h| \,d\xi}_{=:J_{2} }
\end{aligned}
\end{align*}
To estimate $J_1$, we will use the following estimates: by \eqref{pos}, \eqref{pw1},  
\begin{align}
\begin{aligned}\label{pos2}
&\forall i\in \mathcal{S},\,\forall i' ,\, \forall x\in  \mathcal{J}_{i'}, \quad 
\forall v \mbox{ with } p(v)-p(\bar v)>2\delta_1,\\
&
\big| (p(v)-p(\uv_s))(x) \big|^2 \le 
 \left\{ \begin{array}{ll}
    C_{\delta_1}|x-x_0^{i}|^q \big( D_i +\widetilde G_{i}  \big)^q, \quad &\mbox{if}~ i'=i, \\
   C_{\delta_1}|x-x_0^{i}|^q \big( D +\widetilde G  \big)^q,\quad &\mbox{if}~ i'\neq i ,
       \end{array} \right.
\end{aligned}
\end{align}
where $q:=\frac{2\gamma}{\gamma+\alpha}$ ( $1\le q<2$ by $0<\alpha\le \gamma$).\\
By  the definition of $\bar v_s$, we split $J_1$ into two parts:
\begin{align}
\begin{aligned}\label{vs2}
J_1&=\int_{\Omega^c} |(a_i)_x| \big| p(v)-p(\uv_s) \big| |h-\bar h| {\mathbf 1}_{\{\delta_1<p(v)-p(\bar v)\le 2\delta_1\}} \,d\xi \\
&\quad +\int_{\Omega^c} |(a_i)_x| \big| p(v)-p(\uv_s) \big| |h-\bar h| {\mathbf 1}_{\{p(v)-p(\bar v)>2\delta_1\}} \,d\xi \\
&=: J_{11} + J_{12}.
\end{aligned}
\end{align}
To estimate $J_{12}$, we first have as in the proof of \eqref{p2bad}:
\begin{align*}
\begin{aligned}
 J_{12}=   K_i + K_i^c,
\end{aligned}
\end{align*}
where
\begin{align*}
\begin{aligned}
& K_i :=  \int_{ {\mathcal{J}_i}} |(a_i)_x|  \big| p(v)-p(\uv_s) \big|  |h-\bar h| {\mathbf 1}_{\{p(v)-p(\bar v)>2\delta_1\}} dx,\\
&K_i^c :=  \int_{ {\mathcal{J}_i}^c} |(a_i)_x|  \big| p(v)-p(\uv_s) \big|  |h-\bar h| {\mathbf 1}_{\{p(v)-p(\bar v)>2\delta_1\}} dx.
\end{aligned}
\end{align*}
First, 
\[
K_i \le   \left(\int_{ {\mathcal{J}_i}} |(a_i)_x| \big| p(v)-p(\uv_s) \big|^2{\mathbf 1}_{\{p(v)-p(\bar v)>2\delta_1\}} d\xi \right)^{1/2}  \left(\int_\bbr |(a_i)_x| \big| h -\bar h \big|^2 d\xi \right)^{1/2}.
\]
Using \eqref{pos2}, \eqref{xyid} and Lemma \ref{lem:small}, we have
\begin{align}
\begin{aligned}\label{hstech}
&\int_{{\mathcal{J}_i}} |(a_i)_x| \big| p(v)-p(\uv_s) \big|^2 {\mathbf 1}_{\{p(v)-p(\bar v)>2\delta_1\}} d\xi\\
&\quad\le C_{\delta_1} \big( D_i +\widetilde G_{i}  \big)^q \bigg( \frac{1}{|\s_i|^q\delta_*^{q/4}} \int_{|x-y_i|\le \frac{1}{|\s_i|\delta_*^{1/4}}}  |(a_i)_x| Q(v|\bar v) dx + 2 \int_{|x-y_i|\ge \frac{1}{|\s_i|\delta_*^{1/4}}}  |(a_i)_x| |x-y_i|^q dx\bigg)\\
&\quad\leq C_{\delta_1} \frac{|\s_i|\delta_*}{|\s_i|^q\delta_*^{q/4}}  \big( D_i +\widetilde G_{i}  \big)^q .
\end{aligned}
\end{align}
Therefore, 
\[
K_i  \le C_{\delta_1} \sqrt{ \delta_*^{1-\frac{q}{4}} |\s_i|^{1-q}}  \big( D_i +\widetilde G_{i}  \big)^{q/2}  \left(\int_\bbr |(a_i)_x| \big| h -\bar h \big|^2 d\xi \right)^{1/2}.
\]
By Young's inequality with $1<q<2$, we have
\[
K_i  \le \sqrt{\delta_*}  \big( D_i +\widetilde G_{i}  \big) +  \frac{C_{\delta_1}}{\sqrt{\delta_*}} \left( \delta_*^{1-\frac{q}{4}} |\s_i|^{1-q}\right)^{\frac{1}{2-q}} \left(\int_\bbr |(a_i)_x| \big| h -\bar h \big|^2 d\xi \right)^{\frac{1}{2-q}}.
\]
Since Lemma \ref{lem:small} and $1<q<2$ imply
\begin{align*}
\begin{aligned}
&\left( \delta_*^{1-\frac{q}{4}} |\s_i|^{1-q}\right)^{\frac{1}{2-q}} \left(\int_\bbr |(a_i)_x| \big| h -\bar h \big|^2 d\xi \right)^{\frac{1}{2-q}} \\
&\le C \delta_*^{\frac{4-q}{4(2-q)}} |\s_i|^{\frac{1-q}{2-q}} \big(|\s_i|\delta_*\big)^{\frac{q-1}{2-q}}
 \int_\bbr |(a_i)_x| \big| h -\bar h \big|^2 d\xi\\
&=C \delta_*^{\frac{3q}{4(2-q)}} \int_\bbr |(a_i)_x| \big| h -\bar h \big|^2 d\xi \\
&\le C\delta_*^{\frac{3}{4}} \int_\bbr |(a_i)_x| \big| h -\bar h \big|^2 d\xi,
\end{aligned}
\end{align*}
we have
\[
K_i  \le \sqrt{\delta_*}  \big( D_i (U) +\widetilde G_{i}(U)   \big) +  C_{\delta_1}\delta_*^{\frac{1}{4}} G_{1i}^{-}(U).
\]
For $K_i^c$, using \eqref{pos2} and the same estimates as in the proof of \eqref{kicest}, we have
\begin{align*}
\begin{aligned}
K_i^c &\le   \left( \int_{ {\mathcal{J}_i}^c} |(a_i)_x| \big| p(v)-p(\uv_s) \big|^2{\mathbf 1}_{\{p(v)-p(\bar v)>2\delta_1\}} d\xi \right)^{1/2}  \left( \int_{ {\mathcal{J}_i}^c} |(a_i)_x| \big| h -\bar h \big|^2 d\xi \right)^{1/2}\\
&\le  C_{\delta_1}\left( \frac{\nu |\s_i|}{\delta_*}   \big( D +\widetilde G  \big)^q \right)^{1/2} \left( \int_\bbr |(a_i)_x| \big| h -\bar h \big|^2 d\xi \right)^{1/2}.
\end{aligned}
\end{align*}
Using Lemma \ref{lem:small} and $q<2$, we have
\[
K_i^c \le C_{\delta_1} |\s_i| \sqrt\nu  \big( D +\widetilde G  \big)^{q/2} \le C_{\delta_1} |\s_i| \sqrt\nu \Big(  D (U) +\widetilde G (U)  + 1 \Big). 
\]
Thus, 
\[
 J_{12} \le   \sqrt{\delta_*}  \big( D_i (t) +\widetilde G_{i}(t)   \big) + C_{\delta_1}\delta_*^{\frac{1}{4}} G_{1i}^{-}(U) +  C_{\delta_1} |\s_i| \sqrt\nu  \big( D(t) +\widetilde G(t) +1  \big)   .
\]

For $J_{11}$ and $J_2$, we  have
\[
J_{11}+J_2 \le C\delta_1\int_{\Omega^c}  |(a_i)_x|  |h-\bar h| dx \le C \delta_1 G_{1i}^{-}(U) + C \delta_1 \int_\bbr  |(a_i)_x|   {\mathbf 1}_{\{p(v)-p(\bar v)>\delta_1\}} dx.
\]
To control the last term,  we use the following non-linearization: 
\[
|p(v)-p((\uv_s)_{\delta_1/2})|=  (p(v)-p(\bar v)-\delta_1/2)_+ \ge \frac{\delta_1}{2} {\mathbf 1}_{\{p(v)-p(\bar v) >\delta_1\}},
\]
where $(\uv_s)_{\delta_1/2}$ denote $\uv_s$ truncated by $\delta_1/2$ instead of $\delta_1$.\\
In addition, using \eqref{pos2} with the same estimates as before, we have
\begin{align}
\begin{aligned}\label{aderout} 
&\delta_1 \int_\bbr  |(a_i)_x|   {\mathbf 1}_{\{p(v)-p(\bar v)>\delta_1\}} dx \le  C_{\delta_1} \int_{\Omega^c}  |(a_i)_x| |p(v)-p(\uv_s)|^{2/q} dx \\
&\le C_{\delta_1} \big(D_i +\widetilde G_{i} \big) \int_{ {\mathcal{J}_i}}   |(a_i)_x| |x-x_0^{i}| dx+C_{\delta_1}\big(D +\widetilde G \big) \int_{ {\mathcal{J}_i}^c}   |(a_i)_x| |x-x_0^{i}| dx \\
&\le C_{\delta_1}\sqrt\ds \big(D_i +\widetilde G_{i} \big)  + C_{\delta_1}\frac{\nu |\s_i|}{\delta_*}\big(D +\widetilde G \big) .
\end{aligned}
\end{align}
Thus, 
\[
J_{11}+J_2  \le C \delta_1 G_{1i}^{-}(U) +   C_{\delta_1}\sqrt\ds \big(D_i +\widetilde G_{i} \big)  + C_{\delta_1}\frac{\nu |\s_i|}{\delta_*}\big(D +\widetilde G \big) .
\]
Hence we obtain \eqref{ns1}.\\

\noindent {\bf Estimate on  \eqref{ns2} :} 
First, we have
\[
\int_{\Omega^c}  |(a_i)_x|   \Big( Q(\uv|\bar v) +  |\uv-\bar v| \Big) dx \le C  \int_\bbr  |(a_i)_x|   {\mathbf 1}_{\{p(v)-p(\bar v)>\delta_1\}} dx.
\]
Using \eqref{pos2} with the same argument as in \eqref{aderout}, we have
\begin{align*}
\begin{aligned} 
& \int_{\Omega^c}  |(a_i)_x| {\mathbf 1}_{\{p(v)-p(\bar v)>\delta_1\}}  dx \le  C_{\delta_1} \int_{\Omega^c}  |(a_i)_x| |p(v)-p(\uv_s)|^{1/q} dx \\
&\le C_{\delta_1}\sqrt{D_i +\widetilde G_{i} } \int_{ {\mathcal{J}_i}}   |(a_i)_x| \sqrt{|x-x_0^{i}|} dx+C_{\delta_1}\sqrt{D +\widetilde G }  \int_{ {\mathcal{J}_i}^c}   |(a_i)_x| \sqrt{|x-x_0^{i}|} dx \\
&\le C_{\delta_1} |\s_i|^{1/2} \delta_*^{3/4}\sqrt{D_i +\widetilde G_{i} }  + C_{\delta_1}\frac{\nu |\s_i|^{3/2}}{\delta_*}\sqrt{D +\widetilde G } .
\end{aligned}
\end{align*}

\end{proof}

\subsection{Proof of Proposition \ref{prop:big}} 
The proof is based on Lemma \ref{lem:prop}.\\
\noindent {\bf Estimate on  \eqref{badprop} and \eqref{roughbad} :} 
We use \eqref{avr}, \eqref{l5}, \eqref{ns1}, \eqref{p2bad} to have
\begin{align*}
\begin{aligned}
& |B_{1i}(U)-B_{1i}(\uu)| \\
&\le \ds \int_\bbr |(a_i)_x| \left|p(v|\bar v)-p(\uv|\bar v)\right | \,dx \\
 &  \le C \ds  (G_{2i}(U)-G_{2i}(\uu))  + C_{\delta_1} \ds \sqrt{\delta_*}  \big( D_i(t) +\widetilde G_i(t) \big) +  C_{\delta_1} \nu|\s_i|  \big( D(U) +\widetilde G(t) \big),
\end{aligned}
\end{align*}
\begin{align*}
\begin{aligned}
 |B_{2i}^-(U)| &\le  \int_{\Omega^c} |(a_i)_x|  \big| p(v)-p(\bar v) \big|  |h-\bar h|  \le  C_{\delta_1} \sqrt\ds  \big( D_i  +\widetilde G_{i} \big) \\
&\quad  \quad+  C_{\delta_1}|\s_i| \sqrt\nu (1+ \sqrt\nu/\delta_*)   \big( D(U) +\widetilde G  \big)+  ( C\delta_1 + C_{\delta_1} \delta_*^{\frac{1}{4}} )  G_{1i}^{-}(U)+ C_{\delta_1}|\s_i| \sqrt\nu,
\end{aligned}
\end{align*}
and
\begin{align*}
\begin{aligned}
 |B_{2i}^+(U)-B_{2i}^+(\uu)| &\le  \int_\Omega |(a_i)_x| \Big| |p(v)-p(\bar v) |^2 - |p(\uv) -p(\bar v) |^2 \Big|   dx \\
& \le  C_{\delta_1}\sqrt{\delta_*} D_i(t) +  C_{\delta_1} \frac{\nu|\s_i|}{\delta_*} D(U) .
\end{aligned}
\end{align*}
To estimate $B_3(U)$, we first observe that by $|(a_i)_x|\le C |\s_i|^2/\ds$,
\begin{align*}
\begin{aligned}
 \left( \sum_{i\in \mathcal{S}} |(a_i)_x| \right)^2
 &=  \sum_{i\in \mathcal{S}} |(a_i)_x|^2  +  \sum_{i\in \mathcal{S}} |(a_i)_x| \bigg( \sum_{i'\in \mathcal{S}-\{i\}}  |(a_{i'})_x|   \bigg)  \\
& \le C\left( \frac{L(0)^2}{\ds} +\frac{L(0)}{\ds}  \sum_{i'\in \mathcal{S}-\{i\}} |\s_{i'} | \right) \sum_{i\in \mathcal{S}} |(a_i)_x|  \le \frac{CL(0)^2}{\ds} \sum_{i\in \mathcal{S}} |(a_i)_x|.  
\end{aligned}
\end{align*}
Then, using $L(0)\le\ds^2$ and \eqref{l3} with $\sum_{i\in \mathcal{S}} \big( D_i +\widetilde G_i\big) \le C \big( D +\widetilde G\big)$, we have
\begin{align*}
\begin{aligned}
& |B_3(U) - B_3(\uu)|\\
 &\le  \frac{C}{\ds} \int_\bbr \left( \sum_{i\in \mathcal{S}} |(a_i)_x| \right)^2  \Big| v^\beta |p(v)-p(\bar v) |^2 -\uv^\beta |p(\uv)-p(\bar v) |^2 \Big|\, dx \\
& \le  \frac{CL(0)^2}{\ds^2}  \sum_{i\in \mathcal{S}} \int_\bbr  |(a_i)_x| \Big| v^\beta |p(v)-p(\bar v) |^2 -\uv^\beta |p(\uv)-p(\bar v) |^2 \Big|\, dx\\
& \le C\ds^2 \sum_{i\in \mathcal{S}} (G_{2i}(U)-G_{2i}(\uu))  + C_{\delta_1}\sqrt\ds  \big( D(U) +\widetilde G(t) \big) ,
\end{aligned}
\end{align*}
and by $C^{-1}\le \uv \le C$,
\begin{align*}
\begin{aligned}
|B_3(\uu)| &\le \frac{C}{\ds} \int_\bbr \left( \sum_{i\in \mathcal{S}} |(a_i)_x| \right)^2 \uv^\beta |p(\uv)-p(\bar v) |^2 dx
 \le \frac{CL(0)^2}{\ds^2}  \sum_{i\in \mathcal{S}}\int_\bbr  |(a_i)_x| Q(\uv|\bar v) dx \\
& \le C\ds^2 \sum_{i\in \mathcal{S}} G_{2i}(\uu) .
\end{aligned}
\end{align*}
Thus,
\[
|B_3(U)| \le   C\ds^2 \sum_{i\in \mathcal{S}} (G_{2i}(U)-G_{2i}(\uu)) +C\ds^2 \sum_{i\in \mathcal{S}} G_{2i}(\uu) + C_{\delta_1}\sqrt\ds  \big( D(U) +\widetilde G(t) \big) .
\]
Moreover, it holds from Lemma \ref{lem:small} that for each $i$,
\[
 |B_{1i}(\uu)| +  |B_{2i}^+(\uu)| \le C \int_\Omega |(a_i)_x| Q(\uv|\bar v) dx \le C\ds |\s_i|.
\]

\noindent {\bf Estimate on  \eqref{error-in} :}
First, using $|1-\phi_i|\le \mathbf{1}_{ {\mathcal{J}_i}^c}$ by \eqref{phii}, and $C^{-1}\le \uv, \bar v\le C$, we have
\begin{align*}
\begin{aligned}
 | Y^g_i(\uu)-\mathcal{Y}^g_i(\uv)| + |{B}_{1i}(\uu)-\mathcal{I}_{1i}(\uv)|+ |{B}_{2i}^+(\uu)-\mathcal{I}_{2i}(\uv)| \le C \int_{ {\mathcal{J}_i}^c} |(a_i)_x| dx.
\end{aligned}
\end{align*}
Since \eqref{wellom} with the same estimates as in \eqref{kicest} implies
\begin{align}
\begin{aligned} \label{aicest} 
 \int_{ {\mathcal{J}_i}^c} |(a_i)_x| dx  &\le C  \int_{|x-y_i|\ge\nu^{-5/6} }  \frac{\s_i^2}{\delta_*} e^{-C^{-1}|\s_i||x-y_i|}  dx \le  C \frac{|\s_i|}{\delta_*} e^{-C^{-1} \nu^{-1/2}} \le  C\frac{|\s_i| \nu}{\delta_*} ,
\end{aligned}
\end{align} 
we have the desired estimates.\\
Since \cite[Appendix A, (A.5)]{KV-2shock} yields 
\[
Q(\uv|\bar v)\ge \frac{p(\bar v)^{-\frac{1}{\gamma}-1}}{2\gamma}|p(\uv)-p(\bar v)|^2 -\frac{1+\gamma}{3\gamma^2} p(\bar v)^{-\frac{1}{\gamma}-2}(p(\uv)-
p(\bar v))^3,
\]
we use $\tilde\lambda_i (a_i)_x>0$ to have
\begin{align*}
\begin{aligned}
&-( G_{2i}(\uu) - \mathcal{G}_{2i} (\uv) ) \\
&= -\tilde\lambda_i  \int_{\bbr}  (a_i)_x \bigg[Q(\uv|\bar v) - \bigg( \frac{1}{2\gamma}  p(\tilde v_i)^{-\frac{1}{\gamma}-1} \phi_i^2 \big(p(\uv)-p(\bar v)\big)^2\\
&\qquad - \frac{1+\gamma}{3\gamma^2} p(\tilde v_i)^{-\frac{1}{\gamma}-2} \phi_i^3 \big(p(\uv)-p(\bar v)\big)^3 \bigg) \bigg]dx \\
&\le C  \int_{\bbr}  |(a_i)_x| |\bar v -\tilde v_i| dx + C\int_{ {\mathcal{J}_i}^c} |(a_i)_x|.
\end{aligned}
\end{align*}
Thus, using \eqref{aicest} and Lemma \ref{lem:finsout} with $ \int_{\bbr}  |(a^\nu_i)_x| |\bar v_{\nu,\delta} -\tilde v^\nu_i| dx= \int_{\bbr}  |(a_i)_x| |\bar v -\tilde v_i| dx$, we have
\[
-\big( G_{2i}(\uu) - \mathcal{G}_{2i} (\uv) \big) \le \frac{C\nu^{1/2}}{\ds\delta} + C\frac{|\s_i| \nu}{\delta_*} \le  \frac{C\sqrt\nu}{\ds\delta} .
\]
Using $\sum_{i\in\mathcal{S}} \phi_i \le 1$ and $0\le \phi_i^2 \le \phi_i \le 1$, we have
\begin{align*}
\begin{aligned}
D(\uu) &\ge  \int_{\bbr} a\,\Big(\sum_{i\in\mathcal{S}} \phi_i  \Big) \mu_1(\uv ) \uv^\gamma \big| \big( p(\uv) -p(\bar v)\big)_x \big|^2 dx \\
&\ge \sum_{i\in\mathcal{S}}  \int_{\bbr} a\,\phi_i^2 \mu_1(\uv ) \uv^\gamma \big| \big( p(\uv) -p(\bar v)\big)_x \big|^2 dx .
\end{aligned}
\end{align*}
In addition, since
\begin{align*}
\begin{aligned}
\mathcal{D}_{i} (\uv) &= \int_{\bbr} a\, \mu_1(\uv ) \uv^\gamma \big| \big( \phi_i^2 ( p(\uv) -p(\bar v))\big)_x \big|^2 dx \\
&\le (1+\ds) \int_{\bbr} a\, \mu_1(\uv ) \uv^\gamma\phi_i^2 \big| \big(  p(\uv) -p(\bar v)\big)_x \big|^2 dx +\frac{C}{\ds}\int_{\bbr} a\, \mu_1(\uv ) \uv^\gamma |(\phi_i)_x|^2 \big| p(\uv) -p(\bar v) \big|^2 dx,
\end{aligned}
\end{align*}
and $|(\phi_i)_x|\le C\nu^{5/6}$ from \eqref{wellom} and \eqref{phii}, and so
\[
\int_{\bbr} a\, \mu_1(\uv ) \uv^\gamma |(\phi_i)_x|^2 \big| p(\uv) -p(\bar v) \big|^2 dx \le C\nu^{5/3} \int_\bbr a \eta(\uu|\bar U) dx,
\]
we use Lemma \ref{lem:num} to have
\begin{align*}
\begin{aligned}
-D(\uu) &\le  -\frac{1}{1+\ds}\sum_{i\in\mathcal{S}} \mathcal{D}_{i} (\uv) + \frac{C\nu^{5/3}}{\ds}\sum_{i\in\mathcal{S}} \int_\bbr a \eta(\uu|\bar U) dx\\
&\le  -\frac{1}{1+\ds}\sum_{i\in\mathcal{S}} \mathcal{D}_{i} (\uv) + \frac{C\nu^{4/3}}{\delta\ds} \int_\bbr a \eta(U|\overline U) dx.
\end{aligned}
\end{align*}

\noindent {\bf Estimate on  \eqref{finyest} :} 
We use the notations $Y^s_{i1}, Y^s_{i2}, Y^s_{i3}, Y^s_{i4}$ for the terms of $Y^s_{i}$:
\[
Y^s_i = \underbrace{-\int_{\Omega^c} (a_i)_x Q(v|\bar v) dx}_{=:Y^s_{i1}} \underbrace{-\int_{\Omega^c} a (\tilde v_i)_x p'(\bar v)(v-\bar v)dx }_{=:Y^s_{i2}}
\underbrace{ -\int_{\Omega^c} (a_i)_x \frac{|h-\bar h|^2}{2} dx}_{=:Y^s_{i3}} \underbrace{+\int_{\Omega^c} a (\tilde h_i)_x (h-\bar h) dx}_{=:Y^s_{i4}}.
\]
Using \eqref{p2bad}, \eqref{l50} and \eqref{avr}, we first have
\begin{align}
\begin{aligned} \label{ygiuu}
& |Y^g_i(U)- Y^g_i(\uu)| + |Y^s_{i1}(U)-Y^s_{i1}(\uu)| + |Y^s_{i2}(U)-Y^s_{i2}(\uu)| \\
&\le C \int_\Omega |(a_i)_x| \Big| |p(v)-p(\bar v) |^2 - |p(\uv) -p(\bar v) |^2 \Big|   dx \\
&\quad +  \int_\bbr |(a_i)_x| \bigg( \left|Q(v|\bar v)-Q(\uv|\bar v)\right | + |v-\uv| \bigg) \,dx + C\ds \int_\Omega |(a_i)_x| |p(v)-p(\uv) | dx \\
&\le C_{\delta_1}\sqrt{\delta_*} D_i(t) +  C_{\delta_1}\frac{\nu|\s_i|}{\delta_*} D(U) + C(G_{2i}(U)-G_{2i}(\uu) ).
\end{aligned}
\end{align}
We use \eqref{p2bad} and Lemma \ref{lem:small} to have
\begin{align*}
\begin{aligned}
|Y^s_{i3}(U)| +|Y^b_i(U)|  
&\le C\int_{\Omega^c}  |(a_i)_x|  |h-\bar h|^2  dx +  C\int_{\Omega}  |(a_i)_x| \big( |h-\bar h|^2 +|p(v)-p(\bar v)|^2 \big) dx\\
&\le C\int_\bbr  |(a_i)_x|  |h-\bar h|^2  dx+ \int_\Omega |(a_i)_x| \Big| |p(v)-p(\bar v) |^2 - |p(\uv) -p(\bar v) |^2 \Big|   dx \\
&\quad + C \int_\Omega |(a_i)_x| Q(\uv|\bar v) dx\\
&\le C_{\delta_1}\sqrt{\delta_*} D_i(t) +  C_{\delta_1}\frac{\nu|\s_i|}{\delta_*} D(U)  + C \ds |\s_i|.
\end{aligned}
\end{align*}
First, using the assumption \eqref{yicond} and Lemma \ref{lem:small}, we have
\[
 |Y^g_i(U)- Y^g_i(\uu)| + |Y^s_{i1}(U)-Y^s_{i1}(\uu)| + |Y^s_{i2}(U)-Y^s_{i2}(\uu)| + |Y^s_{i3}(U)| +|Y^b_i(U)| \le C_{\delta_1} \ds |\s_i|.
\]
On the other hand, since
\begin{align*}
\begin{aligned}
&\left| \int_\Omega  (a_i)_x  \big(p(v)-p(\bar v)\big)\Big(h-\bar h-\frac{p(v)-p(\bar v)}{\tilde\lambda_i}\Big) dx \right| \\
&\le \ds^{-1/4} G_{1i}^+(U) + C\ds^{1/4} \int_\Omega |(a_i)_x| |p(v)-p(\bar v) |^2 dx\\
&\le \ds^{-1/4} G_{1i}^+(U) +C\ds^{1/4}  \int_\Omega |(a_i)_x| \left(Q(\uv|\bar v)  + \Big| |p(v)-p(\bar v) |^2 - |p(\uv) -p(\bar v) |^2 \Big| \right) dx \\
&\le \ds^{-1/4} G_{1i}^+(U) + C\ds^{1/4} \Big( G_{2i}(\uu) +C_{\delta_1}\sqrt{\delta_*} D_i(t) +  C_{\delta_1}\frac{\nu|\s_i|}{\delta_*} D(U)  \Big),
 \end{aligned}
\end{align*}
we have
\[
|Y^b_i(U)| \le \ds^{-1/4} G_{1i}^+(U) + C\ds^{1/4} \Big( G_{2i}(\uu) +C_{\delta_1}\sqrt{\delta_*} D_i(t) +  C_{\delta_1}\frac{\nu|\s_i|}{\delta_*} D(U)  \Big).
\]
In addition, using $|Y^s_{i3}(U)| \le G_{1i}^-(U)$ and \eqref{ygiuu}, we have
\begin{align*}
\begin{aligned}
&\Big( |Y^g_i(U)- Y^g_i(\uu)| + |Y^s_{i1}(U)-Y^s_{i1}(\uu)| + |Y^s_{i2}(U)-Y^s_{i2}(\uu)| + |Y^s_{i3}(U)| +|Y^b_i(U)| \Big)^2\\
&\le C_{\delta_1} \ds |\s_i| \Big( |Y^g_i(U)- Y^g_i(\uu)| + |Y^s_{i1}(U)-Y^s_{i1}(\uu)| + |Y^s_{i2}(U)-Y^s_{i2}(\uu)| + |Y^s_{i3}(U)| +|Y^b_i(U)| \Big)\\
&\le C_{\delta_1} \ds |\s_i| \Big( \sqrt{\delta_*} D_i(t) + \frac{\nu|\s_i|}{\delta_*} D(U)  + C(G_{2i}(U)-G_{2i}(\uu) ) + G_{1i}^-(U)\\
&\quad\quad+ \ds^{-1/4} G_{1i}^+(U) + C\ds^{1/4} G_{2i}(\uu) \Big).
 \end{aligned}
\end{align*}
For the remaining terms, using \eqref{avr},
\begin{align*}
\begin{aligned}
&|Y^s_{i4}(U)|^2 \le C\int_\bbr |(\tilde h_i)_x| dx  \int_{\Omega^c} |(\tilde h_i)_x| |h-\bar h|^2 dx \le C|\s_i| \delta_* G_{1i}^-(U),\\
&|Y^l_i(U)|^2 \le C\int_\bbr |(\tilde h_i)_x| dx \int_\Omega | (\tilde h_i)_x| \Big|h-\bar h-\frac{p(v)-p(\bar v)}{\tilde\lambda_i}\Big|^2 dx\le C|\s_i| \delta_* G_{1i}^+(U).
 \end{aligned}
\end{align*}
It holds from \eqref{ns2} that
\begin{align*}
\begin{aligned}
|Y^s_{i1}(\uu)| +|Y^s_{i2}(\uu)|  
&\le C\int_{\Omega^c}  |(a_i)_x|  \Big( Q(\uv|\bar v) +  |\uv-\bar v| \Big) dx \\
&\le C_{\delta_1} |\s_i|^{1/2} \delta_*^{3/4}\sqrt{D_i +\widetilde G_{i} }  + C_{\delta_1}\frac{\nu |\s_i|^{3/2}}{\delta_*}\sqrt{D +\widetilde G } .
\end{aligned}
\end{align*}
and so,
\[
|Y^s_{i1}(\uu)|^2 +|Y^s_{i2}(\uu)|^2 \le  C_{\delta_1} |\s_i|\delta_*^{3/2}\Big(D_i(t) +\widetilde G_i(t)\Big)+ C_{\delta_1}\frac{\nu^2 |\s_i|^3}{\delta_*}\Big(D(U) +\widetilde G(t)\Big).
\]
Hence,
\begin{align*}
\begin{aligned}
& |Y^g_i(U)- Y^g_i(\uu)|^2+ |Y^b_i(U)|^2 + |Y^l_i(U)|^2 + |Y^s_i(U)|^2 \\
&\le 2 \Big( |Y^g_i(U)- Y^g_i(\uu)| + |Y^s_{i1}(U)-Y^s_{i1}(\uu)| + |Y^s_{i2}(U)-Y^s_{i2}(\uu)| + |Y^s_{i3}(U)| +|Y^b_i(U)| \Big)^2\\
&\quad + 2 |Y^s_{i4}(U)|^2 + |Y^l_i(U)|^2 +2 |Y^s_{i1}(\uu)|^2 + 2 |Y^s_{i2}(\uu)|^2 \\
&\le C_{\delta_1} \ds |\s_i| \bigg[ \sqrt{\delta_*}\Big(D_i(t) +\widetilde G_i(t)\Big) +  \frac{\nu|\s_i|}{\delta_*^2} \Big(D(U) +\widetilde G(t)\Big) + \big(G_{2i}(U)-G_{2i}(\uu) \big) + G_{1i}^-(U) \\
&\quad+ \ds^{-1/4} G_{1i}^+(U) + \ds^{1/4} G_{2i}(\uu)  \bigg].
 \end{aligned}
\end{align*}

\subsection{Proof of Proposition \ref{prop:smain}}
First, using \eqref{real-R} and \eqref{xyest} with \eqref{bad}, \eqref{good}, \eqref{ggd}, \eqref{bad0}, and choosing $\bar\delta=\delta_1$, we find that whenever $(-1)^{{k_i}-1} Y_i(U) \le \s_i^2$ for all  $i\in \mathcal{S}$,
\begin{align*}
\begin{aligned}
\mathcal{R}(U)
&\le \sum_{i\in \mathcal{S}} G_{Y_i}(U)+ (1+\ds) \sum_{i\in \mathcal{S}} \Big(|{B}_{1i}(U)| +|{B}_{2i}^+(U)|  \Big) +2\sum_{i\in \mathcal{S}}|{B}_{2i}^-(U)| + 2 |B_3(U)|  \\
&\quad
- \sum_{i\in \mathcal{S}} {G}_{1i}^{-}(U)- \sum_{i\in \mathcal{S}}{G}_{1i}^{+}(U)  - \sum_{i\in \mathcal{S}} {G}_{2i}(U)  - \Big(1-\sqrt\ds\Big) D (U),
\end{aligned}
\end{align*}
where
\begin{align*}
\begin{aligned}
G_{Y_i}(U)&:= -\frac{1}{\s_i^4} |Y_i(U)|^2 {\mathbf 1}_{\{0\le  (-1)^{{k_i}-1} Y_i(U)\le \s_i^2\}} 
 -  \frac{|\tilde\lambda_i|}{2 \s_i^2}   |Y_i(U)|^2 {\mathbf 1}_{\{ -\s_i^2  \le  (-1)^{{k_i}-1} Y_i(U)\le 0 \}} \\
 &\qquad -\frac{|\tilde\lambda_i| }{2} |Y_i(U)| {\mathbf 1}_{\{(-1)^{{k_i}-1} Y_i(U) \le -\s_i^2 \}} .
\end{aligned}
\end{align*}
By the condition \eqref{yicond} of Proposition \ref{prop:big}, we may use two different strategies, depending on the strength of the left-hand side of \eqref{yicond} as follows. For a fixed $t$, let
\[
I_* := \{ i \in  \mathcal{S}~|~ \sqrt{\delta_*} D_i(t) +  \frac{\nu|\s_i|}{\delta_*} D(U)  \le 2C^*|\s_i|\ds \} .
\]
Notice that $I_*=\emptyset$ if and only if $\sqrt{\delta_*} D_i(t) +  \frac{\nu|\s_i|}{\delta_*} D(U) > 2C^*\ds \max_{i\in \mathcal{S}} |\s_i|$.\\
Then, we split the above estimate into two parts: using $G_{Y_i}(U)\le 0$,
\[
\mathcal{R}(U) \le \mathcal{R}_*(U) + \mathcal{R}_*^c(U),
\]
where
\begin{align*}
\begin{aligned}
\mathcal{R}_*(U) &:= \sum_{i\in I_*} G_{Y_i}(U)+ (1+\ds) \sum_{i\in I_*} \Big(|{B}_{1i}(U)| +|{B}_{2i}^+(U)|  \Big) +2\sum_{i\in I_*}|{B}_{2i}^-(U)| + 2 |B_3(U)|  \\
&\quad
- \sum_{i\in I_*} {G}_{1i}^{-}(U)- \sum_{i\in I_*}{G}_{1i}^{+}(U)  - \sum_{i\in I_*} {G}_{2i}(U)  - \Big(1-\sqrt\ds\Big) D (U),\\
\mathcal{R}_*^c(U) &:= 2 \sum_{i\in I_*^c} \Big(|{B}_{1i}(U)| +|{B}_{2i}^+(U)| +2|{B}_{2i}^-(U)|  \Big) + 2 |B_3(U)|  \\
&\quad
- \sum_{i\in I_*^c} {G}_{1i}^{-}(U)- \sum_{i\in I_*^c}{G}_{1i}^{+}(U)  -\sum_{i\in I_*^c} {G}_{2i}(U)  - \Big(1-\sqrt\ds\Big) D (U).
\end{aligned}
\end{align*}

\noindent {\bf Estimate of $\mathcal{R}_*^c$ :}
We will apply Proposition \ref{prop:big} to
\begin{align*}
\begin{aligned}
\mathcal{R}_*^c(U) &\le I_B + \bar I_B  - \sum_{i\in I_*^c} {G}_{1i}^{-}(U)- \sum_{i\in I_*^c}{G}_{1i}^{+}(U)  - \sum_{i\in I_*^c} {G}_{2i}(U)  - \Big(1-\sqrt\ds\Big)D (U),
\end{aligned}
\end{align*}
where 
\begin{align*}
\begin{aligned}
I_B &:= 2 \sum_{i\in I_*^c} \Big( |B_{1i}(U)-B_{1i}(\uu)|  +|B_{2i}^+(U)-B_{2i}^+(\uu)| +2|{B}_{2i}^-(U)|  \Big) + 2 |B_3(U)| , \\
\bar I_B & := 2 \sum_{i\in I_*^c} \Big( |B_{1i}(\uu)| +  |B_{2i}^+(\uu)|\Big) . 
\end{aligned}
\end{align*}
Using \eqref{badprop} and $\sum_{i\in \mathcal{S}} \big( D_i +\widetilde G_i\big) \le C \big( D +\widetilde G\big)$, $L(0)\le \ds$, and taking $\ds\ll \deltao$ so that $C_{\delta_1} \ds^{\frac{1}{4}} \le \ds^{\frac{1}{8}}$, we have
\begin{align*}
\begin{aligned}
I_B &\le C\ds \sum_{i\in \mathcal{S}} \Big( G_{2i}(U)-G_{2i}(\uu) \Big) +C\ds^2 \sum_{i\in \mathcal{S}} G_{2i}(\uu)+ C_{\delta_1}\sqrt\ds  \big( D(U) +\widetilde G(t) \big) \\
& +C_{\delta_1} \sum_{i\in I_*^c} \bigg[\ds^{1/4}  \big( D_i (t) +\widetilde G_{i}(t)   \big) +   C_{\delta_1}|\s_i| \sqrt\nu  \big( D(U) +\widetilde G(t)  \big)+  ( C\delta_1 + C_{\delta_1} \delta_*^{\frac{1}{4}} )  G_{1i}^{-}(U)+ C_{\delta_1}|\s_i| \sqrt\nu \bigg] \\
&\le C\ds \sum_{i\in \mathcal{S}} \Big( G_{2i}(U)-G_{2i}(\uu) \Big)  +C\ds^2 \sum_{i\in \mathcal{S}} G_{2i}(\uu) \\
&\quad + C\ds^{\frac{1}{8}}  \big( D(U) +\widetilde G(t)  \big) +  ( C\delta_1 + C_{\delta_1} \delta_*^{\frac{1}{4}} ) \sum_{i\in I_*^c}G_{1i}^{-}(U) + C\sqrt\nu.
\end{aligned}
\end{align*}
In addition, by \eqref{roughbad},
\[
\bar I_B \le 2 \sum_{i\in I_*^c} C^*\ds|\s_i| \le   \sum_{i\in I_*^c} \Big(  \sqrt{\delta_*} D_i(t) +  \frac{\nu|\s_i|}{\delta_*} D(U)\Big) \le \sqrt{\delta_*} D(U). 
\]
Thus,
\beq\label{rcs}
\mathcal{R}_*^c(U)  \le C\ds \sum_{i\in I_*} \Big( G_{2i}(U)-G_{2i}(\uu) \Big)  +C\ds^2 \sum_{i\in I_*}  G_{2i}(\uu) -\frac{1}{2}\sum_{i\in I_*^c} {G}_{2i}(U) + C\ds^{\frac{1}{8}}  \widetilde G(t)+ C\sqrt\nu.
\eeq

\noindent {\bf Estimate of $\mathcal{R}_*$ :}
We here apply Lemma \ref{lem:inside} and Proposition \ref{prop:big}.
For that, we first observe from \eqref{yismall} of Lemma \ref{lem:small} that 
\begin{align*}
\begin{aligned}
-\frac{|\tilde\lambda_i| }{2} |Y_i(U)| {\mathbf 1}_{\{(-1)^{{k_i}-1} Y_i(U) \le -\s_i^2 \}} &\le -\frac{|\tilde\lambda_i| }{2} |Y_i(U)| {\mathbf 1}_{\{-C_0|\s_i|\ds \le (-1)^{{k_i}-1} Y_i(U) \le -\s_i^2 \}} \\
&\le - \frac{|\tilde\lambda_i| }{2C_0|\s_i|\ds} |Y_i(U)|^2 {\mathbf 1}_{\{-C_0|\s_i|\ds \le (-1)^{{k_i}-1} Y_i(U) \le -\s_i^2 \}} .
\end{aligned}
\end{align*}
Then, by $|\s_i|\le L(0) \le \ds\ll \delta_1$,
\[
G_{Y_i}(U)\le -\frac{4}{|\s_i|\delta_1} |Y_i(U)|^2 {\mathbf 1}_{\{-C_0|\s_i|\ds\le  (-1)^{{k_i}-1} Y_i(U)\le \s_i^2\}} .
\]
Since it follows from \eqref{defyi}  that for each $i\in \mathcal{S}$,
$$
Y_i(U)=\mathcal{Y}^g_i(\uv) + \big(Y^g_i(U)-Y^g_i(\uu)\big) +\big(Y^g_i(\uu)-\mathcal{Y}^g_i(\uv)\big) + Y^b_i(U) +Y^l_i(U)+Y^s_i(U),
$$
we have
\begin{align*}
|\mathcal{Y}^g_i(\uv)|^2 &\le 4\Big(|Y_i(U)|^2+\big|Y^g_i(U)-Y^g_i(\uu)\big|^2+ \big|Y^g_i(\uu)-\mathcal{Y}^g_i(\uv)\big|^2 \\
&\quad+ | Y^b_i(U)|^2+| Y^l_i(U)|^2+| Y^s_i(U)|^2\Big),
\end{align*}
and so,
\begin{align*}
\begin{aligned}
-4|Y_i(U)|^2 &\leq -|\mathcal{Y}^g_i(\uv)|^2+4\Big(\big|Y^g_i(U)-Y^g_i(\uu)\big|^2+ \big|Y^g_i(\uu)-\mathcal{Y}^g_i(\uv)\big|^2 \\
&\quad+ | Y^b_i(U)|^2+| Y^l_i(U)|^2+| Y^s_i(U)|^2\Big).
\end{aligned}
\end{align*}
Thus, 
\begin{align*}
\begin{aligned}
G_{Y_i}(U)& \le -\frac{|\mathcal{Y}^g_i(\uv)|^2}{|\s_i|\deltao} 
+ \frac{4}{|\s_i|\deltao} \Big(\big|Y^g_i(U)-Y^g_i(\uu)\big|^2+ \big|Y^g_i(\uu)-\mathcal{Y}^g_i(\uv)\big|^2 \\
&\quad + | Y^b_i(U)|^2+| Y^l_i(U)|^2+| Y^s_i(U)|^2\Big).
\end{aligned}
\end{align*}
We also find from \eqref{error-in} that
\begin{align*}
\begin{aligned}
- {G}_{2i}(U)&= - \Big({G}_{2i}(U)-  {G}_{2i}(\uu)\Big) -\ds \deltao {G}_{2i}(\uu) -(1-\deltao\ds) \Big({G}_{2i}(\uu)-\mathcal{G}_{2i} (\uv) \Big) -(1-\deltao\ds)\mathcal{G}_{2i} (\uv)
\end{aligned}
\end{align*}
and by \eqref{amonod},
\begin{align*}
\begin{aligned}
- \big(1-\sqrt\ds \big) D (U) &\le  - \big(1-\deltao/2 \big) D (\uu) - \big(\deltao/2-\sqrt\ds \big) D (U) \\
&\le - \big(1-\deltao \big) \sum_{i\in \mathcal{S}}\mathcal{D}_{i} (\uv) +  \frac{C\nu^{4/3}}{\delta\ds}  \int_\bbr a \eta(U|\overline U) dx - \big(\deltao/2 -\sqrt\ds \big) D (U).
\end{aligned}
\end{align*}
Let 
\begin{align*}
\begin{aligned}
&J_Y:=  \sum_{i\in I_*} \frac{4}{|\s_i|\deltao} \Big(\big|Y^g_i(U)-Y^g_i(\uu)\big|^2+ \big|Y^g_i(\uu)-\mathcal{Y}^g_i(\uv)\big|^2 + | Y^b_i(U)|^2+| Y^l_i(U)|^2+| Y^s_i(U)|^2\Big),\\
&J_B:=2 \sum_{i\in I_*} \Big( |B_{1i}(U)-B_{1i}(\uu)|+ |{B}_{1i}(\uu)-\mathcal{I}_{1i}(\uv)| \Big)\\
&\quad\quad+ 2 \sum_{i\in I_*} \Big( |B_{2i}^+(U)-B_{2i}^+(\uu)| + |{B}_{2i}^+(\uu)-\mathcal{I}_{2i}(\uv)| \Big)+2\sum_{i\in I_*}|{B}_{2i}^-(U)| + 2 |B_3(U)| \\
&\quad\quad -(1-\deltao\ds) \sum_{i\in I_*} \Big({G}_{2i}(\uu)-\mathcal{G}_{2i} (\uv) \Big) +  \frac{C\nu^{4/3}}{\delta\ds}  \int_\bbr a \eta(U|\overline U) dx.
\end{aligned}
\end{align*}
Then, using the above estimates,
\begin{align*}
\begin{aligned}
\mathcal{R}_*(U) &\le\sum_{i\in I_*} \mathcal{R}^i_{\deltao,\ds}(\uv) + J_Y + J_B - \sum_{i\in I_*} {G}_{1i}^{-}(U)- \sum_{i\in I_*}{G}_{1i}^{+}(U) \\
& \quad -\sum_{i\in I_*} \Big({G}_{2i}(U)-  {G}_{2i}(\uu)\Big) -\ds \deltao \sum_{i\in I_*}{G}_{2i}(\uu) - \big(\deltao/2 -\sqrt\ds \big) D (U).
\end{aligned}
\end{align*}
By Lemma \ref{lem:inside} with Lemma \ref{lem:essy}, we have
\begin{align*}
\begin{aligned}
\mathcal{R}^i_{\deltao,\ds}(\uv)&=-\frac{1}{|\s_i|\deltao}|\mathcal{Y}^g_i(\uv)|^2 +(1+\deltao) |\mathcal{I}_{1i}(\uv)|\\
&\quad\quad+(1+\deltao \ds) |\mathcal{I}_{2i}(\uv)|-\left(1-\deltao\ds \right)\mathcal{G}_{2i}(\uv)-(1-\deltao)\mathcal{D}_i(\uv)\\
&\le 0.
\end{aligned}
\end{align*}
We use  Proposition \ref{prop:big} to have 
\begin{align*}
\begin{aligned}
J_Y &\le   
 C_{\deltao}\frac{\ds}{\deltao}  \sum_{i\in I_*}\bigg[ \sqrt{\delta_*}\Big(D_i(t) +\widetilde G_i(t)\Big) +  \frac{\nu|\s_i|}{\delta_*^2} \Big(D(U) +\widetilde G(t)\Big) \\
&\quad+ \big(G_{2i}(U)-G_{2i}(\uu) \big) + G_{1i}^-(U) 
+ \ds^{-1/4} G_{1i}^+(U) + \ds^{1/4} G_{2i}(\uu)  \bigg] + C\frac{\nu^2}{\deltao}  \sum_{i\in I_*}|\s_i|  .
\end{aligned}
\end{align*}
and 
\begin{align*}
\begin{aligned}
J_B &\le  C\ds^2 \sum_{i\in \mathcal{S}} G_{2i}(U) + C_{\deltao} \sum_{i\in I_*}\bigg[ \ds  (G_{2i}(U)-G_{2i}(\uu))  + \ds^{\frac14}  \big( D_i(t) +\widetilde G_i(t) \big) \\
&\quad + |\s_i| \sqrt\nu \big( D(U) +\widetilde G(t) \big) +( C\delta_1 + C_{\delta_1} \delta_*^{\frac{1}{4}} )  G_{1i}^{-}(U) \bigg] +C_{\delta_1}\sqrt\nu \sum_{i\in I_*}|\s_i| \\
&\quad +\sum_{i\in I_*} \frac{C\sqrt\nu}{\ds\delta} +  \frac{C\nu^{4/3}}{\delta\ds}  \int_\bbr a \eta(U|\overline U) dx.
\end{aligned}
\end{align*}
Thus, retaking $\ds\ll\deltao$ so that $C_{\deltao}\ds/\deltao \le \ds^{\frac78}$, and using Lemma \ref{lem:num},
\begin{align*}
\begin{aligned}
\mathcal{R}_*(U) &\le  -\big(1-\sqrt\delta_*\big)\sum_{i\in I_*}\Big({G}_{2i}(U)-  {G}_{2i}(\uu)\Big) -\ds \big(\deltao-C\ds^{\frac18}\big)\sum_{i\in I_*} {G}_{2i}(\uu) + C\ds^2 \sum_{i\in I_*^c} G_{2i}(U)\\
& \quad  - \big(\deltao-C\ds^{\frac18} \big) D (U) + C\ds^{\frac18}\widetilde G(t) + C\sqrt\nu  + \frac{C\nu^{\frac14}}{\ds\delta^2} +  \frac{C\nu^{4/3}}{\delta\ds} \int_\bbr a \eta(U|\overline U) dx.
\end{aligned}
\end{align*}

\noindent {\bf Conclusion :}
The above estimate and \eqref{rcs} imply
\begin{align*}
\begin{aligned}
\mathcal{R}(U) &\le   -\ds \big(\deltao-C\ds^{\frac18}\big)\sum_{i\in I_*} {G}_{2i}(U) -\frac{1}{4}\sum_{i\in I_*^c} {G}_{2i}(U) + C\ds^{\frac18}\widetilde G(t) \\
&+\mathcal{C} (\delta, \nu)+  \frac{C\nu^{4/3}}{\delta\ds}  \int_\bbr a \eta(U|\overline U) dx,
\end{aligned}
\end{align*}
where $\mathcal{C} (\delta, \nu)$ is the constant that vanishes when $\nu\to0$ for any fixed $\delta>0$. \\
For $\widetilde G(t)$, we use \eqref{avr}, $Q(v|\bar v)\le C p(v|\bar v)$ for all $v>0$, 
to have
\begin{align}
\begin{aligned} \label{tilgt}
\widetilde G (t) &=\int_\bbr |(\bar v)_x| Q(v|\bar v)\,dx \\
&\le \sum_{i\in\mathcal{S}} \int_\bbr |(\tilde v_i)_x| Q(v|\bar v) dx + \sum_{i\in\mathcal{R}} \int_\bbr |(v^R_i)_x| Q(v|\bar v) dx + \sum_{i\in\mathcal{NP}} \int_\bbr |(v^P_i)_x| Q(v|\bar v) dx \\
&\le C\ds\sum_{i\in\mathcal{S}} {G}_{2i}(U) + C \sum_{i\in\mathcal{R}} \int_\bbr |(v^R_i)_x| p(v|\bar v) dx + \sum_{i\in\mathcal{NP}} \int_\bbr |(v^P_i)_x| Q(v|\bar v) dx.
\end{aligned}
\end{align}
Therefore, 
\begin{align*}
\begin{aligned}
\mathcal{R}(U) 
&\le C\ds^{\frac{1}{8}}  \sum_{i\in\mathcal{R}} \int_\bbr |(v^R_i)_x| p(v|\bar v) dx  +C\ds^{\frac18}  \sum_{i\in\mathcal{NP}} \int_\bbr |(v^P_i)_x| Q(v|\bar v) dx \\
&\quad +\mathcal{C} (\delta, \nu)+  \frac{C\nu^{4/3}}{\delta\ds}  \int_\bbr a \eta(U|\overline U) dx .
\end{aligned}
\end{align*}

\section{Proof of Theorem \ref{thm:uniform}} \label{sec:pfthm}
\setcounter{equation}{0}

We here complete the proof of Theorem \ref{thm:uniform}.

\subsection{The decay of weight $a(t,x)$ at interaction time}\label{subsec-time}
By \eqref{defweight1}, we have that 
for each interaction time $t_j$, 
\[
a^\nu(x,t_j+)-a^\nu(x,t_j-) = \frac{1}{\ds} \bigg[ \Big( \Delta \bar L_j+\kappa \Delta \bar Q_j \Big) + \int_{-\infty}^x \sum_{i\in \mathcal{S}_j} \partial_x \Big( -p(\tilde v_{i}+) +p(\tilde v_{i}-) \Big) dy \bigg],
\]
where $\Delta \bar L_j := L(t_j+)-L(t_j-)$, $ \Delta \bar Q_j := Q(t_j+)-Q(t_j-)$, and
\[
\tilde v_{i}+:= \tilde v_{ij}^\nu(x,t_j+),\quad \tilde v_{i}-:=\tilde v_{i,j-1}^\nu(x,t_j-).
\]
Let $i_0$ and $i_0+1$ be two adjacent indices, for which only two waves moving along the trajectories $y_{i_0}$ and $y_{i_0+1}$ interact at the time $t_j$, that is, no interaction of waves along other trajectories $y_i$ for $i\notin\{i_0, i_0+1\}$. We still use the same indices $i_0$ and $i_0+1$ for the two outgoing waves generated by the interaction. \\
Then, since no interaction of waves along other trajectories $y_i$ for $i\notin\{i_0, i_0+1\}$, we have

\beq\label{shsumz}
\int_{-\infty}^x \sum_{i\in \mathcal{S}_j\cap\{i_0, i_0+1\}^c} \partial_x \Big( -p(\tilde v_{i}+) +p(\tilde v_{i}-) \Big) dy =0,
\eeq
and so,
\beq \label{aequal}
a^\nu(x,t_j+)-a^\nu(x,t_j-) = \frac{1}{\ds} \bigg[\Big( \Delta \bar L_j+\kappa \Delta \bar Q_j \Big)+ \int_{-\infty}^x \Big( g_+^S(y) -g_-^S(y) \Big) dy \bigg],
\eeq
where
\begin{align*}
\begin{aligned}
& g_+^S(y) := \sum_{i\in \mathcal{S}_j\cap\{i_0, i_0+1\}} \partial_x \Big( -p(\tilde v_{i}+) \Big),\\
 & g_-^S(y) := \sum_{i\in \mathcal{S}_j\cap\{i_0, i_0+1\}} \partial_x \Big( -p(\tilde v_{i}-) \Big).
\end{aligned}
\end{align*}
Notice that $\partial_x \big( -p(\tilde v_{i}) \big)$ and $\partial_x \tilde v_i$ have the same sign, and so we find that $\partial_x \Big( -p(\tilde v_i^\nu(x,t)) \Big)<0$ for $k_{ij}=1$, whereas $\partial_x \Big( -p(\tilde v_i^\nu(x,t)) \Big)>0$ for $k_{ij}=2$. 
To reflect this monotonicity,  in what follows, we sometimes use the following notations for simplicity:
\beq\label{barsi-0}
\tilde\s_i := 
\left\{ \begin{array}{ll}
        \int_{-\infty}^\infty \partial_x \big(-  p(\tilde v_i) \big) dx,\quad\mbox{if } i\in \mathcal{S}_j,  \\
       \int_{-\infty}^\infty \partial_x \big(-  p(v^R_i) \big)  dx,\quad\mbox{if } i\in \mathcal{R}_j, \\
        \int_{-\infty}^\infty \partial_x\big(-  p(v^P_i) \big) dx,\quad\mbox{if } i\in \mathcal{NP}_j . \end{array} \right.
\eeq
Note that $\tilde\s_i = (-1)^{k_{ij}+1} \bar\s_i$ for the sizes $\bar\s_i$ defined in \eqref{barsi}.  Also, we use the notations: $\bar\s', \bar\s''$ for the sizes of outgoing waves,  $\bar\s_1, \bar\s_2$ for the sizes of incoming waves as in \eqref{barsi}; and as the associated monotonic functions defined in \eqref{barsi-0}, use the notations: $\tilde\s', \tilde\s''$ for the outgoing waves; $\tilde\s_1, \tilde\s_2$ for the incoming waves.\\

\begin{figure}
	\centering
		\includegraphics
	[scale=.6]
	{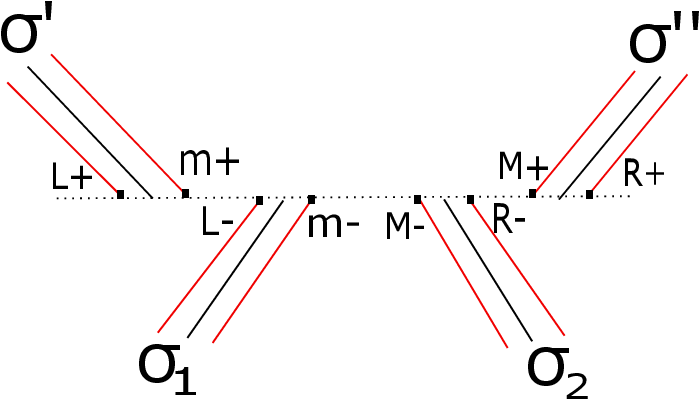}
		\caption{$|x_{m\pm}-x_{L\pm}|=|x_{R\pm}-x_{M\pm}|=\sqrt{\nu}$}\label{Fig1}
\end{figure} 
\begin{figure}
	\centering
		\includegraphics
	[scale=.6]
	{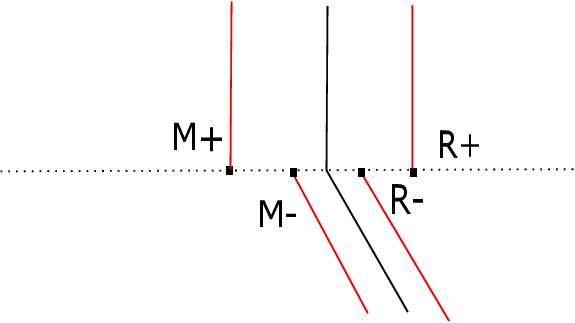}
		\caption{$|x_{R-}-x_{M-}|=\sqrt{\nu}$, \quad$|x_{R+}-x_{M+}|=\sqrt{\rho_\nu}$}\label{Fig2}
\end{figure} 

To estimate the right-hand side of \eqref{aequal}, we will use the modified front tracking algorithm as before. For that, we construct approximations of $\int_{-\infty}^x g_\pm^S(y) dy$ by discretizing them as follows. As in Figures \ref{Fig1} and \ref{Fig2}, let $x_{L\pm}, x_{m\pm}, x_{M\pm}, x_{R\pm}$ be the boundary points of the zones of width $\sqrt\nu$ scale for physical waves, and of width $\rho_\nu$ scale for nonphysical waves. In addition, $v_L$ and $v_R$ be respectively the left and right states of the (inviscid) Riemann solvers as in the proof of Proposition \ref{prop:deca}.    Note that $p(v_R)=p(v_L)+ \tilde\s'+\tilde\s'' = \tilde\s_1+\tilde\s_2$.
We now define an approximation of $ \int_{-\infty}^x  g_+^S(y) dy$ by 
\beq\label{app-ome}
\Omega^S_+(x) :=
\left\{ \begin{array}{ll}
       p(v_L),\quad\mbox{if } x\le x_{L+},  \\
       p(v_L),\quad\mbox{if } x_{L+}< x \le x_{M+} \mbox{, and } i_0\notin \mathcal{S}_j  \\
        \int_{x_{L+}}^x \partial_x \Big( -p(\tilde v_{i_0}+) \Big) dy  \mathbf{1}_{x_{L+}< x<x_{m+}}  + \big( p(v_L) + \tilde \s' \big) \mathbf{1}_{x_{m+}\le x\le x_{M+}}  ,  \quad\mbox{if }  i_0\in \mathcal{S}_j  \\
        p(v_L) ,\quad\mbox{if } x > x_{M+} \mbox{, and } i_0\notin \mathcal{S}_j  \mbox{, and } i_0+1\notin \mathcal{S}_j \\
        p(v_L)+ \tilde \s'  ,\quad\mbox{if } x > x_{M+} \mbox{, and } i_0\in \mathcal{S}_j  \mbox{, and } i_0+1\notin \mathcal{S}_j \\
         \int_{x_{M+}}^x \partial_x \Big( -p(\tilde v_{i_0+1}+) \Big) dy  \mathbf{1}_{x_{M+}< x<x_{R+}}  + \big( p(v_R) - \tilde \s' \big)  \mathbf{1}_{x\ge x_{R+}} , \mbox{if }  i_0\notin \mathcal{S}_j ,\, i_0+1\in \mathcal{S}_j  \\
          \int_{x_{M+}}^x \partial_x \Big( -p(\tilde v_{i_0+1}+) \Big) dy  \mathbf{1}_{x_{M+}< x<x_{R+}}  +p(v_R)  \mathbf{1}_{x\ge x_{R+}} , \mbox{if }  i_0\in \mathcal{S}_j ,\, i_0+1\in \mathcal{S}_j . \end{array} \right.
\eeq
Notice that the above approximation is constructed by discretizing outside each transition zone. 
Likewise, we define $\Omega^S_-(x)$  approximation of $ \int_{-\infty}^x  g_-^S(y) dy$. Notice that
\[
\Omega^S_-(x_{L-})=\Omega^S_+(x_{L+}).
\]

Notice that since $|x_{m\pm}-x_{L\pm}|=|x_{R\pm}-x_{m\pm}|=\sqrt\nu \gg O(\nu) :$ the width of transition zone of viscous shocks, we have
\beq\label{errortail}
\Big|\Omega^S_\pm(x) -  \int_{-\infty}^x  g_\pm^S(y) dy \Big| \le C\exp{(-\nu^{-1/2})} ,\quad\forall x\in\bbr.
\eeq
We will show the negativity of $\Delta \bar L_j+\kappa \Delta \bar Q_j + \Delta\Omega^S(x)$ for all $x$ where 
\[
\Delta\Omega^S(x):= \Omega^S_+(x) - \Omega^S_-(x) .
\]

\begin{proposition}\label{prop:deca}
There exists $\eps_0>0$ such that for all $\eps<\eps_0$, the following holds: there exists $\kappa_0>0$ such that for all $\kappa>\kappa_0$,
\beq\label{bbyfr}
  \Delta \bar L_j+\kappa \Delta \bar Q_j + \Delta\Omega^S(x) <0,\quad \forall x\in\bbr.
\eeq
\end{proposition}

\begin{figure}
	\centering
		\includegraphics
	[scale=.4]
	{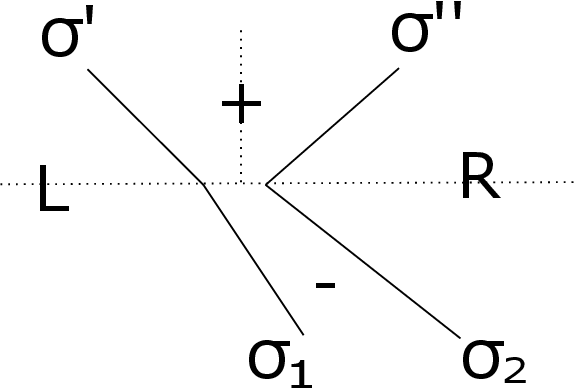}\hfill
		\includegraphics
	[scale=.4]
	{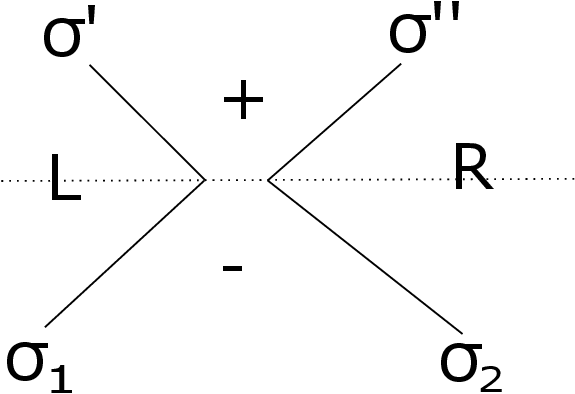}\hfill
		\includegraphics
	[scale=.4]
	{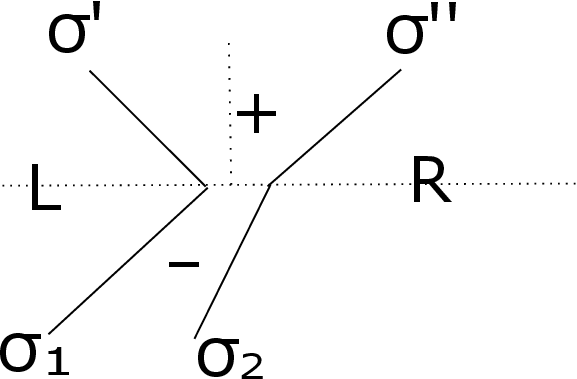}
		\caption{Decay of $a(t,x)$ at interaction point $(t,x)$ (accurate and adjusted solvers). Dot line represents possible non-physical shock.  $t_+$, $t_-$ denote the time before and after interaction, $x_L,\ x_m,\ x_R$ denote the point to the left, between and right of two incoming waves, respectively. For simplicity, we do not add bar on $\sigma$. All waves have width.}\label{pic_a}
\end{figure} 

\begin{figure}
			\includegraphics
	[scale=.38]
	{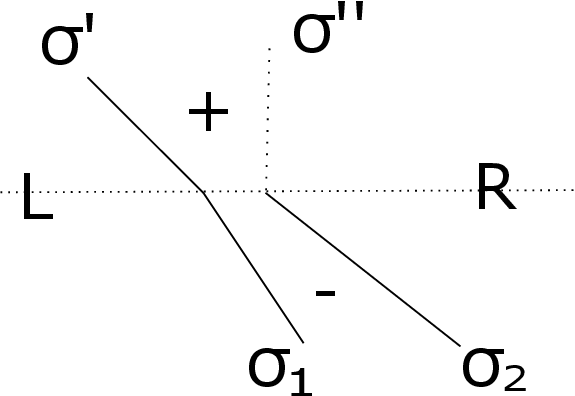}
					\includegraphics
	[scale=.38]
	{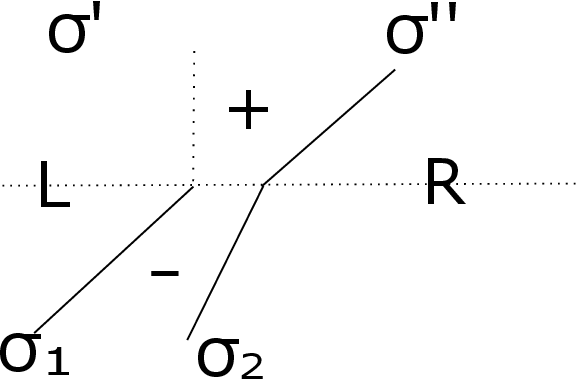}
		\includegraphics
	[scale=.38]
	{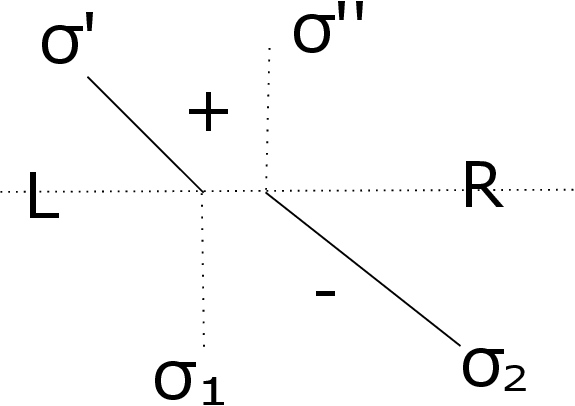}
		\includegraphics
	[scale=.38]
	{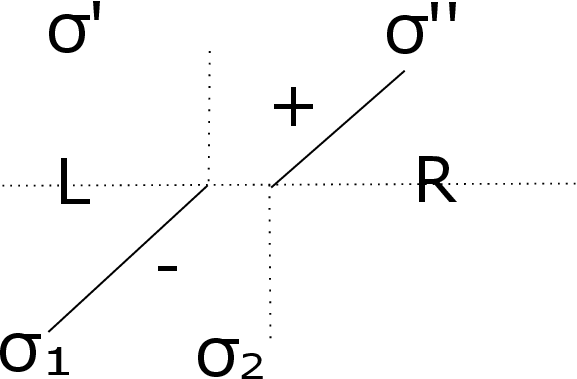}
	\caption{Decay of $a(t,x)$ at interaction point $(t,x)$ (simplified solver). Dot line represents non-physical shock.  $t_+$, $t_-$ denote the time before and after interaction, $x_L,\ x_m,\ x_R$ denote the point to the left, between and right of two incoming waves, respectively. For simplicity, we do not add bar on $\sigma$. All waves have width.} \label{pic_a2}
\end{figure} 

\vskip0.2cm

\begin{figure}
	\centering
		\includegraphics
	[scale=.5]
	{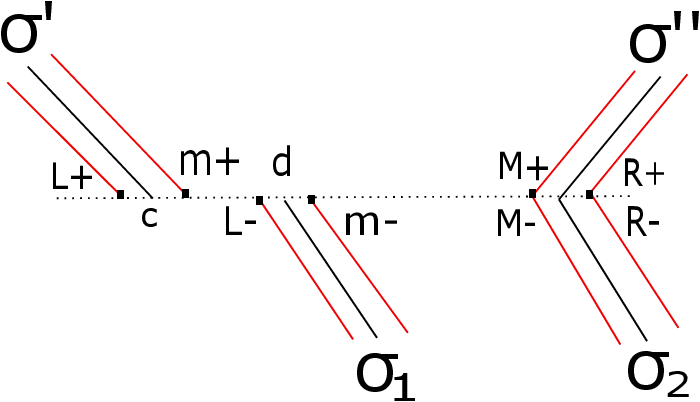}\qquad\qquad
			\includegraphics
	[scale=.5]
	{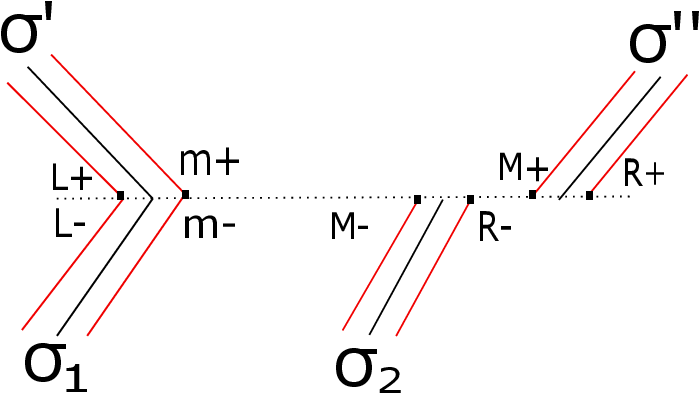}
		\caption{Left: When $\sigma_1$ and $\sigma'$ has same sign, the wave for $\sigma'$ is to the left of the wave for $\sigma_1$, in the sense that $x_{m-}-x_{L-}=x_{L-}-x_{m+}=x_{m+}-x_{L+}=\sqrt{\nu}$. The wave for $\sigma''$ can be a forward wave or a non-physical shock.  Right: A symmetric forward interaction.}\label{pic_a_0}
\end{figure} 
\vskip0.1cm

\begin{proof}

%
%
%
We will show the desired estimates when there is an interaction in Figure \ref{pic_a} or \ref{pic_a2} at $t_j$.

\paragraph{\bf Step 1} First, as before, we have the following decreasing property of $ \Delta \bar L_j+\kappa \Delta \bar Q_j $. The proof is quite standard, and will be given in the appendix.
There exists $\kappa_0>0$ such that for all $\kappa>\kappa_0$,
\beq\label{qdec}
 \Delta \bar Q_j  \le -\frac{3}{4} |\bar\s_1| |\bar\s_2| ,
\eeq
\beq\label{bbyfr1}
 \Delta \bar L_j+\kappa \Delta \bar Q_j  \le -\frac{\kappa}{2} |\bar\s_1| |\bar\s_2|.
\eeq

\paragraph{\bf Step 2}
We will show  that for all $x \in\bbr$,
\beq\label{adecay}
\hbox{either}\quad \Delta \bar L_j + \Delta\Omega^S(x) \leq O(|\bar\sigma_1| |\bar\sigma_2|) \quad \hbox{or}\quad  \Delta\Omega^S(x) \leq O(|\bar\sigma_1| |\bar\sigma_2|)\quad\mbox{holds} .
\eeq
Then, using \eqref{qdec} and \eqref{adecay} with retaking $\kappa_0\gg1$ if necessary, we have  \eqref{bbyfr}.\\
To prove \eqref{adecay}, we will use the rule on configuration of waves as in Figure \ref{pic_a} and Figure \ref{pic_a2}, except for the following adjustment:\\

{\emph{If $\sigma_1$ and $\sigma'$ have the same sign for the first figure of Figure \ref{pic_a}and for the first figure of Figure \ref{pic_a2}, 
then the wave $\sigma_1$ is to the right of the wave $\sigma'$ as in Figure \ref{pic_a_0}.}}

{\emph{If $\sigma''$ and $\sigma_2$ have the same sign for the last figure of Figure \ref{pic_a} and for the second figure of Figure \ref{pic_a2},
then the wave $\sigma''$ is to the right of the wave $\sigma_2$, see Figure \ref{pic_a_0}.}} \\

For example, see the interaction in Figure \ref{pic_a_0}.

For other cases, regardless of the relative positions of waves $\sigma'$ and $\sigma_1$,  and of waves $\sigma''$ and $\sigma_2$, we have the same result \eqref{adecay}. So, for the other cases, we will identify the centers of incoming and outgoing waves, that is, for example, $x_{L-}=x_{L+}, x_{m-}=x_{m+}, x_{M-}=x_{M+}, x_{R-}=x_{R+}$ for the head-on interaction. 
\\

We will prove \eqref{adecay} for each Riemann solver. We always assume two interacting wave fronts are $\bar \sigma_1$ $\bar \sigma_2$ ($\bar \sigma_1$ to the left) leading to outgoing waves $\bar \sigma\rq{}$, $\bar \sigma\rq{}\rq{}$ ($\bar \sigma\rq{}$ to the left). See Figure  \ref{pic_a} and \ref{pic_a2}.

We will crucially use the following monotonicity property of $ \Omega^S_\pm(x)$:
\vskip0.1cm
{\emph{$\Omega^S_+(x)$ increases as it passes $\fa S$ on $[x_{M+}, x_{R+}]$, and decreases as it passes $\ba S$ on $[x_{L+}, x_{m+}]$, but is constant on $(-\infty, x_{L+}]$, $[x_{m+}, x_{M+}]$, $[x_{M+}, \infty)$ and as it passes rarefaction or non-physical shock. This property is also satisfied by $\Omega^S_-(x)$. }}

\vskip0.1cm

\vskip0.3cm
\paragraph{\bf (i) Accurate solvers}
We only prove \eqref{adecay} up to the middle state of the intersection, i.e. for $x\le x_{M\pm}$, since the proof of  \eqref{adecay} on the right state for $x> x_{M\pm}$ is symmetric as in \cite{CKV20}.

\vskip0.1cm

First, consider all cases with $\bar\sigma_1>0$ (i.e. incoming rarefaction). Then, 
\[
\Omega^S_-(x) = \Omega^S_-(x_{L-})=\Omega^S_+(x_{L+}), \quad\forall x\le x_{M-}=x_{M+}.
\]
But, since $\bar\s'$ is always 1-family as in Figure \ref{pic_a}, the monotonicity property implies that
\beq\label{deplus}
\Omega^S_+(x) \le \Omega^S_+(x_{L+}) \quad\forall x\le x_{M\pm}.
\eeq
Thus, \eqref{adecay} holds as
\beq\label{negholds}
 \Delta\Omega^S(x) \le 0, \quad\forall x\le x_{M\pm}.
\eeq

On the other hand, for the head-on interaction (as the second figure in Figure \ref{pic_a}) with $\bar\sigma_1<0$ (i.e. incoming $2-$shock), we have  
\[
\Omega^S_-(x) > \Omega^S_\pm(x_{L\pm}), \quad\forall x\le x_{M\pm}.
\]
This with \eqref{deplus} implies \eqref{negholds}.\\

%
%
\vskip0.1cm


For the overtaking interactions, it is enough to consider  the interactions between two backward waves with $\bar\sigma_1<0$ as in the left picture of Figure  \ref{pic_a}, since the case of forward waves can be proved by a  symmetric argument.\\ 
For $\ba S\ba S$ interaction, by Lemma \ref{prop1}, the strength of outgoing $\ba S$ is bigger than that of any incoming shocks, so \eqref{negholds} holds.\\
We now consider the case: $\ba S \ba R$, i.e. $\bar \sigma_1<0$ and $\bar \sigma_2>0$ as follows. If $\bar \sigma'> 0$, it follows from Lemma \ref{prop1} that 
$|\bar \s' -( \bar \sigma_1+\bar \sigma_2)|+|\bar\s''|\le C|\bar \sigma_1\bar \sigma_2|\ll1$, and so $\bar\s_1+\bar\s_2>0$ and
\[
\Delta \bar L = |\bar\s'|+|\bar\s''|  - (|\bar\s_1|+|\bar\s_2| )  \le |\bar\s_1+\bar\s_2| + C|\bar \sigma_1\bar \sigma_2| - (|\bar\s_1|+|\bar\s_2| ) \le 2 \bar\s_1 +  C|\bar \sigma_1\bar \sigma_2| <0.
\]
Since  $\Omega^S_+(x) = \Omega^S_+(x_{L+})$ for all $x\le x_{M+}$, and  
\beq\label{shestm}
\Omega^S_-(x) \ge \Omega^S_-(x_{L-}) +\bar\s_1,\quad\forall x\le x_{M-}, 
\eeq
we have
$\Delta\Omega^S(x) \le -\bar\s_1$ for all $x\le x_{M\pm}$. Thus, \eqref{adecay} holds as
\[
\Delta \bar L+\Delta\Omega^S(x) \le \bar\s_1 +  C|\bar \sigma_1\bar \sigma_2| \le C|\bar \sigma_1\bar \sigma_2|.
\]
If $\bar \sigma\rq{}< 0$, we use the  adjustment on configuration of waves as in Figure  \ref{pic_a_0}. First, the monotonicity property implies
\beq\label{spinst}
\Omega^S_+(x) \le \Omega^S_\pm(x_{L\pm}) = \Omega^S_-(x),\quad\forall x\le x_{L-},
\eeq
and so $\Delta\Omega^S(x) \le 0$ for $x\le x_{L-}$.
For $x_{L-} < x\le x_{M-}$, observe that since \eqref{shestm} holds and $\Omega^S_+(x) = \Omega^S_+(x_{L+}) +\bar\s'$ for all $x_{L-} < x\le x_{M-}$, we have
\beq\label{omegaline}
\Delta\Omega^S(x) \le \bar\s' -\bar\s_1, \quad x_{L-} < x\le x_{M-}.
\eeq
To control it, we use the bound of $\Delta \bar L$ as follows. Using the same argument as above, we have from Lemma \ref{prop1} that  $\bar\s_1+\bar\s_2<0$ and 
\[ 
\Delta \bar L \le |\bar\s_1+\bar\s_2| + C|\bar \sigma_1\bar \sigma_2| - (|\bar\s_1|+|\bar\s_2| )  \le -2 \bar\s_2 +  C|\bar \sigma_1\bar \sigma_2| \le -2 (\bar\s' -\bar\s_1) +  C|\bar \sigma_1\bar \sigma_2|,
\]
which yields $\Delta \bar L+\Delta\Omega^S(x) \le  C|\bar \sigma_1\bar \sigma_2|$ for $x_{L-} < x\le x_{M-}$.

\vskip0.3cm
\paragraph{\bf (ii) Adjusted solver} To define the adjusted solver, our adjustment on wave strength from the accurate solver is up to the quadratic order, i.e. $O(\bar \sigma_1\bar \sigma_2)$. Clearly \eqref{adecay} still holds for the adjusted solver, since it does for the accurate solver.

\vskip0.3cm
\paragraph{\bf(iii) Simplified Solver}
We now consider simplified solvers in Figure \ref{pic_a2} as a reference. We will prove \eqref{adecay} for all $x$. 

For the interaction between a physical wave and a non-physical wave as in the right two pictures of Figure \ref{pic_a2}, we have
 \[
\Omega^S_-(x) \ge \Omega^S_-(x_{L-}) ,\quad\forall x\le x_{M-}, 
\]
and
 \[
\Omega^S_+(x) \le \Omega^S_+(x_{L+}) ,\quad\forall x\le x_{M+}.
\]
Indeed, if the wave $\bar\s_1$ is non-physical and  $\bar\s'$ is physical as 1-family, then
\[
\Omega^S_-(x) = \Omega^S_-(x_{L-}),\quad \Omega^S_+(x) \le \Omega^S_+(x_{L+}) ,\quad\forall x\le x_{M\pm}.
\] 
If $\bar\s_1$ is physical as 2-family and  $\bar\s'$ is non-physical, then
\[
\Omega^S_-(x) \ge \Omega^S_-(x_{L-}),\quad \Omega^S_+(x) = \Omega^S_+(x_{L+}) ,\quad\forall x\le x_{M\pm}.
\] 
Thus, $ \Delta\Omega^S(x) \le 0$ for all $ x\le x_{M\pm}$.\\
For $x>x_{M\pm}$, it is enough to consider the interaction of a shock and a non-physical wave, since $\Omega^S_\pm$ does not change on rarefaction.
In this case, as before, we use the fact that the shock strength changes in the quadratic order, that is, $|\bar\s'-\bar\s_2|=O(|\bar\sigma_1| |\bar\sigma_2|)$ for $\bar\s', \bar\s_2$ shocks; $|\bar\s''-\bar\s_1|=O(|\bar\sigma_1| |\bar\sigma_2|)$ for $\bar\s'', \bar\s_1$ shocks.
In addition, since for $ \bar\s', \bar\s_2$ shocks of 1-family, 
\[
\Omega^S_+(x) = \Omega^S_+(x_{L+}) +\bar\s',\quad\forall x > x_{M+},
\]
and $\Omega^S_-(x) \ge \Omega^S_-(x_{L-}) +\bar\s_2$ for all $x > x_{M-}$,
we have
\[
\Delta\Omega^S(x) \le \bar\s' - \bar\s_2 = O(|\bar\sigma_1| |\bar\sigma_2|),\quad\forall x > x_{M\pm}.
\]
Likewise, for $\bar\s'', \bar\s_1$ shocks of 2-family, since $\Omega^S_+(x) \le \Omega^S_+(x_{L+}) -\bar\s''$ for all $x > x_{M+}$, 
and $\Omega^S_-(x) = \Omega^S_-(x_{L-}) -\bar\s_1$ for all $x > x_{M-}$, we have
\[
\Delta\Omega^S(x) \le -\bar\s'' + \bar\s_1 = O(|\bar\sigma_1| |\bar\sigma_2|),\quad\forall x > x_{M\pm}.
\]
Thus, we show \eqref{adecay} for all $x$.

\vskip0.3cm

Next, we consider the interactions of two backward physical waves as in the first picture of Figure \ref{pic_a2} (see Figures \ref{Simp1}-\ref{Simp5} as references). 

For the $\ba S\ba S$ interaction as in Figure \ref{Simp1}, using the fact that $\bar\s' = \bar\s_1+\bar\s_2$, we have $\Delta\Omega^S(x) \le 0$ for all $x$.

For the $\ba S\ba R$ interaction as in Figure \ref{Simp2}, 
we use the  adjustment as in Figure \ref{pic_a_0}, since $\bar\s'<0, \bar\s_1<0$. 
First, $\Delta\Omega^S(x) \le 0$ for $x\le x_{L-}$ as in \eqref{spinst}. Also, as in \eqref{omegaline}, we have 
\[
\Delta\Omega^S(x) \le \bar\s' -\bar\s_1, \quad x_{L-} < x\le x_{M+}.
\]
To control it, we use the bound of $\Delta\bar L$ as
\[
\Delta\bar L < |\bar\s'| - |\bar\s_1| = \bar\s_1 - \bar\s'. 
\]
Indeed, recalling the fact that $|\s''|=r_R-r_b<r_R-r_a=\s_2$ and so, $z_a-z_R> z_b-z_R>0$, together with the fact that the pressure $p$ increases in $z$, 
we have $ 0<\bar\s'' = p_b-p_R < p_a -p_R =  \bar\s_2$. Thus, we have the above bound, and so $\Delta \bar L+\Delta\Omega^S(x) \le  0$ for $x_{L-} < x\le x_{M+}$.
Moreover, since $\Omega^S_-, \Omega^S_+$ do not change as they pass the rarefaction $\bar\s_2$ and the non-physical shock $\bar\s''$ respectively, the above estimate $\Delta\Omega^S(x) \le \bar\s' -\bar\s_1$ still holds for $ x> x_{M+}$. Thus, we prove \eqref{adecay} for all $x$.

\vskip0.1cm
For the $\ba S\ba R$ interaction as in Figure \ref{Simp3}, using  $\Delta\Omega^S(x) \le -\bar\s_1$ for all $x$, and $\Delta \bar L =2\bar\s_1$ as before,
we have
\[
\Delta \bar L +\Delta\Omega^S(x) \le \bar\s_1<0, \quad\forall x.
\]
Likewise, for the $\ba R\ba S$ interaction as in Figure \ref{Simp5}, since $\Delta\Omega^S(x) \le -\bar\s_2$ for all $x$, and $\Delta \bar L =2\bar\s_2$, we have
\[
\Delta \bar L +\Delta\Omega^S(x) \le \bar\s_2<0, \quad\forall x.
\]
For the $\ba R\ba S$ interaction in Figure \ref{Simp4}, we first have $\Delta\Omega^S(x)\le 0$ for all $x\le x_{M-}$ by the monotonicity property. Using the following estimates
\beq\label{app_est}
-\bar \sigma\rq{}+\bar \sigma_2\leq C|\bar \sigma_1\bar \sigma_2|,\qquad \bar \sigma\rq{}\rq{}-\bar \sigma_1\leq C|\bar \sigma_1\bar \sigma_2|,
\eeq
which will be proved in the appendix, we have $\Delta \bar L \le -\bar \sigma'+\bar \sigma_2 +C|\bar \sigma_1\bar \sigma_2|$. In addition, since
$\Delta\Omega^S(x)\le \bar\s' - \bar\s_2$ for all $x> x_{M-}$, we find $\Delta \bar L +\Delta\Omega^S(x)\le C|\bar \sigma_1\bar \sigma_2|$ for all $x> x_{M-}$.\\

For the interactions between two forward physical waves as in the second picture of Figure \ref{pic_a2}, the monotonicity property implies
\[
\Omega^S_-(x) = \Omega^S_-(x_{L-}) \le \Omega^S_+(x),\quad\forall x\le x_{M\pm}.
\]
The proof of \eqref{adecay} for $x> x_{M\pm}$ is symmetric to the above proof on the interactions between backward physical waves for $x< x_{m\pm}$. So we omit it.\\
Hence we complete the proof.
\end{proof}

\vskip0.3cm

\subsection{The estimate of weighted norm at the interaction time}
First, we use \eqref{aequal}, \eqref{errortail} and Proposition \ref{prop:deca} to have
\begin{align}
\begin{aligned} \label{aresult}
&a^\nu(x,t_j+)-a^\nu(x,t_j-) \\
&= \frac{1}{\ds} \bigg[\Big( \Delta \bar L_j+\kappa \Delta \bar Q_j +\Delta\Omega^S(x) \Big) -\left(\Delta\Omega^S(x) - \int_{-\infty}^x \Big( g_+^S(y) -g_-^S(y) \Big) dy \right) \bigg]\\
&\le \frac{1}{\ds}f(x)  + \frac{C}{\ds} \exp{(-\nu^{-1/2})},\\
&\mbox{where } f(x):= \Delta \bar L_j+\kappa \Delta \bar Q_j +\Delta\Omega^S(x) <0.
\end{aligned}
\end{align}
Now, we will estimate
 \begin{align*}
\begin{aligned}
\int_{\bbr} \Big( a^\nu(x,t_j+)\eta(U^\nu (x,t)|\overline U_{\nu,\delta}(x,t_j+) ) -a^\nu(x,t_j-) \eta(U^\nu(x,t)|\overline U_{\nu,\delta}(x,t_j-) ) \Big) dx. 
\end{aligned}
\end{align*}
Using \eqref{aresult} with putting $a_\pm := a^\nu(x,t_j\pm)$ and $ \overline U_\pm :=\overline U_{\nu,\delta}(x,t_j\pm)$, we have
 \begin{align*}
\begin{aligned}
&\int_{\bbr} \Big( a_+ \eta(U^\nu |\overline U_+ ) -  a_- \eta(U^\nu |\overline U_- ) \Big) dx\\
&= \int_{\bbr} \big(a_+ - a_-\big) \eta(U^\nu |\overline U_+ ) dx + \int_\bbr  a_- \Big(\eta(U^\nu |\overline U_+ ) - \eta(U^\nu |\overline U_- ) \Big) dx\\
&\le \frac{1}{\ds} \int_{\bbr} f(x) Q(v^\nu| \bar v_+) dx + \frac{C}{\ds} \exp{(-\nu^{-1/2})}  \int_{\bbr} Q(v^\nu| \bar v_+) dx \\
&\quad+ \frac{C}{\ds} \exp{(-\nu^{-1/2})}  \int_\bbr (|h^\nu- \bar h_-|^2 + |\bar h_- - \bar h_+|^2) dx  \\
&\quad +  \int_\bbr  a_- \Big(\nabla\eta(\overline U_-) -\nabla\eta(\overline U_+)  \Big) \cdot U^\nu dx + C \|\overline U_\pm\|_{L^\infty}   \int_\bbr |\overline U_+-\overline U_-|dx.
\end{aligned}
\end{align*}
Using the estimates that by Lemma \ref{lem:tri},
 \begin{align*}
\begin{aligned}
Q(v^\nu| \bar v_+) &= Q(v^\nu| \bar v_-) - Q(\bar v_+| \bar v_-) + \big( Q'(\bar v_+)-Q'(\bar v_-) \big) (\bar v_+ -v^\nu)\\
&\le Q(v^\nu| \bar v_-) + \big( p(\bar v_-)-p(\bar v_+) \big) (\bar v_+ -v^\nu),
\end{aligned}
\end{align*}
and
 \begin{align*}
\begin{aligned}
&\int_\bbr  a_- \Big(\nabla\eta(\overline U_-) -\nabla\eta(\overline U_+)  \Big) \cdot U^\nu dx\\
&=\int_\bbr  a_- \Big(p(\bar v_+) -p(\bar v_-) \Big)  v^\nu dx +  \int_\bbr  a_- \Big(\bar h_- - \bar h_+  \Big)  (h^\nu- \bar h_-) dx +  \int_\bbr  a_- \Big(\bar h_- - \bar h_+  \Big) \bar h_- dx \\
&\le \int_\bbr  a_- \Big(p(\bar v_+) -p(\bar v_-) \Big)  v^\nu dx +C\nu^{\frac{1}{12}}  \int_\bbr |h^\nu-\bar h_-|^2 dx +C\nu^{-\frac{1}{12}}  \int_\bbr ( |\bar h_- - \bar h_+|^2 + |\bar h_- - \bar h_+| ) dx,
\end{aligned}
\end{align*}
and using Lemma \ref{lem:ubdd} (presented at the end of this proof) with taking $\nu\ll1$ so that $ \frac{1}{\ds} \exp{(-\nu^{-1/2})} \le C\nu^{\frac{1}{12}}$, we have
 \begin{align*}
\begin{aligned}
&\int_{\bbr} \Big( a_+ \eta(U^\nu |\overline U_+ ) -  a_- \eta(U^\nu |\overline U_- ) \Big) dx\\
&\le \frac{1}{\ds} \int_{\bbr} f(x) Q(v^\nu| \bar v_+) dx + C\nu^{\frac{1}{12}}  \int_{\bbr} \Big(Q(v^\nu| \bar v_-) + \big( p(\bar v_+)-p(\bar v_-) \big) (v^\nu-\bar v_+) \Big)dx \\
&\quad+ \int_\bbr  a_- \Big(p(\bar v_+) -p(\bar v_-) \Big)  v^\nu dx +C\nu^{\frac{1}{12}}  \int_\bbr |h^\nu- \bar h_-|^2 dx  + C\nu^{\frac{1}{12}}  \\
&\le   \frac{1}{\ds} \int_{\bbr} f(x) Q(v^\nu| \bar v_+) dx  +  \int_\bbr  \Big( a_- + C\nu^{\frac{1}{12}} \Big) \Big(p(\bar v_+) -p(\bar v_-) \Big)  v^\nu dx \\
&\quad + C\nu^{\frac{1}{12}}   \int_\bbr  \eta(U^\nu |\overline U_- )  dx + C\nu^{\frac{1}{12}} .
\end{aligned}
\end{align*}
Observe that by $|\bar v_\pm-v_*|\ll1$, there exists a constant $c_*>2 v_*$ such that
\beq\label{newq}
 Q(v^\nu| \bar v_\pm) \ge p(\bar v_\pm) v^\nu - Q(\bar v_\pm) - p(\bar v_\pm)\bar v_\pm \ge  p(\bar v_\pm) (v^\nu - c_*)_+,
\eeq
 and using Lemma \ref{lem:ubdd},
 \begin{align*}
\begin{aligned}
&  \int_\bbr  \Big( a_- + C\nu^{\frac{1}{12}} \Big) \Big(p(\bar v_+) -p(\bar v_-) \Big)  v^\nu dx \\
&=  \int_\bbr  \Big( a_- + C\nu^{\frac{1}{12}} \Big) \Big(p(\bar v_+) -p(\bar v_-) \Big)\Big( (v^\nu - c_*)_+ + v \mathbf{1}_{v\le c_*} + c_*\mathbf{1}_{v> c_*} \Big)dx \\
&\le  \int_\bbr  \Big( a_- + C\nu^{\frac{1}{12}} \Big) \Big(p(\bar v_+) -p(\bar v_-) \Big)(v^\nu - c_*)_+ dx +  C\nu^{\frac{1}{12}} .
\end{aligned}
\end{align*}
By \eqref{newq} and $f(x)= \Delta \bar L_j+\kappa \Delta \bar Q_j +\Delta\Omega^S(x) <0$,
 \begin{align*}
\begin{aligned}
&\int_{\bbr} \Big( a_+ \eta(U^\nu |\overline U_+ ) -  a_- \eta(U^\nu |\overline U_- ) \Big) dx\\
&\le  \int_{\bbr} \bigg( \frac{1}{\ds}\Big( \Delta \bar L_j+\kappa \Delta \bar Q_j +\Delta\Omega^S(x)\Big) p(\bar v_+) +  \Big( a_- + C\nu^{\frac{1}{12}} \Big) \Big( p(\bar v_+) -p(\bar v_-) \Big) \bigg)  (v^\nu - c_*)_+ dx  \\
&\quad + C\nu^{\frac{1}{12}}   \int_\bbr  \eta(U^\nu |\overline U_- )  dx + C\nu^{\frac{1}{12}} .
\end{aligned}
\end{align*}
It remains to estimate the first term above.
For that, we use
 \begin{align*}
\begin{aligned}
p(\bar v) &= p(v_*)+ \int_{-\infty}^x \partial_x p(\bar v) dy\\
&= p(v_*)+  \int_{-\infty}^x \sum_{i\in \mathcal{S}} \big(p'(\bar v) - p'(\tilde v_i) \big) \partial_x \tilde v_i dy +  \int_{-\infty}^x \sum_{i\in \mathcal{R}} \big(p'(\bar v) - p'(v^R_i) \big) \partial_x v^R_i dy \\
&\quad +  \int_{-\infty}^x \sum_{i\in \mathcal{NP}} \big(p'(\bar v) - p'(v^P_i) \big) \partial_x v^P_i dy + \int_{-\infty}^x \sum_{i\in \mathcal{S}}  \partial_x p(\tilde v_i) dy \\
&\quad + \int_{-\infty}^x \sum_{i\in \mathcal{R}}  \partial_x p(v^R_i) dy+ \int_{-\infty}^x \sum_{i\in \mathcal{NP}}  \partial_x p(v^P_i) dy.
\end{aligned}
\end{align*}
Using the same argument as above with \eqref{shsumz}, we have
 \begin{align}\label{pvpm}
\begin{aligned}
p(\bar v_+) -p(\bar v_-) &\le  \underbrace{C\int_\bbr \sum_{i\in \mathcal{S}}   |\bar v - \tilde v_i| | (\tilde v_i)_x| dx + C \int_\bbr \sum_{i\in \mathcal{R}}   |\bar v - v^R_i| | (v^R_i)_x| dx +C\sum_{i\in \mathcal{NP}}  \int_\bbr |(v^P_i)_x | dx}_{=:J_1}\\
&\quad  - \int_{-\infty}^x \Big( g_+^S(y) -g_-^S(y) \Big) dy  - \int_{-\infty}^x \Big( g_+^R(y) -g_-^R(y) \Big) dy
\underbrace{+ \int_{-\infty}^x \Big( g_+^P(y) -g_-^P(y) \Big) dy }_{=:J_2},
\end{aligned}
\end{align}
where
\begin{align*}
\begin{aligned}
& g_{\pm}^R(y) := \sum_{i\in \mathcal{R}_j\cap\{i_0, i_0+1\}} \partial_x \Big( -p(v^R_{i}\pm) \Big),\\
& g_{\pm}^P(y) := \sum_{i\in \mathcal{NP}_j\cap\{i_0, i_0+1\}} \partial_x \Big( -p(v^P_{i}\pm) \Big).
\end{aligned}
\end{align*}
By Lemmas \ref{lem:waves} and \ref{lem:finsout}, we have 
 \begin{align}\label{j12f}
\begin{aligned}
&J_1 \le  \frac{C\nu^{1/2}}{\delta} +C\sum_{i\in \mathcal{NP}} |\s_i| \le  \frac{C\nu^{1/6}}{\delta},\\
&J_2 \le  \int_{-\infty}^\infty \Big(|g_+^P(y)| + |g_-^P(y)| \Big)  dy \le C\sum_{i\in \mathcal{NP}} |\s_i| \le \frac{C\nu^{1/6}}{\delta}.
\end{aligned}
\end{align}
As in \eqref{app-ome}, we define $\Omega^R_\pm(x)$ approximations of $ \int_{-\infty}^x  g_\pm^R(y) dy$, and define the difference:
\[
\Delta\Omega^R(x):= \Omega^R_+(x) - \Omega^R_-(x) .
\]
By the same fact as in \eqref{errortail} that the width of transition zone of rarefaction has $\nu$ scale,
\beq\label{errortail-r}
\Big|\Omega^R_\pm(x) -  \int_{-\infty}^x  g_\pm^R(y) dy \Big| \le C\exp{(-\nu^{-1/2})} ,\quad\forall x\in\bbr.
\eeq
To control $\Delta\Omega^R(x)$, following the same argument of the proof as in Proposition \ref{prop:deca}, we have 
\beq\label{app_last}
 \Delta \bar L_j+\kappa \Delta \bar Q_j -\Delta\Omega^R(x) \le  0,
\eeq
for all interactions except for the simplified solver in Figure \ref{Simp4} and its symmetric simplified solver between $\fa S\fa R$, where it holds for this special case that
\beq\label{app_last_2}
 \Delta \bar L_j+\kappa \Delta \bar Q_j -\Delta\Omega^R(x) \le C\rho_\nu.
\eeq
Indeed, the proof of \eqref{app_last} is symmetric to that of \eqref{bbyfr}, because $p$ has opposite monotonicity on shock and rarefaction in each family, and so $ \partial_x \Big( -p(\tilde v_{i}) \Big)$ and  $\partial_x p(v^R_{i})$ have the same sign for the same wave.
 Thus, to prove \eqref{app_last}, basically, we only need to change  $\ba S, \ba R,\fa S, \fa R,\bar\sigma$ in the proof of \eqref{bbyfr} to 
$\ba R, \ba S,\fa R, \fa S,-\bar\sigma$, respectively, and so change $\Delta\Omega^S(x)$ to $-\Delta\Omega^R(x)$. We omit the detail.
However, for the simplified solver in Figure \ref{Simp4}, since $\Delta L<\bar\sigma''- \bar\sigma_1$ by \eqref{app_est}, and $-\Delta\Omega^R(x) \leq  \bar\sigma_1\, \forall x\in\bbr$, we have
\[
 \Delta \bar L_j+\kappa \Delta \bar Q_j -\Delta\Omega^R(x) \le \s'' \le C\rho_\nu.
\]
Therefore, since it holds from \eqref{pvpm} and \eqref{j12f} that
 \begin{align*}
\begin{aligned}
p(\bar v_+) -p(\bar v_-) &\le  \frac{C\nu^{1/6}}{\delta} -\Delta\Omega^S(x) -\Delta\Omega^R(x) +\Big[ \Delta\Omega^S(x)- \int_{-\infty}^x \Big( g_+^S(y) -g_-^S(y) \Big) dy \Big] \\
&\quad + \Big[ \Delta\Omega^R(x) - \int_{-\infty}^x \Big( g_+^R(y) -g_-^R(y) \Big) dy \Big],
\end{aligned}
\end{align*}
using  \eqref{errortail}, \eqref{bbyfr}, \eqref{errortail-r}, and \eqref{app_last_2}, we have
\begin{align*}
\begin{aligned}
& \frac{ p(\bar v_+)}{\ds}\Big( \Delta \bar L_j+\kappa \Delta \bar Q_j +\Delta\Omega^S(x)\Big)  +  \Big( a_- + C\nu^{\frac{1}{12}} \Big) \Big( p(\bar v_+) -p(\bar v_-) \Big) \\
&\le  \frac{p(\bar v_+)}{\ds} \left(1-\frac{\ds }{p(\bar v_+)} \Big( a_- + C\nu^{\frac{1}{12}} \Big)  \right) \Big( \Delta \bar L_j+\kappa \Delta \bar Q_j +\Delta\Omega^S(x) \Big) \\
&\quad
+ \Big( a_- + C\nu^{\frac{1}{12}} \Big) \Big( \Delta \bar L_j+\kappa \Delta \bar Q_j -\Delta\Omega^R(x) \Big) +C\exp{(-\nu^{-1/2})} + \frac{C\nu^{1/6}}{\delta} \\
&\le C\rho_\nu +C\exp{(-\nu^{-1/2})} + \frac{C\nu^{1/6}}{\delta} \le C\nu^{\frac{1}{12}},
\end{aligned}
\end{align*}
where the last inequality is obtained by the smallness of $\nu$.\\
Hence,
 \begin{align*}
\begin{aligned}
&\int_{\bbr} \Big( a_+ \eta(U^\nu |\overline U_+ ) -  a_- \eta(U^\nu |\overline U_- ) \Big) dx\\
&\quad \le C\nu^{\frac{1}{12}}  \int_\bbr (v^\nu - c_*)_+ dx + C\nu^{\frac{1}{12}}   \int_\bbr |h^\nu-\bar h_-|^2 dx + C\nu^{\frac{1}{12}} ,
\end{aligned}
\end{align*}
which together with \eqref{newq} yields 
 \begin{align}
\begin{aligned} \label{instest}
&\int_{\bbr} \Big( a_+ \eta(U^\nu |\overline U_+ ) -  a_- \eta(U^\nu |\overline U_- ) \Big) dx\\
&\le  C\nu^{\frac{1}{12}} \int_{\bbr}   \eta(U^\nu| \overline U_-)  dx  + C\nu^{\frac{1}{12}}.
\end{aligned}
\end{align}

\begin{lemma}\label{lem:ubdd}
Let $ \overline U_\pm :=\overline U_{\nu,\delta}(x,t_j\pm)$.
There exists $C>0$ such that
\[
\quad \|\overline U_\pm\|_{L^\infty} \le C,\quad \|\overline U_+-\overline U_-\|_{L^1} \le C\nu^{1/6},\quad \|\overline U_+-\overline U_-\|_{L^2}^2 \le C\nu^{1/6}.
\]
\end{lemma}
\begin{proof}
Since $\|\overline U_\pm-U_*\|_{L^\infty}\ll1$, the first estimate holds.\\
It holds from \eqref{apprr} that
\begin{align*}
\begin{aligned}
\overline U_+-\overline U_- &= \int_{-\infty}^x \bigg(\sum_{i\in \mathcal{S}_j} \partial_x \Big(\tilde U_{i+} -\tilde U_{i-} \Big) +  \sum_{i\in \mathcal{R}_j} \partial_x \Big(U^R_{i+}-U^R_{i-} \Big)+  \sum_{i\in \mathcal{NP}_j} \partial_x \Big(U^P_{i+}-U^P_{i-} \Big) \bigg) dz ,
\end{aligned}
\end{align*}
where 
\begin{align*}
\begin{aligned}
&\tilde U_{i+}:= \tilde U_{ij}^\nu(x,t_j+),\quad \tilde U_{i-}:=\tilde U_{i,j-1}^\nu(x,t_j-),\\
&U^R_{i+}:= U_{ij}^R(x,t_j+),\quad U^R_{i-}:= U_{i,j-1}^R(x,t_j-),\\
&U^P_{i+}:= U_{ij}^P(x),\quad U^P_{i-}:= U_{i,j-1}^P(x) .
\end{aligned}
\end{align*}
Since $\overline U_+-\overline U_- \to 0$ as $x\to+\infty$, we also have
\begin{align*}
\begin{aligned}
\overline U_+-\overline U_- &= -\int_x^{\infty} \bigg(\sum_{i\in \mathcal{S}_j} \partial_x \Big(\tilde U_{i+} -\tilde U_{i-} \Big) +  \sum_{i\in \mathcal{R}_j} \partial_x \Big(U^R_{i+}-U^R_{i-} \Big)+  \sum_{i\in \mathcal{NP}_j} \partial_x \Big(U^P_{i+}-U^P_{i-} \Big) \bigg) dz .
\end{aligned}
\end{align*}
As done before, we have
\begin{align*}
\begin{aligned}
\overline U_+-\overline U_- &= \int_{-\infty}^x \bigg( \sum_{i\in \mathcal{S}_j\cap\{i_0, i_0+1\}} \partial_x \Big(\tilde U_{i+} -\tilde U_{i-} \Big) +   \sum_{i\in \mathcal{R}_j\cap\{i_0, i_0+1\}} \partial_x \Big(U^R_{i+}-U^R_{i-} \Big)\\
&\quad +   \sum_{i\in \mathcal{NP}_j\cap\{i_0, i_0+1\}} \partial_x \Big(U^P_{i+}-U^P_{i-} \Big) \bigg) dz,
\end{aligned}
\end{align*}
or 
\begin{align*}
\begin{aligned}
\overline U_+-\overline U_- &=  -\int_x^{\infty}  \bigg( \sum_{i\in \mathcal{S}_j\cap\{i_0, i_0+1\}} \partial_x \Big(\tilde U_{i+} -\tilde U_{i-} \Big) +   \sum_{i\in \mathcal{R}_j\cap\{i_0, i_0+1\}} \partial_x \Big(U^R_{i+}-U^R_{i-} \Big)\\
&\quad +   \sum_{i\in \mathcal{NP}_j\cap\{i_0, i_0+1\}} \partial_x \Big(U^P_{i+}-U^P_{i-} \Big) \bigg) dz.
\end{aligned}
\end{align*}
Thus,
\begin{align*}
\begin{aligned}
|\overline U_+-\overline U_-| &\le \sum_{i\in \mathcal{S}_j\cap\{i_0, i_0+1\}} \int_{-\infty}^x\Big( |\partial_x \tilde U_{i+}|+|\partial_x \tilde U_{i-}|\Big) dz \\ 
&+ \sum_{i\in \mathcal{R}_j\cap\{i_0, i_0+1\}} \int_{-\infty}^x\Big( |\partial_x U^R_{i+}|+|\partial_x U^R_{i-}|\Big) dz
\\ 
&
+ \sum_{i\in \mathcal{NP}_j\cap\{i_0, i_0+1\}} \int_{-\infty}^x\Big( |\partial_x U^P_{i+}|+|\partial_x U^P_{i-}|\Big) dz\\
&=: J^S_1 + J^R_1 +J^P_1 ,
\end{aligned}
\end{align*}
or
\begin{align*}
\begin{aligned}
|\overline U_+-\overline U_-| &\le \sum_{i\in \mathcal{S}_j\cap\{i_0, i_0+1\}} \int_x^{\infty} \Big( |\partial_x \tilde U_{i+}|+|\partial_x \tilde U_{i-}|\Big) dz \\ 
&+ \sum_{i\in \mathcal{R}_j\cap\{i_0, i_0+1\}} \int_x^{\infty}\Big( |\partial_x U^R_{i+}|+|\partial_x U^R_{i-}|\Big) dz
\\ 
&
+ \sum_{i\in \mathcal{NP}_j\cap\{i_0, i_0+1\}}  \int_x^{\infty}\Big( |\partial_x U^P_{i+}|+|\partial_x U^P_{i-}|\Big) dz\\
&=: J^S_2 + J^R_2 +J^P_2 ,
\end{aligned}
\end{align*}

Consider an interval $I:=\big(y_{i_0}-\sqrt{\rn}, y_{i_0+1}+\sqrt\rn \big)$. Since $|I|\le 3\sqrt\rn$ and $ \|\overline U_\pm\|_{L^\infty(\bbr)} \le C$,  we first find that for each $q=1,2$,
\[
 \|\overline U_+-\overline U_-\|_{L^q(\bbr)}^q \le C\sqrt\rn + C \|\overline U_+-\overline U_-\|_{L^q(I^c)}^q .
\]
Using the fact from Lemma \ref{lem:waves} that for each $i=i_0, i_0+1$,
\[
\int_{-\infty}^x |(U^P_i)_x| dz \le C\int_{-\infty}^x \frac{|\s_i|}{\sqrt\rn} \mathbf{1}_{|z-y_i|\le \sqrt\rn} dz =0,\quad \forall x \le y_{i_0}-\sqrt{\rn},
\]
and
\[
\int_x^{\infty} |(U^P_i)_x| dz \le C\int_x^{\infty} \frac{|\s_i|}{\sqrt\rn} \mathbf{1}_{|z-y_i|\le \sqrt\rn} dz =0,\quad \forall x \ge y_{i_0+1}+\sqrt\rn,
\]
we have
\[
 \|J^P_1\|_{L^q(I^c)}= \|J^P_2\|_{L^q(I^c)} =0.
\]
On the other hand, observe that since 
\begin{align*}
\begin{aligned}
\int_{-\infty}^x |\partial_x \tilde U_{i}^\nu(z,t)| dz& \le \int_{-\infty}^x  \frac{\s_i^2}{\nu} \exp\Big(-\frac{C^{-1}|\s_{i}||z-y_i|}{\nu} \Big) dz \\
&\le C |\s_i| \exp\Big(-\frac{C^{-1}|\s_{i}||x-y_i|}{\nu} \Big),
\end{aligned}
\end{align*}
we have
\begin{align*}
\begin{aligned}
\int_{-\infty}^x \partial_x \tilde U_{i}^\nu(z,t) dz& \le \int_{-\infty}^x  \frac{\s_i^2}{\nu} \exp\Big(-\frac{C^{-1}|\s_{i}||z-y_i|}{\nu} \Big) dz \\
&\le C |\s_i| \exp\Big(-\frac{C^{-1}|\s_{i}||x-y_i|}{\nu} \Big).
\end{aligned}
\end{align*}
Then,
\begin{align*}
\begin{aligned}
 \|J^S_1\|_{L^q(I^c)}^q, \|J^S_2\|_{L^q(I^c)}^q   \le C\nu,
\end{aligned}
\end{align*}
Likewise, since
\[
\int_{-\infty}^x \partial_x U^R_{i}(z,t) dz \le C |\s_i| \exp\Big(-\frac{|\s_{i}||x-y_i|}{\nu} \Big).
\]
\[
 \|J^R_1\|_{L^q(I^c)}^q, \|J^R_2\|_{L^q(I^c)}^q   \le C\nu .
\]
Therefore, we have the desired estimates by $\rn=\nu^{1/3}$.
\end{proof}

\subsection{Global-in-time uniform estimates}
Let 
\[
F_\nu(t):=  \int_{\bbr} a^\nu \eta(U^\nu |\overline U_{\nu,\delta} ) dx.
\]
Note that $F_\nu$ depends also on $\delta$. However, we will fix $\delta$ as a function of $\nu$ below. Therefore we drop the $\delta$ in the notation.
\vskip0.3cm

Given $T>0$, consider the sequence of the interaction times $0<t_1<t_2<\cdots<t_N$, and let $t_0:=0, T:=t_{N+1}$.
By Proposition \ref{prop:main} and \eqref{instest}, we find that for each $t_j$ ($j=0,1,2,...,N$),
\beq\label{fdineq}
\forall  t\in (t_j,t_{j+1}), \quad F'_\nu(t) + G_\nu(t) \le C F_\nu(t) +C \delta^\frac{1}{4}\, \overline{\mathcal{D}}(t) +C \delta^\frac{3}{4} +C   \frac{\delta^2}{\delta_*} + \mathcal{C} (\delta, \nu) ,
\eeq
where $\mathcal{C} (\delta, \nu)$ is the constant that vanishes when $\nu\to0$ for any fixed $\delta>0$, and
\beq\label{fdjump}
F_\nu(t_{j+1}+) - F_\nu(t_{j+1}-) \le C\nu^{\frac{1}{12}} F_\nu(t_{j+1}-)  +C\nu^{\frac{1}{12}}.
\eeq
Notice that by Remark \ref{rem:origin} with the assumption $\int_\bbr \eta\big(U_0^\nu|(v_*,u_*)\big) dx<\infty$,
\[
\int_0^T \overline{\mathcal{D}}(t) dt \le C.
\]

We will apply the following Lemma \ref{lem:gron} to the above estimates by putting
\[
f_\nu(t):= F_\nu(t) + \delta^\frac{1}{8},\quad g_\nu(t):= G_\nu(t),\quad h_\nu(t):=C\, \overline{\mathcal{D}}(t) +C \delta^\frac{1}{2} +C   \frac{\delta}{\delta_*} +  \frac{ \mathcal{C} (\delta, \nu)}{\delta^\frac{1}{4}} .
\]  
First, since $\delta\ll\ds$ and $\mathcal{C} (\delta, \nu)\to 0$ as $\nu\to0$ for any fixed $\delta>0$, we have
\[
C \delta^\frac{1}{2} +C   \frac{\delta}{\delta_*} +  \frac{ \mathcal{C} (\delta, \nu)}{\delta^\frac{1}{4}} \le C,
\]
and so,
\[
\int_0^T h_\nu(t)dt \le C(1+T).
\]
It holds from \eqref{fdineq} and \eqref{fdjump} that
\[
f_\nu'(t)+ g_\nu(t) \le C f_\nu(t) + \delta^\frac{1}{4} \Big(C \overline{\mathcal{D}}(t) +C \delta^\frac{1}{2} +C   \frac{\delta}{\delta_*} +  \frac{ \mathcal{C} (\delta, \nu)}{\delta^\frac{1}{4}} \Big) = C f_\nu(t) + \delta^\frac{1}{4}  h_\nu(t),
\]
and by $\delta^\frac{1}{8}\le f_\nu(t_j-)$
\[
f_\nu(t_j+)-f_\nu(t_j-) \le  C\nu^{\frac{1}{12}} F_\nu(t_{j}-)  +C\nu^{\frac{1}{12}} \frac{f_\nu(t_j-)}{\delta^\frac{1}{8}} \le   \frac{C\nu^{\frac{1}{12}}}{\delta^\frac{1}{8}}f_\nu(t_j-).
\]
Now, using Lemma \ref{lem:gron} with $c_0:= \delta^\frac{1}{8}$ and $c_1:=\frac{C\nu^{\frac{1}{12}}}{\delta^\frac{1}{8}}$, and together with $N\le C\Big(\frac{1}{\delta}\ln(\rho_\nu^{-1}) \Big)^3$ by Lemma \ref{lem:num}, we have
 \begin{align*}
\begin{aligned}
F_\nu(t)&<f_\nu(t) \le  (F_\nu(0)+\delta^\frac{1}{8}) \exp\left(CT+C\delta^\frac{1}{8} + C\Big(\frac{1}{\delta}\ln(\rho_\nu^{-1}) \Big)^3\frac{\nu^{\frac{1}{12}}}{\delta^\frac{1}{8}}\right)\\
&\le  (F_\nu(0)+\delta^\frac{1}{8}) \exp\left(CT+C\delta^\frac{1}{8} + C\frac{\nu^{\frac{1}{24}}}{\delta^5} \right).
\end{aligned}
\end{align*}
and
\[
\int_0^T G_\nu(t) dt \le  (F_\nu(0)+\delta^\frac{1}{8}) \left(1+CT + C\frac{\nu^{\frac{1}{24}}}{\delta^5} \right) \exp\left( CT+C\delta^\frac{1}{8} +  C\frac{\nu^{\frac{1}{24}}}{\delta^5}  \right) + C \delta^\frac{1}{4}.
\]

Finally, we choose $\delta$ as $\delta_\nu$ wisely (for example $\delta_\nu=\nu^\frac{1}{240}$) such that $F_\nu(0)\to 0$ and the above r.h.s. converge to $0$ as $\nu\to 0$. Indeed, $F_\nu(0)\to 0$ as $\nu\to 0$ by \eqref{initialNS} and the construction in Section \ref{sub:conapp}. \\
Hence,
\beq\label{conthm}
\mbox{$ \sup_{t\in (0,T)} F_\nu(t)$ and $\int_0^TG_\nu(t)\,dt$ converge to $0$ as $\nu\to 0$}.
\eeq
This completes the proof of Theorem \ref{thm:uniform}. 

\begin{lemma}\label{lem:gron}
Given $T>0$, consider a increasing sequence $0=:t_0<t_1<\cdots<t_N<t_{N+1}:= T$. Let $f, g, h:[0,T]\to\bbr_+$ be locally integrable nonnegative functions such that for some positive constants $c_0, c_1, C$, the following holds: $\int_0^T h(t)dt \le C$, and
 \begin{align*}
\begin{aligned}
&\mbox{for each } j=0,1,...,N,\quad\forall t\in (t_j, t_{j+1}),\\
 & f'(t) + g(t) \le  Cf(t) + c_0^2 h(t),\quad f(t)\ge c_0, \quad \lim_{t\to 0} f(t) =f(0), \\
& f(t_j+)  \le (1+ c_1)  f(t_j-).
\end{aligned}
\end{align*}
Then, there exists $C>0$ (independent of the constants $c_0, c_1$) such that for a.e. $t\in [0,T]$, 
 \begin{align*}
\begin{aligned}
& f(t) \le  f(0) \exp(CT+Cc_0 + Nc_1),\\
& \int_0^T g(t) dt \le  f(0)(1+CT +Nc_1) \exp( CT+Cc_0 + Nc_1) + Cc_0^2.
\end{aligned}
\end{align*}
\end{lemma}
\begin{proof}
Let $J(t):=\ln f(t)$.\\
Since $f'(t) \le  Cf(t) + c_0^2 h(t)$ and $f\ge c_0$, we have
\[
\mbox{for each } j=0,1,...,N,\quad\forall t\in (t_j, t_{j+1}),\quad J'(t)\le C +  c_0^2 \frac{h(t)}{f(t)} \le C+ c_0 h(t).
\]
Thus, for each $j=0,1,...,N$,  for all $t\in (t_k, t_{k+1})$,
\[
J(t) - J( t_{j}+) \le C( t_{j+1}- t_{j}) + c_0 \int_{t_{j}}^{t_{j+1}} h(t)dt.
\]
and so,
\[
J( t_{j+1}-) - J( t_{j}+) \le C( t_{j+1}- t_{j}) + c_0 \int_{t_{j}}^{t_{j+1}} h(t)dt.
\]
We also have
\[
J( t_{j+1}+) - J( t_{j+1}-) = \ln\left( \frac{f(t_{j+1}+)}{f(t_{j+1}-)} \right) \le \ln ( 1+ c_1)\le c_1.
\]
For a.e. $t\in [0, T]$, there exists $0\le k\le N$ such that $t\in (t_k, t_{k+1})$. Then,
 \begin{align*}
\begin{aligned}
J(t)-J(0) &=  J(t)-J(t_k+) +  \sum_{0\le j \le k-1} (J(t_{j+1}-)-J(t_j+)) +  \sum_{0\le j \le k-1} (J(t_{j+1}+)-J(t_{j+1}-))\\
&\le CT+ c_0\int_0^T h(t)dt + N c_1 \le C(T+c_0) + N c_1. 
\end{aligned}
\end{align*}
Thus,
\[
 \ln\left( \frac{f(t)}{f(0)} \right) \le C(T+c_0) + N c_1,
\]
and so
\[
f(t) \le f(0) \exp(CT+Cc_0 + Nc_1), \quad \mbox{for a.e. } t\in [0,T].
\]
Using the above estimate and
\[
f'(t) + g(t) \le  Cf(t) + c_0^2 h(t),\quad\forall t\in (t_j, t_{j+1}), 
\]
we have
 \begin{align*}
\begin{aligned}
 \int_0^T g(t) dt &= \sum_{0\le j \le N} \int_{t_{j}}^{t_{j+1}} g(t)dt \\
 & =  \sum_{0\le j \le N} \left(\int_{t_{j}}^{t_{j+1}} g(t)dt  + f(t_{j+1}-) - f(t_j+) \right) -  \sum_{0\le j \le N} \left( f(t_{j+1}-) - f(t_j+) \right) \\
 &\le C\int_0^T f(t) dt + c_0^2 \int_0^T h(t) dt -  \sum_{0\le j \le N} \left( f(t_{j+1}-) - f(t_j+) \right) \\
&\le C T f(0) \exp(CT+Cc_0 + Nc_1) + C c_0^2 +  \sum_{0\le j \le N} \left(f(t_j+)- f(t_{j+1}-)   \right).
\end{aligned}
\end{align*}
Since
 \begin{align*}
\begin{aligned}
  \sum_{0\le j \le N} \left(f(t_j+)- f(t_{j+1}-)   \right) &=  \sum_{0\le j \le N}\left(f(t_j+)- f(t_{j}-)   \right) + f(0) - f(t_{N}-) \\
  &\le c_1 \sum_{0\le j \le N}  f(t_{j}-) +f(0) \\
  &\le Nc_1f(0) \exp(CT+Cc_0 + Nc_1) + f(0),
\end{aligned}
\end{align*}
we have the desired estimates.
\end{proof}
\section{From Lagrange to Euler}\label{sec:main}
This section is dedicated to the proof of Proposition \ref{main_prop} from Theorem \ref{thm:uniform}.
Let us write $x$ for the Eulerian variable and $y$ for the Lagrangian variable. We have for every time $t$ fixed
$$
y=\int_0^x \rho(z,t)\,dz, \qquad \frac{dy}{dx}=\rho.
$$
Notice that for any (smooth enough) function $F$, the change of variables gives:
\begin{equation}\label{eqBVE}
\int_\RR |\partial_x F|\,dx=\int_\RR |\partial_y F|\,dy.
\end{equation}
The quantities related to the Euler equation, $U^0$ and $\bar{U}_\nu$, are transformed from Lagrangian to Eulerian as:
\begin{equation*}
\begin{array}{ll}
\displaystyle{v^0(y,t)=1/\rho^0(x,t)},&\displaystyle{h^0(y,t)=u^0(x,t)},\\[0.2cm]
\displaystyle{\bar{v}_\nu(y,t)=1/\bar \rho_\nu(x,t)},&\displaystyle{\bar{h}_\nu(y,t)=\bar u_\nu(x,t).}
\end{array}
\end{equation*}
Note that these quantities do not verify specific equations. This is different for the solutions to the Navier-Stokes equation. In the Eulerian variables $(\rho^\nu, \rho^\nu u^\nu)$ is solution to \eqref{1NS}, while in the Lagrange variables, $(v^\nu,h^\nu)$ verifies the modified Navier-Stokes equation \eqref{NS}. They are linked via the BD effective velocity with $\phi'(\rho)=\bar \mu(\rho)/\rho^{3/2}$:
\begin{eqnarray*}
&&v^\nu(y,t)=1/\rho^\nu(x,t),\\
&&  h^\nu(y,t)=u^\nu(x,t)-\nu\frac{\bar \mu_L(v^\nu)}{v^\nu}v^\nu_y(y,t)=u^\nu(x,t)+\nu\frac{1}{(\rho^\nu(x,t))^{1/2}}[\phi(\rho^\nu(x,t))]_x.
\end{eqnarray*}
Thanks to \eqref{eqBVE} and the fact that $v^0$ and $\bar v_\nu$ are bounded away from 0, the smallness $BV$ condition on the initial value of the hypotheses of Theorem \ref{thm:uniform} derives from the BV condition of the hypotheses of Proposition \ref{main_prop} (with possibly a different $\eps_0>0$).
\vskip0.3cm
We now show the following lemma.
\begin{lemma}\label{lemin}
Fix $(\rho_*,\rho_* u_*)\in \RR^+\times\RR$ with $\rho_*>0$ and let $U^0=(\rho^0,\rho^0u^0)\in E^0.$ Then we have 
$$
\int_\RR \eta^L((v^0,h^0)|(v_*,u_*))\,dy <\infty.
$$
Moreover, if  $U^\nu_0=(\rho^\nu_0,\rho^\nu_0u^\nu_0)$ verifies \eqref{initialNS}, then 
$$
\lim_{\nu\to0} \int_\RR\eta^L((v^\nu_0,h^\nu_0)| (v^0,h^0))\,dy=0.
$$
\end{lemma}
\begin{proof}
First, we have:
$$
\int_\RR \eta^L((v^0,h^0)|(v_*,u_*))\,dy=\int_\RR \eta((\rho^0,\rho^0u^0)|(\rho_*,\rho_*u_*))\,dx.
$$
Therefore, the first statement of Lemma \ref{lemin} comes directly from \eqref{einiq}. Because of the BD effective velocity,
\begin{eqnarray*}
\int_\RR\eta^L((v^\nu_0,h^\nu_0)| (v^0,h^0))\,dy&=&\int_\RR\eta\left(\left(\rho^\nu_0,\rho^\nu_0\left(u^\nu_0+\nu\frac{\mu(\rho^\nu_0)}{(\rho^\nu_0)^2}(\rho_0^\nu)_x\right)\right)|(\rho^0,\rho^0u^0)\right)\,dx\\
&=&\int_\RR\rho^\nu_0\left|u_0^\nu-u^0-\nu\frac{\mu(\rho^\nu_0)}{(\rho^\nu_0)^2}(\rho^\nu_0)_x\right|^2+\frac{p(\rho^\nu_0|\rho^0)}{\gamma-1}\,dx\\
&\leq&\int_\RR\eta\left((\rho^\nu_0,\rho^\nu_0u^\nu_0)|(\rho^0,\rho^0u^0)\right)\,dx+\nu^2\int_\RR |\partial_x \phi(\rho^\nu_0)|^2\,dx,
\end{eqnarray*}
which converges to 0 thanks to the last statement of \eqref{initialNS}.
\end{proof}

From this lemma, the hypotheses of Proposition \ref{main_prop} imply the remaining hypotheses  of Theorem \ref{thm:uniform}. We can then use its statement.
Thanks to \eqref{eqBVE} and  the fact that $\bar v_\nu$ is bounded below, the small $BV$ condition \eqref{BVest} derives from the small $BV$ condition \eqref{BVbarU} (possibly with a different $C>0$).
\vskip0.3cm
Moreover,  for all $t>0$,  
\begin{equation}\label{eqlem}
\int\eta^L(U^\nu|\bar U_\nu)\,dy=\int_\RR \left(\frac{|h^\nu-\bar h_{\nu}|^2}{2} + Q(v^\nu| \bar v_{\nu})\right) dy \to 0,\quad \mbox{as $ \nu\to 0$}.
\end{equation}
Especially, setting $\bar \rho_{\nu} := 1/\bar v_{\nu}$ and using the change of coordinates, 
\[
\int_\RR \frac{p(\rho^\nu |\bar \rho_{\nu})}{\gamma-1} dx = \int_\RR Q(v^\nu| \bar v_{\nu}) dy \to 0,
\]
which implies
\begin{equation}\label{eqlim}
\rho^\nu - \bar \rho_{\nu}  \to 0 \quad \mbox{in $L^1_{\mathrm{loc}}((0,T)\times\bbr)$ as $ \nu\to 0$}.
\end{equation}
For the limit of $\rho^\nu u^\nu$, we first notice that 
\begin{equation}\label{eqontient}
\rho^\nu u^\nu-\bar\rho_\nu \bar u_\nu=(\rho^\nu-\bar\rho_\nu)\bar u_\nu+\rho^\nu(h^\nu-\bar u_\nu)-\nu \sqrt{\rho^\nu}\phi(\rho^\nu)_x =: R_1 + R_2 +R_3.
\end{equation}
The rest of this section shows that the three terms $R_1, R_2, R_3$ in \eqref{eqontient} converge to 0, which will finish the proof. 
\vskip0.3cm\noindent{\bf Convergence of $R_1$:} This term converges to 0 in $L^1_{\mathrm{loc}}$ thanks to \eqref{eqlim}, and the uniform bound \eqref{BVest} on $\bar u_\nu$.

\vskip0.3cm\noindent{\bf Convergence of $R_2$:} We use the change of coordinates  and \eqref{eqlem} to find that for any compact $K$ in $(0,T)\times\bbr$,
\[
\int_K \rho^\nu |h^\nu-\bar u_{\nu}|^2 dxdt = \int_\bbr |h^\nu-\bar h_{\nu}|^2 dydt \to 0,\quad \mbox{as $ \nu\to 0$}.
\]
Together with the uniform  bound of $\rho^\nu-\bar \rho_\nu$ in $L^1$, this  implies
\[
\rho^\nu (h^\nu-\bar u_{\nu})  \to 0 \quad \mbox{in $L^1_{\mathrm{loc}}((0,T)\times\bbr)$ as $ \nu\to 0$}.
\]
\vskip0.3cm\noindent{\bf Convergence of $R_3$:} The difficulty comes from the small values of $\rho$. For any $\kappa>0$, we consider the  function $\phi_\kappa$ such that $\phi_\kappa(\kappa)=0$, and
$$
\phi'_\kappa(\rho)=\phi'(\rho) {\bf 1}_{\{\rho\leq\kappa\}}.
$$
Then, we decompose $R_3$ as
\[
\nu \sqrt{\rho^\nu}\phi(\rho^\nu)_x = \nu \sqrt{\rho^\nu} \partial_x \phi_k(\rho^\nu) +  \nu \sqrt{\rho^\nu} \partial_x \big( \phi(\rho^\nu) -\phi_k(\rho^\nu) \big) =: d_1 + d_2.
\]
Since $\nu\sqrt{\rho^\nu}|\partial_x\phi_\kappa(\rho^\nu)|\leq \nu\sqrt{\kappa} |\partial_x\phi(\rho^\nu)|$,
we use Lemma \ref{lem:derho} to have
\[
\|d_1\|^2_{L^2_{\mathrm{loc}}((0,T)\times\bbr)} \leq \kappa \|\nu \partial_x\phi(\rho^\nu)\|^2_{L^2_{\mathrm{loc}}((0,T)\times\bbr)} \leq C\kappa.
\]
For $d_2$, we first observe that 
\[
d_2=\nu\sqrt{\rho^\nu}\partial_x(\phi(\rho^\nu)-\phi_\kappa(\rho^\nu))=\nu\sqrt{\rho^\nu} \phi'(\rho^\nu) {\bf 1}_{\{\rho^\nu>\kappa\}} \rho^\nu_x = \nu \frac{\bar\mu(\rho^\nu)}{\rho^\nu} {\bf 1}_{\{\rho^\nu>\kappa\}} \rho^\nu_x.
 \]
Consider the function $\psi_k$ such that $\psi_k'(\rho)= \nu \frac{\bar\mu(\rho)}{\rho} {\bf 1}_{\{\rho>\kappa\}}$. Since it holds from the assumption \eqref{hyp:mu}
that 
\[
|\psi_\kappa(\rho^\nu)|\leq \nu C+ \nu C_\kappa \,p(\rho^\nu| \rho_*),
\] 
Lemma \ref{lem:origin} implies
\[
\|\psi_\kappa(\rho^\nu)\|_{L^1_{\mathrm{loc}}((0,T)\times\bbr)} \le \nu C_\kappa.
\]  
Thus, $d_2 = \partial_x \psi_\kappa(\rho^\nu)$ vanishes weakly in $W^{-1,1}_{\mathrm{loc}}((0,T)\times\bbr)$ as $\nu\to 0$.\\
Since $\kappa>0$ is arbitrary,  $R_3$ converges to $0$ weakly in $(L^2+W^{-1,1}_{\mathrm{loc}}) ((0,T)\times\bbr)$ as $\nu\to 0$.

\section{Inviscid limit via Compensated compactness}\label{sec:13}
\setcounter{equation}{0}
This section is dedicated to the proof of Proposition \ref{prop:cc}. 
We basically follow the similar argument as in Chen-Perepelitsa paper \cite{CP} to show that solutions to \eqref{1NS} converge to entropy solutions to the Euler system. 
However, since the viscous coefficient is not constant and some part of \cite{CP} need to be improved, we here present the main parts of \cite{CP} with possible variants of the desired estimates.\\
As in the case of  Chen and Perepelitsa in \cite{CP}, the most challenging part of the proof is to obtain the convergence on the flux of entropy and the decay of the entropy dissipation in Section \ref{sub:lim}.
For this part, our proof is different from the one of Chen and Perepelitsa in \cite{CP}. 
We  use crucially our  uniform bounds obtained in Theorem \ref{thm:uniform} as in Section \ref{sub:lim}. This allows us to use an entropy pair $(\eta^M, q^M)$ corresponding to a truncated density function $\psi_M$, and to pass into the limit when $M$ goes to infinity.

\subsection{Uniform estimates for higher integrability of solutions to \eqref{1NS}} 
As in \cite[Section 3]{CP}, in addition to the uniform bounds as in Lemmas \ref{lem:origin} and \ref{lem:derho}, we here present the estimates for higher integrability.

\begin{lemma}\label{lem:high1}
Assume \eqref{inif1} and
\beq\label{asshigh1}
\gamma>1,\quad \alpha=\gamma-1.
\eeq
Then, given $T>0$, for all $ t\in (0,T)$ and any compact set $K\subset\bbr$, there exists $C>0$ (independent of $\nu$) such that 
\[
 \int_0^t \int_K  (\rho^\nu)^{\gamma+1}  dx ds \le C.
\]
\end{lemma}
\begin{proof}
The proof basically follows \cite[Lemma 3.3]{CP}. By \eqref{1NS} and the proof of \cite[Lemma 3.3]{CP}, we find that 
for any $\omega\in C_c^\infty(\bbr)$ with $\omega=1$ on $K$,
\begin{align*}
\begin{aligned}
\int_0^t \int_\bbr \rho^{\gamma+1} \omega^2 dx ds 
&=\underbrace{\nu \int_0^t \int_\bbr \rho \bar\mu(\rho)u_x \omega^2 dx ds}_{=:J_1} \underbrace{- \nu \int_0^t \int_\bbr \rho \omega \bigg( \int^x_{-\infty}  \bar\mu(\rho) u_x  \omega_x \bigg) dx ds}_{=:J_2} \\
&\quad -  \int_\bbr \rho \omega \bigg( \int^x_{-\infty}  \rho u \omega \bigg) dx + \int_\bbr \rho_0 \omega \bigg( \int^x_{-\infty}  \rho_0 u_0 \omega \bigg) dx \\
&\quad +\int_0^t \int_\bbr \rho u \omega_x \bigg( \int^x_{-\infty}  \rho u  \omega \bigg) dx ds + \int_0^t \int_\bbr \rho \omega \bigg( \int^x_{-\infty}  (\rho u^2+ p)  \omega_x \bigg) dx ds.
\end{aligned}
\end{align*}
Since the viscous coefficient $\bar\mu$ is the only difference of \eqref{1NS} and the Navier-Stokes system \cite[(1.1)]{CP},
it is enough to estimate the viscosity terms $J_1, J_2$ as follows.\\
Since it holds from \eqref{assmu} that
\[
C^{-1}(1+\rho^\alpha) \le \bar\mu(\rho) \le C(1+\rho^\alpha),
\]
using Lemma \ref{lem:origin}, we have
\begin{align*}
\begin{aligned}
J_1 &\le \nu \int_0^t \int_\bbr  \bar\mu(\rho) |u_x|^2 dx ds + \nu \int_0^t \int_\bbr \rho^2 \bar\mu(\rho) \omega^4 dx ds\\
&\le C+ C\nu \int_0^t \int_\bbr \big(\rho^2 + \rho^{2+\alpha}\big) \omega^2 dx ds. 
\end{aligned}
\end{align*}
Thus, by \eqref{asshigh1},
\[
J_1\le C(1+T) + C\nu \int_0^t \int_\bbr  \rho^{\gamma+1} \omega^2 dx ds. 
\]
Likewise, since it holds from Lemma \ref{lem:origin} that
\beq\label{locq1}
\|\rho\|_{L^\infty(0,T;L^1_{\mathrm{loc}})} \le CT+ \|\rho\|_{L^\infty(0,T;L^\gamma_{\mathrm{loc}})} \le CT + C\|p(\rho|\rho_*)\|_{L^\infty(0,T;L^1)}  \le C(1+T),
\eeq
we have
\begin{align*}
\begin{aligned}
J_2 &\le \|\rho\omega\|_{L^\infty(0,T;L^1)} \sqrt{ \nu \int_0^t \int_\bbr  \bar\mu(\rho) |u_x|^2 dx ds} \sqrt{ \nu \int_0^t \int_\bbr \bar\mu(\rho) |\omega_x|^2 dx ds} \\
&\le C(1+T)^2 + C\nu \int_0^t \int_{\mbox{supp}(\omega)} \rho^\alpha dx ds \\
&\le C(1+T)^2 + C\nu \int_0^t \int_{\mbox{supp}(\omega)} \rho^\gamma dx ds \le C(1+T)^2.
\end{aligned}
\end{align*}
Hence we have the desired result.
\end{proof}

Notice that the following assumption \eqref{asshigh2} holds by \eqref{initialNS} and \eqref{initialEuler}.

\begin{lemma}\label{lem:high2}
Assume $\gamma>1$ with $\alpha\le\gamma$, and \eqref{inif1}, \eqref{inif2} and there exists $C>0$ (independent of $\nu$) such that 
\beq\label{asshigh2}
\int_\bbr \rho_0^\nu | u_0^\nu- u_*| dx \le C.
\eeq
Then, given $T>0$, for all $ t\in (0,T)$ and any compact set $K\subset\bbr$, there exists $C>0$ (independent of $\nu$) such that 
\[
 \int_0^t \int_K  \big(\rho^\nu |u^\nu|^3 + (\rho^\nu)^{\gamma+\theta} \big)  dx ds \le C,\quad \theta:=\frac{\gamma-1}{2}.
\]
\end{lemma}
\begin{proof}
The proof is the same as in \cite[Lemma 3.4]{CP}, except for the estimates on the viscous terms as follows.  \\
Consider the weak entropy pair $(\eta^\#, q^\#)$ corresponding to the $\psi_\#(s)=\frac{s|s|}{2}$, and consider the entropy pair $(\check\eta,\check q)$ corresponding to the $\psi(s)=\psi_\#(s-u_*)$. Then, as in \cite[(3.25)]{CP}, it holds from \eqref{1NS} that 
\begin{align*}
\begin{aligned}
&\int_{-\infty}^x (\check\eta(\rho,m)-\check\eta(\rho_0,\rho_0 u_0)) dz - t q^\#(\rho_*,0)
+\int_0^t q^\#(\rho,\rho(u-u_*)) - u_* \eta^\# (\rho,\rho(u-u_*)) ds \\
&\quad \underbrace{-\nu\int_0^t \check\eta_m  \bar\mu(\rho) u_x ds}_{=:J_1} +\underbrace{ \nu \int_0^t \int_{-\infty}^x \check\eta_{mu} \bar\mu(\rho) |u_x|^2 dz ds}_{=:J_2}
+\underbrace{ \nu \int_0^t \int_{-\infty}^x \check\eta_{m\rho} \rho_x  \bar\mu(\rho) u_x dz ds}_{=:J_3} =0.
\end{aligned}
\end{align*}
First, using Lemmas \ref{lem:origin}, \ref{lem:derho} with
\[
|\check\eta_{mu}| \le C, \quad |\check\eta_{m\rho}| \le C\rho^{\theta-1} = C \rho^{\frac{\gamma-3}{2}} ,
\]
we estimate $J_2, J_3$ as
\begin{align*}
\begin{aligned}
\left|\nu \int_0^t \int_{-\infty}^x \check\eta_{mu} \bar\mu(\rho) |u_x|^2 dz ds \right| &\le C\nu \int_0^t \int_\bbr \bar\mu(\rho) |u_x|^2 dx ds \le C;\\
\left| \nu \int_0^t \int_{-\infty}^x \check\eta_{m\rho} \rho_x  \bar\mu(\rho) u_x \right| &\le \sqrt{ \nu \int_0^t \int_\bbr \bar\mu(\rho) |u_x|^2} \sqrt{  \nu \int_0^t \int_\bbr \bar\mu(\rho) | \check\eta_{m\rho}|^2 |\rho_x|^2}\\
&\le C \sqrt{ \nu \int_0^t \int_\bbr \bar\mu(\rho) \rho^{\gamma-3} |\rho_x|^2 } \le C.
\end{aligned}
\end{align*}
Then, integrating the above equality over a compact set $K$, it remains to estimate $\int_K J_1 dx$.\\
Using $|\check\eta_m| \le C (|u|+\rho^\theta)$, we have
\begin{align*}
\begin{aligned}
\left|\int_K J_1 dx \right| &\le C \nu\int_0^t \int_K   \bar\mu(\rho) |u_x|^2  + C \nu\int_0^t \int_K  \bar\mu(\rho) (|u|+\rho^\theta)^2\\
&\le C + C \underbrace{\nu\int_0^t \int_K  \bar\mu(\rho) \rho^{2\theta} }_{=:J_{11}} + C \underbrace{\nu\int_0^t \int_K  \bar\mu(\rho) |u|^2}_{=:J_{12}}  .
\end{aligned}
\end{align*}
Since $2\theta=\gamma-1$, $ C^{-1}(1+\rho^\alpha)\le \bar\mu(\rho) \le C(1+\rho^\alpha)$ and $\frac{\alpha+\gamma-1}{2}<\gamma$, we use \eqref{locq1} and the Poincare inequality to have 
\begin{align*}
\begin{aligned}
J_{11} &\le C\nu\int_0^t \int_K \rho^{\gamma-1} + C\nu\int_0^t \int_K \left(\rho^\frac{\alpha+\gamma-1}{2} \right)^2 \\
&\le C(1+T) + C\nu \int_0^t \left(\frac{1}{|K|} \int_K \rho^\frac{\alpha+\gamma-1}{2} \right)^2 + C\nu \int_0^t \int_K |\partial_x \rho^\frac{\alpha+\gamma-1}{2} |^2 \\
&\le C(1+T)^2 +C\nu \int_0^t \int_K \rho^\alpha \rho^{\gamma-3} |\rho_x|^2  \\
&\le C(1+T)^2 +C\nu \int_0^t \int_K \bar\mu(\rho) \rho^{\gamma-3} |\rho_x|^2 \le C(1+T)^2.
\end{aligned}
\end{align*}
To control $J_{12}$, for a given $C_0:=\int_\bbr \eta\big((\rho_0^\nu,u_0^\nu)|(\rho_*,u_*)\big) dx$, retake $K$ if necessary such that $|K|\ge \frac{2C_0}{C_*}$ where $C_*:=\frac{p(\rho_*/2 | \rho_*)}{\gamma-1}$.
Since it holds from Lemma \ref{lem:origin} that 
\[
C_0 \ge \int_K \frac{p(\rho | \rho_*)}{\gamma-1} dx \ge \int_{K\cap\{\rho<\rho_*/2\}} \frac{p(\rho_*/2 | \rho_*)}{\gamma-1} dx,
\]
we have
\[
|K\cap\{\rho<\rho_*/2\}| \le C_0/C_*,
\]
which implies
\[
|A(t)| \ge  C_0/C_*, \quad A(t):=K\cap\{\rho\ge\rho_*/2\}.
\]
Using (as in \eqref{locq1})
\[
\frac{1}{|A(t)|} \int_{A(t)} |u| \le \frac{2C_*}{\rho_*C_0} \int_K \rho|u| \le C \left(|K|+\int \eta((\rho,u)|(\rho_*,u_*)) \right) \le C,
\]
we find that for all $y\in K$,
\[
|u(t,y)| \le \frac{1}{|A(t)|} \int_{A(t)} |u| + \int_K |u_x| \le C+ C\int_K \sqrt{\bar\mu(\rho)} |u_x| \le C+ \int_K \bar\mu(\rho) |u_x|^2.
\]
Then, using the fact that $\alpha\le \gamma$ and so, as in \eqref{locq1},
\[
\|\bar\mu(\rho)\|_{L^\infty(0,T;L^1(K))} \le CT + \|\rho\|_{L^\infty(0,T;L^\gamma(K))} \le C(T+1),
\]
we have
\[
J_{12} \le  \|\bar\mu(\rho)\|_{L^\infty(0,T;L^1(K))}  \int_0^t \nu\|u(t)\|_{L^\infty(K)}^2 \le C(T+1) \left(T+ \nu  \int_0^t\int_K \bar\mu(\rho) |u_x|^2 \right) \le C.
\]
\end{proof}

\subsection{$H^{-1}$-compactness of entropy dissipation measures}
We here present \cite[Proposition 4.1]{CP} on $H^{-1}$-compactness of dissipation measures of weak entropy, which is a key part for the compensated compactness. However, we need an addition assumption on $\gamma, \alpha$ for the desired compactness as follows.

\begin{lemma}\label{lem:h1}
 Assume \eqref{inif1}, \eqref{inif2} and \eqref{asshigh2}. Assume $\alpha\le\gamma$ and also that either \eqref{asshigh1} or $\gamma>5/3$ holds.
Let $\psi:\bbr\to\bbr$ be any compactly supported $C^2$ function, and $(\eta^\psi, q^\psi)$ be a weak entropy pair generated by $\psi$. Then, for the solution $(\rho^\nu,m^\nu)$ with $m^\nu=\rho^\nu u^\nu$ to \eqref{1NS}, the sequence of dissipation measures
\[
\{\partial_t\eta^\psi(\rho^\nu,m^\nu) + \partial_x q^\psi(\rho^\nu,m^\nu) \}_{\nu>0}
\]
is in a compact subset of $H^{-1}_{\mathrm{loc}}(\bbr^2_+)$.
\end{lemma}
\begin{proof}
It holds from \eqref{1NS} that
\[
\partial_t\eta^\psi(\rho^\nu,m^\nu) + \partial_x q^\psi(\rho^\nu,m^\nu)=\nu(\eta^\psi_m \bar\mu u_x)_x - \nu \eta^\psi_{mu} \bar\mu |u_x|^2  -\nu \eta^\psi_{m\rho} \bar\mu \rho_x u_x
\]
First, by the definition of the weak entropy:
\[
\eta^\psi_m(\rho,\rho u) = \rho \int_{-1}^1 \psi(u+\rho^\theta s) (1-s^2)^\lambda ds,\quad \theta:=\frac{\gamma-1}{2},\quad\lambda:=\frac{3-\gamma}{2(\gamma-1)},
\]
we have
\[
|\eta^\psi_m(\rho,\rho u)|\le C,\quad |\eta^\psi_{mu}(\rho,\rho u)|\le C, \quad |\eta^\psi_{m\rho}(\rho,\rho u)|\le C\rho^{\frac{\gamma-3}{2}}.
\]
Then, since
\[
\big|\nu \eta^\psi_m \bar\mu u_x \big| \le C \nu\bar\mu |u_x| \le C\sqrt\nu\sqrt{\nu\bar\mu |u_x|^2}\sqrt{\bar\mu},
\]
using Lemma \ref{lem:origin} and the fact that 
\beq\label{intmu}
\mbox{$\bar\mu(\rho)\le C+C\rho^\gamma$ and so, $\bar\mu^{\frac{\gamma+\theta}{\gamma}}$ is bounded in $L^1_{\mathrm{loc}}(\bbr^2_+)$ uniformly in $\nu$ by Lemma \ref{lem:high2}, } 
\eeq
we find that
\[
\nu ( \eta^\psi_m(\rho^\nu,m^\nu) \bar\mu(\rho^\nu) u^\nu_x )_x
\]
are confined in a compact set of $W^{-1,q_1}_{\mathrm{loc}}(\bbr^2_+)$ for some $1<q_1<2$.\\
Likewise, since
\[
 \nu \eta^\psi_{mu} \bar\mu |u_x|^2 \le C\nu \bar\mu |u_x|^2,
\]
and
\[
\big|\nu \eta^\psi_{m\rho} \bar\mu \rho_x u_x\big| \le C\sqrt{\nu\bar\mu \rho^{\gamma-3} |\rho_x|^2} \sqrt{\nu\bar\mu |u_x|^2},
\]
Lemmas \ref{lem:origin}, \ref{lem:derho} imply that
\[
 \nu \eta^\psi_{mu}(\rho^\nu,m^\nu)\bar\mu(\rho^\nu) |u^\nu_x|^2 \quad\mbox{and}\quad  \nu \eta^\psi_{m\rho}(\rho^\nu,m^\nu) \bar\mu(\rho^\nu) \rho^\nu_x u^\nu_x
\]
are confined in a compact set of $W^{-1,q_1}_{\mathrm{loc}}(\bbr^2_+)$.\\
Thus, 
\beq\label{compactq2}
\partial_t\eta^\psi(\rho^\nu,m^\nu) + \partial_x q^\psi(\rho^\nu,m^\nu)\quad\mbox{are confined in a compact set of $W^{-1,q_1}_{\mathrm{loc}}(\bbr^2_+)$}.
\eeq

On the other hand, observe that (by \cite[Lemma 2.1]{CP})
\[
|\eta^\psi(\rho,\rho u)|\le C\rho,\quad |q^\psi(\rho,\rho u)| \le 
\left\{ \begin{array}{ll}
       C\rho,\quad \gamma\in(1,3],\\
        C\rho+C\rho^{1+\theta},\quad \gamma>3  . \end{array} \right.
\]
Then, if \eqref{asshigh1} holds, it follows from Lemmas \ref{lem:high1}, \ref{lem:high2} that
\beq\label{compactq1}
\eta^\psi(\rho^\nu,m^\nu), \quad q^\psi(\rho^\nu,m^\nu)\quad\mbox{are uniformly bounded in $L^{q_2}_{\mathrm{loc}}(\bbr^2_+)$ for some $q_2>2$},
\eeq
where $q_2=\gamma+1>2$ when $\gamma\in(1,3]$ and $q_2=\frac{\gamma+\theta}{1+\theta}>2$ when $\gamma>3$. \\
If $\gamma>5/3$, \eqref{compactq1} with $q_2=\gamma+\theta>2$ holds by Lemma \ref{lem:high2}.\\
Thus, \eqref{compactq1} implies
\[
\partial_t\eta^\psi(\rho^\nu,m^\nu) + \partial_x q^\psi(\rho^\nu,m^\nu)\quad\mbox{are uniformly bounded in  $W^{-1,q_2}_{\mathrm{loc}}(\bbr^2_+)$}.
\]
This and \eqref{compactq2} with the interpolation compactness (as in \cite{CP}) imply the desired result.
\end{proof}

\subsection{Inviscid limit} \label{sub:lim}
We here complete the proof of Proposition \ref{prop:cc}. 
First, by the method of compensated compactness based on Lemma \ref{lem:h1}, there exists a pair of measurable functions $(\rho, u):\bbr^2_+\to\bbr^2_+$ such that (up to subsequence)
\beq\label{pwconv}
\rho^\nu\to \rho,\quad u^\nu\to u \quad\mbox{a.e. on } \bbr^2_+.
\eeq
This with Lemma \ref{lem:high2} especially implies that
\[
\rho^\nu \to \rho,\quad \rho^\nu u^\nu \to \rho u \quad \mbox{in } L^q_{\mathrm{loc}} (\bbr_+\times\bbr) \quad\mbox{for some } q>1,
\]
and $(\rho, u)$ solves the Euler system in the sense of distributions.
Therefore, it remains to show that the limit $(\rho,u)$ satisfies the entropy inequality, that is,
\[
\partial_t\eta(\rho,\rho u) + \partial_y q(\rho,\rho u) \le 0,\qquad \mbox{in the sense of distributions}.
\]
To this end, consider a cutoff function $\psi_M(s)$ for each integer $M>0$, and the corresponding entropy pair $(\eta^M, q^M)$ as follows:
\[
\psi_M(s)=\int_0^s\int_0^w \psi_M''(\tau)d \tau dw,\qquad  
\psi_M''(s)= \left\{ \begin{array}{ll}
       0\quad \mbox{if } |s|\ge M\\
       1-\frac{|s|}{M} \quad\mbox{if } |s|\le M, \end{array} \right.
\]
and
\begin{align*}
\begin{aligned}
&\eta^M(\rho,\rho u) = \rho \int_{-1}^1 \psi_M(u+\rho^\theta s) (1-s^2)^\lambda ds,\\
&q^M(\rho,\rho u) = \rho \int_{-1}^1 (u+\theta\rho^\theta s )  \psi_M(u+\rho^\theta s) (1-s^2)^\lambda ds.
\end{aligned}
\end{align*}
Then, since $|\psi_M(s)|\le \psi_*(s):=|s|^2/2$ and  the entropy pair $(\eta,q)$ corresponds to $\psi_*$,
\begin{align}
\begin{aligned} \label{uniformm}
&|\eta^M(\rho,\rho u)| \le  \rho \int_{-1}^1 \psi_*(u+\rho^\theta s) (1-s^2)^\lambda ds=\eta(\rho,\rho u),\\
&|q^M(\rho,\rho u)| \le C(|u|+\rho^\theta) \rho \int_{-1}^1 \psi_*(u+\rho^\theta s) (1-s^2)^\lambda ds \le C(\rho|u|^3+\rho^{\theta+\gamma}).
\end{aligned}
\end{align}
Since $\eta(\rho^\nu,\rho^\nu u^\nu)$ and $\rho^\nu|u^\nu|^3+(\rho^\nu)^{\theta+\gamma}$ are uniformly bounded in $L^1_{\mathrm{loc}}$ by Lemmas \ref{lem:origin} and \ref{lem:high2}, it holds from \eqref{pwconv} and \eqref{uniformm} that for a fixed $M$,
\begin{align}
\begin{aligned}\label{temmcon1}
&\eta^M(\rho^\nu,\rho^\nu u^\nu) \to \eta^M(\rho,\rho u) ,\\
&q^M(\rho^\nu,\rho^\nu u^\nu) \to q^M(\rho,\rho u),\quad \mbox{ in $L^1_{\mathrm{loc}}(\bbr^2_+)$ as $\nu\to0$}.
\end{aligned}
\end{align}
Furthermore, using
\[
|\psi_M(s) - \psi_*(s)| \le \int_0^s\int_0^w |\psi_M''(\tau)-1| d \tau dw \le C|s|^3/M,\quad |s|\le M,
\]
we have
\begin{align*}
\begin{aligned}
&| \eta^M(\rho,\rho u)- \eta(\rho,\rho u)| \le C\frac{\rho (|u|+\rho^\theta)^3}{M},\\
& |q^M(\rho,\rho u)- q(\rho,\rho u)| \le C\frac{\rho (|u|+\rho^\theta)^4}{M}, \quad \mbox{for any $\rho, u$ with $|u|+\rho \le M$},
\end{aligned}
\end{align*}
which together with \eqref{uniformm} and $\eta(\rho,\rho u), \rho|u|^3+\rho^{\theta+\gamma} \in L^1_{\mathrm{loc}}$ implies
\beq\label{temmcon2}
\eta^M(\rho,\rho u) \to \eta(\rho,\rho u),\quad q^M(\rho,\rho u)\to q(\rho,\rho u),\quad \mbox{in $L^1_{\mathrm{loc}}(\bbr^2_+)$ as $M\to\infty$ }.
\eeq
We will apply the above convergences to the entropy inequality: (by \eqref{1NS})
\[
\partial_t\eta^M(\rho^\nu,m^\nu) + \partial_x q^M(\rho^\nu,m^\nu)=\nu(\eta^M_m  \bar\mu(\rho^\nu)  u_x)_x - \nu \eta^M_{mu}  \bar\mu(\rho^\nu)  |u_x|^2  -\nu \eta^M_{m\rho}  \bar\mu(\rho^\nu)  \rho_x u_x.
\]
Notice that it holds from $\psi_M''\ge 0$ that
\[
\eta^M_{mu}(\rho,\rho u) =  \int_{-1}^1 \psi_M''(u+\rho^\theta s) (1-s^2)^\lambda ds \ge 0,
\]
we have
\beq\label{lastentin}
\partial_t\eta^M(\rho^\nu,m^\nu) + \partial_x q^M(\rho^\nu,m^\nu) \le \nu(\eta^M_m  \bar\mu(\rho^\nu)  u^\nu_x)_x-\nu \eta^M_{m\rho}  \bar\mu(\rho^\nu)  \rho^\nu_x u^\nu_x.
\eeq
It remains to control the two dissipation terms.\\
Using $|\psi_M'|\le M$ and so $|\eta^M_m| \le CM$, and using also Lemma \ref{lem:origin} and \eqref{intmu}, we have 
\[
\|\nu \eta^M_m  \bar\mu(\rho^\nu)  u^\nu_x\|_{L^1_{\mathrm{loc}}(\bbr^2_+)} \le CM\sqrt\nu \|\sqrt{\nu  \bar\mu(\rho^\nu) } |u^\nu_x|\|_{L^2(\bbr^2_+)} \|  \bar\mu(\rho^\nu)  \|_{L^1_{\mathrm{loc}}(\bbr^2_+)}^{1/2} \le CM\sqrt\nu,
\]
which implies that for a fixed $M$,
\beq\label{nuconv1}
 \nu(\eta^M_m  \bar\mu(\rho^\nu)  u^\nu_x)_x \to 0\quad \mbox{ in $\mathcal{D}'$ as $\nu\to0$}. 
\eeq
To estimate the last dissipation term, we first use Lemma  \ref{lem:origin} to have
\begin{align}
\begin{aligned} \label{lacon1}
\|\nu \eta^M_{m\rho}  \bar\mu(\rho^\nu)  \rho^\nu_x u^\nu_x\|_{L^1_{\mathrm{loc}}(\bbr^2_+)} 
&\le \sqrt{\nu \int_0^t \int_\bbr  \bar\mu(\rho^\nu)  |u^\nu_x|^2}  \sqrt{\nu \int_0^t \int_\bbr  \bar\mu(\rho^\nu) |\eta^M_{m\rho}(\rho^\nu,m^\nu) |^2 |\rho^\nu_x|^2} \\
&\le  \sqrt{\nu \int_0^t \int_\bbr  \bar\mu(\rho^\nu) |\eta^M_{m\rho}(\rho^\nu,m^\nu) |^2 |\rho^\nu_x|^2}.
\end{aligned}
\end{align}
To make the last term vanish at the limit, we will use Lemma \ref{lem:derho} and Theorem \ref{thm:uniform}, due to the observation:
\[
\eta^M_{m\rho}(\rho,\rho u) = \theta\rho^{\theta-1} \int_{-1}^1 \psi_M''(u+\rho^\theta s)  s(1-s^2)^\lambda ds.
\]
Indeed, since $\int_{-1}^1 s(1-s^2)^\lambda ds=0$ and so
\[
|\eta^M_{m\rho}(\rho,\rho u)| \mathbf{1}_{\rho\in (\rho_*/2, 2\rho_*)} \le  \theta\rho^{\theta-1} \int_{-1}^1 \big|\psi_M''(u+\rho^\theta s) -\psi_M''(u)\big|\mathbf{1}_{\rho\in (\rho_*/2, 2\rho_*)} (1-s^2)^\lambda ds \le \frac{C}{M}\rho^{\theta-1} ,
\]
it holds from Lemma   \ref{lem:derho}  that
\beq\label{lacon2}
\nu \int_0^t \int_\bbr  \bar\mu(\rho^\nu) |\eta^M_{m\rho}(\rho^\nu,m^\nu) |^2 \mathbf{1}_{\rho^\nu\in (\rho_*/2, 2\rho_*)}  |\rho^\nu_x|^2
\le \frac{C\nu}{M} \int_0^t \int_\bbr  \bar\mu(\rho^\nu)  (\rho^\nu)^{\gamma-3} |\rho^\nu_x|^2 \le  \frac{C}{M}.
\eeq
On the other hand, using $|\eta^M_{m\rho}(\rho,\rho u)|\le C\rho^{\theta-1}$ and the change of coordinates, we have
\begin{align*}
\begin{aligned}
&\nu \int_0^t \int_\bbr  \bar\mu(\rho^\nu) |\eta^M_{m\rho}(\rho^\nu,m^\nu) |^2 |\rho^\nu_x|^2\mathbf{1}_{\rho^\nu\notin (\rho_*/2, 2\rho_*)} dx d\tau\\
& \le \nu  \int_0^t \int_\bbr  \bar\mu(v^\nu) (v^\nu)^\gamma |p(v^\nu)_y|^2 \mathbf{1}_{v^\nu\notin (v_*/2, 2v_*)} dy d\tau \\
&\le C \nu  \int_0^t \int_\bbr  \mu_2(v^\nu) (v^\nu)^\gamma |p(v^\nu)_y|^2  + C\nu  \int_0^t \int_\bbr \mu_1(v^\nu) (v^\nu)^{\gamma} |(p(v^\nu)-p(\bar v_{\nu,\delta}))_y|^2  \\
&\quad + C \nu  \int_0^t \int_\bbr  \mu_1(v^\nu) (v^\nu)^{\gamma} |p(\bar v_{\nu,\delta})_y|^2 \mathbf{1}_{v^\nu\notin (v_*/2, 2v_*)}  ,
\end{aligned}
\end{align*}
where $\bar v_{\nu,\delta}$ denotes the approximation as in Theorem \ref{thm:uniform} .\\
Moreover, since $\mu_1(v^\nu) (v^\nu)^{\gamma} \le C (v^\nu)^{\gamma-\alpha} \le C v^\nu$ by $\gamma-\alpha\in [0,1]$, and so  
\[
\mu_1(v^\nu) (v^\nu)^{\gamma} \mathbf{1}_{v^\nu\notin (v_*/2, 2v_*)} \le Q(v^\nu| \bar v_{\nu,\delta}) \le C p(v^\nu| \bar v_{\nu,\delta}),
\]
using  \eqref{apprr} on the definition of $\bar v_{\nu,\delta}$, with Lemmas \ref{lem:waves} and \ref{lem:num}, we have
\begin{align*}
\begin{aligned}
& \nu  \int_0^t \int_\bbr  \mu_1(v^\nu) (v^\nu)^{\gamma} |p(\bar v_{\nu,\delta})_y|^2 \mathbf{1}_{v^\nu\notin (v_*/2, 2v_*)} dyd\tau \\
& \le C \nu  \int_0^t \int_\bbr \bigg[\bigg( \sum_{i\in \mathcal{S}} |(\tilde v_{i}^\nu)_y|^2 + \sum_{i\in \mathcal{R}} |(v^{R,\nu}_{i})_y|^2 \bigg)  p(v^\nu| \bar v_{\nu,\delta}) +\sum_{i\in \mathcal{NP}} |(v^{P,\nu}_{i})_y|^2  Q(v^\nu| \bar v_{\nu,\delta})  \bigg] \\
&\le C  \int_0^t \int_\bbr \bigg( \sum_{i\in \mathcal{S}} |(\tilde v_{i}^\nu)_y| + \sum_{i\in \mathcal{R}} |(v^{R,\nu}_{i})_y| \bigg)  p(v^\nu| \bar v_{\nu,\delta}) 
+C\frac{\nu^{1/6}}{\delta^2}  \int_0^t \int_\bbr   Q(v^\nu| \bar v_{\nu,\delta}).
\end{aligned} 
\end{align*}
Thus, it holds from the result \eqref{conthm} of Theorem \ref{thm:uniform} that for a given $T>0$,
\begin{align*}
\begin{aligned}
&\nu \int_0^t \int_\bbr  \bar\mu(\rho^\nu) |\eta^M_{m\rho}(\rho^\nu,m^\nu) |^2 |\rho^\nu_x|^2\mathbf{1}_{\rho^\nu\notin (\rho_*/2, 2\rho_*)} dx d\tau \\
&\le \int_0^T G_\nu(t) dt + CT \frac{\nu^{1/6}}{\delta^2}  F_\nu(t) \to 0 \quad\mbox{as } \nu\to 0.
\end{aligned}
\end{align*}
This together with \eqref{lacon1} and \eqref{lacon2} implies
\[
\|\nu \eta^M_{m\rho}  \bar\mu(\rho^\nu)  \rho^\nu_x u^\nu_x\|_{L^1_{\mathrm{loc}}(\bbr^2_+)} \to 0 \quad\mbox{as $\nu\to 0$}.
\]
Hence, using the above convergence and \eqref{temmcon1}, \eqref{temmcon2}, \eqref{lastentin}, \eqref{nuconv1} and passing to limit as $\nu\to0$, and then $ M\to \infty$,  we have the desired result.

\appendix
\section{A short introduction to wave interactions in p-system} \label{app:algorithm}
To introduce wave interactions for p-system, we follow \cite{ChenJenssen} by Chen and Jenssen, with a little extension to give some estimates needed for our modified front tracking scheme.

\subsection{Wave curves}
The characteristic speeds of \eqref{eq:Euler} are $\lambda_1$ and $\lambda_2$. The Riemann invariants we use are
\beq\label{s_r_def}
	r:=h-z\qquad \text{and}\qquad s:=h+z, \qquad z=\frac{2\sqrt{\gamma}}{\gamma-1}v^{-\frac{\gamma-1}{2}}.
\eeq

We next give the parametrizations in the $({z},h)$-plane of the backward and forward 
wave curves of the ``first type.'' That is, given a base point $(\bar {z},\bar h)$, we consider
the curves of points $({z},h)$ such that the Riemann problem with left state $(\bar {z},\bar h)$
and right state $({z},h)$ yields a single backward or forward wave (rarefaction or entropic 
shock). Letting $b$ and $f$ denote the ${z}$-ratios ${z}_\text{right}/{z}_\text{left}$ across 
backward and forward waves, respectively, the parametrizations are given by:
\begin{align}
	&\text{Backward waves:}\qquad
	{\ba{V}}(b;\bar {z},\bar h) = \left(\!\!\begin{array}{c}
		b\bar {z} \\
		\bar h -{\ba{\phi}}(b)\bar {z}\\
	\end{array} \!\!\right)
	\qquad\left\{\begin{array}{ll}
	\ba R:\quad 0<b<1 \\
	\ba S:\quad b>1
	\end{array}\right.\label{bkwd_wave}\\
	&\text{Forward waves:}\,\,\,\qquad
	{\fa{V}}(f;\bar {z},\bar h) = \left(\!\!\begin{array}{c}
		f\bar {z} \\
		\bar h +{\fa{\phi}}(f)\bar {z}\\
	\end{array} \!\!\right)
	\qquad\left\{\begin{array}{ll}
	\fa R:\quad f>1 \\
	\fa S:\quad 0<f<1,
	\end{array}\right.\label{frwd_wave}
\end{align}
where the auxiliary functions $\ba\phi$ and $\fa\phi$ are given by
\[{\ba{\phi}}(x) = \left\{\begin{array}{ll}
	 (x-1) \quad & 0<x<1\\\\
	\frac{\gamma-1}{2\sqrt{\gamma}}\sqrt{(1-x^{-\frac{1}{\alpha}})(x^{\frac{\gamma}{\alpha}}-1)}\quad & x>1
\end{array}\right\}\] 
\[{\fa{\phi}}(x) = \left\{\begin{array}{ll}
	-\frac{\gamma-1}{2\sqrt{\gamma}}\sqrt{(1-x^{-\frac{1}{\alpha}})(x^{\frac{\gamma}{\alpha}}-1)} \quad & 0<x<1\\\\
	  (x-1) \quad  & x>1.
\end{array}\right\}\] 
Note that the auxiliary functions $\ba\phi$ and $\fa\phi$ are both strictly increasing and satisfy
\beq\label{phi_reln}
	\fa \phi(x)=-x{\ba{\phi}}\big(\textstyle\frac{1}{x}\big)\,.
\eeq
A calculation shows that the map
$x\mapsto \sqrt{(1-x^{-\frac{1}{\alpha}})(x^{\frac{\gamma}{\alpha}}-1)}$
is convex up for $x>1$. It follows that $\ba\phi$ is convex up, while
$\fa\phi$ is convex down. Furthermore, it is immediate to verify that 
\beq\label{phi_prime_infty}
	\lim_{x\to+\infty}\ba\phi '(x)=+\infty\qquad\text{and}\qquad 
	\lim_{x\to0^+}\fa\phi(x)=-\infty.
\eeq
The wave curves in the $(r,s)$-plane are illustrated in Figure \ref{riem_coords}. Note that the no-vacuum
region $\{\,v>0\,\}$ corresponds to the half-plane $\{\,s>r\,\}$ in $(r,s)$-coordinates.

By the definition of $\ba \phi$ and $\fa \phi$, it is easy to show that  there exist $K_1$ and $K_2$, such that
	\beq\label{si-p}\left.\begin{array}{ll}
		0\leq K_1(p_R-p_L)\leq-\sigma=r_L-r_R\leq K_2(p_R-p_L) & \hbox{for $\ba S$},\\
		0\leq K_1(p_L -p_R)\leq-\sigma=s_L-s_R\leq K_2(p_L -p_R)  &\hbox{for $\fa S$},
		\end{array}\right.
	\eeq 
	where the subscriptions $R$ and $L$ denote the right and left states, respectively. 
%
\begin{figure}\label{riem_coords}
	\centering
	\includegraphics[scale=.3]{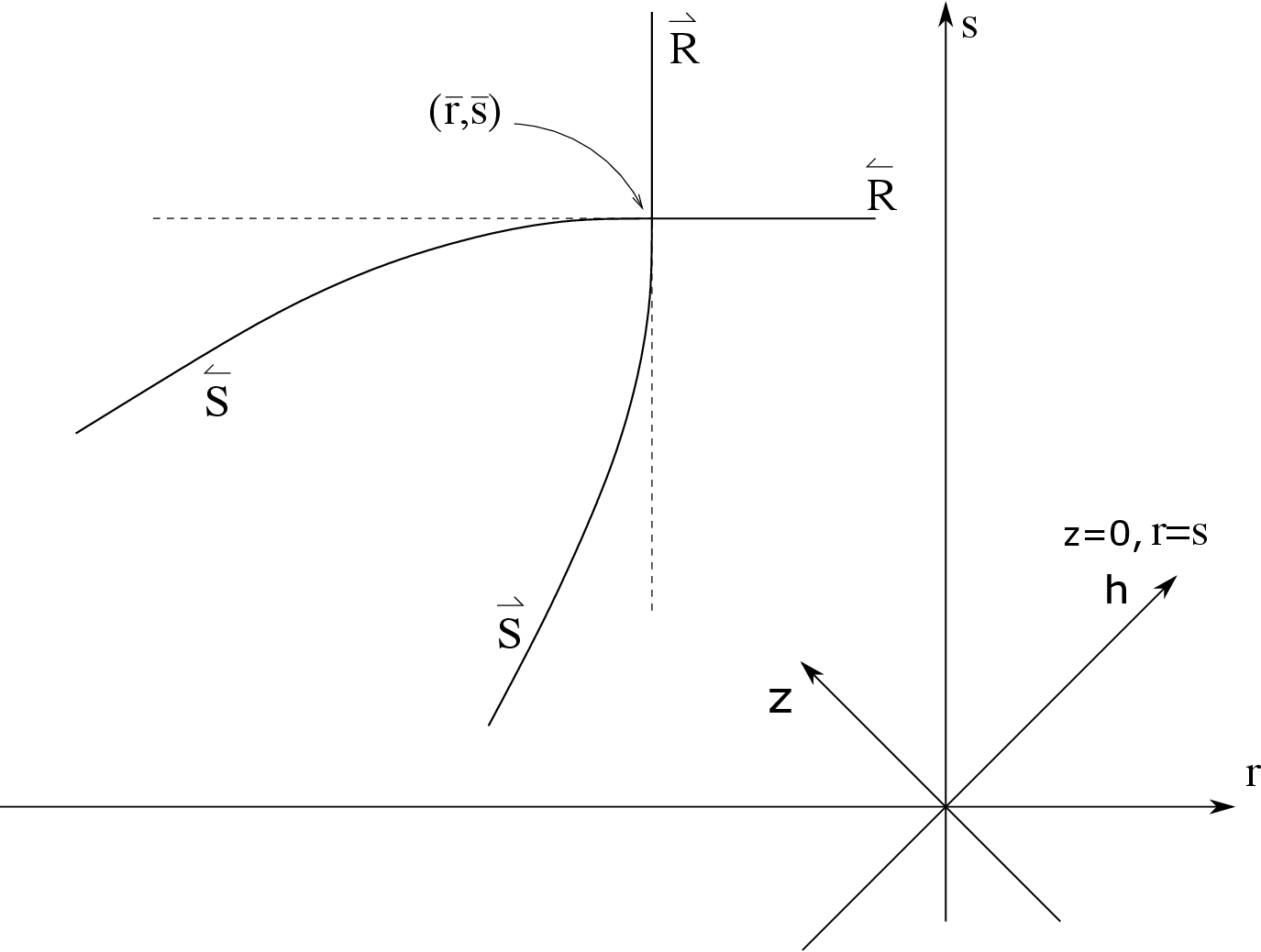}
	\caption{The wave curves in the Riemann coordinate plane. }
\end{figure}

\subsection{Riemann problems and vacuum criterion}
Consider the Riemann problem with left state $(\bar{z},\bar h)$ and right state $({z},h)$,
and let $b$ and $f$ denote the ${z}$-ratios of the resulting backward and forward waves,
respectively. A short calculation shows that $b$ is given as the root of
\beq\label{b_eqn}
	\bar{z}\ba\phi(b)+{z}\ba\phi\left(\textstyle\frac{b\bar{z}}{{z}}\right)=\bar h- h,
\eeq
and $f=\frac{z}{b\bar{z}}$. As the left-hand side of \eqref{b_eqn} is increasing with respect to $b$, it follows that
the Riemann problem $\big((\bar{z},\bar h), \, ({z},h)\big)$ has a unique solution without vacuum provided
\beq\label{no_vac}
	h-\bar h< {z}+\bar{z}\qquad\qquad\text{(no vacuum)}.
\eeq
When the initial data are away from the vacuum, and for our small data problem with sufficiently small $\eps$, \eqref{no_vac} is always satisfied.

\subsection{Pairwise interactions in isentropic flow}\label{interactions}
There are six essentially distinct types of pairwise wave interactions:
\beq\label{gr1}
	\left.
	\begin{array}{ll}
		\mbox{Ia}: & {\fa{R}}{\ba{R}}\\
		\mbox{Ib}: & {\fa{R}}{\ba{S}}\\
		\mbox{Ic}: & {\fa{S}}{\ba{S}}
	\end{array}	
	\qquad\right\}\qquad \mbox{head-on interactions}
\eeq
\beq\label{gr2}
	\left.
	\,\,\,\begin{array}{ll}
		\mbox{IIa}: & {\ba{S}}{\ba{S}}\\
		\mbox{IIb}: & {\ba{S}}{\ba{R}}\\
		\mbox{IIc}: & {\ba{R}}{\ba{S}}
	\end{array}	
	\qquad\right\}\qquad \mbox{overtaking interactions.}
\eeq
The head-on interaction 
\[\mbox{Ib'}: \fa{S}\ba{R},\]
and the overtaking interactions 
\beq\label{gr2'}
	\begin{array}{ll}
		\mbox{IIa'}: & {\fa{S}}{\fa{S}}\\
		\mbox{IIb'}: & {\fa{R}}{\fa{S}}\\
		\mbox{IIc'}: & {\fa{S}}{\fa{R}},
	\end{array}	
\eeq
are qualitatively the same as those in Ib and IIa-IIc, respectively. 
Finally, for cases IIb, IIc, IIb', and IIc', there are two possible outcomes with always reflected shock depending 
on the relative strengths of the incoming waves.

\subsection{The shock curve}
Let\rq{}s only discuss $\ba S$. The analysis on the $\fa S$ is symmetric. Set
\[
\eta=|\Delta r|,\qquad \beta=|\Delta s| 
\]
along $\ba S$ from left state $(\bar h,\bar z)$ to the right state $(h,z)$. Let $G(x)$ denotes the unique $d$ root of 
\[
\ba \phi(d)+(d-1)=x,
\]
then it is easy to show that 
\beq\label{betaeta1}
\beta=\bar {z} H(b):=\bar {z} (b+2(1-G(b)))
\eeq
with $b=\frac{\eta}{\bar{z}}$. It is easy to check that $H(b)= k b^3 +o(b^3)$ for some $k>0$ near $b=0$. On the other hand, if we set 
\[
\beta=|\Delta s|,\qquad \eta=|\Delta r|,
\]
then
\beq\label{betaeta2}
\eta=\bar {z} K(\frac{\beta}{\bar {z}}),\quad \hbox{with}\  K(b)= k b^3 +o(b^3)
\eeq
 for some $k>0$ near $b=0$.

\subsection{Some useful results}

We first prove \eqref{key_est}. In fact,
$z_L>z_a$ and $0<s_L-s_{a\rq{}}<s_a-s_R$, so 
$$r_L-r_{a\rq{}}=z_L K(\frac{s_L-s_{a\rq{}}}{z_L})<z_a K(\frac{s_a-s_R}{z_L})=r_a-r_R$$ by the monotonicity of $K$ in \eqref{betaeta2}, when $\eps$ is sufficiently small. Then clearly $r_R-r_{a\rq{}}<r_a-r_L$.
\vskip0.3cm

Using Taylor expansion, it is not hard to prove Case 2 and 3 in Lemma \ref{lemma_key3}. One can also find the proof in \cite{BCZ1,BCZ2}, or check other references as \cite{Young,sm,ChenJenssen}. 

\vskip0.3cm
Finally, we prove \eqref{qdec} , \eqref{bbyfr1} and \eqref{app_est}.

\begin{proof}

For accurate solvers, Lemma \ref{prop1} still holds for the new strength $\bar \sigma$, see \cite{Bressan, BCZ2,Glimm}, so \eqref{qdec} and \eqref{bbyfr1} still hold after we change $L$, $Q$ and $\sigma$ to
$\bar L$, $\bar Q$ and $\bar \sigma$. When we compare the adjusted solver and the corresponding accurate solver, there is at most a quadratic difference on each outgoing wave strength, so \eqref{qdec} and \eqref{bbyfr1} hold for all adjusted solvers, too.
 \vskip0.1cm
 
Then we consider the simplified solvers. We will only prove \eqref{qdec} and \eqref{bbyfr1} for the interactions in Figure \ref{Simp1}-\ref{Simp_non}. Other symmetric interactions can be treated similarly.
\vskip0.1cm

For the $\ba S\ba S$ interaction in Figure \ref{Simp1}, since the only change in the outgoing wave is on the wave speed of the right outgoing wave:  change from a $\fa R$ to non-physical shock, while all wave strengths are preserved. Then \eqref{qdec} and \eqref{bbyfr1} hold by Lemma \ref{prop1}.
\vskip0.1cm

For interactions in Figure \ref{Simp3} and \ref{Simp5}, the outgoing pressure function is monotonic in $x$, i.e. decays in $x$. So $\bar L$ decays. On the other hand, since the sum of strengths for outgoing $\ba R$ and non-physical shock is less than the strength of incoming $\ba R$, $Q$ decays. For the interaction in Figure \ref{Simp2}, $\bar L$ and $\bar Q$ also decays.
So it is easy to see that for all these three interactions, \eqref{qdec} and \eqref{bbyfr1} hold.
\vskip0.1cm

For the interaction in Figure \ref{Simp4}, we only have to prove \eqref{app_est} when $\kappa$ is sufficiently large enough. 

In fact, to show the first inequality 
$-\bar \sigma\rq{}+\bar \sigma_2\leq K_1|\bar \sigma_1\bar \sigma_2|$, in \eqref{app_est}, we only have to prove that
\beq\label{app_1}
(z_c-z_L)-(z_R-z_a)\leq K_1|\bar \sigma_1\bar \sigma_2|,
\eeq
for some constant $K_1$, because $z_c-z_L\geq z_b-z_L$ and $p$ is increasing on $z$. Since $z=\frac{1}{2}(s-r)$, we only have to show the corresponding variation on $-s$ and $r$ are bounded by $O(|\bar \sigma_1\bar \sigma_2|)$.

Recall, on Figure \ref{Simp4}, the state $b$ is between $a\rq{}$ and $c$, where the shock $L$-$c$ has the same strength of the incoming shock $a$-$R$, i.e. 
\beq\label{key_last1}
r_L-r_c=r_a-r_R\geq \rho_\nu.
\eeq 
So we know
\[-\sigma_2=r_R-r_a=r_c-r_L.\] By Taylor expansion, it is easy to show that when $\eps$ is small enough, the variation on $-s$ is 
\beq\label{app_2}
-s_c+s_L-(-s_R+s_a)=
z_L H(\frac{\sigma_2}{z_L})-z_a H(\frac{\sigma_2}{z_a})\leq K_0|\sigma_1\sigma_2|\leq K_1|\bar \sigma_1\bar \sigma_2|
\eeq
for some $K_0,K_1>0$, where we also use the fact that the incoming shock dominates the incoming rarefaction, $|z_L-z_a|$ is equivalent to $|\sigma_1|$ and \eqref{si-p}. So we prove \eqref{app_1}. And we also prove the second inequality $\bar \sigma\rq{}\rq{}-\bar \sigma_1\leq K_1|\bar \sigma_1\bar \sigma_2|$ in \eqref{app_est}, since $-\bar \sigma\rq{}+\bar \sigma_2= \bar \sigma\rq{}\rq{}-\bar \sigma_1.$
\vskip0.1cm

Using the similar method we can prove \eqref{qdec} and \eqref{bbyfr1}  for interactions between a physical wave and a non-physical shock, where in fact the growth of $\bar L$ is at most in the order of $|\bar \sigma_1\bar \sigma_2|$ which is smaller than the decay of $\kappa \bar Q$, when $\kappa$ is sufficiently large.
\vskip0.1cm

So we complete the proof of \eqref{qdec} , \eqref{bbyfr1} and \eqref{app_est}.
\end{proof}

\section{Adjustment} \label{app:mod}
For some interactions, there are more than two outgoing waves. So we need to shift the position of waves to give enough room for additional outgoing waves.
We will show the $L^2$ error caused by these shifts approaches zero as $\nu\to 0$.
\bigskip

\paragraph{\bf Case 1.}
When the accurate shock-shock overtaking interaction produces more than two reflected rarefactions, we need more rooms to put more than two outgoing waves and leave them at least $d_j$ apart. In this case, after stopping clock, we need to shift all waves.


Since the total number of additional outgoing rarefaction wave in all time is bounded by $O(1/\delta)$.
Correspondingly, we totally order up to $O(1/\delta)$ shifts with width $d_j$ for some $j\geq 1$, and each order impacts all waves at that time.\\
By Lemma \ref{lem:num}, the total number of physical waves is bounded by 
\[
C\Big(\frac{1}{\delta}\ln \rho_\nu^{-1}\Big),
\]
so clearly, at each time, there are at most this many physical waves.
On the other hand, the total number of non-physcial waves is bounded by 
\[
C\Big(\frac{1}{\delta}\ln \rho_\nu^{-1}\Big)^2.
\]
Note, if we perturb a single physical wave horizontally by $d_j$, then the $L^2$ error is bounded by $\eps \sqrt d_j$,
since the wave strength is always less than $\eps$.\\
If we perturb a single non-physical wave horizontally by $d_j$, then the $L^2$ error is bounded by
$C\rn \sqrt d_j$,
since the wave strength of non-physical wave is always less than $C\rn$.

Therefore, the total $L^2$ error caused by the shifts in this case is bounded by 
\[
\frac{C}{\delta}\cdot \Big((\frac{1}{\delta}\ln \rho_\nu^{-1}) \cdot \eps \sqrt d_j +
\big(\frac{1}{\delta}\ln \rho_\nu^{-1} \big)^2\rn \sqrt d_j\Big).
\]

Here to make rooms for additional rarefaction waves, we first  shift waves right after interactions, then shift waves right before interactions if needed, and if necessary, trace further back along time, to keep the relative position for waves and configuration of front tracking solutions, i.e. we do not change the order of wave interactions. In this case, some wave may interact with a wave which is not the closest one to interact with.

More precisely, we describe more details on how to shift. See Figure \ref{fig_last}. Suppose there are $n$ rarefaction fronts reflected from shock-shock overtaking interaction, at time $t_j$. with $n>1$. Without loss of generality, assume the reflected rarefactions are 1-rarefactions. At $t=t_j^+$, if there is a 2-wave or a non-physical wave to the left and within $2d_j$ distance of the reflected rarefactions (for example, the left red wave in Figure \ref{fig_last}), then there are no enough room to keep reflected rarefactions $d_j$ away from each other, then avoid them to be too close to the approaching wave to the left. So we need to shift the outgoing waves of interaction, except the left-most reflected 1-rarefaction. These waves will be shifted up to $(n-1)d_j$ distance horizontally to the right of its original position. If the new position(s) of some shifted wave(s) is to the right of other waves  at $t=t_j^+$, we need to also move those waves to right, up to  $(n-1)d_j$.


Next, we check the position of waves when $t\in[t_{j-1},t_j)$. To keep the relative positions for waves in the way of Section 11.1, which guarantees that $a(t,x)$ decays in time, we may need to shift waves in this time interval by up to  $(n-1)d_j$. Then we trace back to even earlier time intervals if needed.

Note any wave not involved in interaction is considered as one wave, even if it exists for multiple connected time intervals $t\in\sum_{k=n_1}^{n_2}[t_{k-1},t_k).$ If we need to shift such a wave in any time interval, we will shift the wave in all time intervals where it exists, to keep the relative position of waves. 

\begin{figure}
	\centering
	\includegraphics[scale=1]{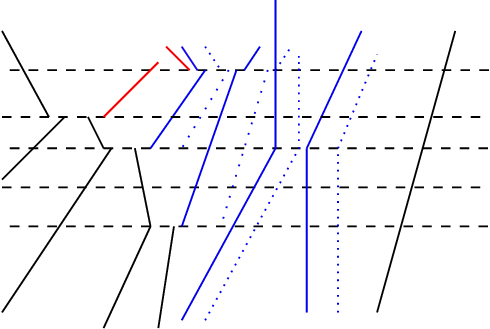}
	\caption{The red 2-wave is too close to the reflected 1-rarefaction fan (including two wave fronts one red and one blue), so we shift the right 1-rarefaction front and other blue solid waves to blue dot waves. After shifts, we keep the configuration of front tracking solution and relative position of waves. All shifted waves are shifted horizontally with distance in the order of $d_j$ for some $j$.\label{fig_last}}
\end{figure} 

To guarantee the inequality
\[
d_{j+1}\ge  \rn^{1/4}  - C\left(\frac{1}{\delta} \log \frac{1}{\rho_\nu}\right)^3 \rn^{1/3} \gg 2\sqrt\rn,
\]
in Remark \ref{rem:djtj}, we need
\[
\delta\gg \rn^{\frac{1}{36}} = \nu^{\frac{1}{108}}.
\]
That is verified by the choice $\delta=\nu^{\frac{1}{240}}$ in Section 11.3. 
Then the total $L^2$ error satisfies
\[
\frac{C}{\delta}\cdot \Big((\frac{1}{\delta}\ln \rho_\nu^{-1}) \cdot \epsilon \sqrt d_j +
\big(\frac{1}{\delta}\ln \rho_\nu^{-1}  \big)^2\rn \sqrt d_j\Big) \ll  \rn^{\frac{1}{16}}\rightarrow 0
\]
as $\nu\rightarrow 0$. The $L^1$ error, if needed, is even smaller.
\bigskip

\paragraph{\bf Case 2.} Another case is when there are more than two outgoing waves produced is the adjusted solver, where an additional non-physical wave is produced.
After stopping clock at this interaction, we shift each wave by $d_j$ for some $j\geq 1$, to give rooms for the additional outgoing wave.\\
By Lemma \ref{lem:num}, the total number of physical wave is bounded by 
\[
C\Big(\frac{1}{\delta}\ln \rho_\nu^{-1}\Big).
\]
Since two incoming waves for adjusted solver are both physical waves, the total number of adjusted solver is at most in the order of
\[
C\Big(\frac{1}{\delta}\ln \rho_\nu^{-1}\Big)^2.\]
So the total $L^2$ error caused by the shifts in this case is bounded by 
\[
C\Big(\frac{1}{\delta}\ln \rho_\nu^{-1}\Big)^2\cdot 
 \Big((\frac{1}{\delta}\ln \rho_\nu^{-1}) \cdot \epsilon \sqrt d_j +
\big(\frac{1}{\delta}\ln \rho_\nu^{-1}  \big)^2\rn \sqrt d_j\Big) \ll  \rn^{\frac{1}{24}-b}\rightarrow 0
\]
as $\nu\rightarrow 0$, for any $b>0$. 

The way of shifting is similar to Case 1, so we omit the details.

\bibliography{Euler_NS_IVP_final}
\end{document}